\documentclass[11pt,a4pahap,twoside]{article}
\usepackage[french,english]{babel}%pour anglais prioritaire
\usepackage{amsfonts}
\usepackage{amsmath,amssymb,amscd}%% avec amscd et l'environnement CD, on peut faire des diagrammes.
\usepackage{mathrsfs}  %???
\usepackage{amsthm}%%
\usepackage{a4wide}   % activer pour une page agrandie
\usepackage[T1]{fontenc} % % accents dans le DVI
\usepackage{newlfont} %??
\usepackage{color}
\usepackage{stackrel}
\usepackage{ulem}
\usepackage{comment}
\numberwithin{equation}{subsection} 
\usepackage{imakeidx}
\usepackage[left,modulo]{lineno}

\newcommand{\cache}[1]{}

\usepackage[%dvipdfm,colorlinks,
backref=page,bookmarksopen=true]{hyperref} % pour faire des r\'ef\'erences dynamiques et indiquer o\`u est cit\'e chaque \'el\'ement de la biblio

\usepackage{graphics, graphicx, epsfig}

\usepackage[all]{xy}  % pour dessiner des beaux diagrammes

\usepackage[applemac]{inputenc}       %% Activer pour utiliser la codification Mac OS Roman

\newtheorem{theo}{Theorem}[section]%%% num\'erotage par section
\def\choixcompteur{theo}
\newtheorem{prop}[\choixcompteur]{Proposition}
\newtheorem{lemm}[\choixcompteur]{Lemma}
\newtheorem{coro}[\choixcompteur]{Corollary}

\theoremstyle{definition}

\newtheorem{rema}[\choixcompteur]{Remark}
\newtheorem{remas}[\choixcompteur]{Remarks}

\newtheorem*{exem*}{Example}
\newtheorem*{exems*}{Examples}
\newtheorem*{exam*}{Example}
\newtheorem*{exams*}{Examples}
\newtheorem*{rema*}{Remark}
\newtheorem*{remas*}{Remarks}
\newtheorem*{NB}{N.B}

\theoremstyle{definition}
\newtheorem*{defi*}{Definition}
\newtheorem*{defiprop*}{Definition-Proposition}

\theoremstyle{plain}
\newtheorem*{prop*}{Proposition}
\newtheorem*{lemm*}{Lemma}
\newtheorem*{coro*}{Corollary}
\newtheorem*{theo*}{Theorem}
\newtheorem*{conj*}{Conjecture}

 \def\cdr@enoncedef{%
 \newenvironment{enonce*}[2][plain]%
 {\let\cdrenonce\relax \theoremstyle{##1}%
 \newtheorem*{cdrenonce}{##2}%
 \begin{cdrenonce}}%
 {\end{cdrenonce}}   }%

 \AtBeginDocument{%
    \cdr@enoncedef}

\def\cf{{\it cf.\/}\ }
\def\ie{{\it i.e.\/}\ }
\def\eg{{\it e.g.\/}\ }
\def\etc{{\it etc.\/}\ }
\def\lc{{\it l.c.\/}\ }

\def\resp{{\it resp.,\/}\ }

\def\truc{\unskip\kern 3pt\penalty 500
\hbox{\vrule\vbox to 5pt{\hrule width 4pt\vfill\hrule}\vrule}\kern 3pt}

\def\qed{\nobreak\hfill $\truc$\par\goodbreak}

\def\into{\hookrightarrow}

\def\vect{\overrightarrow}
\def\un{\underline}
\def\ov{\overline}
\newcommand{\set}[1]{\{ #1\} } %  \set0 donne {0}

\def\parni{\par\noindent}

\def\eds{ editors}
\def\N{{\mathbb N}}    %lettres grasses sp\'eciales: \H, \L, \O, \P et \S d\'ej\`a pris
\def\Z{{\mathbb Z}}
\def\Q{{\mathbb Q}}
\def\R{{\mathbb R}}
\def\C{{\mathbb C}}

\def\F{{\mathbb F}}
\def\A{{\mathbb A}}

\def\PP{{\mathbb P}}
\def\k{{\Bbbk}}

\newcommand{\g}[1]{\mathfrak{#1}} %lettres gothiques

\def\qa{\alpha}     % lettres grecques
\def\qb{\beta}
\def\qd{\delta}
\def\qe{\varepsilon}
\def\qf{\varphi}
\def\qg{\gamma}
\def\qi{\iota}
 
 \def\ql{\lambda}
\def\qm{\mu}
\def\qn{\nu}
\def\qo{\omega}
\def\qp{\varpi}
\def\qr{\rho}
\def\qs {\sigma}
\def\qt{\tau}
 \def\qth{\theta}
 
\def\qx{\xi}
 \def\qz{\zeta}

\def\QD{\Delta}
\def\QF{\Phi}
\def\QG{\Gamma}

\def\QO{\Omega}

\def\sha{{\mathcal A}}   %lettres calligraphiques
\def\shb{{\mathcal B}}
\def\shc{{\mathcal C}}

\def\she{{\mathcal E}}

\def\shi{{\SHI}}

\def\shk{{\mathcal K}}

\def\sho{{\mathcal O}}

\def\shr{{\mathcal R}}
\def\shs{{\mathcal S}}
\def\sht{{\mathcal T}}

\def\SHC{{\mathscr C}}

\def\SHI{{\mathscr I}}

\def\SHT{{\mathscr T}}

\def\zm{^{\mathrm{m}}}

\def\zu{^{\mathrm{u}}}
\def\zv{^{\mathrm{v}}}

\def\zw{^{\mathrm{w}}}

\def\zzv{\mathrm{v}}
\def\zzw{\mathrm{w}}

\def\pr{\mathrm{proj}}
\def\loo{\mathrm{loop}}

\newenvironment{psmallmatrix}
  {\left(\begin{smallmatrix}}
  {\end{smallmatrix}\right)}

\makeindex[columns=3]

\hypersetup{
pdfauthor={Bardy-Panse, H\'{e}bert, Rousseau},
pdftitle={Twin masures},
}

\begin{document}

%\title[titre abrege]{Spherical Hecke algebras for Kac-Moody groups \goodbreak over local fields}  %% pour amsart?
\title{Twin masures associated with Kac-Moody groups over Laurent polynomials}
\author{Nicole Bardy-Panse, Auguste H\'ebert \& Guy Rousseau}

\maketitle

%\author[S. Gaussent]{St\'ephane Gaussent*}
%\address{adresse}
%\email{gaussent@}
%\thanks{*ANRS}
%\urladdress{}
%
%\author[G. Rousseau]{Guy Rousseau**}
%\address{adresse}
%\email{rousseau@}
%\thanks{**ANRS}
%\urladdress{}

%\dedicatory{to him}   %% pour amsart

%\subjclass{17B67}   %% pour amsart
%\keywords{Hecke algebra, Kac--Moody group, local field}  %% pour amsart
%\altkeywords{alg\`ebre de Hecke , groupe de Kac--Moody, corps local}  %% pour smfart

%\par Version de travail \hskip 5em (\the\day \quad \the\month \quad \the\year) %% activer pour brouillons

%%%%%%%%%%%%%%%%%%%%%%%%%%%%%%%%%%%%%%%%%%%%%%%%%%%
%%%%%%%%%%%%%%%%%%%%%%%%%%%%%%%%%%%%%%%%%%%%%%%%%%%
%                                                TEXTE
%%%%%%%%%%%%%%%%%%%%%%%%%%%%%%%%%%%%%%%%%%%%%%%%%%%
%%%%%%%%%%%%%%%%%%%%%%%%%%%%%%%%%%%%%%%%%%%%%%%%%%%

\begin{abstract} 
Let $\g G$ be a split reductive group, $\k$ be a field and $\qp$ be an indeterminate. In order to study $\g G(\k[\qp,\qp^{-1}])$ and $\g G(\k(\qp))$, one can make them act on their twin building $\shi=\shi_\oplus\times \shi_{\ominus}$, where $\shi_{\oplus}$ and $\shi_{\ominus}$ are related via a ``codistance''.

 Masures are generalizations of Bruhat-Tits buildings adapted to the study of Kac-Moody groups over valued fields. Motivated by the work of Dinakar Muthiah on Kazhdan-Lusztig polynomials associated with Kac-Moody groups, we study the action of $\g G(\k[\qp,\qp^{-1}])$ and $\g G(\k(\qp))$ on their ``twin masure'', when $\g G$ is a split Kac-Moody group instead of a reductive group. 

\end{abstract}

%%%%%%%%%%%%%%%%%%%%%%%%%%%%%%%%%%

\tableofcontents

%\bigskip
\setcounter{tocdepth}{1}    %%%%  = only sections
%\tableofcontents
\section{Introduction}

\subsection{Context}

\paragraph{Split reductive groups over valued fields and Bruhat-Tits buildings}

Let $\g G$ be a split reductive group with maximal split torus $\g T$. Let $\shk$ be a field,  $G=\g G(\shk)$ and $T=\g T(\shk)$.  If $\omega: \shk\rightarrow \R\cup\{\infty\}$ is a nontrivial valuation of $\shk$, one can  construct a Bruhat-Tits building $\shi_\omega=\shi(\g G,\shk,\omega)$ on which $G$ acts, and study $G$ via its action on $\shi_\omega$. This building is a union of apartments, which are all translates by an element of $G$ of a standard apartment $\A_\omega$. 

The action of $G$ on  $\shi_\omega$ takes into account the valuation $\omega$. More precisely, let $\Phi$ be the root system of $(G,T)$, which can be regarded as a subset of the dual $\A_\omega^*$ of  the real vector space $\A_\omega$. 
Then $G=\langle T, x_\alpha(u),\alpha\in \Phi, u\in \shk\rangle$, where for each $\alpha\in \Phi$, $x_\alpha:(\shk,+)\hookrightarrow (G,.)$ is an algebraic group morphism.  Let $N$ be the normalizer of $T$ in $G$. Then $N$ is the stabilizer of $\A_\omega$ in $G$ and $T$ acts by translation on $\A_\omega$. If $t\in T$, then $t$ acts by translation on $\A_\omega$ by a vector depending on the values of $\omega(\chi(t))$, where $\chi$ runs over the characters of  $T$.
 If $\alpha\in \Phi$ and $u\in \shk$, $x_\alpha(u)$ fixes the half-apartment (or half-space) $\A_\omega\cap x_\alpha(u).\A_\omega=\{a\in \A_\omega\mid\alpha(a)+\omega(u)\geq 0\}$. 

\paragraph{Twin building of $\g G(\k[\qp,\qp^{-1}])$}

Suppose now that  $\shk=\k(\qp)$, where $\k$ is a field and $\qp$ is an indeterminate. Let $\omega_\oplus,\omega_\ominus$ be the valuations on $\shk$, trivial over $\k$ and such that $\omega_\oplus(\qp)=1=\omega_\ominus(\qp^{-1})$. Let $\sho=\k[\qp,\qp^{-1}]$. In order to study $G=\g G(\shk)$ and $G_{\sho}=\g G(\sho)$, it is natural to make them act on $\shi=\shi_{\oplus}\times \shi_{\ominus}$, where $\shi_\oplus=\shi(\g G,\shk,\omega_\oplus)$ and $\shi_\ominus=\shi(\g G,\shk,\omega_\ominus)$. The buildings $\shi_\oplus$ and $\shi_\ominus$ are related by a $G_\sho$-invariant codistance $d^*:\shc(\shi_\oplus)\times \shc(\shi_\ominus)\rightarrow W$, where $\shc(\shi_\oplus), \shc(\shi_{\ominus})$ are the sets of local chambers of $\shi_\oplus$ and $\shi_\ominus$ and $W$ is the affine Weyl group of $\A_\oplus:=\A_{\omega_\oplus}$ (which is isomorphic to the affine Weyl group of $\A_{\ominus}:=\A_{\omega_\ominus}$).  Equipped with this codistance, $\shi_\oplus\times \shi_\ominus$ is called a twin building (see \cite{RT94} for the case of $\g G=\mathrm{SL}_2$ and \cite{AB08} for a general study of twin buildings).   

%An other approach to this twining is as follows. 
This codistance is also called a twinning and it is deduced from some Birkhoff decomposition in $G$.
We may describe it slightly differently.
Let $C_\infty$ be the ``fundamental local chamber of $\A_\ominus$'', $C_0^+$ be the ``fundamental local chamber'' $C_0^+$ of $\A_\oplus$,  $I_\infty$ be the fixator of $C_\infty$ in $G_\sho$ and  $I$ be the fixator of $C_0^+$ in $G$. Then using the Birkhoff decomposition $G=I_\infty N I$, one can prove that there exists a unique $I_\infty$-invariant retraction $\rho_{C_\infty}:\shi_{\oplus}\twoheadrightarrow \A_{\oplus}$ (see \ref{n3.15.1}). We can then recover $d^*$ from $\rho_{C_\infty}$.

\paragraph{Kazhdan-Lusztig polynomials}

Let $(W',S')$ be a Coxeter group. In their fundamental paper \cite{KL79}, Kazhdan and Lusztig associated to this data a family $(P_{\mathbf{v},\mathbf{w}})_{\mathbf{v},\mathbf{w}\in W'}$ of polynomials of $\Z[\mathbf{q}]$, where $\mathbf{q}$ is an indeterminate. These polynomials are now known as  the Kazhdan-Lusztig polynomials. 
In order to define them, they began by defining  auxiliary polynomials -~called ``$R$-polynomials''~- $R_{\mathbf{v},\mathbf{w}}\in \Z[\mathbf{q}]$, for $\mathbf{v},\mathbf{w}\in W'$.  When $W'=W$, these polynomials are defined by the following equation (see \cite[(1.2)]{Mu19b})
 \begin{equation}\label{eq_df_Rpol} R_{\mathbf{v},\mathbf{w}}(q)= |(I\overset{\cdot}{\mathbf{w}} I\cap  I_\infty \overset{\cdot}{\mathbf{v}} I)/I|,\text{ for } \mathbf{v},\mathbf{w}\in W, \text{ for all prime power } q,
\end{equation} with $I=I(q)$ and $I_\infty=I_\infty(q)$   defined as above  in  $G=G_q=\g G(\F_q(\qp))$, with $\F_q$ the field of cardinality $q$, where $\overset{\cdot}{\mathbf{v}}, \overset{\cdot}{\mathbf{w}}$ are liftings   of $\mathbf{v},\mathbf{w}$ in $N\subset G$. 
This formula, is implicitly used by D. Kazhdan and G. Lusztig in
\cite{KL79}, and was proven by Z. Haddad (\cite{haddad1985coxeter}). 
%\red{\textbf{Bon j'ai copie ce que Paul a mis dans l'intro de l'article, mais je ne trouve pas le resultat en question chez Haddad, est-ce que vous le voyez ?}}
%\blue{Il me semble que dans l'article de Haddad, c'est les 2 ou 3 lignes qui suivent (2.4). Il fait reference a un article de Deodhar (inventiones 1985 ?); mais ce n'est pas direct. Je crois que ce que tu as indique est suffisant.}

\paragraph{Split Kac-Moody groups over valued field and masures}

Split Kac-Moody groups are infinite dimensional generalizations of split reductive groups. There are many possible definitions of such groups but in this paper, we are mainly interested in  the minimal one defined in \cite{T87} (although we also use  Mathieu's completion).  Let  $\g G$ be such a group, $\shk$ be a field equipped with a nontrivial valuation $\omega:\shk \rightarrow \R\cup\{\infty\}$ and $G=\g G(\shk)$. In \cite{R16}, generalizing results of  \cite{GR08}, Rousseau defined a ``masure''  $\shi_\omega=\shi(\g G,\shk,\omega)$ on which $G$ acts. This masure is a kind of Bruhat-Tits building adapted to the Kac-Moody framework. We still have $\shi_\omega=\bigcup_{g\in G} g.\A_\omega$, where $\A=\A_\omega$ is the ``fundamental apartment''. This apartment is an affine space of the same dimension as $\g T$ equipped with an arrangement of hyperplanes.   Using $\shi_\omega$, one can define the Iwahori subgroup $I=I_\omega$ of $G$, which is the fixator of the fundamental local chamber $C_0^+$ of $\A$. 
The Borel subgroup $B^\pm=T.U^\pm$ is well known (\cf  \ref{ss_m_KM_grp}).
In the following, a Bruhat or Birkhoff decomposition will be called more precisely a Bruhat-Borel or Birkhoff-Borel (\resp Bruhat-Iwahori or Birkhoff-Iwahori) decomposition, when it involves $B^\pm$ (\resp $I$).
As the Iwahori case is frequently used, we often omit this name Iwahori.

Let $Y$ be the cocharacter lattice and  $W\zv$ be the vectorial Weyl group of  $(\g G,\g T)$. 
Then,  $W:=N/T=W\zv\ltimes Y$ and the Bruhat decomposition does not hold in $G$: $IW I\subsetneq G$ (where we regard $W$ as a subset of $N$ by choosing for each element  of $W$ a lifting in $N$). Because of this, one often restricts attention to a subsemi-group $G^+=G^+_\omega$ of $G$ defined as follows. Let $C^v_f$ be the fundamental vectorial chamber of $\A$, $\sht:=\bigcup_{w\in W^v} w.\overline{C^v_f}$ be the Tits cone, $Y^+=Y\cap \sht$ and $W^+=W\zv\ltimes Y^+$. Then $G^+:=IW^+I$ is  a set of elements of $G$ admitting a Bruhat decomposition. An equivalent definition of $G^+$ is as follows. If $x,y\in \A$, we write $x\leq y$ if $y-x\in\sht$. Then $\leq$ extends to a $G$-invariant preorder $\leq$ on $\shi$ and we have  $G^+=\{g\in G\mid g.0\geq 0\}$ (where $0$ is the vertex of $C_0^+$).

\paragraph{Kazhdan-Lusztig polynomials in the Kac-Moody setting}

In general, neither $W$ nor $W^+$, which is not even a group (except if $\g G$ is reductive), is a Coxeter group. 
In \cite{Mu19b}, Muthiah suggests to take \eqref{eq_df_Rpol}, for $\mathbf{v},\mathbf{w}\in W^+$, as a definition of the $R$-polynomials associated with $\g G$ and then to define the Kazhdan-Lusztig polynomials. With this approach, two questions naturally arise: are the cardinalities in \eqref{eq_df_Rpol} finite and how to compute them if they are? 

%\red{Parler des resultats avec Paul?} 
In \cite{Mu19b}, Muthiah partially solves these questions, when $\g G$ is untwisted affine of type A, D or E, under the assumption that the retraction $\rho_{C_\infty}:\shi_{\oplus}\twoheadrightarrow \A_{\oplus}$  is well-defined (for every prime power $q$, where $\shi_\oplus=\shi\left(\g G,\F_q\left(\qp\right),\omega_{\oplus}\right)$), or at least that it is well-defined on a sufficiently large subset of $\shi_{\oplus}$. 
These works are generalized to general Kac-Moody groups in   \cite{hebert2024kazhdan}, under the same assumption on the retraction $\rho_{C_\infty}$, with similar techniques. Muthiah's method is as follows. Let $\mathbf{v},\mathbf{w}\in W^+$. Then the set involved in \eqref{eq_df_Rpol} is in bijection with a set $E_{\mathbf{v},\mathbf{w}}$  of local chambers of $\shi_\oplus$, which are in some ``sphere'', and whose image by $\rho_{C_\infty}$ is in $\mathbf{v}.C_0^+$. 
He proves that the image by $\rho_{C_\infty}$ of a line segment of $\shi_{\oplus}$ (satisfying certain conditions) is an $I_\infty$-Hecke path of $\A_{\oplus}$, \ie it is a piecewise linear path satisfying certain precise conditions. He proves finiteness results for the number of these $I_\infty$-Hecke paths in $\A_\oplus$ (in the untwisted affine case of type A, D or E) and proves that for a given $I_\infty$-path, the number of line segments of $\shi_{\oplus}$ retracting on it is finite and polynomial in $q$ (in the general case, not necessarily affine). However, he does not study the existence of $\rho_{C_\infty}$.

\subsection{Content of this paper}

Let  $\k$ be  a field (not necessarily finite) and $\g G$ be a split Kac-Moody group.
 In this paper, we study the action of $G=\g G(\k(\qp))$ and $G_{twin}:=G_\sho=\g G(\sho)$ on $\shi_\oplus\times \shi_\ominus$.  As $\sho$ is not a field, the meaning  of $\g G(\sho)$ is not clear, but we give a definition of it as a subgroup of $G$ in \ref{ss_m_KM_grp}.  We begin by studying the action of $G_\sho$ on a single masure $\shi_\oplus$ or $\shi_{\ominus}$. 
 We actually study  the slightly more general situation where $\sho$ is replaced by $\shr$,  a dense  subring of a field $\shk$ equipped with a discrete valuation (satisfying the additional assumption \eqref{eq_Assumption_R}, i.e,  such that $\omega(\shr^*)=\omega(\shk^*)=\Z$).  We prove that $G_\shr$ admits Bruhat and Iwasawa decompositions, using the corresponding decompositions of $\g G(\shk)$. For $\epsilon\in \{-,+\}$, set $U^{\epsilon\epsilon}_\shr=\langle x_\alpha(\shr)\mid\alpha\in \Phi^\epsilon\rangle\subset G_\shr$ (where $\Phi^+$ and $\Phi^-$ are the sets of positive and negative real roots respectively).
 Note that greater groups $U^{\epsilon}_\shr$ will be defined in \ref{ss_m_KM_grp}.
 Set $I_\shr=I\cap G_\shr$ and $N_\shr=N\cap G_\shr$. Then we prove the following theorem (see Corollary~\ref{cor_Bruhat} and Corollary~\ref{cor_Iwasawa}):

\begin{theo*}
We have  
\begin{enumerate}
\item $G_\shr=U^{\epsilon\epsilon}_\shr N_\shr I_\shr$, for both $\epsilon\in \{-,+\}$,

\item $G_\shr\cap G^+=I_\shr W^+ I_\shr$.
\end{enumerate}
\end{theo*}

We then go back to the situation where $\shr=\sho=\k[\qp,\qp^{-1}]$ and  study the action of $G_{\sho}$ on $\shi_\oplus\times \shi_{\ominus}$.   
We do not prove the existence of $\rho_{C_\infty}$, but we prove that if $(G_\sho)_{\oplus}^+:=\{g\in G_\sho\mid g.0_{\oplus}\geq 0_{\oplus}\}$ admits a Birkhoff decomposition (see \S \ref{4.0} for the precise meaning), then $\rho_{C_\infty}$ is well-defined on $\shi_{\oplus}^{\geq {0_\oplus}}=\{x\in \shi_\oplus\mid x\geq 0_\oplus\}$ (see \S\ref{n3.15.1}). 
Following the ideas of Muthiah, we conjecture that this decomposition holds (see \S \ref{4.4}) and that the same decompositions with $(G_\sho)_{\oplus}^+$ replaced by $(G_\sho)_{\oplus}^-:=\{g\in G_\sho\mid g.0_{\oplus}\leq 0_{\oplus}\}$  hold, which would be sufficient for applying Muthiah's method.
With such Birkhoff decompositions, we might really say that $\shi_\oplus$ and $ \shi_{\ominus}$ are twin masures.
Unfortunately the decompositions proved by M. Patnaik in \cite{patnaik2024local} concern a completion of $\g G(\sho)$, see \ref{4.4}.
%\red{Je dirais "a completion" plutot que "the completion" etant donne que c'est la completion de Carbone-Garland et non celle de Mathieu qu'il utilise il me semble}

We then study the image by $\rho_{C_\infty}$ of a   line segment $[x,y]$, with $x\leq y$ or $y\leq x$ and  such that $\rho_{C_\infty}(z)$ is well-defined for every $z\in [x,y]$ (the second condition is always satisfied  if our conjecture is true).
 We prove that they are $C_\infty-$Hecke paths.
 We then obtain a formula counting the number of liftings of a given $C_\infty-$Hecke path, and proving that it is polynomial in the cardinality of $\k$ (see Theorem~\ref{5.7}). 
 
 %These results recover Muthiah's results of \cite{Mu19b}, but are more detailed.  
 
 \par To get this number of liftings of a $C_\infty-$Hecke path as a line segment, we first  prove that, after choosing some superdecorations, it is the product of the numbers of local liftings around a finite number of points  (the points where the path crosses a wall in some specific way).
 Then we compute each of these numbers of local liftings.
 We get a precise formula, which seems more explicit than Muthiah's formula in \cite{Mu19b} (where our paths are called $I_\infty$-Hecke paths).   

Eventually, we study the case where $\g G$ is affine $\mathrm{SL}_2$. We prove that $G \supsetneq I_\infty N I$: the Birkhoff decomposition does not hold on the entire $G$. This was expected since this is already the case for the Bruhat decomposition. We give an example of an element $g\in  G\setminus I_\infty N I$. As $g\notin G_{\oplus}^+\cup G_{\oplus}^-$,  this does not contradict our conjecture. We also study explicit examples of $C_\infty-$Hecke paths.

\begin{remas*}
\begin{enumerate}
\item Our conventions differ from the one of \cite{Mu19b}. Our Tits cone is the opposite of the Tits cone for Muthiah, and thus what Muthiah denotes $G^+$ corresponds to $G^-$ for us. For this reason, our definition of $C_\infty-$Hecke path and our formulas slightly differ from the one of \cite{Mu19b}.

\item The fixators of objects in the masure (like $I$ or $I_\infty$) are   subgroups of $G$ or $G_\sho$  defined by  sets of generators. Even in the affine case, it is a delicate issue to describe them explicitly. For example, if $\g G(\shk)=\mathrm{SL}_2(\k(\qp)[u,u^{-1}])$, where $u$ is an indeterminate, then the fixator of $0_\oplus$ in $G$ is $\mathrm{SL}_2(\sho_{\oplus}[u,u^{-1}])$, where $\sho_{\oplus}=\{a\in \k(\qp)\mid\omega_\oplus(a)\geq 0\}$ (see Lemma~\ref{lemDescriptionK}). However, for $I_\infty$, we prove that \[I_\infty\subset \begin{psmallmatrix}
\qp^{-1}\k[\qp^{-1}][u,u^{-1}]+\k[u^{-1}]&\  & \qp^{-1}\k[\qp^{-1}][u,u^{-1}]+u^{-1}\k[u^{-1}] \\
\qp^{-1}\k[\qp^{-1}][u,u^{-1}]+\k[u^{-1}]&\  & \qp^{-1}\k[\qp^{-1}][u,u^{-1}]+\k[u^{-1}] \end{psmallmatrix},\] (see Lemma~\ref{lemDescription_I_infty}), but we do not know if it is an equality.
\end{enumerate} 
\end{remas*}

\medskip

The paper is organized as follows.

In section~\ref{s1}, we introduce the general framework, in particular Kac-Moody groups and masures.

In section~\ref{s2}, we study $G_{\shr}$ for $\shr$ a dense subring of a valued field $\shk$ (satisfying Assumption~\eqref{eq_Assumption_R}). We prove the Bruhat and  Iwasawa decompositions of $G_\shr$.

In section~\ref{s3}, we study the action of  $G_{twin}:=G_{\sho}$, where $\sho=\k[\qp,\qp^{-1}]$  on $\shi_\oplus\times \shi_{\ominus}$. We define $\rho_{C_\infty}$ under some conjecture.

In section~\ref{s5}, we study $C_\infty-$Hecke paths and their liftings in $\shi_\oplus$.

In section~\ref{s6}, we study the case where $\g G$ is affine $\mathrm{SL}_2$.

\medskip

\textbf{Acknowledgements}

We are especially grateful to Dinakar Muthiah for his suggestion to look at these problems. We also thank St\'ephane Gaussent and Manish Patnaik for interesting discussions on this subject. 
We also thank the anonymous referees for their remarks, questions and suggestions.

%%%%%%%%%%%%%%%%%%%%%%%%%%%%%%%%%%%%%%%%%%%%
\section{Split Kac-Moody groups over valued fields and masures}\label{s1}

\subsection{Standard apartment of a masure}\label{subRootGenSyst}

\subsubsection{Root generating system}\label{ss_Root_generating}

A \textbf{ Kac-Moody matrix} (or { generalized Cartan matrix}) is a square matrix $A=(a_{i,j})_{i,j\in I}$ indexed by a finite set $I$, with integral coefficients, and such that :
\begin{enumerate}
\item[\tt $(i)$] $\forall \ i\in I,\ a_{i,i}=2$;

\item[\tt $(ii)$] $\forall \ (i,j)\in I^2, (i \neq j) \Rightarrow (a_{i,j}\leq 0)$;

\item[\tt $(iii)$] $\forall \ (i,j)\in I^2,\ (a_{i,j}=0) \Leftrightarrow (a_{j,i}=0$).
\end{enumerate}
A \textbf{root generating system} is a $5$-tuple $\mathcal{S}=(A,X,Y,(\alpha_i)_{i\in I},(\alpha_i^\vee)_{i\in I})$\index{s@$\mathcal{S}$}\index{Y@$Y$}\index{a@$\alpha_i$, $\alpha_i^\vee$}\index{X@$X$} made of a Kac-Moody matrix $A$ indexed by the finite set $I$, of two dual free $\Z$-modules $X$ and $Y$ of finite rank, and of a free family $(\alpha_i)_{i\in I}$ (resp.  a free family $(\alpha_i^\vee)_{i\in I}$) of elements in $X$ (resp. $Y$) called \textbf{simple roots} (resp. \textbf{simple coroots}) that satisfy $a_{i,j}=\alpha_j(\alpha_i^\vee)$ for all $i,j$ in $I$. Elements of $X$ (respectively of $Y$) are called \textbf{characters} (resp. \textbf{cocharacters}).

Fix such a root generating system $\mathcal{S}=(A,X,Y,(\alpha_i)_{i\in I},(\alpha_i^\vee)_{i\in I})$ and set $\A:=Y\otimes \R$\index{A@$\A$}. Each element of $X$ induces a linear form on $\A$, hence $X$ can be seen as a subset of the dual $\A^*$. In particular, the $\alpha_{i}$'s (with $i \in I$) will be seen as linear forms on $\A$. This allows us to define, for any $i \in I$, an involution $r_{i}$\index{r@$r_i$} of $\A$ by setting $r_{i}(v) := v-\alpha_i(v)\alpha_i^\vee$ for any $v \in \A$. One defines the \textbf{Weyl group of $\mathcal{S}$} as the subgroup $W\zv$\index{W@$W\zv$} of $\mathrm{GL}(\A)$ generated by $\{r_i\mid i\in I\}$. The pair $(W\zv, \{r_i\mid i\in I\})$ is a Coxeter system.

The following formula defines an action of the Weyl group $W\zv$ on $\A^{*}$:  
\[\displaystyle \forall \ x \in \A , w \in W\zv , \alpha \in \A^{*} , \ (w.\alpha)(x):= \alpha(w^{-1}.x).\]
Let $\Phi:= \{w.\alpha_i\mid(w,i)\in W\zv\times I\}$\index{P@$\Phi,\Phi^\vee,\Phi^\pm$} (resp. $\Phi^\vee=\{w.\alpha_i^\vee\mid(w,i)\in W\zv\times I\}$) be the set of \textbf{real roots} (resp. \textbf{real coroots}): then $\Phi$ (resp. $\Phi^\vee$) is a subset of the \textbf{root lattice} $Q := \displaystyle \bigoplus_{i\in I}\Z\alpha_i$ (resp. \textbf{coroot lattice} $Q^\vee=\bigoplus_{i\in I}\Z\alpha_i^\vee$). If $\alpha\in \Phi$, there exist $i\in I$, $w\in W\zv$ such that $\alpha=w.\alpha_i$. One sets $\alpha^\vee=w.\alpha_i^\vee$ and $r_\alpha=r_{\alpha^\vee}=wr_iw^{-1}\in W\zv$. This does not depend on the choice of $i$ and $w$. By \cite[1.2.2 (2)]{Kum02}, one has $\R \alpha^\vee\cap \Phi^\vee=\{\pm \alpha^\vee\}$ and $\R \alpha\cap \Phi=\{\pm \alpha\}$ for all $\alpha^\vee\in \Phi^\vee$ and $\alpha\in \Phi$.  We set $Q^+=\bigoplus_{i\in I} \N \alpha_i$, $Q^{\vee,+}=\bigoplus_{i\in I} \N \alpha_i^\vee$\index{Q@$Q^+$, $Q^{\vee,+}$},   $\Phi^+=\Phi\cap Q^+$ and $\Phi^-=\Phi\cap -Q^+=-\Phi^+$. We define $\mathrm{ht}:Q\otimes \R\rightarrow \R$ by $\mathrm{ht}(\sum_{i\in I} n_i\alpha_i)=\sum_{i\in I} n_i$ for $(n_i)\in \R^I$ and we call $\mathrm{ht}$ the height. %%%\index{P@$\Phi^+$, $\Phi^-$}

\subsubsection{Local chambers, sectors, chimneys}  %%%%%%%%%%%%%%%%%%%%%
\label{2.2b}

\par\quad\; {\bf(1)} Vectorial facets. Let $(\qa_{i})_{1\leq i\leq\ell}$ be the above basis of the system $\QF$ of roots.
Then $C\zv_{f}=\set{v\in \A \mid \qa_{i}(v)>0,\forall i}$ is the canonical vectorial chamber.
Its facets are the cones $F\zv(J)=\set{v\in \A \mid \qa_{i}(v)=0,\forall i\in J,\qa_{i}(v)>0,\forall i\not\in J}$ for $J\subset\set{1,\ldots,\ell}=I$.
{The facet $F\zv(J)$ and $J$ are said spherical if the group $W\zv(J)$ generated by the reflections $r_{i}=r_{\qa_{i}}$ for $i\in J$ is finite.} 
 
\par A positive (\resp negative) vectorial facet of type $J$ is a conjugate   by $W\zv$ of $F\zv(J)$ { (\resp $-F\zv(J)$)}.
It is a chamber if $J=\emptyset$ and a panel if $\vert J\vert =1$.

\par The Tits cone $\sht$ {(\resp its interior $\sht^\circ$)}\index{T@$\sht,\sht^\circ$} is the union of all positive {(\resp and spherical)} vectorial facets. It is a convex cone.

\medskip
\par{\bf(2)} Local facets and segment germs. A local facet in $\A$ is the germ $F(x,F\zv)=germ_{x}(x+F\zv)$ where $x\in\A$ and $F\zv$ is a vectorial facet (\ie $F(x,F\zv)$ is the filter of all neighbourhoods of $x$ in $x+F\zv$).
It is a local chamber, a local panel, positive, or negative %, or of type $J$
 if $F\zv$ has this property, it is of type $0$\index{type $0$} if $x\in Y\subset \A$.
We denote by $C_0^+$\index{C@$C_0^+$} the  fundamental local chamber, \ie $C_0^+=germ_0(C\zv_f)$.
 
Let $x,y$ in $\A$ be such that $x\neq y$. The germ of $[x,y]$ at $x$ is the filter $[x,y)=germ_x([x,y])$\index{$[x,y)$} consisting of the subsets of the form $\Omega\cap [x,y]$, where $\Omega$ is a neighbourhood of $x$ in $\A$. It is said to be  preordered if $y-x\in \pm \sht$.

\medskip
\par{\bf(3)} Sectors and sector germs.  A sector in $\A$ is a subset $\g q=x+C\zv$, for $x$ a point in $\A$ and $C\zv$ a vectorial chamber.
Its sector germ is the filter $\g Q=germ_{\infty}(\g q)$ of subsets of $\A$ containing another sector $x+y+C\zv$, with $y\in C\zv$.
It is entirely determined by its direction $C\zv$.
This sector or sector germ is said positive (\resp negative) if $C\zv$ has this property.

\par For example, we consider  $\g Q_{\pm\infty}=germ_{\infty}(\pm C\zv_{f})$\index{Q@$\g Q_{\pm \infty}$}.

\medskip
\par{\bf(4)}  A half-apartment (\resp an open-half-apartment, a  wall) of $\A$ is a set of form $D(\qa-k)= \alpha^{-1}([k,+\infty[)$\index{D@$D(\qa+k),D^\circ(\qa+k)$} (\resp $D^\circ(\qa-k)=\alpha^{-1}(]k,+\infty[)$, $M(\qa-k)= \alpha^{-1}(\set{k})$)\index{M@$M(\qa+k)$}, where $k\in \Z$ and $\alpha\in  \Phi$.

 A subset $E$ of $\A$ is said to be enclosed if it is the intersection of a finite number of half-apartments.
The enclosure $cl(E)$ of a subset (or filter) $E$ of $\A$ is the filter consisting of the subsets containing an enclosed set containing $E$.
\medskip
\par{\bf(5)} Chimneys. Let $F = F(x, F_{1}\zv)$ be a local facet and $F\zv$ be a vectorial facet. 
The chimney $\g r(F, F\zv)=cl(F+F\zv)$ is the filter consisting of the sets containing an enclosed set containing $F + F\zv$. 
A shortening of a chimney $\g r(F, F\zv)$, with $F = F(x, F_{1}\zv)$ is a chimney of the form
$\g r(F(x + \qx, F_{1}\zv), F\zv)$ for some $\qx \in F\zv$. 
The germ $\g R=germ_{\infty}(\g r)$ of a chimney $\g r$ is the filter of subsets of
$\A$ containing a shortening of $\g r$.
{The chimney $\g r(F, F\zv)$ or its germ $\g R$ is said splayed of sign $\qe$ if its direction $F\zv$ is a spherical facet of sign $\qe$. A sector is a splayed chimney.}

\subsection{Split Kac-Moody groups over valued fields}

\subsubsection{Minimal split Kac-Moody groups}\label{ss_m_KM_grp}

 Let  $\g G=\g G_{\mathcal{S}}$\index{G@$\g G$} be the group functor associated in \cite{T87} with  the  root generating system $\mathcal{S}$, see also \cite[8]{Re02}. Let $(\shk,\omega)$\index{K@$\shk$}\index{o@$\omega$} be a valued field where $\omega:\shk\twoheadrightarrow \Z\cup\{+\infty\}$ is a normalized, discrete valuation. 
 Let $G=\g G(\shk)$ be the \textbf{split Kac-Moody group over $\shk$ associated with $\mathcal{S}$}.  The group $G$ is generated by the following subgroups:\begin{itemize}
\item the fundamental torus $T=\g T(\shk)$\index{T@$T$}, where $\g T=\mathrm{Spec}(\Z[X])$,

\item the root subgroups $U_\alpha=\g U_\alpha(\shk)$\index{U@$U_\alpha$}, each isomorphic to $(\shk,+)$ by an isomorphism $x_\alpha$\index{x@$x_\alpha$}.
\end{itemize}

The groups $X$ and $Y$ correspond to the character lattice and cocharacter lattice of $\g T$ respectively. One writes $\g U^{\pm}$\index{U@$\g U^{\pm},U^{\pm}$} the subgroup of $\g G$ generated by the $\g U_{\alpha}$, for $\alpha\in \Phi^{\pm}$ and $U^{\pm}=\g U^{\pm}(\shk)$.     %\index{U@$U^{\pm}$}.

Let $\shr$\index{R@$\shr$} be a subring of $\shk$ (with $1\in \shr$). In this paper, we are interested in the  group of $\shr$-points of $\g G$.  It seems that there is currently no consensus on what this should mean. We mainly study the case where $\shr=\sho=\k[\qp,\qp^{-1}]\subset \shk=\k(\qp)$, for $\k$ a field and $\qp$ an indeterminate. 
When $\g G$ is a split reductive group over $\k$,  one knows that $\g G(\sho)$ is given by some well known generators.
 This is a consequence of $\sho$ being a principal ideal domain by \cite{T85} top of page 205.
 So in this paper, we take the same kind of generators and set \index{G@$G_{\shr}$}
  \[G_\shr:=\langle x_\alpha(\shr),\g T(\shr)\mid\alpha\in \Phi\rangle \subset \g G(\shk)=G.\]

\par For $\epsilon\in \{-,+\}$, we set $U^\epsilon_\shr=G_\shr\cap U^\epsilon= G_\shr\cap  \langle x_\alpha(\shk)\mid \alpha\in \Phi^\epsilon\rangle$\index{U@$U^\epsilon_\shr,U^{\epsilon\epsilon}_\shr$}. Let  $U^{\epsilon \epsilon}_\shr=\langle x_\alpha(u)\mid u\in \shr,\alpha\in \Phi^\epsilon\rangle$.
 We have $U^{\epsilon \epsilon}_\shr\subset U^\epsilon_\shr$. However, this inclusion is strict in general, see  \cite{T87} 3.10.d page 555 for a counter-example.
 
 \par  Timoth\'ee Marquis \cite[Def. 8.126]{Mar18} defines a minimal Kac-Moody group functor $\g G_{\shs}^ {min}$ and proves [\lc proof of Prop. 8.128] that the morphism $\g G^ {min}_{\shs}(k_{1}) \to \g G^ {min}_{\shs}(k_{2})$ is injective when $k_{1}\into k_{2}$ is an injective morphism of rings. Moreover when 
 $\shr$ is a Euclidean ring (\eg $\shr=\sho=\k[\qp,\qp^ {-1}]$), we know that $\mathrm{SL}_{2}(\shr)$ is generated by its torus and root subgroups \cite[Exer 7.2 (3)]{Mar18}.
 So our $G_{\shr}$ is equal to the group $\g G_{\shs}^ {min}(\shr)$ defined by Timoth\'ee Marquis.
  It is perhaps not equal to $\g G(\shr)$ as the morphisms $\qi(\shr):\g G(\shr) \to \g G^{pma}(\shr)$ (see below in \S\ref{2.1}) and $\g G(\shr)\to\g G(\shk)$ might be non injective.

Note that general Kac-Moody groups over rings are defined and studied in \cite{allcock2016presentation}, \cite{allcock2016steinberg}, and \cite{allcock-carbone2016presentation}. 
It seems more difficult to relate them with the group we study.

\begin{rema}
We chose to work with any discretely valued field $(\shk,\qo)$. For our main purpose, which is to develop a Kazhdan-Lusztig theory in the Kac-Moody setting, we only need the case where the residual cardinality of $\shk$ is finite, and even where $\shk=\k(\qp)$, where $\k$ is a finite field. However, as  it would not really simplify our proofs to impose these restrictions on $\shk$, we work in this more general frameworks.

\end{rema}

\subsubsection{Subgroups $N$ and $N_\shr$}\label{ss_N}
Let $\g N$\index{N@$\g N$} be the group functor on rings such that if $\shr'$ is a ring, $\g N(\shr')$ is the subgroup of $\g G(\shr')$ generated by $\g T(\shr')$ and the $\tilde{s}_{\alpha_i}$, for $i\in I$, where $\tilde{s}_{\alpha_i}$ is defined in \cite[1.6]{R16}. Then if $\shr'$ is a field with at least $4$ elements, $\g N(\shr')$ is the normalizer of $\g T(\shr')$ in $\g G(\shr')$. 

Let $N=\g N(\shk)$\index{N@$\g N$} and $\mathrm{Aut}(\A)$ be the group of affine automorphisms of $\A$. Then by \cite[4.2]{R16}, there exists a group morphism $\nu:N\rightarrow \mathrm{Aut}(\A)$\index{n@$\nu$} such that:\begin{enumerate}
\item for $i\in I$, $\nu(\tilde{s}_{\alpha_i})$ is the simple reflection $r_i\in W\zv$, it fixes $0$,

\item for $t\in \g T(\shk)$, $\nu(t)$ is the translation on $\A$ by the vector $\nu(t)$ defined by $\chi(\nu(t))=-\omega(\chi(t))$, for all $\chi\in X$. This action  is compatible with the action of $W\zv$ on $\A$,

\item we have $\nu(N)=W\zv\ltimes Y:=W$\index{W@$W$}.

\end{enumerate}

Let $\shr$ be a dense subring of $\shk$.  We often  assume: \begin{equation}\label{eq_Assumption_R}
\exists \qp\in \shr^*\mid \omega(\qp)=1.
\end{equation}

 This assumption is in particular satisfied by $\shr=\k[\qp,\qp^{-1}]$, $\shk =\k(\qp)$ or $\k(\!(\qp)\!)$, for $\k$ a field and $\qp$ an indeterminate or by $\shr=\Z[\frac{1}{p}]$,  $\shk=\Q$ or $\Q_p$, for $p$ a prime number.
 
 Let $N_\shr=\g N(\shr)\subset N$\index{N@$N_\shr$}. Then $N_\shr$ normalizes $T_\shr:=\g T(\shr)$\index{T@$\g T_\shr$}.   For $\ql\in Y= Hom(\g{Mult},\g T)$, we set $\qp^\ql:=\ql(\qp){\in} \g T(\shr)$\index{p@$\qp^\ql$}. Then $\nu(\qp^\ql)$ is the translation on $\A$ by the vector $-\lambda$. Moreover, $\tilde{s}_{\alpha_i}\in N_\shr$. In particular, we have: \begin{equation}\label{eq_Nu(NR)}
 \nu(N_\shr)=W\zv\ltimes Y=W.
 \end{equation}

\subsubsection{The completion $\g G^ {pma}$ of the Kac-Moody group $\g G$} %%%%%%%%%%
\label{2.1}
\par In order to study the group $G=\g G(\shk)$ (for $\shk$ a field), we consider the group-functor homomorphism $\qi:\g G \to \g G^{pma}$    from $\g G$ to the (positive) completion $\g G^{pma}$\index{G@$G^{pma}$} of $\g G$ (we shall also use the negative completion $\g G^{nma}$).
We know that $\qi(\shk): \g G(\shk) \to \g G^{pma}(\shk)$ is injective for any field $\shk$ \cite[Prop. 3.13]{R16}, so we consider $G$ as a subgroup of $ \g G^{pma}(\shk)$.
 Actually $ \g G^{pma}$ is the Kac-Moody group defined by Olivier Mathieu in \cite{Ma89}  as a functor on the category of rings; we refer here to \cite[\S 3]{R16}.
This group is hard to define. However the following important subgroups have simpler definitions.

\par One starts with the split Kac-Moody algebra $\g g_{\Z}$\index{g@$\g g_\Z$}  over $\Z$ (see \cite[Definition 7.5]{Mar18} for the definition of $\g g_\Z$),   with system of (real or imaginary) roots $\QD=\QD^+\sqcup\QD^-\subset Q$ (see \cite[1.2.2]{Kum02} for the definition of $\Delta$). We have $\Phi\subset \Delta$. The elements of $\Phi=\Delta_{re}$ are called real roots and the elements of $\Delta_{im}= \Delta\setminus \Phi$\index{D@$\Delta$, $\Delta_{re}$, $\Delta_{im}$} are called imaginary roots.
To each $\qa\in\QD$ is associated a subgroup $\g U_{\qa}$.

\par Let  $\Psi\subset \Delta^+$. We say that $\Psi$ is closed if for all $\alpha,\beta\in \Psi$, for all $p,q\in \N^*$, $p\alpha+q\beta\in \Delta^+$ implies $p\alpha+q\beta\in \Psi$. Let $\Psi$ be a closed  subset of $\QD^+$ and $R$ a ring (commutative with unit), then a pro-unipotent group scheme $\g U^ {ma}_{\Psi}$ is described as follows in \cite[Prop 3.2 + 3.4]{R16}:

\begin{equation}
\g U^ {ma}_{\Psi}(R)= \prod_{\qa\in \Psi}\,X_{\qa}(\g g_{\qa,\Z}\otimes R).
\end{equation}

\par One chooses an order on $\Psi$, \eg such that the height of $\qa$ is increasing.

\par $\g g_{\qa,\Z}$\index{g@$\g g_{\qa,\Z}$} is the eigenspace associated to $\qa$ in $\g g_{\Z}$ and $X_{\qa}: \g g_{\qa,\Z}\otimes R \to \g U^ {ma}_{\Psi}(R),\ \sum_{x\in\shb_{\qa}}\, \ql_{x}.x \mapsto \prod_{x\in\shb_{\qa}}\, [exp]\ql_{x}.x$\index{X@${X_\alpha}$} is one to one (where $\shb_\alpha$ is a $\Z$-basis of $\g g_{\alpha,\Z}$).  

\par When $\qa$ is real (\ie $\qa\in\QF=\QD_{re}$), then $\g U_{\qa}(R)=X_{\qa}(\g g_{\qa,\Z}\otimes R)$.
One chooses $e_{\qa}$ (one of the two bases of $\g g_{\qa,\Z}$) and one writes $x_{\qa}(a) =X_{\qa}(a.e_{\qa})$ for $a\in R$.
One gets thus an isomorphism $x_{\qa}: (R,+)\to \g U_{\qa}(R), a\mapsto x_{\qa}(a)$ and $\g x_{\qa}:\g{Add}\to \g U_{\qa}$.

\par When $\qa$ is imaginary (\ie $\qa\in \QD_{im}$), then $\g U_{\qa}(R)=\prod_{n\geq1}\,X_{n\qa}(\g g_{n\qa,\Z}\otimes R)$.

\medskip
\par $\g U^ {ma}_{\Psi}$ may be seen as ``topologically generated'' by the $\g U_{\qa}$ for $\qa\in\Psi$.
\par One writes $\g U^ {ma+}=\g U^ {ma}_{\QD^+}$. It contains $\g U^+$.
The positive Borel subgroup of $\g G^{pma}$  is $\g T\ltimes \g U^ {ma+}=\g B^ {ma+}$.

\subsubsection{Parahoric subgroups}\label{3.1}%%%%%%%%%%%%

In \cite{R16} and \cite{GR08}, the masure $\shi$ of $G$ is constructed as follows. To each $x\in \A$ is associated a group $\hat{P}_x=G_x$. Then $\shi$ is defined  in such a way that $G_x$ is the fixator of $x$ in $G$ for the action on $\shi$ (see \ref{sub_Masure}).  We actually associate to each filter $\Omega$ on $\A$ a subgroup $G_\Omega\subset G$ (with $G_{\{x\}}=G_x$ for $x\in \A$).  Even though the masure is not yet defined, we use the terminology ``fixator'' to speak of $G_\Omega$, as this will be the fixator of $\Omega$ in $G$. The definition of $G_\Omega$ involves the completed groups $G^{pma}$ and $G^{nma}$.   

\parni{\bf 1)} Let $\QO\subset\A$ be a non empty set or filter.
One defines a function $f_{\QO}:\QD\to \Z\cup \{+\infty\}$, $f_{\QO}(\qa)=\inf\set{r\in\Z \mid \QO\subset D(\qa+r)}=\inf\set{r\in\Z \mid \qa(\QO)+r\subset[0,+\infty[}$\index{f@$f_\Omega$} and, for $r\in\Z$, $\shk_{\qo\geq r}=\set{x\in\shk \mid \qo(x)\geq r}$, $\shk_{\qo= r}=\set{x\in\shk \mid \qo(x) = r}$.  The filter $\Omega$ is said almost open (resp. narrow) if for all $\alpha\in \Phi$, $f_\Omega(\alpha)+f_\Omega(-\alpha) \geq 1$ (resp. $f_\Omega(\alpha)+f_\Omega(-\alpha)\leq 1$). For example, segment germs and local facets are narrow and local chambers and sectors are almost open.

\medskip
\parni{\bf 2)} If $\Omega$ is a set, we  define the subgroup $U^ {ma+}_{\QO}=\prod_{\qa\in\QD^+}\, X_{\qa}(\g g_{\qa,\Z}\otimes \shk_{\qo\geq f_{\QO}(\qa)})$\index{U@$U^{ma+}_{\QO}$}, see \ref{2.1}.
Actually, for $\qa\in\QF^+=\QD_{re}^+$, $X_{\qa}(\g g_{\qa,\Z}\otimes \shk_{\qo\geq f_{\QO}(\qa)})=x_{\qa}(\shk_{\qo\geq f_{\QO}(\qa)})=:U_{\qa,\QO}$\index{U@$U_{\alpha,\Omega}$}. 
We then define $U_{\QO}^ {pm+}=U_{\QO}^ {ma+}\cap G=U_{\QO}^ {ma+}\cap U^+$, see \cite[4.5.2, 4.5.3 and 4.5.7]{R16}. When $\Omega$ is a filter, we set $U^ {ma+}_{\QO}:=\cup_{S\in\QO}\, U^ {ma+}_{S}$ and $U_{\QO}^ {pm+}:=U_{\QO}^ {ma+}\cap G$

\par We may also consider the negative completion $G^ {nma}=\g G^ {nma}(\shk)$ of $G$, and define the subgroup $U^ {ma-}_{\QO}=\prod_{\qa\in\QD^-}\, X_{\qa}(\g g_{\qa,\Z}\otimes \shk_{\qo\geq f_{\QO}(\qa)})$.
For $\qa\in\QF^-=\QD_{re}^-$, $X_{\qa}(\g g_{\qa,\Z}\otimes \shk_{\qo\geq f_{\QO}(\qa)})=x_{\qa}(\shk_{\qo\geq f_{\QO}(\qa)})=:U_{\qa,\QO}$.
We then define $U_{\QO}^ {nm-}=U_{\QO}^ {ma-}\cap G=U_{\QO}^ {ma-}\cap U^-$.

\medskip
\parni{\bf 3)} Let $\Omega$ be a filter on $\A$. We denote by $N_\Omega$ the fixator of $\Omega$ in $N$ (for the action of $N$ on $\A$). If $\Omega$ is not a set, we have $N_\Omega=\bigcup_{S\in \Omega} N_S$. Note that we drop  the hats used in \cite{R16} to avoid confusions with the hats related to the completion  $\hat{\shk}_\omega$ of $\shk$, that we shall consider in section \ref{s3}. When $\QO$ is almost open  one has $N_{\QO}=N_{\A}=T_{0}:=\g T(\shk_{\qo\geq0})=\g T(\shk_{\qo=0})$\index{T@$T_0$} (written $H$ in \lc~, but we avoid this here), see [\lc 4.3.1].

  If $x\in \A$, we set $G_x=U_x^{pm+}.U_x^{nm-}.N_x$. This is a subgroup of $G$. If $\Omega\subset \A$ is a set, we set $G_\Omega=\bigcap_{x\in \Omega} G_x$\index{G@ $G_x$, $G_\Omega$} and if $\Omega$ is a filter, we set $G_\Omega=\bigcup_{S\in \Omega} G_S$. Note that in \cite{R16}, the definition of   $G_x$ is much more complicated (see \cite[D\'efinition 4.13]{R16}). However it is equivalent to this one by \cite[Proposition 4.14]{R16}.

A filter is said to have a ``good fixator'' if it satisfies \cite[D\'efinition 5.3]{R16}. There are many examples of filters with good fixators [\lc 5.7]: points, preordered segment germs, local facets, sectors, sector germs, $\A$, walls, half apartments, \ldots For such a filter $\Omega$, we have: \[G_{\QO}=U_{\QO}^ {pm+}.U_{\QO}^ {nm-}.N_{\QO}=U_{\QO}^ {nm-}.U_{\QO}^ {pm+}.N_{\QO}.\]

\par We then have: $U_{\QO}^ {pm+}=G_{\QO}\cap U^+=: U^+(\QO)$ and $U_{\QO}^ {nm-}=G_{\QO}\cap U^-=: U^-(\QO)$, as $U^-\cap U^+.N=U^+\cap N=\set1$, by [\lc Remarque 3.17]  and \cite[1.2.1 (RT3)]{Re02}.

\par Note that for the sector germ $\Omega=\g Q_{+\infty}$, $U_\Omega^{nm-}=\set{1}$, $N_\Omega=N_\A=T_0$ and $U_{\Omega}^{pm+}=U^+$. So $G_{\g Q_{+\infty}}=T_0 U^+$. Similarly, $G_{\g Q_{-\infty}}=T_0 U^-$. 

When $\Omega$ is a   local facet, $G_\Omega$ is called a parahoric subgroup (this is a little more general than in \cite{BrT72}). 

When $\Omega=C_0^+=germ_0(C^v_f)$ is the (fundamental) positive local chamber in $\A$, $I=G_\Omega$ is called the (fundamental) Iwahori subgroup. \index{I@$I$}

\par 

\medskip
\parni{\bf 4)} For $\QO$ a set or a filter, one defines:

\par  $U_{\QO}=\langle U_{\qa,\QO} \mid \qa\in\QF \rangle$\quad, \quad $U^\pm_{\QO}=U_{\QO}\cap U^\pm$\quad and \quad$U_{\QO}^ {\pm\pm}=\langle U_{\qa,\QO} \mid \qa\in\QF^\pm \rangle$.

\parni Then one has $U_{\QO}=U_{\QO}^-.U_{\QO}^+.N_{\QO}\zu=U_{\QO}^+.U_{\QO}^-.N_{\QO}\zu$, where $N_{\QO}\zu=U_{\QO}\cap N\subset N_{\QO}$, see [\lc 4.6.1].
And also $U_{\QO}^ {+}\subset U_{\QO}^ {pm+}$, $U_{\QO}^ {-}\subset U_{\QO}^ {nm-}$, see [\lc 4.3.2 and 4.5.3].

\par The inclusion $U_{\QO}^ {\pm\pm}\subset U_{\QO}^ {\pm}$ is clear, but it is not always an equality, see [\lc 4.3.2 and 4.12.3.a].

 When $\QO$ is narrow and has a good fixator, then $G_{\QO}=U_{\QO}^ {pm+}.U_{\QO}^ {-}.N_{\QO}=U_{\QO}^ {nm-}.U_{\QO}^ {+}.N_{\QO}$, see [\lc 4.13.4 and 5.3]. 
 
 \subsection{Masure associated with $G$}\label{sub_Masure}
 
 \subsubsection{Masure}\label{n1.3.1}
 
 We now define the masure  $\shi=\shi(\g G,\shk,\omega)$. As a set, $\shi=G\times \A/\sim$, where $\sim$ is defined as follows: \[\forall (g,x),(h,y)\in G\times \A, (g,x)\sim (h,y)\Leftrightarrow \exists n\in N\ \mid y=\nu(n).x\text{ and }g^{-1}hn\in G_x.\]
  We regard $\A$ as a subset of $\shi$ by  identifying $x$ and $(1,x)$, for $x\in \A$. The group $G$ acts on $\shi$ by $g.(h,x)=(gh,x)$, for $g,h\in G$ and $x\in \A$. An apartment is a set of the form $g.\A$, for $g\in G$. The stabilizer of $\A$ in $G$ is $N$ and if $x\in \A$, then the fixator of $x$ in $G$ is $G_x$. More generally, when $\Omega\subset \A$ has a good fixator, then $G_{\Omega}$ is the fixator of $\Omega$ in $G$ and $G_{\QO}$ permutes transitively the apartments containing $\QO$.
   If $A$ is an apartment, we transport all the notions that are preserved by $W$ (e.g segments, walls, facets, chimneys, etc.) to $A$.  Then by \cite[Corollary 3.7]{hebert2022new}, $\shi$ satisfies the following properties: 
 
 \medskip
 
 (MA II) : Let $A,A'$ be two apartments.  Then $A\cap A'$ is a finite intersection of half-apartments and there exists $g\in G$ such that $g.A=A'$ and $g$ fixes $A\cap A'$.

\medskip

\par (MA III): if $\g R$ is the germ of a splayed chimney and if $F$ is a facet or a germ of a chimney, then there exists an apartment containing $\g R$ and $F$.

\medskip

We also have: \begin{itemize}
\item The stabilizer of $\A$ in $G$ is $N$ and $N$ acts on $\A\subset \shi$ via $\nu$.

\item If $\QO$ has a good fixator, $N.G_{\QO}=\set{g\in G \mid g.\QO\subset\A}$.

\item  The group $U_{\qa,r}:=\{x_\alpha(u)\mid u\in \shk, \omega(u)\geq r\}$\index{U@$U_{\alpha,r}$}, for $\qa\in\QF,r\in\Z$, fixes the half-apartment $D(\qa+r)=\set{x\in \A\mid \qa(x)+r\geq0}$. 
 It is actually the fixator in $U_{\qa}$ of any point in the wall $M(\qa+r)=\set{x\in \A\mid \qa(x)+r=0}$. 
 It acts simply transitively on the set of apartments in $\SHI$ containing $D(\qa+r)$.
\end{itemize}

\par For $x,y\in\SHI$, we write $x\leq y$ (\resp $x\mathring{<}y$, $x\mathring{\leq}y$) if there exists $g\in G$ such that $g.y,g.x\in \A$ and $g.y-g.x\in \sht$ (\resp $g.y-g.x\in \mathring{\sht}$, $g.y-g.x\in \mathring{\sht}\cup \{0\}$). \index{$\leq,\mathring{<},\mathring{\leq}$}
Note that by (MA II), if $x\leq y$, then for all $g'\in G$ such that $g'.x,g'.y\in \A$, we have $g'.y-g'.x\in \sht$.   The relation $x\leq y$ {(\resp $x\stackrel{\circ}{\leq}y$)} is $G$-invariant and  is a preorder relation by \cite[Th\'eor\`eme 5.9]{R11}; in particular it is transitive.

\medskip

Let $H$  be a subgroup of $G$. An $H$-apartment is a set of the form $h.\A$, where $h\in H$.  We denote by $\sha(H)$  the set of $H$-apartments. Note that implicitly, an apartment is a $G$-apartment.  As we shall see (Corollary~\ref{cor_Masure_GR}),  every point of $\shi$ lies in a $G_\shr$-apartment. However,  $\sha(G_\shr)$\index{A@$\sha(G),\sha(G_\shr)$}  can be strictly smaller than $\sha(G)$. 

\par Let $\Omega_1,\Omega_2$ be two filters on $\shi$. We say that $\Omega_1$ and $\Omega_2$ are $H$-friendly\index{h@$H$-friendly} if there exists $A\in \sha(H)$ containing $\Omega_1\cup \Omega_2$.

\medskip

\par Let $H$ be a subgroup of $G$. 
Then one may consider the semigroups $H^+:=\set{g\in H\mid g.0\geq0}$ and $H^-:=\set{g\in H\mid g.0\leq0}$. 
We will often apply this definition with $H=G$ or  $H=G_\shr$ and consider the semigroups $G^+$\index{G@$G^+,G_\shr^+$} and $G_{\shr}^+$. %\index{$G_\shr^+$}.}

\begin{rema}
In \ref{ss_Root_generating}, we made the assumption that the family $(\alpha_i^\vee)_{i\in I}$ is free. This is more convenient and it enables us to use the results of \cite{hebert2022new} for example. However this assumption is not necessary to define Kac-Moody groups (see \cite{Mar18} for example). For example,  $G:=\mathrm{SL}_n(\shk[u,u^{-1}])\rtimes \shk^*$ is naturally equipped with the structure of a Kac-Moody group associated with a root generating system $\mathcal{S}$ having nonfree coroots. This group is particularly interesting for examples, since it is one of the only Kac-Moody groups in which we can make explicit computations.  To handle this kind of group, a solution is to consider a central extension $\widetilde{G}$ of $G$ having free coroots. Then if $\widetilde{\shi}$ is the masure associated with $\widetilde{G}$, we have a natural surjection $\pi:\widetilde{\shi}\twoheadrightarrow \shi$, that is compatible with the actions of $G$ and $\widetilde{G}$. Then we can deduce properties of $\shi$ and $\widetilde{G}$ from properties of $\widetilde{\shi}$ and $\widetilde{G}$.  We detail this reasoning in section~\ref{s6} for the case $n=2$. It should be possible to study  groups with  non necessarily free coroots in general with the same reasoning, using the results of \cite[7.4.5]{Mar18}.
\end{rema}

\subsubsection{Decompositions of subgroups of $G$, retractions}\label{ss_Retractions}

Let $H$ be a subgroup of $G$ and $E_{1},E_{2}$ be two subsets or filters in $\A$.
We write $N_{H}(\A)$ the stabilizer of $\A$ in $H$ and $H_{E_{i}}$ the (pointwise) fixator of $E_{i}$ in $H$.
We are interested in decompositions $H=H_{E_{1}}.N_{H}(\A).H_{E_{2}}$ or $H^+=H_{E_{1}}.(N_{H}(\A)\cap H^+).H_{E_{2}}$, where $H^+$ is a subsemigroup of $H$.
We say that it is a Bruhat (\resp Iwasawa; mixed Iwasawa) decomposition if the pair $(E_{1},E_{2})$ is made of two local chambers (\resp a local chamber and a sector germ; a local chamber and a chimney germ).

\par There is a geometric translation of such a decomposition, when each $H_{E_{i}}$ is transitive on the set of apartments in $\sha(H)$ containing $E_{i}$ (here $\sha(H)=\set{h.\A \mid h\in H}$). 
Then such a decomposition  (involving $H$ and not $H^+$) means that, for any $h_{1},h_{2}\in H$, the subsets or filters $h_{1}E_{1}$ and $h_{2}E_{2}$ are in a same apartment of $\sha(H)$ (they are ``$H-$friendly''). Actually, the axiom (MA III) is a geometric translation of decompositions of $G$.

\par  Let $A$ be an apartment of $\shi$ and  $\g Q$ be a sector germ of $A$. Let $x\in \shi$. Then by (MA III), there exists an apartment $B$ containing $x$ and $\g Q$. By (MA II), there exists $h\in G$ such that $h.B=A$ and $h$ fixes $A\cap B$. Then $h.x$ does not depend on the choices of $B$ and $h$ and we set $\rho_{A,\g Q}(x)=h.x$. The map $\rho_{A,\g Q}:\shi\twoheadrightarrow A$ is the retraction onto $A$ centred at $\g Q$. When $\g Q=\g Q_{\pm \infty}$, i.e when $\g Q$ is the germ at infinity of $\pm C^v_f$ and $A=\A$, we write $\rho_{\pm \infty}$\index{r@$\rho_{\pm \infty}$} instead of $\rho_{\A,\g Q_{\pm \infty}}$.

\subsection{A precise decomposition of $G_{\QO}$, for $\QO$ a local chamber.}\label{3.2}%%%%%%%%%%%%

 \begin{prop*}%%%%%%%%%%%%%%
 Let $\QO\subset \A \subset \SHI$ be a non empty set or filter.
 Suppose that $\QO$ is narrow, almost open and has a good fixator (for example $\QO$ is a local chamber). Then:
 
 \par\qquad\qquad $G_{\QO}=U_{\QO}^ {+}.U_{\QO}^ {-}.T_{0}=U_{\QO}^ {+}.T_{0}.U_{\QO}^ {-}=U_{\QO}.T_{0}=\langle T_{0}, (U_{\qa,\QO})_{\qa\in\QF} \rangle$,
 
 \parni actually $U_{\QO}^ {pm+}=U_{\QO}^ {+}=U^+\cap G_{\QO}=:U^+(\QO)$ and $U_{\QO}^ {nm-}=U_{\QO}^ {-}=U^-\cap G_{\QO}=:U^-(\QO)$.
 \end{prop*}

 \begin{proof} By \ref{3.1} and the fact that $T_{0}$ normalizes $U_{\QO}^ {\pm},U_{\QO}^ {pm+},U_{\QO}^ {nm-}$, one has $G_{\QO}=U_{\QO}^ {pm+}.U_{\QO}^ {-}.T_{0}=U_{\QO}^ {pm+}.T_{0}.U_{\QO}^ {-}=U_{\QO}^ {nm-}.T_{0}.U_{\QO}^ {+}=U_{\QO}^ {+}.T_{0}.U_{\QO}^ {nm-}$.
 But $G$ is a Kac-Moody group, so one has the Birkhoff-Borel decomposition $G=\sqcup_{n\in N}\, U^+.n.U^-$ and the uniqueness result $U^-\cap N.U^+=U^+\cap N=\set1$, see  \cite[Remark 3.17]{R16} and \cite[1.2.4 (i) + (RT3)]{Re02}.
 In particular in the subset $U^+.T.U^-$ of $G$, the decomposition is unique.
 So the third and fifth formula for $G_{\QO}$ above give $U_{\QO}^ {pm+}=U_{\QO}^ {+}$ and $U_{\QO}^ {nm-}=U_{\QO}^ {-}$.
 \end{proof}
 
 \begin{NB} The proposition above is a simple improvement of Property 4.13.4 in \cite{R16} when $\Omega$ is moreover almost open.
But the trick below in Consequence \ref{3.3} 2), enables us to get the decomposition of $G_x$ guessed in Property 4.13.5 of \cite{R16}.
 \end{NB}

 \subsubsection{Consequences}\label{3.3}%%%%%%%%%%%%%
\par\quad\; 1) In particular the Iwahori group $I=G_{C_{0}^+}$ (fixator in $G$ of the fundamental local chamber $C_{0}^+=germ_{0}(C\zv_{f})$) is $\langle T_{0}, (U_{\qa,C_{0}^+})_{\qa\in\QF} \rangle$.
 This is the same definition as in \cite{BrKP16} (given there in the untwisted affine case).
 This result was also proved in \cite[7.2.2]{BPGR16}, using the results of \cite{BrKP16}.
 We get here a direct proof and a more general result.
 
 \par 2) Let $x\in\A$ and $C^\pm_{x}=germ_{x}(x\pm C\zv_{f})$ be the two opposite chambers at $x$ with respective directions $\pm C\zv_{f}$.
 Then $U^ {ma+}_{C^+_{x}}=U^ {ma+}_{x}$, hence $U^ {pm+}_{C^+_{x}}=U^ {pm+}_{x}$.
 So $U^ {pm+}_{x}=U^ {pm+}_{C^+_{x}}=U^ {+}_{C^+_{x}}\subset U^+_{x}\subset U^ {pm+}_{x}$ and $U^ {pm+}_{x}=U_{x}^+=G_{x}\cap U^+$.
 Similarly $U^ {nm-}_{x}=U_{x}^-=G_{x}\cap U^-$.

 \par So (as $x$ has a good fixator) we get $G_{x}=U_{x}^ {pm+}.U_{x}^ {nm-}.N_{x}=U_{x}^ {+}.U_{x}^ {-}.N_{x}=U_{x}^ {-}.U_{x}^ {+}.N_{x}=U_{x}.N_{x}=\langle N_{x},(U_{\qa,x})_{\qa\in\QF} \rangle$.
 
 \par When $x$ is a special point $N_{x}/T_{0}=W\zv$ and $N_{x}=N_{x}\zu.T_{0}$, so $G_{x}=\langle T_{0},(U_{\qa,x})_{\qa\in\QF} \rangle$. 
 Moreover $G_{x}=P_{x}^ {min}=P_{x}^ {pm}=P_{x}^ {nm}$ with the notations of \cite[4.6.a]{R16}.

\begin{lemm*}\label{lem_Fixator_Ain}
 Let $\A_{in}=\bigcap_{\alpha\in \Phi} \ker(\alpha)= \bigcap_{i\in I} \ker(\alpha_i)$ and $\Omega$ be a filter on $\A$. Then we have $G_\Omega=G_{\Omega+\A_{in}}$.

 \end{lemm*}

\begin{proof}
We begin by the case where $\Omega=\{x\}$, for some $x\in \A$. Let $y\in x+\A_{in}$.   Then we have $U_{\alpha,y}=U_{\alpha,x}$ for all $\alpha\in \Phi$, since $\alpha(x)=\alpha(y)$.  Let $n\in N_x$ and $\underline{w}\in W$ be the automorphism of $\A$ induced by $n$. Write $\underline{w}=a+w$, where $a\in Y$ and $w\in W\zv$. Then we have $a=x-w.x$. As $W\zv$ fixes $\A_{in}$, we deduce $y-w.y=a$ and hence $\underline{w}$ fixes $y$. Otherwise said, $n$ fixes $y$ and we have  $N_x\subset N_y$. By symmetry, $N_x=N_y$ and thus $G_x=G_y$. Let now $\Omega$ be a nonempty set. Then $G_\Omega=\bigcap_{x\in \Omega}G_x=\bigcap_{x\in\Omega}\bigcap_{y\in x+\A_{in}} G_y=G_{\Omega+\A_{in}}$. 

Assume now that $\Omega$ is a filter.  Let $S$ be a subset of $\A$. Then $S\in \Omega+\A_{in}$ if and only if there exists $S'\in \Omega$ such that  $S=S'+\A_{in}$. Therefore \[G_{\Omega+\A_{in}}=\bigcup_{S\in \Omega+\A_{in}}G_S=\bigcup_{S'\in \Omega} G_{S'+\A_{in}}=G_\Omega.\]
\end{proof}

 \par 3) In particular the fixator $K=G_{0}$ of the origin  point  in $\A$ is $K=G_{0}=\langle T_{0},(U_{\qa,0})_{\qa\in\QF} \rangle$.
 This is the same definition as in \cite{BrKP16} (given there in the untwisted affine case).
 This result was also proved in \cite[Remark 3.4]{GR14}, using the results of \cite{BrKP16}.
 We get here a direct proof and a more general result.

 \par 4) Let $x\in\A$ and $F_{x}\subset \ov{C^+_{x}}$ be a segment germ or a local facet.
 Then $U^ {ma+}_{C^+_{x}}=U^ {ma+}_{F_{x}}$ hence $U^ {pm+}_{C^+_{x}}=U^ {pm+}_{F_{x}}$.
  So $U^ {pm+}_{F_{x}}=U^ {pm+}_{C^+_{x}}=U^ {+}_{C^+_{x}}\subset U^+_{F_{x}}\subset U^ {pm+}_{F_{x}}$ and $U^ {pm+}_{F_{x}}=U_{F_{x}}^+=G_{F_{x}}\cap U^+$.
 If $F_{x}\subset \ov{C^-_{x}}$, then we get  $U^ {nm-}_{F_{x}}=U_{F_{x}}^-=G_{F_{x}}\cap U^-$.
 But we do not get the two series of equalities in general.

\subsubsection{Generalization of Proposition~\ref{3.2} to the almost-split case} 

In  \ref{3.3}, we obtained a decomposition of the fixator $G_{\Omega}$ of certain filters $\Omega\subset \A$ and deduced a decomposition of  $G_x$, for $x\in \A$. The advantage of these decompositions is that they involve only the minimal Kac-Moody group $G$ and not its completions. As this result could be interesting on its own, we extend it to almost-split Kac-Moody groups below. This result will not be used in the sequel.

 We consider an almost split Kac-Moody group $\g G$ over a field $\shk$ endowed with a real valuation $\qo$.
 We suppose that $\g G$ is quasi-split over a tamely ramified extension and, if $\g G$ is not split, that the valuation $\qo$ may be extended functorially and uniquely to any separable extension of $\shk$ (\eg $\qo$ is complete).
 Then, by \cite[6.9]{R17}, there exists a  masure $\SHI$ with a strongly transitive action of $G=\g G(\shk)$ and the fixators $G_{x}$ of the points in the canonical apartment $\A$ are a very good family of parahorics.
 For $\QO\subset\A$, we write $U_{\QO}=\langle U_{\qa,\QO} \mid \qa\in\QF \rangle\subset G_{\QO}$, $U^\pm_{\QO}=U_{\QO}\cap U^\pm \subset G_{\QO}$ and $N_{\QO}=N\cap G_{\QO}$, where $G_{\QO}$ is the fixator of $\QO$
in $G$.

\begin{prop*} For any point $x\in\A$, one has $G_{x}=U^+_{x}.U^-_{x}.N_{x}=U^-_{x}.U^+_{x}.N_{x}=U_{x}.N_{x}$, $U^\pm_{x}=G_{x}\cap U^\pm$.
And for any local chamber $\QO$ in $\A$, one has $G_{\QO}=U^+_{\QO}.U^-_{\QO}.N_{\QO}=U^-_{\QO}.U^+_{\QO}.N_{\QO}=U_{\QO}.N_{\QO}$, $U^\pm_{\QO}=G_{\QO}\cap U^\pm$.
\end{prop*}

\begin{NB} This result is also true if $\QO\subset\A$ is narrow, non empty, almost open, with good fixator.
\end{NB}

\begin{proof} When $\g G$ is actually split, the proof is exactly the same as above in \ref{3.1}, \ref{3.2} and \ref{3.3} (1), (2), (3).
In the general almost split case, we have mainly to replace $T$ by the centralizer $Z$ of a maximal split subtorus of $\g G$ [\lc 2.7].
For any vectorial chamber $C\zv=\pm wC\zv_{f} \subset \A$,  we write $U(C\zv)=wU^\pm w^ {-1}$ and $U_{\QO}(C\zv)=U_{\QO}\cap U(C\zv)$.
When $\QO\subset\A$ has a good fixator, we have $G_{\QO}=U_{\QO}^ {(+)}.U_{\QO}^ {(-)}.N_{\QO}=U_{\QO}^ {(-)}.U_{\QO}^ {(+)}.N_{\QO}$, where $U_{\QO}^ {(\pm)}=G_{\QO}\cap U^\pm\supset U^\pm_{\QO}$ [\lc 4.4.b, 4.5].
We shall use this for $\QO$ a point or a local chamber.

\par When $\QO$ is a local chamber, $N_{\QO}=Z_{0}:=Z\cap G_{\QO}$, $G_{\QO}=U_{\QO}^ {(+)}.Z_{0}.U_{\QO}^ {(-)}$ and the Iwasawa decomposition [\lc 4.3.3] gives $G=U^+.N.U_{\QO}$, so $G_{\QO}=(U^+.N\cap G_{\QO}).U_{\QO}$.
Now, by the uniqueness in the Birkhoff-Borel decomposition [\lc 1.6.2], $U^+.N\cap G_{\QO}=U_{\QO}^ {(+)}.Z_{0}.U_{\QO}^ {(-)}\cap U^+.N.\set1=U_{\QO}^ {(+)}.Z_{0}$; so $G_{\QO}=U_{\QO}^ {(+)}.Z_{0}.U_{\QO}$.
But, for $C\zv_{1},C\zv_{2}\subset \A$ adjacent chambers along the wall $\ker\qa$ (with $\qa(C\zv_{1})\geq0$), we get from [\lc 4.4.a] $U_{\QO}((C\zv_{1})):=G_{\QO}\cap U(C\zv_{1})=U_{\qa,\QO}\ltimes(G_{\QO}\cap U(C\zv_{1}) \cap U(C\zv_{2}))$.
From this we deduce, as in \cite[Prop. 3.4]{GR08}, that $U_{\QO}(C\zv_{1}).U_{\QO}(-C\zv_{1}).Z_{0}$ is independent of the choice of the (positive) chamber $C\zv_{1}$ and $U_{\QO}\subset U_{\QO}(C\zv_{1}).U_{\QO}(-C\zv_{1}).Z_{0}=U_{\QO}^+.Z_{0}.U_{\QO}^-$.
So $G_{\QO}=U_{\QO}^ {(+)}.Z_{0}.U_{\QO}^-$ and, symmetrically, $G_{\QO}=U_{\QO}^ {+}.Z_{0}.U_{\QO}^{(-)}$.
The uniqueness in the Birkhoff-Borel decomposition gives $U_{\QO}^ {(\pm)}=U_{\QO}^\pm$, hence $G_{\QO}=U^+_{\QO}.U^-_{\QO}.N_{\QO}=U^-_{\QO}.U^+_{\QO}.N_{\QO}=U_{\QO}.N_{\QO}$.

\par For $x\in\A$ and $C_{x}^\pm=germ_{x}(x\pm C\zv_{f})$, we have $U_{x}^ {(\pm)}=U_{C_{x}^\pm}^ {(\pm)}$ \cite[beginning of 4.5.3]{R17}.
So $U_{x}^ {(\pm)}=U_{C_{x}^\pm}^ {(\pm)}=U_{C_{x}^\pm}^ {\pm}\subset U_{x}^\pm \subset U_{x}^ {(\pm)}$ and $U_{x}^ {(\pm)}=U_{x}^\pm$.
Now $G_{x}=U^+_{x}.U^-_{x}.N_{x}=U^-_{x}.U^+_{x}.N_{x}$ is equal to $U_{x}.N_{x}$, as $U_{x}^\pm \subset U_{x} \subset G_{x}$.
\end{proof}

%%%%%%%%%%%%%%%%%%%%%%%%%%%%%%%%%%%%%%%%%%%%

\section{Study of $G_\shr$, for $\shr$ a dense subring of a valued field $\shk$}\label{s2}

Let $\shr$ be a dense subring of $\shk$  (for the main applications, we make the additional assumption~\eqref{eq_Assumption_R}). In this section, we study decompositions of $G_\shr$. Our main results are the Bruhat decomposition and the Iwasawa decompositions of $G_\shr$ (see Corollaries~\ref{cor_Bruhat} and \ref{cor_Iwasawa}).  To do that, we study the action of $G_\shr$ on the masure $\shi$ of $G$. Given a subset $P$ of a $\shk$-apartment, we study the existence of an $\shr$-apartment containing $P$ (see Theorem~\ref{thm_Bounded_sets_R_friendly}). We then deduce the desired decompositions from the  corresponding decompositions of $G$.

\subsection{Commutators in $\g U^ {ma+}$} %%%%%%%%%%%%%%%
\label{2.2}

Let $\beta\in \Phi$. We want to understand the actions of  $x_\beta(u)$ on $\shi$, for $u\in \shk$ satisfying $\omega(u)\gg 0$. To do so, we begin by studying commutators in $\g U^{ma+}$.

\par For $\qa,\qb\in\QD^+$, one would like a formula for the commutators in $[\g U_{\qa},\g U_{\qb}]$.

\par  Assume $\alpha$ and $\beta$ are not collinear. Let $\Psi'=\set{p\qa+q\qb\in \QD^+ \mid p\geq1, q\geq0}$ and $\Psi=\Psi'\cup ((\N_{>0}\qb)\cap\QD)$.
They are closed subsets of $\QD^+$.
Moreover $\Psi'$ is an ideal   of $\Psi$; so $\g U^ {ma}_{\Psi'}(R)\triangleleft \g U^ {ma}_{\Psi}(R)$ by \cite[Lemma 3.3]{R16}. 

\par In particular:

\par\qquad\qquad\qquad $X_{\qb}(u_{\qb}).X_{\qa}(u_{\qa}).X_{\qb}(u_{\qb})^ {-1}=\prod^ {p\geq1,q\geq0}_{p\qa+q\qb\in\QD}\, X_{p\qa+q\qb}(v_{p\qa+q\qb}).$ 
\parni One chooses an order such that \eg the height of $p\qa+q\qb$ is increasing and $u_{\qa}\in \g g_{\qa,\Z}\otimes R$, $u_{\qb}\in \g g_{\qb,\Z}\otimes R$.
Then $v_{p\qa+q\qb}\in \g g_{p\qa+q\qb,\Z}\otimes R$.

\par We now restrict  to the case where $\qb$ is real.

\begin{prop*} Let $\alpha\in \Delta^+$,  $\beta\in \Phi^+$,  $c_\alpha\in \g g_{\alpha,\Z}$  and $u,v\in R$.   Then \[x_{\qb}(u).X_{\qa}(v.c_{\qa}).x_{\qb}(-u)=\prod^ {p\geq1,q\geq0}_{p\qa+q\qb\in\QD}\, X_{p\qa+q\qb}(v^pu^q.c_{p\qa+q\qb}),\]

\parni for some $c_{p\qa+q\qb}\in \g g_{p\qa+q\qb,\Z}$ independent of $u$ and $v$.
\end{prop*}

\begin{NB} 1) For $p=1,q=0$, $c_{p\qa+q\qb}$ is certainly equal to $c_{\qa}$, \ie the factor on the left of the right hand side is $X_{\qa}(v.c_{\qa})$.
This is suggested by the notation, but not proven here.

\par 2) When $\qa$ is imaginary and $p\geq2, q=0$, one should have $c_{p\qa}=0$.
But we do not prove this here.
\end{NB}

\begin{proof} 
If  $\alpha$ and $\beta$ are collinear, then $\alpha=\beta$, $\{(p,q)\in \N^*\times \N\mid p\alpha+q \beta\in \Delta^+\}=\{(1,0)\}$ and $x_\beta(u)$ and $x_\alpha(v)$ commute so the formula is clear in this case. We now assume that $\alpha$ and $\beta$ are not collinear.
From the above formula, $x_{\qb}(u).X_{\qa}(v.c_{\qa}).x_{\qb}(-u)=\prod^ {p\geq1,q\geq0}_{p\qa+q\qb\in\QD}\, X_{p\qa+q\qb}(v_{p\qa+q\qb}(u,v))$, with $v_{p\qa+q\qb}(u,v)\in  \g g_{p\qa+q\qb,\Z}\otimes R$ and the map $R^2\to \g g_{p\qa+q\qb,\Z}\otimes R, (u,v)\mapsto v_{p\qa+q\qb}(u,v)$ is polynomial (defined over $\Z$), as we have unipotent groups defined over $\Z$ by \cite[\S 3.4]{R16}.
One will determine this polynomial map by using $R=\C$ and $u,v\in\C^*$ (we can assume $u,v$ algebraically independent over $\Q$).

\par There exists $t\in\g T(\C)$ such that $\qa(t)=v$ and $\qb(t)=u$, hence $(p\qa+q\qb)(t)=v^pu^q$.
Following the first paragraph of \cite[\S 3.5]{R16}, one has $t.X_{\qg}(v_{\qg}).t^ {-1}=X_{\qg}(\qg(t).v_{\qg})$ for $\gamma\in \Delta^+$ and $v_\gamma\in \g g_{\gamma}$. Hence:

\parni $x_{\qb}(u).X_{\qa}(v.c_{\qa}).x_{\qb}(-u)=t.x_{\qb}(1).X_{\qa}(c_{\qa}).x_{\qb}(-1).t^ {-1}=\prod^ {p\geq1,q\geq0}_{p\qa+q\qb\in\QD}\, t.X_{p\qa+q\qb}(v_{p\qa+q\qb}(1,1)).t^ {-1}=X_{p\qa+q\qb}(v^pu^q.c_{p\qa+q\qb})$, if one writes $c_{p\qa+q\qb}=v_{p\qa+q\qb}(1,1)\in  \g g_{p\qa+q\qb,\Z}$.
\end{proof}

\begin{lemm}\label{2.4}%%%%%%%%%%%% 
One writes $\ov \sht$ the closed Tits cone in $\A=Y\otimes\R=\g h_{\R}$, $\ov \sht^\vee$ its analogue in the dual $\A^*=X\otimes\R=\g h_{\R}^*$ and $ \ov Z=\ov{conv}(\QD^+_{im}\cup\set0)$\index{Z@$\ov Z$} the closed convex hull in $\A^*$ (some notations come from \cite[\S 5.8]{K90}). Then,

\par (a) If $\QD$ is of indefinite type, for any $\qa\in\QF=\QD_{re}$, one has $\qa^\vee\not\in \pm \ov \sht$, %%%\index{D@$\Delta_{re}$}

\par (b) If $\QD$ is of indefinite type, for any $\qa\in\QF=\QD_{re}$, one has $\qa\not\in \pm \ov \sht^\vee$,

\par (c) $\ov Z\subset -\ov \sht^\vee$,

\par (d) $\QD_{re}\cap\pm \ov Z=\emptyset$.
\end{lemm}

\begin{proof} a) By \cite{K90}, 5.8.1 and Theorem 5.6.c, one has $\qa_{i}^\vee\not\in  \ov \sht$, $\forall i$. Conjugating by the Weyl group, we get (a).
Now (b) is the result dual to (a).

\par c) One may suppose $\QD$ indecomposable. The result is clear if $\QD$ is of finite type ($\ov Z=\set0$).
In the affine or indefinite case, one considers $K=\set{\qa\in\sum\N\qa_{i} \mid \qa(\qa_{j}^\vee)\leq0, \forall j \text{\ and\ } supp(\qa) \text{\ connected}}$ \cite[5.3]{K90}.
By \cite[5.8 c) or b)]{K90} $K\subset -\ov \sht^\vee$.
But $\QD^+_{im}=\cup_{w\in W\zv}\,w(K)$ by \cite[5.4]{K90}; so $\QD^+_{im}\subset -\ov \sht^\vee$ and $\ov Z \subset -\ov \sht^\vee$.

\par d) One may suppose $\QD$ indecomposable. The result is clear if $\QD$ is of finite type ($\ov Z=\set0$) or of affine type ($\ov Z=[0,+\infty[\qd$ and no real root is collinear to $\qd$).
In the indefinite case (d) is a consequence of (b) and (c).
\end{proof}

\subsection{Study of the action of root subgroups on $\shi$}

The aim of this subsection is to prove the following lemma. It will enable us to obtain decompositions of $G_\shr$ from decompositions of $G$.
In the reductive case, this lemma is already known, see \cite[Prop. 7.4.33]{BrT72}. The difficulty here is that the number of roots is infinite.

\begin{lemm}\label{lem_Ua++_fix}
Let $x\in \shi$. Then there exists $a\in \A$ such that $U_{a}^{pm+}$ fixes $x$. In particular, if $\alpha\in \Phi^+$, then for $u\in \shk$ such that $\omega(u)\gg 0$,  $x_\alpha(u)$ fixes $x$. 
\end{lemm}

Recall that $\mathrm{ht}:Q\otimes \R\rightarrow \R$ is defined as follows: if $(n_i)\in \R^I$, then $\mathrm{ht}(\sum_{i\in I} n_i\alpha_i)=\sum_{i\in I} n_i$.  

\begin{lemm}\label{lem_LB}
Let $\beta\in \Phi^+$. Then $\inf \{\mathrm{ht}(\frac{\tau}{q})\mid(q,\tau)\in \N^*\times( Q_+\setminus\{0\}), \tau+q\beta\in \Delta\}>0$.
\end{lemm}

\begin{proof}
Suppose this is not the case and choose $(q_n)\in (\N^*)^\N$ and $(\tau_n)\in (Q_+\setminus\{0\})^\N$ such that for $n\in \N$, $q_n \beta+\tau_n\in \Delta$ and $\frac{1}{q_n}\mathrm{ht}(\tau_n)\underset{n\to +\infty}{\longrightarrow} 0$. Then $\frac{1}{q_n}\tau_n\underset{n\to +\infty}{\longrightarrow} 0$. Up to choosing a subsequence of $((q_n,\tau_n))_{n\in \N}$, we may assume that one of the following two possibilities holds:\begin{itemize}
\item $q_n \beta +\tau_n\in \Delta_{im}^+$, for all $n\in \N$. In this case,  $\beta+\frac{1}{q_n}\tau_n\in \ov Z=\ov{conv}(\QD^+_{im}\cup\set0)$.
So $\qb\in\ov Z$: this is impossible since $\QD_{re}\cap\ov Z=\emptyset$ (see Lemma \ref{2.4}).

\item $q_n \beta+\tau_n\in \Delta_{re}^+$, for all $n\in \N$. Then the rays \[\R_{+}^*\left(q_n \beta+\tau_n\right)=\R^*_+ \left(\beta+\frac{1}{q_n}\tau_n\right) ,\] which are generated by  real roots, converge to  the ray $\R_{+}^*.\qb$. 
By \cite[Lemma 5.8]{K90} one has $\qb\in \ov Z$: this is impossible (similarly as above).
\end{itemize}
\end{proof}

\begin{lemm}\label{lem_Conj_unp}
Let $b\in \A$, $\beta\in \Phi^+$ and $v\in \shk$.  Then there exists $a\in b-C^v_f$ such that $x_{\beta}(v)U_{a}^{pm+}x_{\beta}(-v)\subset  U_b^{pm+}$. 
\end{lemm}

\begin{proof}
 Let $a\in \A$ and $h\in U_a^{pm+}$. 
By definition of $U_a^{pm+}$, we can write $h=\prod_{\alpha\in \Delta^+} X_\alpha( u_\alpha.c_\alpha)$, where $c_\alpha\in \g g_{\alpha,\Z}$, $u_\alpha\in \shk$ and $\alpha(a)+\omega(u_{\qa})\geq 0$ for all $\alpha\in \Delta^+$, where $\Delta_+$ is equipped with a total order such that the height is an increasing map for $\leq$. 
 
Let $\alpha\in \Delta^+$. Set \[E_\alpha=\{(p,q)\in \N^*\times \N\mid p\alpha+q\beta\in \Delta^+\}.\] We equip $E_\alpha$ with a total order $\leq$ such that for all $(p,q),(p',q')\in E_\alpha$, \[(p,q)\leq (p',q')\Rightarrow \mathrm{ht}(p\alpha+q\beta)\leq \mathrm{ht}(p'\alpha+q'\beta).\]  By \ref{2.2} Proposition, we have \begin{equation}\label{eq_commut}
x_{\beta}(v) X_\alpha(u_\alpha.c_\alpha) x_{\beta}(-v)=\prod_{(p,q)\in E_\alpha}X_{p\alpha+q\beta}(u_\alpha^p v^qc_{(p,q),\alpha}),
\end{equation}
 
where  $c_{(p,q),\alpha}\in \g g_{p\alpha+q\beta,\Z}$, for $(p,q)\in E_\alpha$.
 
 Therefore  \begin{equation}\label{eq_commut2}
x_\beta(v)hx_\beta(-v)=\prod_{\alpha\in \Delta^+} \prod_{(p,q)\in E_\alpha}X_{p\alpha+q\beta}(u_\alpha^p v^qc_{(p,q),\alpha})
\end{equation} (the right hand side of this product is well-defined, as for any $m\in\N$, there exist at most finitely many triples $(\alpha,p,q)$ with $\alpha\in \Delta^+$ and $(p,q)\in E_\alpha$  satisfying $\mathrm{ht}(p\alpha+q\beta)=m$).

Let  $\alpha\in \Delta^+$. Set \[\Omega_\alpha(u_\alpha)=\bigcap_{(p,q)\in E_{\qa}}\{a'\in \A\mid  (p\alpha+q\beta)(a')+\omega(u_\alpha^p v^q)\geq 0\}.\] 
By \eqref{eq_commut2}, $x_\beta(v) X_\alpha(u_\alpha.c_\alpha) x_\beta(-v)$ belongs to $U_{\Omega_\alpha(u_\alpha)}^{ma+}$.  Moreover, \[\begin{aligned}\Omega_\alpha(u_\alpha)&=\bigcap_{(p,q)\in E_\alpha}\{a'\in \A\mid  p\alpha(a')+q\beta(a')+ p\omega(u_\alpha)+q\omega(v)\geq 0\}\\
&\supset  \bigcap_{(p,q)\in E_\alpha}\{a'\in \A\mid \frac{p}{q+1}(\alpha(a')+\omega(u_\alpha))\geq \max\left(0,-\beta(a')-\omega(v)\right)\} \\ 
& \supset \Omega'_\alpha(a):=\bigcap_{(p,q)\in E_\alpha}\{a'\in \A\mid\frac{p}{q+1}(\alpha(a')-\alpha(a))\geq \max\left(0,-\beta(a')-\omega(v)\right)\} .
\end{aligned}\]

We are looking for $a\in \A$ such that $b\in \bigcap_{\alpha\in \Delta^+}\Omega_\alpha'(a)$. Otherwise said,  we are looking for $a\in \A$ such that, for all $\alpha\in \Delta^+$   we have \begin{equation}\label{eq_ineq_condition}
\frac{p}{q+1}\left(\alpha(b)-\alpha(a)\right)\geq \max(0,-\beta(b)-\omega(v)),\ \forall (p,q)\in E_\alpha.
\end{equation}

Let $\lambda\in \A$ be such that $\alpha_i(\lambda)=1$ for all $i\in I$. Then $\lambda\in C^v_f$.  We search for $a$   in the form  $b-n\lambda$, where   $n\in \R_+$. Then \eqref{eq_ineq_condition} 
becomes \begin{equation}\label{eq_ineq_condition2}
\frac{p}{q+1}n \alpha(\lambda) =n\frac{\mathrm{ht}(p\alpha)}{q+1}\geq \max(0,-\beta(b)-\omega(v)),\ \forall (p,q)\in E_\alpha.
\end{equation}

If $(p,q)\in E_\alpha$, then $\frac{\mathrm{ht}(p\alpha)}{q+1}=\mathrm{ht}(p\alpha)$ if $q=0$ and $\frac{\mathrm{ht}(p\alpha)}{q+1}=\frac{\mathrm{ht}(p\alpha)}{q}\frac{q}{q+1}\geq \frac{1}{2}\inf\{\frac{\mathrm{ht}(\tau)}{q} \mid\tau\in Q^+, q\in \N^*, \tau+q\beta\in \Delta^+\}>0$ if $q>0$ (by Lemma~\ref{lem_LB}). Therefore \eqref{eq_ineq_condition2} is satisfied for $n\gg 0$, which proves the lemma. 
\end{proof}

\begin{lemm}\label{lem_Conj_unp_2}
Let $b\in \A$ and $g\in U^+$. Then there exists $a\in b-C^v_f$ such that $gU_{a}^{pm+}g^{-1}\subset  U_b^{pm+}$. 
\end{lemm}

\begin{proof}
Write $g=x_{\beta_1}(v_1)\ldots x_{\beta_k}(v_k)$, with $k\in \N$, $\beta_1,\ldots,\beta_k\in \Phi^+$ and $v_1,\ldots,v_k\in \shk$. We proceed by induction on $k$. If $k=1$, this is Lemma~\ref{lem_Conj_unp}. We assume that  $k\geq 2$ and that there exists $a'\in b-C^v_f$ such that $x_{\beta_1}(v_1)\ldots x_{\beta_{k-1}}(v_{k-1}) U_{a'}^{pm+} x_{\beta_{k-1}}(-v_{k-1})\ldots x_{\beta_1}(-v_1)\subset U_b^{pm+}$. By Lemma~\ref{lem_Conj_unp}, there exists $a\in a'-C^v_f$ such that $x_{\beta_k}(v_k) U_{a}^{pm+} x_{\beta_k}(-v_k)\subset U_{a'}^{pm+}$. Then $gU_{a}^{pm+} g^{-1}\subset U_b^{pm+}$, which proves the lemma. 
\end{proof}

We can now prove Lemma~\ref{lem_Ua++_fix}: if $x\in \shi$, then there exists $a\in \A$ such that $U_{a}^{pm+}$ fixes $x$. Indeed, we  have $x\in U^+.\rho_{+\infty}(x)$, where $\rho_{+\infty}$ is defined in \ref{ss_Retractions}.  Therefore there exist $g\in U^+$, $b\in \A$ such that $x=g.b$. By Lemma~\ref{lem_Conj_unp}, there exists  $a\in  \A$ such that $g^{-1} U_a^{pm+} g\subset U_b^{pm+}$. Then $U_{a}^{pm+}$ fixes $x$.

\subsection{Bruhat and Iwasawa decomposition}\label{ss_Br_Iw_dec}

\begin{theo}\label{thm_Bounded_sets_R_friendly}
Let $A\in \sha(G)$ and $P$ be a bounded subset of $A$. Then there exists $\tilde{A}\in \sha(G_\shr)$ such that $\tilde{A}$ contains $P$. If moreover $A$ contains $\g Q_{\epsilon\infty}$, for some $\epsilon\in \{-,+\}$, then we can choose $\tilde{A}=u.\A$, for some $u\in U_\shr^{\epsilon\epsilon}$.
\end{theo}

\begin{proof}
Write $A=g.\A$, with $g\in G$. 
By \cite[Proposition 1.5]{R06}, we can write $g=x_{\beta_1}(u_1)\ldots x_{\beta_k}(u_k)t$, for some $k\in \N$, $\beta_1,\ldots,\beta_k\in \Phi$, $u_1,\ldots,u_k\in \shk$ and $t\in T$. As $t.\A=\A$, we may assume that $t=1$. For $1\leq i\leq k$, we choose a sequence $(u_i^{(n)})_{n\in \N}\in \shr^\N$ such that $u_i^{(n)}\rightarrow u_i$. 

Let $a\in \A$. Then by Lemma~\ref{lem_Ua++_fix}, for $n\gg 0$, $x_{\beta_1}(u_1^{(n)})^{-1}x_{\beta_1}(u_1)$ fixes $x_{\beta_2}(u_2)\ldots x_{\beta_k}(u_k).a$ and thus we have \[x_{\beta_1}(u_1^{(n)})^{-1}x_{\beta_1}(u_1) x_{\beta_2}(u_2)\ldots x_{\beta_k}(u_k).a=x_{\beta_2}(u_2)\ldots x_{\beta_k}(u_k).a,\] for $n\gg 0$. For $n\gg 0$, we have $x_{\beta_2}(u_2^{(n)})^{-1}x_{\beta_2}(u_2) x_{\beta_3}(u_3)\ldots x_{\beta_k}(u_k).a= x_{\beta_3}(u_3)\ldots x_{\beta_k}(u_k).a$.  By induction, we deduce that if $\tilde{g}(n)=x_{\beta_1}(u_1^{(n)})\ldots x_{\beta_k}(u_k^{(n)})$, for $n\in \N$, then we have $\tilde{g}(n)^{-1} g.a=a$ for $n\gg 0$. 

Let $a_1,\ldots,a_m\in \A$ be such that $\mathrm{conv}(a_i\mid1\leq i\leq m)\supset g^{-1}.P$. Let $n\in \N$ be sufficiently large such that  $\tilde{g}(n)^{-1}g$ fixes $a_i$, for all $i\in \{1,\ldots,m\}$. Then $a_i\in \tilde{g}(n)^{-1}g.\A\cap \A$ for all $i$ and as $\tilde{g}(n)^{-1}g.\A\cap \A$ is convex, we have \[g^{-1}.P\subset \A\cap \tilde{g}(n)^{-1}g.\A.\] Let $h\in G$ be such that $h.\A= \tilde{g}(n)^{-1}g.\A$ and such that $h$ fixes $\A\cap \tilde{g}(n)^{-1}g.\A$. Then $h^{-1}\tilde{g}(n)^{-1}g$ stabilizes $\A$ and induces an affine morphism on it. In particular $h^{-1}\tilde{g}(n)^{-1}g$ fixes $\mathrm{conv}(a_i\mid1\leq i\leq m)$. Therefore $\tilde{g}(n)^{-1}g.x=x$,  for all $x\in g^{-1}.P$ and in particular, $P\subset \tilde{g}(n).\A$.

Suppose now that $A$ contains $\g Q_{\epsilon\infty}$, for some $\epsilon\in \{-,+\}$. 
Then we can assume that $g$ fixes $A\cap \A$ and thus that $g$ fixes $\g Q_{\epsilon\infty}$.
 Then $g\in G_{\g Q_{\epsilon\infty}}$ and by \ref{3.1} 3)  we can assume that $\beta_i\in \Phi^\epsilon$ for all $i\in \{1,\ldots,k\}$. 
 Then $\tilde{g}(n)\in U^{\epsilon\epsilon}_{\shr}$, which concludes the proof of the theorem.
\end{proof}

\begin{coro}\label{cor_Masure_GR}
(1) We have $\shi=G_\shr.\A$.

\par (2) For any local chamber $C$ in $\SHI$, there is $u\in U^ {\qe\qe}_{\shr}$ such that $C\subset u.\A$; in particular $C$ and $\g Q_{\qe\infty}$ are in a same $G_{\shr}-$apartment.
\end{coro}

\begin{proof}
Let $x\in \shi$ (\resp $C\subset\SHI$).
Let $A\in\sha(G)$ containing $x$ (\resp containing $C\cup \g Q_{\qe\infty}$, by (MAIII) in \ref{n1.3.1}).
 Then by applying Theorem~\ref{thm_Bounded_sets_R_friendly} to $P=\{x\}$\ (\resp $P=C$), we get $g\in G_{\shr}$ (\resp $u\in U^ {\qe\qe}_{\shr}$) such that $x\in g.\A$ (\resp $C\subset u.\A$).
\end{proof}

We now assume that $\shr^*$ contains an element $\qp$ such that $\omega(\qp)=1$ (this is Assumption~\ref{eq_Assumption_R}). Recall that we have $\nu(N_\shr)=W\zv\ltimes Y$. 
Let $W^+=W\zv\ltimes Y^+\subset W\zv\ltimes Y$, where $Y^+=Y\cap \sht$.

\begin{prop}\label{2.11}%%%%%%%%%%%%%%%%%
Let   $A_1,A_2\in \sha(G_\shr)$. Then there exists $g\in G_{\shr}$ fixing $A_1\cap A_2$ such that $A_{2}=g.A_{1}$.
\end{prop}

\begin{NB} In this proposition, we may replace $G_{\shr}$ by any subgroup $G'$ of $G$ containing $G_{\shr}$.
\end{NB}

\begin{proof} We may assume $A_{1}=\A$. Let $g_1\in G_{\shr}$ be such that $A_{2}=g_{1}.{\A}$. By (MA II), there exists $g_{2}\in G$ fixing $\A\cap A_2$ such that $A_{2}=g_{2}.\A$.
Hence $g_{1}^ {-1}g_{2}$ stabilizes $\A$ and thus it belongs to $N$. As $\nu(N)=\nu(N_\shr)=W$, there exists $n_\shr\in N_\shr$ such that $n_\shr^{-1}g_1^{-1} g_2$ fixes $\A$. Then $g:=g_1n_\shr$ satisfies the condition of the proposition.
\end{proof}

Recall that two filters $\Omega_1,\Omega_2$ are said to be $G_\shr$-friendly if there exists $A\in \sha(G_{\shr})$ containing $\Omega_1\cup \Omega_2$. Recall that $C_0^+=germ_0(C\zv_f)$. 
The following result is probably not new in the reductive case, but we could not find a reference in this case.

\begin{coro}\label{cor_Bruhat}(Bruhat decomposition)
\begin{enumerate}

\item Let $x,y\in \shi$ and $F_x,F_y$ be two facets based at $x$ and $y$ respectively. Then if $x,y$   are $G$-friendly, $F_x,F_y$ are $G_\shr$-friendly. This is in particular the case if $x\leq y$.

\item Let $I_\shr$ be the fixator of $C_0^+$ in $G_\shr$. Then \[G_\shr^+=I_\shr W^+ I_\shr.\]
\end{enumerate}
\end{coro}

\begin{proof}
By \cite[Proposition 5.17]{Heb20}, there exists $A\in  \sha(G)$ containing $F_x\cup F_y$. Let $P\subset A$ be a bounded element of  $F_x\cup F_y$. Then by Theorem~\ref{thm_Bounded_sets_R_friendly}, there exists $\tilde{A}\in \sha(G_\shr)$ containing $P$. Then $\tilde{A}$ contains $F_x\cup F_y$, which proves (1).

Let $h\in G_\shr^+$. Then $h.0\geq 0$ and thus there exists $A\in \sha(G)$ containing $C_0^+$ and $h.C_0^+$. Let $g\in G$ be such that $A=g.\A$ and $g$ fixes $A\cap \A$. Then by  Theorem~\ref{thm_Bounded_sets_R_friendly} and Proposition~\ref{2.11}, there exists $\tilde{g}\in G_\shr$ such that $\tilde{g}.\A$ contains $C_0^+$ and $h.C_0^+$ and such that $\tilde{g}$ fixes $C_0^+$. We have $h.0\geq 0$ and hence $\tilde{g}^{-1}h.0\geq \tilde{g}^{-1}.0=0$.  Therefore $\tilde{g}^{-1}h.C_0^+\subset \A$ is an element of $W^+.C_0^+$ and hence there exists $n\in N_\shr$ (inducing an element of $W^+$ on $\A$) such that $\tilde{g}^{-1}h.C_0^+=n.C_0^+$. Then $n^{-1}\tilde{g}^{-1}h\in I_\shr$ and thus $h\in \tilde{g}n I_\shr=I_\shr W^+ I_\shr$.  
\end{proof}

Recall the definition of ``narrow'' and of the $f_\Omega$ from \ref{3.1}.
\begin{coro}\label{cor_Iwasawa}(Iwasawa decomposition)
Let $\epsilon\in \{-,+\}$ and $\Omega$ be a   narrow  filter on $\A$. Then we have $G_\shr=U^{\epsilon\epsilon}_\shr. N_\shr.(G_\Omega\cap G_\shr)$. In particular, we have $G_\shr=U^{\epsilon\epsilon}_\shr. N_\shr.I_\shr$.
\end{coro}

\begin{proof}
By definition of the $f_\Omega$, we have $\Omega\subset D(\alpha,f_\Omega(\alpha))$, for all $\alpha\in \Phi$. In particular, $\Omega\subset \bigcap_{i\in I} D(\alpha_i,f_\Omega(\alpha_i))\cap D(-\alpha_i,f_\Omega(-\alpha_i))$, for all $i\in I$. As $\Omega$ is narrow, we deduce that $\Omega\subset \bigcap_{i\in I} D(\alpha_i,f_\Omega(\alpha_i))\cap D(-\alpha_i,-f_\Omega(\alpha_i)+1).$ Therefore the image of $\Omega$ in $\A/\A_{in}$ is bounded, where $\A_{in}=\bigcap_{i\in I} \ker(\alpha_i)$. Hence there exists a bounded filter $\Omega'\subset \A$ such that $\Omega\subset \Omega'+\A_{in}$. By Lemma~\ref{lem_Fixator_Ain}, we have $G_{\Omega}=G_{\Omega'}$ and  thus we can assume that $\Omega$ is bounded. 

Let $g\in G_\shr$. Then by the Iwasawa decomposition (\cite[Proposition 4.7]{R16}), there exists $A\in \sha(G)$ containing $\g Q_{\epsilon\infty}$ and $g.\Omega$. 
By Theorem~\ref{thm_Bounded_sets_R_friendly}, there exists $u\in U_{\shr}^{\epsilon\epsilon}$ such that $u.\A$ contains $g.\Omega$. Then $u^{-1}g.\Omega\subset \A$. 

Let $h\in G_\shr$ be such that $hu^{-1}g.\A=\A$ and $h$ fixes $\A\cap u^{-1}g.\A$, see Proposition \ref{2.11}. Then $hu^{-1}g.\A=\A$ and thus $n:=hu^{-1}g\in N_ \shr$. We have $n|_\Omega=u^{-1}g|_\Omega$, so $n^{-1}u^{-1} g\in G_\shr\cap G_\Omega$.
\end{proof}

\begin{rema}\label{rk_Iwa_Gpol}
Let $G'$ be a subgroup of $G$ containing $G_\shr$ (or more generally a subgroup of $G$ containing $U_{\shr}^{\epsilon\epsilon}$ and $N_\shr$, for some $\epsilon\in \{-,+\}$). Then the proof of Corollary~\ref{cor_Iwasawa} actually shows that $G'$ admits an Iwasawa decomposition: \[G'=U_{\shr}^{\epsilon\epsilon}.N_\shr.(G_\Omega\cap G'),\text{ for }\epsilon\in \{-,+\}.\]  

\par If we write an element of $G'$, $g=unh$, with $u\in U^ {\qe\qe}_{\shr}$ (or $u\in U^\qe$), $n\in N_{\shr}$ (or $n\in N$) and $h\in G_{\QO}$, then we have clearly that $\qr_{\qe\infty} (g.\QO)=n.\QO$.  
So the class of $n$ in $W=N_{\shr}/H=N/\g T(\shr)$ is well determined by $g$, up to the right multiplication by the fixator in $W$ of $\QO$.

\end{rema}

\subsection{The twin building at infinity, sector germs and $G_{\shr}-$apartments}\label{n2.14}

\par {\bf(1)} The Kac-Moody group $G=\g G(\shk)$ acts on a twin building $\zv\!\!\SHI$, see \eg \cite{Re02}.
It is the disjoint union of two buildings, the positive one $\zv\!\!\SHI^+$ and the negative one $\zv\!\!\SHI^-$.
Actually $\zv\!\!\SHI^\pm$ is covered by a family $\zv\!\!\sha^\pm(G)$ of vectorial $G-$apartments permuted transitively by $G$, more precisely in bijection with $G/N$, hence also in bijection with the set $\sha(G)$ of $G-$apartments in the masure $\SHI$.

\par The canonical apartment of sign $\pm$ is $\zv\!\A^\pm=\pm\sht\subset\A$, with its vectorial facets of sign $\pm$ (as defined in \ref{subRootGenSyst}).
The stabilizer (and fixator) of the canonical vectorial chamber $\pm C\zv_{f}$ is the Borel subgroup $B^\pm=TU^\pm$.
As $G$ acts transitively on the vectorial chambers of sign $\pm$, the set of these chambers is $G/B^\pm$.

\par One writes $\zv\!\!\sha^\pm(G_{\shr})=G_{\shr}.{\zv\!\A^\pm}$ the set of vectorial $G_{\shr}-$apartments of sign $\pm$.

\par {\bf(2)} On another hand, $G$ permutes transitively the sector germs of sign $\pm$ in $\SHI$ and the fixator of $\g Q_{\pm\infty}=germ_{\infty}(\pm C\zv_{f})$ is  $G_{\g Q_{\pm\infty}}=T_{0}U^\pm$ (see \ref{3.1}.3).
Clearly $B^\pm=TU^\pm$ stabilizes $\g Q_{\pm\infty}$, and the stabilizer is actually reduced to $B^\pm$: 
as $(B^\pm,N)$ is a BN pair in $G$, a subgroup of $G$ strictly greater than $B^\pm$ should contain a simple reflection in $W\zv$, which does not stabilize $\g Q_{\pm\infty}$.

\par So we get bijections $\set{\mathrm{sector\ germs\ of\ sign\ }\pm} \leftrightarrow G/B^\pm \leftrightarrow \set{\mathrm{vectorial\ chambers\ of\ sign\ }\pm}$, $g.\g Q_{\pm\infty}\leftrightarrow \ov g\in G/B^\pm \leftrightarrow g.(\pm C\zv_{f})$.
 These bijections are compatible with the above bijections between apartments:
$g.\g Q_{\pm\infty} \subset h.\A \iff h^ {-1}g\in N.G_{\g Q_{\pm\infty}}=N.U^\pm=N.B^\pm \iff g.(\pm C\zv_{f})\subset h.{\zv\!\A^\pm}$, for any $g,h\in G$.

\begin{lemm}\label{n2.15} We assume that $\shr$ is principal and that $\shk$ is its ring of fractions.
Then any sector germ in $\SHI$ (\resp any vectorial chamber in $\zv\!\!\SHI^\pm$) is contained in a $G_{\shr}-$apartment (\resp a vectorial $G_{\shr}-$apartment).
\end{lemm}

\begin{NB}\begin{enumerate}
\item The hypothesis that $\shk$ is the field of fractions of $\shr$ is clearly necessary, as we know that some sector germs in the masure $\widehat\SHI$ of $\g G$ over the completion $\widehat\shk$ of $\shk$ are not in a $G-$apartment.

\item Actually for this result, there is no need to assume that $\shr$ is dense in $\shk$.
\end{enumerate} 
\end{NB}

\begin{proof} By \S \ref{n2.14}.2, in particular the last equivalences, we may concentrate on the case of 
 $\zv\!\!\SHI^\pm$.
We use induction on the distance of a vectorial chamber to a vectorial $G_{\shr}-$apartment.
Using galleries, we are reduced to prove that, if $C_{1}, C_{2}$ are adjacent chambers in $\zv\!\!\SHI^\pm$ and $C_{1}$ is in a vectorial $G_{\shr}-$apartment, then so is $C_{2}$.
The set of chambers containing the common panel of $C_{1}$ and $C_{2}$ is isomorphic to the projective line $\PP_{1}(\shk)$ and the induced action of the fixator in $G$ (\resp $G_{\shr}$) of this panel on $\PP_{1}(\shk)$ is induced by an action of $\mathrm{SL}_{2}(\shk)\subset G$ (\resp $\mathrm{SL}_{2}(\shr)\subset G_{\shr}$).
But, as $\shr$ is a principal ideal domain with field of fractions $\shk$, we know that $\mathrm{SL}_{2}(\shr)$ acts transitively on $\PP_{1}(\shk)$ (see \eg \cite[1.17]{BeMR03}  or \cite[8.124 page 265]{Mar18}). Our result follows.
\end{proof}

\begin{prop}\label{n2.16} We assume that $\shr$ is a principal ideal domain, that $\shk$ is its field of fractions (and that $\shr$ is dense in $\shk$ for the valuation $\omega$). 
We assume moreover that   $\shr$ satisfies  assumption \eqref{eq_Assumption_R}.
If a sector germ $\g Q\subset \SHI$ and a bounded set $P\subset \SHI$ are $G-$friendly (\ie contained in a same $G-$apartment), then they are also $G_{\shr}-$friendly (\ie contained in a same $G_{\shr}-$apartment)
\end{prop}

\begin{remas*} (a) A sector germ and a bounded subset of an apartment are not always contained in a same apartment (even for the complete system of apartments of an affine building).
Think to the case of a tree.

\par (b) This proposition generalizes Theorem \ref{thm_Bounded_sets_R_friendly} (for some $\shr$) in a framework similar to Iwasawa decomposition.
But it is actually a simple corollary of this theorem.

\par (c) As a particular case of this proposition, we have that any local chamber (or facet) and any sector germ in $\shi$ are contained in a $G_\shr$-apartment.
\end{remas*}

\begin{proof} By Lemma \ref{n2.15}, one may suppose (up to the action of $G_{\shr}$) that $\g Q\subset \A$ and even $\g Q=\g Q_{\pm\infty}$ (using the action of $N_{\shr}$).
Then the proposition is an easy consequence of Theorem \ref{thm_Bounded_sets_R_friendly}.
\end{proof}

%%%%%%%%%%%%%%%%%%%%%%%%%%%%%%%%%%%%%%%%%%%%
%\section{Miscellanea}\label{s3}
\section{Study of the action of $G_{twin}$ on the twin masure}\label{s3}

 Let $\k$\index{k@$\k$} be any field, $\shk=\k(\qp)$\index{K@$\shk$} and $\sho=\k[\qp,\qp^{-1}]$\index{O@$\sho$}, where $\qp$\index{p@$\qp$} is an indeterminate. In this section, we study the groups $G=\g G(\shk)$, $G_{twin}=G_\sho$ (see \ref{ss_m_KM_grp} for the definitions of $G$ and $G_\sho$) and an other group $G_{pol}$ lying between $G$ and $G_{twin}$ (see \ref{1.2} for the definition). Let $\omega_\oplus:\k(\qp)\twoheadrightarrow \Z\cup\{\infty\}$ (resp. $\omega_{\ominus}:\k(\qp)\twoheadrightarrow \Z\cup\set{\infty}$) be the valuation such that $\omega_\oplus(\qp)=1$ (resp. $\omega_{\ominus}(\qp^{-1})=1$).  Let $\shi_\oplus$ (resp. $\shi_\ominus$) be the masure associated with $\left(\mathfrak{G},\k(\qp),\omega_\oplus\right)$ (resp. $\left(\mathfrak{G},\k(\qp^{-1}),\omega_\ominus \right)$). 
 We study the action of these three groups on the twin masure $\shi_{\oplus}\times \shi_{\ominus}$. 
 
 \medskip
 
In \S \ref{ss_Gtwin_Gpol} we introduce the framework.
 
 In \S \ref{ss_Existence_iso_fixing}, we prove the existence, for any two apartments $A_1,A_2$ of $\shi_\oplus\times \shi_{\ominus}$, of an element $g\in G_{twin}$ (or $G_{pol}$) such that $g.A_1=A_2$ and $g$ fixes $A_1\cap A_2$.
 
In \S \ref{2.17}, we study the existence of an apartment of $\shi_\oplus\times \shi_{\ominus}$ containing $E_\oplus\cup E_{\ominus}$,  for certain pairs of filters $E_\oplus \subset \shi_\oplus$, $E_{\ominus}\subset \shi_{\ominus}$. Equivalently, we are interested in certain decompositions of $G_{twin}$ (or $G_{pol}$).

\subsection{The groups $G_{twin}$ and $G_{pol}$}\label{ss_Gtwin_Gpol}

\subsubsection{The field}\label{1.1}  %%%%%%%%%%%%%

Let $\k$ be any field (\eg a finite field) and $\qp$ be  an indeterminate. 
The field of rational functions over $\k$ is written $\shk=\k(\qp)$.  Then $\shk$ is a global field when $\k$ is finite and is a function field over $\k$ in any case. We refer to \cite[1]{Sti09} for more details on this subject.

\par A valuation ring on $\shk/\k$ is a ring $\sho'\subset \shk$ such  that $\k\subsetneq \sho'\subsetneq \shk$ and such that for all $z\in \sho'$, we have either $z\in \sho'$ or $z^{-1}\in \sho'$. Such a ring is local (i.e it has a unique maximal ideal $\g v_{\sho'}$). A set of the form $\g v=\g v_{\sho'}$, for a valuation ring $\sho'$, is called a place of $\shk$ (over $\k$). Then $\sho'$ is uniquely determined by $\g v$.  

If $P$ is a monic irreducible polynomial of $\k[\qp]$, then there exists a unique valuation $\omega_P:\k(\qp)\twoheadrightarrow \Z\cup\{\infty\}$ such that  $\omega_P(\k^*P)=\{1\}$. Then $\g v_P:=\{z\in \k(\qp)\mid\omega_P(z)>0\}$ is a place of $\shk$. We write $\omega_{\oplus}$\index{o@$\omega_\oplus$, $\omega_\ominus$} instead of $\omega_\qp$. Let $\omega_{\ominus}:\k(\qp)\twoheadrightarrow \Z\cup\{\infty\}$ be the valuation such that   $\omega_{\ominus}(\k^* \qp^{-1})=\{1\}$. Then $\omega_{\ominus}$ defines a place of $\shk$. We denote by $\oplus$ (resp. $\ominus$) the place associated with $\omega_{\oplus}$ (resp. $\omega_{\ominus}$). By \cite[Theorem 1.1.2]{Sti09}, every place of $\shk$ is either equal to $\ominus$ or to $\g v_P$ for some monic irreducible element $P$ of $\k[\qp]$. Note that $\ominus$ is often called the place at infinity of $\shk$, which explains why we sometimes index the objects related to $\ominus$ with an ``$\infty$''.  If $\g v$ is a place of $\shk$, we denote by $\omega_\g v$ (resp. $\sho_{\g v}=\set{x\in\shk \mid \qo_{\g v}(x)\geq0}{=\shk_{\qo_{\g v}\geq0}}$\index{O@$\sho_{\g v}$, $\sho_{\oplus}$, $\widehat{\sho}_{\g v}$}) the associated valuation (resp. valuation ring). We have  $\sho_{\oplus}=\k[\qp][(1+\qp\k[\qp])^ {-1}]$  and $\sho_{\infty}=\sho_{\ominus}=\k[\qp^ {-1}][(1+\qp^ {-1}\k[\qp^ {-1}])^ {-1}]$.

\par We also set  $\sho=\k[\qp,\qp^ {-1}]=\bigcap_{\g v\neq0,\infty} \sho_{\g v}$.

One may write $\widehat\shk_{\g v}$\index{K@$\widehat{\shk}_{\g v}$} the completion of $\shk$ with respect to $\qo_{\g v}$ and $\widehat\sho_{\g v}$\index{O@$\sho_{\g v}$, $\sho_{\oplus}$, $\widehat{\sho}_{\g v}$} its ring of integers; $\widehat\shk_{\g v}$ is a ``local'' field (a true local field if $\k$ is finite).
In particular $\widehat\shk_{\oplus}=\k((\qp))$ (\resp $\widehat\shk_{\infty}=\widehat\shk_{\ominus}=\k((\qp^ {-1}))$) and $\widehat\sho_{\oplus}=\k[[\qp]]$ (\resp $\widehat\sho_{\infty}=\widehat\sho_{\ominus}=\k[[\qp^ {-1}]]$).

\begin{rema}
Our main motivation for this work is the definition of Kazhdan-Lusztig polynomials in the Kac-Moody setting. For this, we could restrict ourselves to the case where $\k$ is finite. This assumption is important when we count the number of lifts of a path (to obtain finiteness results) but for many results, it would not simplify our proofs to make this assumption. This is why for most results we make no assumption on  $\k$. 
\end{rema}

%\par Perhaps the notations $\qo_{+},\qo_{-}=\qo_{\infty},{^+\sho},\sho_{\infty}={^-\sho}, {^+\widehat\shk},\widehat\shk_{\infty}={^-\widehat\shk}, {^+\widehat\sho},\widehat\sho_{\infty}={^-\widehat\sho}$ have to be prefered.

\subsubsection{The Kac-Moody group, masures  and the groups $G_{twin}$ and $G_{pol}$}\label{1.2} %%%%%%%%%%%%

 \par {\bf(1) The masures} 
 
 \par $\bullet$ Let  $\mathcal{S}=(A,X,Y,(\alpha_i)_{i\in I},(\alpha_i^\vee)_{i\in I})$ be a root generating system (as defined in \ref{ss_Root_generating}) and $\g G=\g G_{\mathcal{S}}$ be the associated Kac-Moody group described in \ref{ss_m_KM_grp}. We set $G=\g G(\shk)$.

\par $\bullet$ Let $\g v$ be a place of $\shk$. We denote by $\widehat{\shi_\g v}$ the masure associated with $(\g G,\widehat{\shk}_\g v, \omega_{\g v})$ and by $\shi_{\g v}$\index{I@$\shi_{\g v}$ $\widehat{\shi}_{\g v}$} the masure associated with $(\g G,\shk,\omega_{\g v})$ (see \ref{sub_Masure}). Let $G_\g v= \g G(\widehat{\shk}_{\g v})$\index{G@$G_{\g v}$}. By \cite[5.8 3]{R17}, the  inclusion   $G\times \A_{\g v}\hookrightarrow G_{\g v}\times \A_{\g v}$  induces a $G$-equivariant inclusion $\shi_\g v\rightarrow \widehat{\shi}_{\g v}$ and we identify $\shi_\g v$ with its image in $\widehat{\shi}_{\g v}$.
 
 \par$\bullet$ The apartments of $\widehat\SHI_{\g v}$ (\resp $\SHI_{\g v}$) are the subsets $g.\A_{\g v}\subset\widehat\SHI_{\g v}$, for $g\in G_{\g v}$ (\resp $g\in G$).
 One writes $\sha_{\g v}(G_\g v)$\index{A@$\sha_{\g v}(G_{\g v}), \sha_{\g v}(G)$} (\resp $\sha_{\g v}(G)$) the set of these apartments.
 They are associated respectively to the set of maximal split tori of $\g G$ over  $\widehat\shk_{\g v}$ and $\shk$.
By Corollary~\ref{cor_Masure_GR} $\widehat\SHI_{\g v}$ is the union of all apartments in $\sha_{\g v}(G)$ (hence also in $\sha_{\g v}(G_\g v)$). Otherwise said, $\shi_{\g v}=\widehat{\shi_{\g v}}$ as a set.

 \par$\bullet$ The group $G_{\g v}=\g G(\widehat\shk_{\g v})$ acts on $\widehat\SHI_{\g v}$.
 The stabilizer of $\A_{\g v}$ in $G_{\g v}$ (\resp $G$) is $\g N(\widehat\shk_{\g v})$ (\resp $N=\g N(\shk)$).
  
 \par$\bullet$ The group $\g T(\widehat\shk_{\g v})$ acts by translations: to $t\in \g T(\widehat\shk_{\g v})$ is associated the translation of vector $v$, where $v\in \A$ is determined by $\chi(v)=-\qo_{\g v}(\chi(t))$, for any $\chi\in X$ (hence $\chi$ in the dual of $\A$).
 The group of vectors of all these translations is $Y$.
   
 \par$\bullet$ The action of $n\in\g N(\widehat\shk_{\g v})$ is affine with associated linear map the action of the class $\ov n$ of $n$ in the Weyl group $W\zv=\g N(\widehat\shk_{\g v})/\g T(\widehat\shk_{\g v})=\g N(\shk)/\g T(\shk)=\g N(\k)/\g T(\k)$ (this group acts $\Z-$linearly on $Y$, hence $\R-$linearly on $\A$).

\par$\bullet$ One may choose an origin $0_{\g v}$ of $\A_{\g v}$ in such a way that $\g N(\k)$ fixes $0_{\g v}$.
Then the image $W_{\g v}$ of $\g N(\widehat\shk_{\g v})$ or $\g N(\shk)$ in the affine group of $\A_{\g v}$ is identified with $W\zv\ltimes Y$.

\par $\bullet$ If $\g v\in \{\ominus,\oplus\}$, we set $C\zv_{f,\g v}=\{x\in \A_{\g v}\mid\alpha_i(x)>0, \ \forall i\in I\}$\index{C@$C\zv_{f,\g v}$}. 
We set $C_\oplus=germ_{0_{\oplus}}(C\zv_{f,\oplus})\subset \A_{\oplus}$\index{C@$C_{\oplus}, C_{\ominus}, C_{\infty}$} and  $C_\infty=C_{\ominus}=germ_{0_{\ominus}}(-C\zv_{f,\ominus})\subset \A_{\ominus}$. 
These are the fundamental local chambers of $\shi_{\oplus}$ and $\shi_{\ominus}$.

 \medskip 
 \par {\bf(2) The twin group} We want to study the group of $\sho$-points of $\g G$ (where $\sho=\k[\qp,\qp^{-1}]$). As mentioned before, this notion is not well defined. We studied the group $G_\sho=\langle \g N(\sho),(\g U_{\qa}(\sho))_{\qa\in\QF}\rangle$ in section~\ref{s2}. 
 We now denote this group by $G_{twin}$. As suggested by Muthiah, it   seems also natural  to study the group $G_{pol}$, more ``adelic'' in nature,    defined below. We will use the fact that $G_{twin}$ is a subgroup of $G_{pol}$ in our study of $G_{twin}$\index{G@$G_{twin},G_{pol}$}. % \index{G@$G_{twin}$}

 The group $G_{pol}$ is the subgroup of $G$ consisting of the elements $g\in G$ such that for every place $\g v$ of $\shk$ different from $\oplus$ and $\ominus$, we have $g\in \g G(\widehat{\sho}_\g v)$.  \index{G@$G_{twin},G_{pol}$}

As $\widehat\sho_{\g v}$ is not a field, there are several possible definitions for $\mathfrak{G}(\widehat\sho_{\g v})$. We define it as the fixator of the point $0_{\g v}$ for the action of $G$ on the masure $\shi_{\g v}=\shi(\mathfrak{G},\shk,\qo_{\g v})$.
%\sout{associated to $\g G$ on the valued field $(\shk,\qo_{\g v})$}
By \cite[Proposition 3.1]{hebert2023topologies}, we actually have $\mathfrak{G}(\widehat\sho_{\g v})=\mathfrak{G}^{min}(\widehat\sho_{\g v})$, where $\mathfrak{G}^{min}$ is the minimal group defined by Marquis. 
%\sout{As $\g G(\widehat\sho_{\g v})$ is perhaps not very well defined, in the following, we replace the condition ``$g \in \g G(\widehat\sho_{\g v})$'' by the condition ``$g$ fixes the point $0_{\g v}$ for the action of $G$ on the masure $\SHI_{\g v}=\SHI(\g G, \shk,\qo_{\g v})$ associated to $\g G$ on the valued field $(\shk,\qo_{\g v})$''.} 
The group $G_{pol}$ contains $G_{twin}$.

 \medskip
 \par Actually $N_{twin}=\g N(\sho), T_{twin}=\g T(\sho)$ and $U_{\qa,twin}=\g U_{\qa}(\sho)$ are well defined as $\g N, \g T, \g U_{\qa}$ are algebraic groups over $\k$. We have $N_{twin}=N_\sho$, for the notation of \ref{ss_N}.
 
 \par We denote by $I_{\oplus}$\index{I@$I_{\oplus},I_{\ominus},$} (resp. $I_{\ominus}$) the fixator of $C_{\oplus}$ (resp. $C_\ominus$) in $G$.
  We denote by $I_{twin}$\index{I@$I_{twin},I_{pol},I_{\infty}$} (resp. $I_\infty$) the fixator of $C_\oplus$ (\resp $C_{\ominus}$) in  $G_{twin}$ and by $I_{pol}$ the fixator of $C_\oplus$ in $G_{pol}$.

\begin{rema*} When $\g G$ is a split reductive group over $\k$, it is a well defined functor over the $\k-$algebras  and we saw in \ref{ss_m_KM_grp} that $G_{\sho}$ (as defined in 1.2.1) is equal to $\g G(\sho)$.
% This is a consequence of $\sho$ being a principal ideal domain by \cite{T85} top of page 205.
 So $G_{twin}=\g G(\sho)=\g G(\cap_{\g v\neq 0,\infty} \sho_{\g v})=\cap_{\g v\neq 0,\infty} \g G(\sho_{\g v})$.
 And $\g G(\sho_{\g v})$ is the fixator in G of $0_{\g v}\in\SHI_{\g v}$, by \cite[6.13.b, 7.1 and 7.4.4]{BrT72}.
 So $G_{twin}=G_{pol}$ in this reductive case.
 \end{rema*}

 \par One may ask wether $G_{twin}=G_{pol}$ in general. The answer is unknown.
 For affine $\mathrm{SL}_{n}$ and $n=2$, the answer is unknown, but for $n\geq3$ there is equality, see Remark {\ref{6.4c}}.

\subsubsection{Affine roots} \label{1.3} %%%%%%%%%%%%%%%%%%%%%

\par Following \cite[Appendix B]{BrKP16} there is a system of affine roots:

\par\qquad $\QF_{a}=\QF\times\Z=\set{\un\qa=\qa+r\xi \mid \qa\in \QF, r\in\Z}$, where $\qx$ is a symbol (see also below).

\par $\QF^+_{a+}=\set{\qa+r\xi \mid \qa\in \QF^+, r\geq0}$ \quad ; \quad $\QF^+_{a-}=\set{\qa+r\xi \mid \qa\in \QF^+, r<0}$

\par $\QF^-_{a+}=\set{\qa+r\xi \mid \qa\in \QF^-, r>0}$ \quad ; \quad $\QF^-_{a-}=\set{\qa+r\xi \mid \qa\in \QF^-, r\leq0}$

\par\qquad $\QF_{a+}=\QF^+_{a+}\cup \QF^-_{a+}$ \quad and \quad $\QF_{a-}=-\QF_{a+}=\QF^-_{a-}\cup \QF^+_{a-}$

\par So $\QF_{a+}$ may be considered as a system of positive roots in $\QF_{a}$; but there is no associated basis (as $\QF^-$ has no smallest root).

\bigskip
\par One may consider the vector space $\A_{twin}=\A\oplus\R$.\index{a@$\A_{twin}$}
So $\QF_{a}$ is a set of linear forms on $\A_{twin}$: $\qa\in \QF\subset X$ is a linear form on $\A$ and we set $\qa(\R)=\set0$; $\qx(\A)=\set0$ and $\qx\vert_{\R}=Id\vert_{\R}$.
$\A_{twin}$ contains three interesting subspaces $\A_{\oplus}=\A\oplus\set{1}$, $\A_{\ominus}=\A\oplus\set{-1}$ (affine subspaces) and $\zv\A=\A\oplus\set0$.

\par If  $\g v=\oplus$ or $\g v=\ominus$, $\A_{\g v}$ is the (canonical) apartment associated to $\g T$ in the masure $\SHI_{\g v}=\SHI(\g G,\shk,\qo_{\g v})$, see \S \ref{1.2}.2 above. 
%As a set this masure is the same as the masure $^\pm\widehat\SHI$ of $\g G$ over $\widehat K_{\pm}$ (consequence \eg of \ref{2.3}), but its system of apartments $^\pm\sha=\sha_{\pm}=G.{^\pm\A}$ is smaller than the system $^\pm\widehat\sha=\widehat\sha_{\pm}=G_{\pm}.{^\pm\A}$ of $^\pm\widehat\SHI$.

\par $\zv\A=\A$ is, more or less, the (twin) apartment associated to $\g T$ in the twin building $\zv\!\!\SHI={\zv\!\!\SHI^+}\cup{\zv\!\!\SHI^-}$ of $\g G$ over $\shk$.
Actually the (twin) apartment is the union of $\zv\A^+=\sht\subset {\zv\!\!\SHI^+}$ and $\zv\A^-=-\sht\subset {\zv\!\!\SHI^-}$, where $\sht$ is the Tits cone in $\A$ (see \ref{2.2b}.1).

 \subsubsection{The affine Weyl group}\label{1.4}  %%%%%%%%%%%
 
 \par To each $\un\qa=\qa+s\qx\in\QF_{a}$ is associated a reflection $r_{\un\qa}$ in $\A_{twin}$ with respect to the hyperplane (=wall) $M_{twin}(\un\qa)$\index{M@$M_{twin}(\qa+k\qx)$} with equation $(\qa+s\qx)(x,p)=0$: \quad
 $r_{\qa+s\qx}(x,p)=(x-(\qa(x)+sp)\qa^\vee,p)$.
 
 \par On $\zv\A=\A$ it acts as $r_{\qa}$ (reflection associated to the root $\qa$, with respect to the wall $\ker\qa$).
 On $\A_{\oplus}=\A\oplus\set{1}\simeq \A$ (\resp $\A_{\ominus}=\A\oplus\set{-1}\simeq \A$) it acts as the usual reflection $r^\oplus_{\qa,s }$ (\resp $r^\ominus_{\qa,-s }$) with respect to the affine hyperplane (=wall) $M_{\oplus}(\qa+s)$  (\resp $M_{\ominus}(\qa-s)$) with equation $\qa(x)+ s=0$ (\resp $\qa(x)- s=0$); its associated linear map is $r_{\qa}$.
 
 \par Clearly the generated group is $W_{a}=W\zv\ltimes Q^\vee$ where $Q^\vee=\sum_{\qa\in\QF} \Z\qa^\vee=\oplus \Z\qa^\vee_{i}$ acts by transvections: $\qa^\vee*(x,p)=(x+p\qa^\vee,p)$.
 The group $W_{a}$ is not a Coxeter group in general.

  \subsubsection{The root groups in $G_{twin}$}\label{1.5}
  
  \par For $\un\qa=\qa+s\qx\in\QF_{a}$ there is a group embedding $x_{\un\qa}: (\k,+)\to U_{\qa},\ a\mapsto x_{\qa}(\qp^s.a)$.
  Its image is the group $U_{\qa+s\qx}=x_{\qa+s\qx} (\k)\subset G_{twin}\subset G$.
  Then $U_{\qa,twin}=\g U_{\qa}(\sho)=\langle U_{\qa+s\qx} \mid s\in \Z\rangle = \bigoplus_{s\in \Z} \, U_{\qa+s\qx}$.
  
  \par The link with the groups $^{\g v}U_{\qa,r}$ of  \ref{n1.3.1}  is as follows:  $^{\oplus}U_{\qa,r}=(^{\oplus}U_{\qa,r+1})\times U_{\qa+r\qx}$, $^{\oplus}U_{\qa,r}/(^{\oplus}U_{\qa,r+1})\simeq U_{\qa+r\qx}$.
  But $^{\ominus}U_{\qa,r}=x_{\qa}(\shk_{\qo_{-}\geq r})=(^{\ominus}U_{\qa,r+1})\times U_{\qa-r\qx}$, $^{\ominus}U_{\qa,r}/(^{\ominus}U_{\qa,r+1})\simeq U_{\qa-r\qx}$.  

  \par We may consider the action of $G$ on $\SHI_{\oplus}\sqcup\SHI_{\ominus}\sqcup{\zv\SHI}\supset \A_{\oplus}\sqcup\A_{\ominus}\sqcup{\zv\A}$.
  Then, by \ref{n1.3.1}, the fixed point set of $x_{\qa+s\qx}(k)$ (for $k\in\k^*$) in $\A_{\oplus}\sqcup\A_{\ominus}\sqcup{\zv\A}$ is the intersection $D_{\oplus}(\qa+s)\sqcup D_{\ominus}(\qa-s)\sqcup D\zv(\qa)$ of the half-apartment $D_{twin}(\qa+s\qx)=\set{a\in\A \mid (\qa+s\qx)(a)\geq0}$\index{D@$D_{twin}(\qa+k\qx)$} with $\A_{\oplus}\sqcup\A_{\ominus}\sqcup{\zv\A}$. 
  (Recall that $\qx=1$ (\resp $\qx=-1$, $\qx=0$) on $\A_{\oplus}$ (\resp $\A_{\ominus}$, $\zv\A$).

\begin{lemm*} For any $\qa\in \QF$, one has $U_{\qa,twin}=U_{\qa}\cap G_{twin}=U_{\qa}\cap G_{pol}$.
\end{lemm*}
\begin{proof} One has $U_{\qa,twin}\subset U_{\qa}\cap G_{twin}\subset U_{\qa}\cap G_{pol}$.
If $x_{\qa}(a)\in U_{\qa}\cap G_{pol}$ (with $a\in\shk$), then, $\forall\g v\neq0,\infty$, $x_{\qa}(a)$ fixes $0_{\g v}$ in $\SHI(\g G,\shk,\qo_{\g v})$, so $\qo_{\g v}(a)\geq0$ and $a\in\sho$, $x_{\qa}(a)\in U_{\qa,twin}$.
\end{proof}

  \medskip
    \par %As $\g U^\pm$ is pro-unipotent, $\g U^\pm(\sho)$ is more or less well defined. 
  For $\qe=+$ or $\qe=-$, one considers $U^ {\qe\qe}_{twin}=U^{\epsilon\epsilon}_{\sho} = \langle U_{\qa+s\qx} \mid \qa+s\qx \in \QF^\qe_{a-} \cup \QF^\qe_{a+} \rangle \subset  U^\qe$.
  
  \par Let us define also $U^ {\qe}_{twin} :=  U^\qe \cap G_{twin}$ and $U^\qe_{pol}:=U^\qe\cap G_{pol}$. 
  
\parni Clearly $U^ {\qe\qe}_{twin} \subset U^ {\qe}_{twin} \subset U^ {\qe}_{pol}$. 
As we saw in \ref{ss_m_KM_grp}, the first inclusion is strict in general. For the second inclusion one does not know wether it may be an equality.

%\parni{\bf Questions:} Is $U^ {\pm\pm}_{twin}$ equal to $U^ {\pm}_{twin} := \g U^\pm(\shk) \cap G_{twin}$ ? Is $U^\pm_{twin}$ equal to $U^\pm_{pol}:=U^\pm\cap G_{pol}$ ?

\subsubsection{The group $N_{twin}=\g N(\sho)$ ($=N_{pol}$)}\label{1.6}   %%%%%%%%%%%%%%%%%%%

We have $\g T(\k)\subset \g T(\sho)=T_{twin}\subset T=\g T(\shk)$.
For $\ql\in Y= Hom(\g{Mult},\g T)$, we may define $\qp^\ql:=\ql(\qp){\in} \g T(\sho)=T_{twin}$, as $\qp\in\sho^*$.

\par Then one has: \qquad $T_{twin}=\g T(\sho)=\set{h.\qp^\ql \mid h\in \g T(\k), \ql \in Y }$,

\qquad\qquad\qquad\qquad\quad $N_{twin}=\g N(\sho)=\set{n_{0}.\qp^\ql \mid n_{0}\in \g N(\k), \ql \in Y }$,

\par and the Weyl group is \qquad $W:=N_{twin}/\g T(\k)=\set{w.\qp^\ql \mid w\in W\zv, \ql \in Y }=W\zv\ltimes Y$.

\parni Actually the image of $n_{0}.\qp^\ql\in N_{twin}$ in $N_{twin}/\g T(\k)$ is $w.\qp^\ql$ if the class of $n_{0}\in \g N(\k)$ in $\g N(\k)/\g T(\k)$ is $w$.

\par All this may be seen \eg from \cite{T85} page 204: $\g N(\sho)$ is generated by $\g T(\sho)$ and elements $m_{i}$ such that $m_{i}tm_{i}^ {-1}=r_{i}(t)$ (for $t\in \g T(\sho)$), the $m_{i}$ satisfy the braid relations and $m_{i}^2=\eta_{i}\in Hom(Y,\C^*)$ such that $\eta_{i}(\ql)=(-1)^ {\langle\ql,\qa_{i}^\vee\rangle}$, \ie with classical notation $\eta_{i}=(-1)^ {\qa_{i}^\vee}$ (see \eg the relation $\tilde s(-1)=(\tilde s)^ {-1}=\tilde s.(-1)^ {\qa_{i}^\vee}$ in \cite{Re02} page 196).

\begin{NB} 1) In particular, for $\g v=\oplus$ or $\g v=\ominus$, $W_{\g v}=\g N(\shk)/\set{t\in\g T(\shk) \mid \qo_{\g v}(\chi(t))=0, \forall\chi\in X}=\g N(\shk)/\g T(\sho_{\g v})$ is also equal to $W=N_{twin}/\g T(\k)$: any action of an element of $\g N(\shk)$ on $\A_{\g v}$ is induced by the action of an element of $N_{twin}$.
The same things are true for the action on $\zv\A$.

\par 2) We shall see below in \S \ref{1.7}, Lemma \ref{1.7b}, that $N\cap G_{twin}=N_{twin}=N\cap G_{pol}=:N_{pol}$ and $T\cap G_{twin}=T_{twin}=T\cap G_{pol}=:T_{pol}$.

\par 3) %We saw in \ref{1.5} that $U^ {\qe\qe}_{twin}=U^ {\qe\qe}_{pol}$.
By the Iwasawa decomposition (Remark \ref{rk_Iwa_Gpol})  $G_{twin}=G_{pol}\iff I_{twin}=I_{pol}$. %\marge{$I_{twin},I_{pol}$ d\'efinis o\`u ?}
 
\end{NB}
  
\subsubsection{Stabilizers or fixators in $G_{twin}$ or $G_{pol}$ of canonical apartments $\A_{\g v}$ or $\zv\A$}   %%%%%%%%%%%%%%%%%%
\label{1.7}

\par Following \cite[cor. 10.4.3]{Re02}, the fixator (\resp stabilizer) of $\zv\A$
in $G=\g G(\shk)$ is $T$ (\resp $N$).
Let now $\g v=\oplus$ or $\g v=\ominus$.
We know  that $\zv\A$ is at infinity of $\A_{\g v}$, that $\zv\!\!\SHI$ is at infinity of $\SHI_{\g v}$, and that the action of $G$ on $\SHI_{\g v}$ induces at infinity its action on $\zv\!\!\SHI$. %\marginpar{r\'ef\'erences?}
So it follows that the stabilizer of $\A_{\g v}$ in $G$ is $N=\g N(\shk)$ and, then, that its fixator is $\g T(\sho_{\g v})$.

\par {\bf a)} We prove below that the fixator (\resp stabilizer) in $G_{twin}$ or $G_{pol}$ of $\zv\A$ is $\g T(\shk)\cap G_{twin}=\g T(\shk)\cap G_{pol}=\g T(\sho)=T_{twin}$ (\resp $\g N(\shk)\cap G_{twin}=\g N(\shk)\cap G_{pol}=\g N(\sho)=N_{twin}$).

\par {\bf b)} We have the inclusions $\g T(\shk)\cap G_{pol}\supset\g T(\shk)\cap G_{twin}\supset\g T(\sho)=T_{twin}$.
Let us prove $\g T(\shk)\cap G_{pol}\subset\g T(\sho)$.
We have $\g T\simeq \g{Mult}^ {d}$ and $(p_{1},\ldots,p_{{d}})\in\g T(\shk)=(\shk^*)^ {d}$ fixes $0_{\g v}$ in $\SHI(\g G,\shk,\qo_{\g v})$ for all $\g v\neq0,\infty$ if, and only if, $\forall j, \forall \g v, \qo_{\g v}(p_{j})=0$ if, and only if, $\forall j, p_{j}\in \k[\qp,\qp^ {-1}]^*$.
We get that the above inclusions are equalities.

\par {\bf c)} We remarked above (in \S \ref{1.6}) that $\g N(\shk)/\g T(\shk)$ is equal to $\g N(\sho)/\g T(\sho)$ and $\g N(\sho)$ is in $G_{twin}\subset G_{pol}$.
So $\g N(\shk)\cap G_{twin}=\g N(\shk)\cap G_{pol}=\g N(\sho)=N_{twin}$ follows from b). And a) is proved.

\par {\bf d)} Now, for $\g v=\oplus$ or $\g v=\ominus$, the fixator (\resp stabilizer) in $G_{twin}$ or $G_{pol}$ of $\A_{\g v}$ is $\g T(\sho_{\g v})\cap \g T(\sho)=\g T(\k)$ (\resp $\g N(\shk)\cap G_{twin}=\g N(\shk)\cap G_{pol}=\g N(\sho)=N_{twin}$).

\begin{lemm}\label{1.7b}   %%%%%%%%%%%%%%%
\quad $(\g U^\pm(\shk).\g N(\shk))\cap G_{twin}=U^\pm_{twin}.\g N(\sho)$ \ and \ $\g N(\shk)\cap G_{twin}=\g N(\sho)=N_{twin}$\index{N@$N_{twin}$}.
\end{lemm}

\begin{NB} We write $U^\pm_{twin}=\g U^\pm(\shk)\cap G_{twin}$. The same things are true with $G_{pol}$ instead of $G_{twin}$ (just replacing $U^\pm_{twin}$ by $\g U^\pm(\shk)\cap G_{pol}=U^\pm_{pol}$) and with $\g T$ instead of $\g N$.
\end{NB}

\begin{proof} The last equality is proved above in \S \ref{1.7}.d.
Let $g=u.n$ with $g\in G_{pol}, n\in N$ and $u\in U^\pm$.
Let $\g v$ be a place of $\shk$, $\g v\neq 0,\infty$.
As $g\in G_{pol}$, it fixes $0_{\g v}$ for the action of $G$ on $\SHI(\g G,\shk,\qo_{\g v})$.
Let us consider the retraction $\qr$ onto the canonical apartment $\A_{\g v}$ of $\SHI(\g G,\shk,\qo_{\g v})$ associated to $U^\pm$ \ie to $\g Q_{\pm\infty}$ (see \ref{ss_Retractions}).
Then the maps from $\A_{\g v}$ to itself given by $x\mapsto n.x$ and $x\mapsto \qr(g.x)$ coincide.
So $n$ fixes $0_{\g v}$; we have proved that $n\in N\cap G_{pol}=N_{twin}$ (\S \ref{1.7}.d above) and thus $u\in U^\pm \cap G_{pol}$ (and  $u\in U^\pm \cap G_{twin}=U_{twin}^ {\pm}$
 if $g\in G_{twin}$).
\end{proof}

\subsubsection{(Linear) action of $\g N(\sho)=N_{twin}$ on $\A_{twin}$}\label{1.8}  %%%%%%%%%%%%%%%%%%%%%%%
\par We shall  define an action $\qn:N_{twin}\to Aut(\A_{twin})$.  

\par By \S \ref{1.6} $N_{twin}=\set{n_{0}.\qp^\ql \mid n_{0}\in \g N(\k), \ql \in Y }$, we ask that:

\par\quad$\bullet$ $n_{0}$ acts linearly on $\A_{twin}=\A\oplus\R$, trivially on $\R$ and by its linear action $\qn\zv$ on $\A$ (as $W\zv=\g N(\k)/\g T(\k)$).

\par\quad$\bullet$ $t\in T_{twin}=\g T(\sho)$ acts by transvections: $\qn(t)=tr_{v}:\A_{twin}\to\A_{twin}, x\mapsto x+v\qx(x)$, with $v\in \A$ determined by $\chi(v)=-\qo_{\oplus}(\chi(t))$, $\forall \chi \in X$.

\par In particular for $t=\qp^\ql$, $v=-\ql\in Y\subset Y\otimes\R=\A$ (see \eg \cite[2.9]{BPGR19}).

\par This action induces the known actions of $N_{twin}\subset N$ on $\zv\A$, $\A_{\oplus}$ and $\A_{\ominus}$.
For $\A_{\ominus}$, one has to remark that $\qp^\ql$ acts by a translation of vector $v'$ given by $\chi(v')=-\qo_{\ominus}(\chi(\qp^\ql))=\qo_{\oplus}(\chi(\qp^\ql))=\chi(\ql)$, $\forall \chi \in X$.
This agrees with the fact that $\qx=-1$ on $\A_{\ominus}$.

\subsubsection{Root datum in $G_{twin}$ or $G_{pol}$ ?}\label{1.9}   %%%%%%%%%%%%%%%

\par We want to indicate some other relations between the groups defined above. 
For this we consider the definition of root datum given in \cite[1.5 p. 505]{R06}.
This is close to the definition of Bruhat and Tits in \cite{BrT72} or of R\'emy (as ``donn\'ee radicielle jumel\'ee'') in \cite{Re02}.
We shall not get all the axioms and moreover, mainly as $\QF_{a}$ is associated to $W_{a}$ which is not a Coxeter group, the known results for these more classical root data would not be available.

\par One considers the triple\qquad\qquad $(G_{twin}, (U_{\qa+r\qx})_{\qa+r\qx\in\QF_{a}},H=\g T(\k))$.

\medskip
\parni{\bf (1) (DR1)} $H$ is a subgroup of $G_{twin}\subset G_{pol}$, the $U_{\qa+r\qx}$ are non trivial subgroups normalized by $H$. 

\par This is clear.

\medskip
\parni{\bf(2) (DR2)} For $\set{\qa,\qb}\subset\QF$ prenilpotent  and $r,s\in \Z$, the commutator subgroup $[U_{\qa+r\qx},U_{\qb+s\qx}]$ is contained in the group generated by the $U_{p\qa+q\qb+(pr+qs)\qx}$ for $p,q\in\N\setminus\set0$ and $p\qa+q\qb\in \QF$.  

\par This comes from the explicit commutation relations of $\g U_{\qa}$ and $\g U_{\qb}$ (\cf \cite[9.2.2 p.207]{Re02}): $[x_{\qa}(u),x_{\qb}(v)]=\prod_{p,q}\, x_{p\qa+q\qb}(u^p.v^q.C^ {\qa,\qb}_{p,q})$ with $C^ {\qa,\qb}_{p,q}\in\Z$.

\medskip
\parni{\bf(3)} There is no need of {\bf(DR3)} as the system $\QF_{a}$ is reduced.

\medskip
\parni{\bf(4) (DR4)} For $\un\qa=\qa+s\qx\in\QF_{a}$ and $u\in U_{\qa+s\qx}, u\neq1$, there exist $u', u''\in U_{-\qa-s\qx}=U_{-\un\qa}$ such that $m(u)=u'uu''$ conjugates $U_{\qg+t\qx}$ into $U_{r_{\qa+s\qx}(\qg+t\qx)}$, for all $\qg+t\qx\in \QF_{a}$.
Moreover, $\forall u,v\in U_{\qa+s\qx}, u,v\neq1$, one asks $m(u)H=m(v)H$.

\par We prove this in three steps:

\par a) Let $u=x_{\qa+s\qx}(a)=x_{\qa}(\qp^s.a)\in U_{\qa+s\qx}\setminus\set1 \subset U_{\qa}\setminus\set1$ (\ie $a\in\k^*$).
To calculate in $\langle \g U_{\qa},\g U_{-\qa}\rangle$, one may use the group $\g{\mathrm{SL}}_{2}$ and the classical formula:

$$\left(\begin{matrix}1&0  \cr -d^ {-1}&1\cr \end{matrix}\right)\left(\begin{matrix}1&d \cr 0&1\cr \end{matrix}\right)\left(\begin{matrix}1&0  \cr -d^ {-1} &1\cr \end{matrix}\right)=\left(\begin{matrix}0 & d \cr -d^ {-1} &0\cr \end{matrix}\right)=\left(\begin{matrix}1&d  \cr 0&1\cr \end{matrix}\right)\left(\begin{matrix}1&0 \cr -d^ {-1}&1\cr \end{matrix}\right)\left(\begin{matrix}1&d  \cr 0 &1\cr \end{matrix}\right).$$

\par So one defines $u'=u''=x_{-\qa}(-(\qp^sa)^ {-1})=x_{-\qa-s\qx}(-a^ {-1})$.
Then $m_{\qa+s\qx}(u)=m_{\qa}(u)=u'uu''\in\g N(\sho)=N_{twin}$.
Clearly $m_{\qa+s\qx}(u).H=m_{\qa+s\qx}(v).H$ in the above situation, for $v\in U_{\qa+s\qx}\setminus\set1$ (by a calculation in $\g{\mathrm{SL}}_{2}$).

\par b) One has to identify the action of  $m_{\qa+s\qx}(u)\in N_{twin}$ on $\A_{twin}$ by the action $\qn$ of \S \ref{1.8}.
\par Let $\g v=\oplus,\qe=+$ or $\g v=\ominus,\qe=-$.
On $\A_{\g v}$, $\qn(m_{\qa+s\qx}(u))=\qn(m_{\qa}(u))$ is the reflection of $W$ with respect to the following wall of $\A_{\g v}$: $M(\qa+\qo_{\g v}(\qp^s.a))=M(\qa+\qe s)={\A_{\g v}}\cap M_{twin}(\qa+s\qx)$, where $M_{twin}(\qa+s\qx)$ is $\ker(\qa+s\qx)$.
On $\zv\A$, $\qn(m_{\qa+s\qx}(u))=\qn(m_{\qa}(u))=r_{\qa}$.

\par So the action of  $m_{\qa+s\qx}(u)\in N_{twin}$ on $\A_{twin}$ is the reflection $r_{\qa+s\qx}$ defined in \S \ref{1.4}.

\par c) One has to deduce from this that $m_{\qa+s\qx}(u)$ conjugates $U_{\qg+t\qx}$ into $U_{r_{\qa+s\qx}(\qg+t\qx)}$.

\par Actually, using the known results for $G$ acting on $\SHI_{\oplus}=\SHI(\g G,\shk,\qo_{\oplus})$, one gets that $m_{\qa+s\qx}(u)$ conjugates $U_{\qg+t\qx}$ into a subgroup of $\widehat U_{r_{\qa}(\qb),n}\cap G_{twin}$ (if $r_{\qa+s\qx}(\qg+t\qx)=r_{\qa}(\qb)+n\qx$), where $\widehat U_{r_{\qa}(\qb),n}=\prod_{m\geq n}\, U_{r_{\qa}(\qb)+m\qx}=\g x_{r_{\qa}(\qb)}(\qp^n\widehat\sho_{\oplus})$.
Now, if we calculate with $G$ acting on $\SHI_{\ominus}=\SHI(\g G,\shk,\qo_{\ominus})$, one gets that $m_{\qa+s\qx}(u)$ conjugates $U_{\qg+t\qx}$ into a subgroup of $G_{twin}$ and $^{\ominus}\widehat U_{r_{\qa}(\qb),n}=\prod_{m\leq n}\, U_{r_{\qa}(\qb)+m\qx}=\g x_{r_{\qa}(\qb)}(\qp^n\widehat\sho_{\infty})$.
As $\qp^n\widehat\sho_{\oplus} \cap \sho \cap \qp^n\widehat\sho_{\infty}=\qp^n\k$, one gets the expected result using Lemma \ref{1.5}.

\begin{remas*} 1) It is easy to prove that $m(u')=m(u'')$ acting on $\A_{twin}$ is also $r_{\qa+s\qx}=r_{-\qa-s\qx}$.

\par 2) One would like to say that $u$ (\resp $u',u''$) fixes the half-apartment $D_{twin}(\qa+s\qx)=\set{(x,p)\in\A_{twin} \mid (\qa+s\qx)(x,p)=\qa(x)+sp\geq0}$ (\resp $D_{twin}(-\qa-s\qx)$).
The boundary of these half-apartments is the wall $M(\qa+s\qx)=\ker(\qa+s\qx)$, fixed point set of $r_{\qa+s\qx}$.

\par Actually this is satisfied if we consider the restricted actions on $\A_{\g v}\subset{\SHI_{\g v}}$ and $\zv\A\subset {\zv\SHI}$.
\end{remas*}

\medskip
\parni{\bf(5) (DR5 ?)} For $\qe=\pm$, let $U_{\qe}:=\langle U_{\qa+r\qx} \mid \qa+r\qx \in \QF_{a\qe} \rangle$.
 Is it true that $H.U_{\qe}\cap U_{-\qe}=\set1$   ?

\par It seems difficult to answer these two questions (which are actually equivalent).

\par If we look at $G$ acting on $\SHI_{\oplus}$, then $H.U_{+}$ fixes the fundamental local chamber $C_{\oplus}\subset {\A_{\oplus}}$ (\ie $H.U_{+}\subset {I_{twin}}$, ``positive'' Iwahori subgroup of $G_{twin}$). But, if $\qa+r\qx\in\QF_{a-}=\QF^+_{a-}\cup\QF^-_{a-}$ and $u\in U_{\qa+r\qx}\setminus\set1$, then $u$ does not fix $C_{\oplus}$; so we get only the following weaker axiom.

\par{\bf(DR5'')} \quad $H.U_{\qe}\cap U_{\qa+r\qx}=\set1$, for any $\qa+r\qx\in\QF_{a(-\qe)}$.

\begin{NB} 1) The axiom (DR5') of \cite{R06} (weaker than (DR5'')) has no meaning here, as it involves ``simple roots'', which do not exist in $\QF_{a}$.

\par 2) To deduce (DR5) from (DR5''), one should generalize Theorem 3.5.4 in \cite{Re02}.
This is not at all clear (at least up to now).

\par 3) A good question may be: is $H.U_{+}$ equal to $ I_{twin}$ ? (see \ref{3.6}) 
\end{NB}

\medskip
\parni{\bf(6) (DRG ?)} Is $G_{twin}$ equal to $\langle H, (U_{\qa+r\qx})_{\qa+r\qx\in \QF_{a}} \rangle$   ?

\par This fails in general, even if this looks like the definition of $G_{twin}$: $G_{twin}\supset G'_{twin}:=\langle H, (U_{\qa+r\qx})_{\qa+r\qx\in \QF_{a}} \rangle$.
But in $G'_{twin}$ one has, a priori, only a subgroup of $\g N(\sho)=N_{twin}$, due to the fact that one finds only a subgroup of $\g T(\sho)=T_{twin}$.
It seems that $G'_{twin}\cap T_{twin}$ is generated by $H$ and the $m_{\qa+r\qx}(u)m_{\qa+s\qx}(v)^ {-1}$.
In particular the Weyl group associated to $G'_{twin}$ is certainly $W_{a}=W\zv\ltimes Q^\vee
$.

\subsubsection{Twin and twinnable apartments}\label{2.6}%%%%%%%%%%%%

We saw that the system of apartments $\sha_{\oplus}(G)=G.{\A_{\oplus}}=\g G(\shk).\A_{\oplus}$ of $\SHI_{\oplus}$ is smaller than the system of apartments  $\sha_{\oplus}(G_{\oplus})=G_{\oplus}.{\A_{\oplus}}=\g G(\widehat{\shk}_\oplus).\A_{\oplus}$   of $\widehat{\shi}_\oplus$ associated to the completion $\widehat\shk_{\oplus}=\k((\qp))$. As in section~\ref{s2}, we also consider the still smaller system of apartments (called twinnable apartments) $\sha_{\oplus}(G_{twin})=\sha_{\oplus twin}=G_{twin}.\A_{\oplus}$. 
By \S \ref{1.7}, $\sha_{\oplus}(G_{twin})$ is in bijection with $G_{twin}/N_{twin}$ or with the set $\SHT_{twin}$ of maximal split tori in $G$ conjugated to $\g T$ by $G_{twin}$ (that we may call ``twin maximal split tori'').

\par There are analogous things on the negative side:  $\sha_{\ominus}(G_{twin})=G_{twin}.{\A_{\ominus}}$.
The bijections  $\sha_{\oplus}(G_{twin}) \leftrightarrow G_{twin}/N_{twin} \leftrightarrow \SHT_{twin}   \leftrightarrow{\sha_{\ominus}(G_{twin})}$ tell that a positive (\resp negative) twinnable apartment has a unique twin  in $\sha_{\ominus twin}$ (\resp $\sha_{\oplus twin}$).
Classically a twin apartment is a pair $({A_{\oplus}},{A_{\ominus}})=g.({\A_{\oplus}},{\A_{\ominus}})\in {\sha_{\oplus}(G_{twin})}\times{\sha_{\ominus}(G_{twin})}$ (for $g\in G_{twin}$). We denote by $\sha_{twin}$\index{A@$\sha_{twin},\sha_{pol}$} the set of twin apartments: $\sha_{twin}=G_{twin}.(\A_{\oplus},\A_{\ominus})$. If $\g v\in \{\ominus, \oplus\}$, we call the apartments of $\sha_{\g v}(G_{twin})$ ``twinnable apartments''.

\par There is also a notion of twinnable apartment in the twin building $\zv\!\!\SHI={\zv\!\!\SHI^+}\sqcup{\zv\!\!\SHI^-}$ of $G$: $\zv\sha_{twin}=G_{twin}.{\zv\A}$ (\cf \ref{1.3}) and, as $\zv\sha_{twin}=G_{twin}/N_{twin}$ (\cf \S \ref{1.7}), the three sets $\zv\sha_{twin}$, $\sha_{\oplus}(G_{twin})$, $\sha_{\ominus}(G_{twin})$ are in one to one correspondance.

\par Note that the apartments of $\zv\!\!\SHI$ are often called twin in the classical litterature (see \S \ref{1.3}).
Of course we shall (now) avoid this terminology.

There are also analogous systems of apartments for $G_{pol}$. We define similarly $\sha_{\oplus}(G_{pol})=G_{pol}.\A_{\oplus}\simeq \sha_{\ominus}(G_{pol})=G_{pol}.\A_{\ominus}$ and $\sha_{pol}=G_{pol}.(\A_{\ominus},\A_{\oplus})$.\index{A@$\sha_{twin},\sha_{pol}$}
This is similar to the case of $G_{twin}$ since $G_{pol}/N_{pol}\simeq \mathscr{T}_{pol}$.
As $\sha_{twin}\simeq G_{twin}/N_{twin}$, $\sha_{pol}\simeq G_{pol}/N_{pol}$ and $N_{twin}=N_{pol}$ (\S \ref{1.6}), one has $\sha_{twin}=\sha_{pol}\iff G_{twin}=G_{pol}$.

\par Implicitly, we will refer to $G_{twin}$ instead  of $G_{pol}$: a twin apartment is a $G_{twin}$-twin apartment. We will sometimes refer to $G_{pol}$-twin apartments (or $G_{pol}$-twinnable apartments).

\par We say that two sets or filters $\QO_{1},\QO_{2}$ in $\SHI_{\oplus}\cup\SHI_{\ominus}$ are twin-friendly\index{twin-friendly} (\resp pol-friendly) if there exists $A\in\sha_{twin}$ (\resp $A\in\sha_{pol}$) containing $\QO_{1}\cup\QO_{2}$.

\begin{prop}\label{4.1}%%%%%%%%%%
Let $(x,y)\in {\SHI_{{\oplus}}}\times {\SHI_{{\ominus}}}$ be a twin-friendly pair (\ie there is a twin apartment $A_{{\oplus}}\times {A_{{\ominus}}}$ such that $x\in {A_{{\oplus}}}$ and $y\in {A_{{\ominus}}}$).         
One considers local chambers $C_{x}\subset {\SHI_{{\oplus}}}$, $C_{y}\subset {\SHI_{{\ominus}}}$ with respective vertices $x,y$.
Then $(C_{x},C_{y})$ is a twin-friendly pair (\ie there is a twin apartment $A_{{\oplus}}'\times {A_{{\ominus}}'}$ such that $C_{x}\in {A_{{\oplus}}'}$ and $C_{y}\in {A_{{\ominus}}'}$).
\end{prop}

\begin{NB} We may replace the local chambers by local facets or preordered segment germs.
\end{NB}

\begin{proof} We are easily reduced to prove that, if $(x,y)$ (\resp $(x,C_{y})$) is twin friendly, then $(C_{x},{y})$ (\resp $(C_{x},C_{y})$) is  twin-friendly.
And we may suppose $x\in {\A_{{\oplus}}}$ and $y\in {\A_{{\ominus}}}$ (\resp $C_{y}\subset {\A_{{\ominus}}}$).
Let $C_{1}$ be a local chamber in ${\A_{{\oplus}}}$ at $x$, with the same sign as $C_{x}$ and $(C_{1},C_{2},\ldots,C_{n}=C_{x})$ be a gallery of local chambers (in the tangent space $\sht_{x}({\SHI_{{\oplus}}})$).
We argue by induction on $n$, the case $n=1$ is clear and we are reduced to prove the case $n=2$: $C_{1}$ and $C_{x}$ are adjacent.
One writes $F$ the local panel common to $C_{1}$ and $C_{x}$.
If $F$ is in no wall, then $C_{x}\subset cl(C_{1})$ is in ${\A_{{\oplus}}}$, and we are done.
Otherwise $F$ is in a wall $M_{{\oplus}}(\qa+r)=M_{twin}(\qa+r\qx)\cap {\A_{{\oplus}}}$.
One of the two half-apartments $D_{twin}(\pm(\qa+r\qx))$ contains $y$ (\resp $C_{y}$), we may suppose it is $D_{twin}(\qa+r\qx)\supset {D_{{\oplus}}(\qa+r)}$.
Now there is an apartment $A$ of ${\SHI_{{\oplus}}}$  containing ${D_{{\oplus}}(\qa+r)}\cup C_{x}$ and $u\in {^{{\oplus}}U_{\qa,r}}$ such that $A=u.{\A_{{\oplus}}}$ (see \cite[1.4.3]{BPGR19} and \cite[5.7.7]{R16}).
Now $ {^{{\oplus}}U_{\qa,r+1}}$ fixes $u^ {-1}.C_{x}$ and $ {^{{\oplus}}U_{\qa,r}}=U_{\qa+r\qx}\times  {^{{\oplus}}U_{\qa,r+1}}$ (by \S \ref{1.5}).
So there is $u'\in  U_{\qa+r\qx}$ such that $C_{x}\subset u'.( {\A_{{\oplus}}})$.
As $U_{\qa+r\qx}\subset G_{twin}$ fixes $D_{twin}(\qa+r\qx)\cap{\SHI_{{\ominus}}}$, we are done.
\end{proof}

\subsection{Existence of an isomorphism fixing the intersection of two apartments }\label{ss_Existence_iso_fixing}

In this subsection, we prove that if $A$ and $B$ are twin apartments, then there exists $g\in G_{twin}$ such that $g.A=B$ and $g$ fixes $A\cap B$ (i.e, $g$ fixes $(A_\oplus\cap B_\oplus)\cup (A_{\ominus}\cap B_{\ominus})$) (see Theorem~\ref{thmIsomorphisms_fixing_twin_apartments}). This result is crucial in order to define a retraction centered at $C_\infty$ for example.

To that end, we begin by studying, for any place $\g v$ on $\shk$,  the properties of $G_{0_\g v}\cap U^+U^-N$, where $G_{0_\g v}$ is the fixator in $G$ of $0_{\g v}\in \shi_\g v$. We then  deduce a description of $G_{pol}\cap U^{\pm}$. Using these results, we prove a weak version of Theorem~\ref{thmIsomorphisms_fixing_twin_apartments}: we prove it in the case where $A_{\oplus}\cap B_{\oplus}$ and $A_{\ominus}\cap B_{\ominus}$ contain a chamber based at vertices of type $0$ (i.e elements of $G.0_{\oplus}$ or $G.0_{\ominus}$). We then deduce the theorem.

\subsubsection{Intersections of $G_{0_\g v}$ (fixator of $0_{\g v}$ in $G$) with $U^+U^-N$ or $U^+U^-$}\label{3.4}%%%%%%%%%%%%

\par Let $\g v$ be a place on $\shk$ with associated valuation $\qo$. We work in $\shi_{\g v}$. 

{
%\green{
\par One defines $Q^\vee_{\R,+}=\oplus_{i=1}^\ell\, \R_{\geq0}\qa_{i}^\vee\subset \A_{\g v}$ and, for $\qm={\sum_{i=1}^\ell}\, a_{i}\qa_{i}^\vee$, $\mathrm{ht}(\qm)={\sum_{i=1}^\ell}\, a_{i}$. 
One also chooses an element $\qz\in C\zv_{f}\cap Y\subset \A_{\g v}$.

\par The action of $T=\g T(\shk)$ on $\A_{\g v}$ is given by translations.
More precisely $t\in T$ acts by the translation $\qn(t)=\qn_{\qo}(t)$ of vector $\qn(t)=\qn_{\qo}(t)\in \A_{\g v}=Y\otimes\R$ given by: $\chi(\qn(t))=-\qo(\chi(t))$ for any $\chi\in X$.
In particular $\qn(\qp_{\g v}^\ql)=-\ql$ (if $\qp_{\g v}$ is a uniformizing parameter for $\qo_{}$).

\par We define $T_{\qo}(Q^\vee_{\R,+}):=\qn_{\qo}^ {-1}(Q^\vee_{\R,+})$.

\begin{lemm*} 1) $(U^+U^-N)\cap G_{0_\g v}\subset U^+U^-T_{\qo}(Q^\vee_{\R,+})W\zv$ and $(U^+U^-T)\cap G_{0_{\g v}}\subset U^+U^-T_{\qo}(Q^\vee_{\R,+})$.

\par 2) We have $(U^+U^-)\cap G_{0_{\g v}}=(U^+\cap G_{0_{\g v}})(U^-\cap G_{0_{\g v}})=U_{0_{\g v}}^+U_{0_{\g v}}^-$.
\end{lemm*}

\begin{proof}
1) Let $u^+\in U^+,u^-\in U^-$ and $n\in N$ be such that $u^+u^-n\in G_{0_{\g v}}$.
We write $n=t\tilde w$, with $t\in T$ and $\tilde w$ any representative of $w\in W\zv=N/T$ fixing $0_{\g v}$ (\eg $\tilde w\in \g N(\k)$). So $u^+u^-t\in G_{0_{\g v}}$.
We write $\qm=t.0_{\g v}\in \A_{\g v}$ (\ie $\qm=\qn(t)\in \A_{\g v}$).
We consider the retractions $\qr_{\pm\infty}$ of $\SHI$ onto $\A_{\g v}$ with center $\g Q_{\pm\infty}=germ_{\infty}(\pm C\zv_{f})$.
Now $x:=u^-t(0_{\g v})=u^-(\qm)$ satisfies $\qr_{-\infty}(x)=\qm$ and $\qr_{+\infty}(x)=0_{\g v}$ (as $u^+(x)=0_{\g v}$).
By \cite [7.6.1]{Heb18f}=\cite [6.5.1]{Heb18e} or \cite[3.1]{He17}, one has $-\qm\in-Q^\vee_{\R,+}$, so $\qn(t)=\qm\in Q^\vee_{\R,+}$ and $t\in T_\omega(Q^\vee_{\R,+})$.

2) Let $u^+\in U^+,u^-\in U^-$ be such that $u^+u^-\in G_{0_{\g v}}$. Let $x=u^-.0_{\g v}$. Then we have $\rho_{-\infty}(x)=0_{\g v}$ and $\rho_{+\infty}(x)=u^+u^-.0_{\g v}=0_{\g v}$, since $\rho_{+\infty}(x)$ is the unique element of $U^+.x\cap \A_{\g v}$.  Using  \cite[Corollary 4.4]{He17}, we deduce $x\in \A_{\g v}$, and hence $x=\rho_{-\infty}(x)=u^-.0_{\g v}=0_{\g v}=u^+u^-.0_{\g v}$, which proves the lemma. 
\end{proof}

\subsubsection{Application to $G_{pol}$}\label{3.5}%%%%%%%%%%%%

\par We consider now all the places of $\shk$ and the associated valuations.

\medskip
\par We are first looking at $U^\pm\cap G_{pol}=:U^\pm_{pol}\supset U^\pm_{twin}$.

\par From \ref{3.1}.2 we know that, for $\qo=\qo_{\g v}$, $\g v\neq \oplus,\ominus$, $U^ {ma+}_{0_{\g v}}=\prod_{\qa\in\QD^+}\, X_{\qa}(\g g_{\qa,\Z}\otimes \shk_{\qo\geq 0})$, where $\shk_{\qo\geq 0}=\{x\in\shk \mid \qo(x)\geq0 \}=\sho_{\g v}$ and $U^ {pm+}_{0_\g v}=U^ {ma+}_{0_{\g v}}\cap G$ is the fixator of $0_{\g v}$ in $U^+$ for the action on $\SHI_{\qo}$ (\cf \ref{3.1}.3).
As the product decomposition of $U^ {ma+}$ is unique (\cf \ref{2.1}) and ${\sho}=\cap_{\g v\neq\oplus,\ominus}\, \shk_{\qo\geq0}$, one gets:
 \qquad\qquad $$U^\pm\cap G_{pol}=(\prod_{\qa\in\QD^+}\, X_{\qa}(\g g_{\qa,\Z}\otimes {\sho}))\cap G.$$

\par And clearly, if $\QO\subset{\SHI_{\g v}}$ ($\g v=\oplus$ or $\ominus$), its fixator in $U^\pm\cap G_{pol}$ is: 
$$U^\pm(\QO)\cap G_{pol}=(\prod_{\qa\in\QD^+}\, X_{\qa}(\g g_{\qa,\Z}\otimes {\sho}_{\qo\geq f_{\QO}(\qa)}))\cap G.$$
where $\sho_{\qo\geq f_{\QO}(\qa)}=\{ x\in{\sho} \mid \qo(x)\geq f_{\QO}(\qa)\}$. 

\par One may also write a formula for $U^\pm({\QO_{\oplus}}\cup{\QO_{\ominus}})\cap G_{pol}$ when ${\QO_{\oplus}}\subset {\SHI_{\oplus}}$, ${\QO_{\ominus}}\subset {\SHI_{\ominus}}$.

\subsubsection{A particular case of Theorem~\ref{thmIsomorphisms_fixing_twin_apartments}}\label{3.7}%%%%%%%%%%%%

\par 1) We may include $\qa_{1}^\vee,\ldots,\qa_{\ell}^\vee$ in a $\Q-$basis of $Y\otimes\Q$.  
So, taking a ``dual basis'' there is $(\chi_{1},\ldots,\chi_{d})\in X^d$ that is  an $\R-$basis of $\A^*$ (\ie a $\Q-$basis of $X\otimes\Q$) and satisfies $\chi_{i}(\qa_{j}^\vee)=m_{i}\qd_{i,j}$ for $1\leq i\leq d, 1\leq j\leq\ell$ with $m_{i}\in\N_{>0}$.
Actually in the simply connected case (\ie when $\oplus_{i=1}^\ell\Z\qa_{i}^\vee$ is a direct factor in $Y$), one may suppose that $(\chi_{1},\ldots,\chi_{d})$   is a $\Z-$basis of $X$ and $m_{i}=1$.
We have $Q^\vee_{\R,+}=\{ x\in \A \mid \chi_{i}(x)\geq0 \text{\ for\ }1\leq i\leq\ell ; \chi_{i}(x)=0 \text{\ for\ }i>\ell\}$. 
And for $\qm=\sum_{i=1}^\ell a_{i}\qa_{i}^\vee$, we have $a_{i}=\chi_{i}(\qm)/m_{i}$ and $\mathrm{ht}(\qm)=\sum_{i=1}^\ell \chi_{i}(\qm)/m_{i}$ (notation of \ref{3.4}).

\par 2) Let $\g v$ be a place on $\shk$  (typically $\g v\neq \oplus,\ominus$),  and $\omega=\omega_{\g v}$.  
 We write $\qn_{\qo}$ the action of $T$ on $\A_{\g v}\subset \SHI_{\g v}$ associated  $\g v$ and $T_{\qo}(Q^\vee_{\R,+})=\qn_{\qo}^ {-1}(Q^\vee_{\R,+})$.
As $\chi_{i}(\qn_{\qo}(t))=-\qo(\chi_{i}(t))$, we have $\qn_{\qo}(t)=-\sum_{i=1}^d\,\frac{\qo(\chi_{i}(t))}{m_{i}}\qa_{i}^\vee$ for any $t\in T$.
So $t\in T_{\qo}(Q^\vee_{\R,+})\iff \qo(\chi_{i}(t))\leq0$ for $1\leq i\leq\ell$ and $\qo(\chi_{i}(t))=0$ for $i>\ell$. And then $\mathrm{ht}(\qn_{\qo}(t))=-\sum_{i=1}^d\,\frac{\qo(\chi_{i}(t))}{m_{i}}$.

\par 3) Let us now consider $u^+\in U^+, u^-\in U^-$ and $t\in T$ such that $u^+u^-t\in G_{pol}$ (actually by the proof of Lemma \ref{3.4}.1, the study of $U^+U^-N\cap G_{pol}$ may be reduced to this case).
By Lemma \ref{3.4}.1, we have then $t\in T_{\qo}(Q^\vee_{\R,+})$, $\forall \qo\neq\qo_{\oplus},\qo_{\ominus}$.
So $\qo(\chi_{i}(t^ {-1}))\geq0$ for $1\leq i\leq\ell$ and $\qo(\chi_{i}(t^ {-1}))=0$ for $i>\ell$. 
This means that $\chi_{i}(t^ {-1})\in\sho$ for $1\leq i\leq\ell$ and $\chi_{i}(t^ {-1})\in\sho^*$  for $i>\ell$.

\begin{lemm}\label{4.3}  Let $C_{x}\subset \SHI_{{\ominus}}$ and  $C_{y}\subset \SHI_{{\oplus}}$ be local chambers with respective vertices $x$ and $y$.
We suppose $x$ and $y$ of type $0$, i.e they are conjugated by $G$ to $0_{\ominus}$ and $0_{\oplus}$ respectively .
We consider two twin apartments $A_{1},A_{2}\in{\sha_{twin}}$ containing  $C_{x}\cup C_{y}$.
Then there is $g\in G_{twin}$ fixing $C_{x}$ and $C_{y}$ such that $A_{2}=g.A_{1}$.
\end{lemm}

 %\marge{{new $\downarrow$}}
 
%\par (2) We deal below only with the case {$C_{x}={C_{\ominus 0}^\pm}$}, the fundamental positive/negative chamber in {$\SHI_{{\ominus}}$}. %(perhaps the general case is not much more difficult).
%It is sufficient for the definition of $\qr_{I_{\infty}}$.
%Unfortunately the following proof is written for $C_{\ominus 0}^+$; but for ${C_{\infty}=C_{\ominus 0}^-}$ one has only to exchange the signs (of $U^\pm$).

%\par (3) Obviously we may modify $(C_{x},y)$ or $(C_{x},C_{y})$ by the action of $G_{twin}$, it does not change the result of the Lemma.
%So this Lemma (or its analogues in (1) above) is proved under the unique (additional) condition that $x$ is a vertex of type $0$ in $\SHI_{{\ominus}}$ (\ie in the orbit of $0$ under $G_{twin}$ or $G$).

\begin{proof} The action of $G_{twin}$ permutes transitively the twin apartments and the action of the stabilizer $N_{twin}$ of $\A$ in $G_{twin}$ permutes transitively the local chambers in $\A_{\oplus}$, of a given sign and with a vertex of type $0$.
So one may suppose $(A_{1},A_{2})=(\A,A)$, $y=0_{\oplus}$, $C_{y}=C_{\oplus}\subset{\A_{{\oplus}}}, C_{x}\subset{\A_{{\ominus}}}$, both contained in $A\cap\A$. 
Then, by Proposition \ref{2.11} and \S \ref{2.6}, there exist $^+g\in G_{twin}\cap G_{C_{\oplus}}$ and $^-g\in G_{twin}\cap G_{C_{x}}$ such that $A={^+g}\A={^-g}\A$. We would like that ${^+g}={^-g}$ or, more generally, that ${^+g}={^-g}t$ with $t\in T$ fixing $\A$.
But from ${^+g}\A={^-g}\A$ and $^+g,{^-g}\in G_{twin}$, we get only ${^+g}={^-g}n$, with $n\in N_{twin}=N\cap G_{twin}$. 

\par One writes $^+g=u_{1}^+u_{1}^-t_{1}$ and $^+g={^-gn}=u_{2}^+u_{2}^-t'_{2}n=u_{2}^+u_{2}^-n_{2}$, with $u_{1}^+,u_{2}^+ \in U^+$, $u_{1}^-,u_{2}^- \in U^-$, $t'_{2},t_{1}\in T$, $n_{2}=t'_{2}n\in N$ and moreover $u_{1}^+u_{1}^-t_{1}\in G_{twin}\cap G_{C_{\oplus}}$ (so $u_{1}^+,u_{1}^-\in G_{C_{\oplus}}$ and $t_{1}$ fixes $\A_{{\oplus}}$ by Proposition \ref{3.2}) and $u_{2}^+u_{2}^-t'_{2}={^-g} \in  G_{C_{x}}\cap G_{twin}$ (so $u_{2}^+,u_{2}^- \in  G_{C_{x}}$ and $t'_{2}$ fixes $\A_{{\ominus}}$ by Proposition \ref{3.2}).
We want to prove that $n_{2}$ fixes $C_{x}$ and $C_{y}$.

\par One writes $n_{2}=t_{2}\tilde w$ with $t_{2}\in T$ and $\tilde w$ any representative of $w\in W\zv=N/T$ in $\g N(\k)\subset G_{twin}$.
In particular $\tilde w$ fixes $0_{\g v}$ in any masure $\SHI_{\g v}$.

\par (a) But $^+g=u_{1}^+u_{1}^-t_{1}=u_{2}^+u_{2}^-t_{2}\tilde w$ is in $G_{pol}$ and fixes $0_{\oplus}$ in $\SHI_{{\oplus}}$, so $g_{2}:=u_{2}^+u_{2}^-t_{2}$ is in $G_{pol}$. %and fixes $0$ in $^+\SHI$.
By {\S \ref{3.7}}.3, we get  $\chi_{i}(t_{2}^ {-1})\in\sho$, $\forall i=1,\ldots,d$.
Moreover  $g_{2}={^+g}\widetilde w^ {-1}$ fixes $0_{\oplus}$ in $\SHI_{{\oplus}}$, so $\qo_{\oplus}(\chi_{i}(t_{2}^ {-1}))\geq0$ (by \S \ref{3.4}.1 and \ref{3.7}.2), $\chi_{i}(t_{2}^ {-1})\in\k[\qp]$ and $\qo_{\ominus}(\chi_{i}(t_{2}^ {-1}))\leq0$.

\par(b) Now $u_{2}^+$ and $u_{2}^-$ fix $C_{x}\subset{\SHI_{{\ominus}}}$, so one may write $u_{2}^+u_{2}^-=u_{3}^-u_{3}^+t_{3}$ with $u_{3}^-\in U^-,u_{3}^+\in U^+,t_{3}\in T$, all fixing $C_{x}$ (by \ref{3.2}).
In particular $\qo_{{\ominus}}(\chi_{i}(t_{3}))=0$, $\forall i=1,\ldots,d$ by \S {\ref{3.7}.2} (formula for $\qn_{\qo}(t)$).

\par (c) But $g_{2}=u_{3}^-u_{3}^+t_{3}t_{2}\in G_{pol}$, so, by \S \ref{3.7}.3 and \ref{3.4}.b, $\chi_{i}(t_{3}t_{2})\in\sho$, $\forall i=1,\ldots,d$.
We also know that $g_{2}$ fixes $0_{\oplus}$ in $\SHI_{{\oplus}}$.
So, by \S \ref{3.4}.b, $t_{3}t_{2}\in T_{\qo_{{\oplus}}}(-Q^\vee_{\R,+})$, \ie (by \S \ref{3.7}.2) $\qo_{{\oplus}}(\chi_{i}(t_{3}t_{2}))\geq0$.
We deduce from this that $\chi_{i}(t_{3}t_{2})\in\k[\qp]$, hence $\qo_{{\ominus}}(\chi_{i}(t_{3}t_{2}))\leq0$.
But $\qo_{{\ominus}}(\chi_{i}(t_{3}))=0$ (by (b) above), so $\qo_{{\ominus}}(\chi_{i}(t_{2}))\leq0$.
Comparing with (a), we get $\qo_{{\ominus}}(\chi_{i}(t_{2}^ {\pm1}))=0$.
But $\chi_{i}(t_{2}^ {-1})\in\k[\qp]$ by (a), so $\chi_{i}(t_{2})\in\k$.
Hence $t_{2}$ fixes $\A_{{\ominus}}$ and $\A_{\oplus}$, $n_{2}=t_{2}\widetilde w$  fixes $0_{{\ominus}}$ and $0_{\oplus}$.
%\subset{\SHI_{{\ominus}}}$ and $^+g=u_{2}^+u_{2}^-t_{2}\tilde w$ fixes $0$ in $\A_{{\ominus}}\subset{\SHI_{{\ominus}}}$.

\par (d) Now $\g T=\g{Mult}^d$, we write $\qth_{j}$ the $j^ {th}$ coordinate map.
As $\chi_{1},\ldots,\chi_{d}\in X$ is a $\Q-$basis of $X\otimes_{\Z}\Q$, we have $n_{j}\in\Z_{>0}$ and $b_{j,i}\in\Z$ with $n_{j}\qth_{j}=\sum_{i}b_{j,i}\chi_{i}$.
So $\qth_{j}(t_{2})^ {n_{j}}=\prod_{i}\chi_{i}(t_{2})^ {b_{j,i}}\in\k$.
As $\qth_{j}(t_{2})\in\shk=\k(\qp)$, we get $\qth_{j}(t_{2})\in\k$, \ie $t_{2}\in\g T(\k)\subset G_{twin}$ and $n_{2}=t_{2}\widetilde w \in G_{twin}$.
So $u_{2}^+u_{2}^-={^+g}n_{2}^ {-1}\in G_{twin}$ and $t'_{2}\in G_{twin}$; one may replace $^-g=u_{2}^+u_{2}^-t'_{2}$ by $u_{2}^+u_{2}^-$ \ie suppose $t'_{2}=1$.
Symmetrically we get also $t_{1}\in G_{twin} \cap \g T(\shk)$ and one may replace $^+g$ by $u_{1}^+u_{1}^-$ \ie suppose $t_{1}=1$.

%As $^+g=u_{1}^+u_{1}^-t_{1}$ is in $G_{pol}$, one has $\chi_{i}(t_{1}^ {-1})\in\sho$, $\forall i=1,\ldots,d$ (\cf 
%But $t_{1}$ fixes $0$ in $\SHI_{{\oplus}}$, so $\qo_{{\oplus}}(\chi_{i}(t_{1}^ {-1}))=0$, $\chi_{i}(t_{1}^ {-1})\in\k[\qp]\setminus \qp\k[\qp]$ and $\qo_{{\ominus}}(\chi_{i}(t_{1}^ {-1}))\leq0$.
%As $^+g$ fixes $0$ in $\A_{{\ominus}}\subset{\SHI_{{\ominus}}}$, one has $\qo_{{\ominus}}(\chi_{i}(t_{1}^ {-1}))\geq0$ (\cf \ref{3.7}.2 and \ref{3.4}.1).
%Finally $\qo_{{\ominus}}(\chi_{i}(t_{1}^ {-1}))=0$, hence $\chi_{i}(t_{1}^ {-1})\in\k$, $t_{1}\in \g T(\k))=H$.

\par (e) %From this we deduce that $u_{1}^+u_{1}^-$ fixes $0$ in $\A_{{\ominus}}\subset{\SHI_{{\ominus}}}$, and so do $u_{1}^+$, $u_{1}^-$ by \ref{3.4}.3.
We argue now in the tangent twin building $\sht_{0_{\oplus}}({\SHI_{{\oplus}}})$ and use that $^+g=u_{1}^+u_{1}^-=u_{2}^+u_{2}^-t_{2}\widetilde w$ with $u_{1}^\pm$ fixing $C_{\oplus}$, $t_{2}$ fixing $\A_{\oplus}$. %, \tilde w$ fixing $0$ in $\A_{{\ominus}}\subset{\SHI_{{\ominus}}}$.
But $u_{2}^+u_{2}^-={^+g}(t_{2}\widetilde w)^ {-1}$ fixes $0_{\oplus}$ in $\SHI_{{\oplus}}$, and so do $u_{2}^+$, $u_{2}^-$ by \S \ref{3.4}.2.
Hence $u_{2}^+$ fixes $C_{\oplus}=germ_{0_{\oplus}}(C\zv_{f})$ and $u_{2}^-$ fixes $C_{0}^-:=germ_{0_{\oplus}}(-C\zv_{f})\subset\A_{\oplus}$.
We have $C_{0}^-=u_{2}^-.C_{0}^-=(u_{2}^+)^ {-1}u_{1}^+u_{1}^-(t_{2}\tilde w)^ {-1}.C_{0}^-=(u_{2}^+)^ {-1}u_{1}^+u_{1}^-\tilde w^ {-1}.C_{0}^-$.
We consider now the retraction $\qr^+$ of $\sht_{0_{\oplus}}({\SHI_{{\oplus}}})$ onto $\sht_{0_{\oplus}}({\A_{{\oplus}}})$ with center $C_{\oplus}$.
As $u_{2}^+,u_{1}^+$ and $u_{1}^-$ fix $C_{\oplus}$, we get $C_{0}^-=\qr^+(C_{0}^-)=\widetilde w^ {-1}.C_{0}^-$.
We have proved that the class $w$ of $\widetilde w$ in $W\zv$ is trivial.
We could have taken $\widetilde w=1$ and then $n_{2}=t_{2}$ fixes $\A$ as expected.
%\par (Actually this use of the tangent building may certainly be replaced by some Birkhoff decomposition of $\g G(\k)$, considered as a subquotient of $G$ in which $u_{2}^-$ is trivial).
\end{proof}
%\eject %%%%%%%%%%%%%%%%%%%%%%%

\subsubsection{Conclusion}

We now extend the result of Lemma~\ref{4.3} to arbitrary pairs $A,B$ of $\sha_{twin}$. We begin with the case where $A_{\oplus}\cap B_{\oplus}$ and $A_{\ominus}\cap B_{\ominus}$ have nonempty interior and then drop this condition.

\begin{lemm}\label{lemIsomorphisms_fixing_twin_apartments_nonempty_interior}
Let $A,B\in \sha_{twin}$ be such that $A_{\oplus}\cap B_{\oplus}$ and $A_{\ominus}\cap B_{\ominus}$ have non-empty interior. Then there exists $g\in G_{twin}$ such that $g.A=B$ and $g$ fixes $A\cap B$ (i.e $g$ fixes pointwise $(A_{\oplus}\cap B_{\oplus})\sqcup (A_{\ominus}\cap B_{\ominus})$).
\end{lemm}

\begin{proof}
Using isomorphism of apartments, we may assume that $A=\A$. 
We fix an element of $y\in \A_{\ominus}\cap B_{\ominus}$.
  As $\A_{\oplus}\cap B_{\oplus}$ (\resp $\A_{\ominus}\cap B_{\ominus}$) has non-empty interior, there exists $n\in \N^*$ such that $\A_{\oplus}\cap B_{\oplus}$ (\resp $\A_{\ominus}\cap B_{\ominus}$) contains an element $C_{x}$ of $G_{twin}.(\frac{1}{n}Y+C_{\oplus })$ (\resp $C_{y}$ of $G_{twin}.(\frac{1}{n}Y+C_{\ominus })$).  
  Let $\shk^{(n)}=\k(\qp^{1/n})$, where $\qp^{1/n}$ is an indeterminate such that $(\qp^{1/n})^n=\qp$. Let $G^{(n)}=\g G(\shk^{(n)})$. We add an exponent $(n)$ when we consider an object corresponding to $G^{(n)}$ (for example we have $\shi_{\oplus}^{(n)},\shi_{\ominus}^{(n)}$, $G_{twin}^{(n)}$, $\A_{\oplus}^{(n)}$, $\ldots$).  We have $\shi_{\oplus}^{(n)}\supset  \shi_{\oplus}$ and $\shi_{\ominus}^{(n)}\supset  \shi_{\ominus}$. As an affine space, $\A_{\oplus}^{(n)}$ can be identified with $\A_{\oplus}$. 
However, it contains more walls, and we have $Y^{(n)}=\frac{1}{n}Y$. 
Therefore by Lemma~\ref{4.3} applied with $G^{(n)}_{twin}$ instead of $G_{twin}$, there exists $g_{y}\in G_{twin}^{(n)}$ fixing $C_{x}\cup C_{y}$ and such that $g_{y}.\A=B$. 
By Proposition~\ref{2.11}, there exists $h_y\in G_{twin}$ such that $h_{y}.\A_{\oplus}=B_{\oplus}$ (hence $h_{y}.\A=B$) and $h_{y}$ fixes $\A_{\oplus}\cap B_{\oplus}$.
 Then $g_{y}^{-1}h_{y}$ stabilizes $\A_\oplus$ and is an element of $G^{(n)}_{twin}$.  
 Therefore $g^{-1}_yh_y$ is an element of $N_{twin}^{(n)}$.
  Moreover $g^{-1}_yh_y$ fixes $C_{x}$ and thus $g^{-1}_yh_y$ fixes $\A_{\oplus}$. 
  Using \S \ref{1.7} we deduce that  $g^{-1}_yh_y$ fixes $\A_{\ominus}$.
   Hence $h_y$ fixes $(\A_{\oplus}\cap B_{\oplus})\sqcup C_{y}$.
 By Proposition \ref{2.11}, there exists $h_{x}\in G_{twin}$ such that $h_{x}.\A=B$ and $h_{x}$ fixes $\A_{\ominus}\cap B_{\ominus}$.
  So $h_{x}^ {-1}h_{y}$ stabilizes $\A_{\ominus}$ and fixes $C_{y}$: it is the identity on $\A_{\ominus}$. This  
  proves that $h_y$ fixes $\A_{\ominus}\cap B_{\ominus}$ and completes the proof of the lemma. 
\end{proof}

The following proposition corresponds to \cite[Proposition 2.9 1)]{R11} in the twin case.

\begin{prop}\label{propRou11_2.10}
Let $\g v\in \{\ominus,\oplus\}$.  Let $A_{\g v}$ be a twinnable apartment in the masure ${\SHI}_{\g v}$. Let $M$ be a wall of $A_{\g v}$ and $C$ be a (local) chamber of $\shi_\g v$ not in $A_{\g v}$, but dominating a (local) panel of $M$. Then there exist two twinnable apartments $A_{1,\g v}$ and $A_{2,\g v}$ of $\shi_{\g v}$ such that:\begin{enumerate}
\item $A_{1,\g v}$ and $A_{2,\g v}$ contain $C$,

\item $A_{1,\g {v'}}\cap A_{\g {v'}}$ and $A_{2,\g {v'}}\cap A_{\g {v'}}$ (\resp $A_{1,\g {v'}}\cap A_{\g {v'}}$ and $A_{1,\g {v'}}\cap A_{2,\g {v'}}$, $A_{2,\g {v'}}\cap A_{\g {v'}}$ and $A_{1,\g {v'}}\cap A_{2,\g {v'}}$) are two opposite half-apartments of $A_{\g {v'}}$ (\resp $A_{1,\g {v'}}$, $A_{2,\g {v'}}$) for both $\g {v'}\in \{\ominus,\oplus\}$.

\end{enumerate} 
\end{prop}

\begin{proof}
Using apartment isomorphisms, we may assume that $A_{\g v}=\A_{\g v}$. 
 Let $D_{\g v}$ be a half-apartment of $A_{\g v}$ delimited by $M$.   By \cite[Proposition 2.9 1)]{R11}, there exists an apartment $\tilde{B}_{\g v}$ of ${\SHI}_{\g v}$ containing $D_{\g v}$ and $C$.  
 By {\ref{n1.3.1}}, we can write $\tilde{B}_{\g v}=x_\alpha(y).\A_{\g v}$, for some $\alpha\in \Phi$ and $y\in \shk$, with $x_\alpha(y)$ fixing $D_{\g v}$.
  Let $z\in \k^* \qp^\Z$  be such that $\omega_{\g v}(y-z)>\omega_{\g v}(y)$. 

 Let $A_{1,\g v}=x_\alpha(z).\A_{\g v}$. Then \[\begin{aligned}A_{1,\g v}\cap \tilde{B}_{\g v} &=x_\alpha(y).(x_\alpha(-y).A_{1,\g v}\cap x_\alpha(-y).\tilde{B}_{\g v})\\
&= x_\alpha(y).(x_\alpha(z-y).\A_{\g v}\cap \A_{\g v}).\end{aligned}\]

As $C\not \subset \A_{\g v}$, we have $\tilde{B}_{\g v}\cap \A_{\g v}=D_{\g v}$. Moreover $D_{\g v}=\{a\in \A_{\g v}\mid\alpha(a)+\omega_{\g v}(y)\geq 0\}$ and $\A_{\g v}\cap x_\alpha(z-y).\A_{\g v}=\{a\in \A_{\g v}\mid\alpha(a)+\omega_{\g v}(z-y)\geq 0\}$. 
Therefore $ \A_{\g v}\cap x_\alpha(z-y).\A_{\g v}\supsetneq D_{\g v}$ and thus $ \A_{\g v}\cap x_\alpha(z-y).\A_{\g v}$ contains any local chamber of $\A_{\g v}$ which dominates some local panel of $M$. 
Therefore $A_{1,\g v}$ contains $D_{\g v}$ and $C$. 
%Moreover $A_{1,\g v'}\cap A_{\g v'}=D_{\g v'}$.
%\red{$D_{\g v'}$ n'a pas ete defini?}
%Let now $A_{2}=x_{-\alpha}(z^{-1}).\A$. 
%Then $A_{2,\g v'}\cap A_{\g v'}=\overline{A_{\g v'}\setminus D_{\g v'}}$ and  $r:=x_{-\alpha}(-z^{-1})x_\alpha(z)x_{-\alpha}(-z^{-1})\in N_{twin}$ induces reflections with respect to the wall $\{a\in \A_{\g v'}\mid\alpha(a)+\omega_{\g v'}(z)=0\}$, for both $\g v'\in \{\ominus,\oplus\}$. Hence we have (2) and thus we have (1), which proves the proposition. 
%\red{a partir de "Moreover $A_{1,\g v'}\cap A_{\g v'}=D_{\g v'}$", la preuve me parait bizarre. Je propose de remplacer par ceci:}
%\blue{  OK  }
Moreover if $\g v'\in \{\ominus,\oplus\}$, then $A_{1,\g v'}\cap A_{\g v'}=: D_{\g v'}$ is a half-apartment. Let now $A_{2}=x_{-\alpha}(z^{-1}).\A$. 
Then $A_{2,\g v'}\cap A_{\g v'}=\overline{A_{\g v'}\setminus D_{\g v'}}$ and  $r:=x_{-\alpha}(-z^{-1})x_\alpha(z)x_{-\alpha}(-z^{-1})\in N_{twin}$ induces reflections with respect to the wall $\{a\in \A_{\g v'}\mid\alpha(a)+\omega_{\g v'}(z)=0\}$. Hence we have (2) and thus we have (1), which proves the proposition.
\end{proof}

\begin{lemm}\label{lemIso_fixing_twin_pairs_points}
Let $A,B\in \sha_{twin}$. Then for all $(x,y)\in (A_\oplus\cap B_\oplus)\times (A_{\ominus}\cap B_{\ominus})$, there exists $g\in G_{twin}$ fixing $x,y$ and such that $g.A=B$. 
\end{lemm}

\begin{proof}
Considering local chambers $C_{x}\subset B_{\oplus}$, $C_{y}\subset A_{\ominus}$ and a third twin apartment $B'$ containing $C_{x}\cup C_{y}$ (by Proposition \ref{4.1}), we are reduced to consider the case where $A_\oplus\cap B_\oplus$ or $A_\ominus\cap B_\ominus$ contains a local chamber.
We choose the case $A_\ominus\cap B_\ominus\supset C_{y}$; the other case is similar.
  Let $C$ (resp. $C'$) be a positive local chamber of $A_{\oplus}$ (resp $B_{\oplus}$) based at $x$ and $\Gamma=(C_1,\ldots,C_n)$ be a minimal gallery of local chambers at $x$ from $C=C_1$ to $C'=C_n$.   Let $P$ be the panel dominated by both $C_1$ and $C_2$. There are two cases: either  the panel $P$ is not contained in any wall of $A_\oplus$, or  the panel $P$ is contained in exactly one wall of $A_\oplus$.

In the first case, any half-apartment containing $C_1$ contains $C_2$ and thus  the enclosure of $C_1$ contains $C_2$. By (MA II) we deduce that   $A_\oplus$ contains $C_2$ so we can replace $\Gamma$ by the gallery $(C_2,\ldots,C_n)$.

We now assume that we are in the second case.    Let $D_{1,\oplus}$, $D_{2,\oplus}$ be the two half-apartments of $A_{\oplus}$ delimited by $P$. By Proposition~\ref{propRou11_2.10}, there exist  twin apartments $A_{1}$ and $A_{2}$ such that $A_\oplus\cap A_{i,\oplus}=D_{i,\oplus}$ for both $i\in \{1,2\}$. Then $A_{\ominus}\cap A_{1,\ominus}$ and $A_{\ominus}\cap A_{2,\ominus}$ are two opposite half-apartments of $A_{\ominus}$. 
Therefore $A_{1,\ominus}$ or $A_{2,\ominus}$ contains $C_{y}$ and there exists $i\in \{1,2\}$ such that $A\cap A_i\supset D_{i,\oplus}\cup C_{y}$. 
By Lemma~\ref{lemIsomorphisms_fixing_twin_apartments_nonempty_interior}, there exists $g\in G_{twin}$ such that $g.A=A_i$ and $g$ fixes $x$ and $C_{y}$. 
By induction, we deduce that we can assume that $A\cap B$ contains $C_n$ and $C_{y}$. Then by Lemma~\ref{lemIsomorphisms_fixing_twin_apartments_nonempty_interior}, there exists $g\in G_{twin}$ fixing $x,y$ and such that $g.A=B$, which proves the lemma.
\end{proof}

\begin{theo}\label{thmIsomorphisms_fixing_twin_apartments}
Let $A,B\in \sha_{twin}$. Then there exists $g\in G_{twin}$ such that $g.A=B$ and such that $g$ fixes $A\cap B$ (i.e $g$ fixes pointwise $(A_{\oplus}\cap B_{\oplus})\sqcup (A_{\ominus}\cap B_{\ominus})$). 
\end{theo}

\begin{proof}
We identify $A$ and $\A$. We assume that $\A_{\oplus}\cap B_{\oplus}$ and  $\A_{\ominus}\cap B_{\ominus}$ are non-empty, since otherwise we can use Proposition~\ref{2.11}. Fix $y\in \A_{\ominus}\cap B$.
By (MAII) in \ref{n1.3.1}, $\A_{\oplus}\cap B_{\oplus}$ is a finite intersection of half-apartments in $\A_{\oplus}$.
In particular it is convex and the closure of its relative interior $(\A_{\oplus}\cap B_{\oplus})^\bullet$ (the interior of $\A_{\oplus}\cap B_{\oplus}$ considered inside the support $V_{0}$ of $\A_{\oplus}\cap B_{\oplus}$ in $\A_{\oplus}$).
   We regard $\A_{\oplus}$ as an $\R$-vector space and $V_0$ as an affine subspace of $\A$. Let $\vec{V_0}$ be the direction of $V_0$. 
   If $\vec{V}$ is a vector subspace of $\vec{V_0}$, we say that  $\vec{V}$ satisfies the property $\mathscr{P}$  if for all $x\in (\A_{\oplus}\cap B_{\oplus})^\bullet$, there exists $h_{x,\vec{V}}\in G_{twin}$ such that $h_{x,\vec{V}}.\A=B$ and  $h_{x,\vec{V}}$ fixes $(x+\vec{V})\cap \A_{\oplus}\cap B_{\oplus}$ and $y$. 
   Then $\{0_\oplus\}$ satisfies  $\mathscr{P}$ by Lemma~\ref{lemIso_fixing_twin_pairs_points}. 
   Let $\vec{V}$ be a vector subspace of $\vec{V_0}$ satisfying $\mathscr{P}$. Assume $\vec{V}\neq \vec{V_0}$ and take $v\in \vec{V_0}\setminus \vec{V}$. 
   Let $h\in G_{twin}$ be such that $h.\A=B$ and such that $h$ fixes $\A_{\oplus}\cap B_{\oplus}$ (the existence of such an $h$ is provided by Proposition~\ref{2.11}). 
 For $x\in (\A_{\oplus}\cap B_{\oplus})^\bullet$, define $n_x=h^{-1}.h_{x,\vec{V}}$; it is in $ N_{twin}$ and fixes $(x+\vec V)\cap A_{\oplus}\cap B_{\oplus}$ (hence all $x+\vec V$). 
   Let $w_x$ be the image of $n_x$ in the Weyl group $W=N_{twin}/\g T(\k)$, that we regard as a group of automorphisms of the affine space $\A_{\oplus}$. 
   As $W$ is countable, there exist $x',x''\in (\A_{\oplus}\cap B_{\oplus})^\bullet$ such that $x',x''\in x+\R v$,  $x'\neq x''$ and $w_{x'}=w_{x''}$. 
   Then $w_{x'}$ fixes $x''+\vec{V}$ and $x'+\vec{V}$ and thus it fixes $x+(\vec{V}+\R v)$. 
 So $h_{x',\vec{V}}$ fixes $(x+\vec V+\R v)\cap\A_{\oplus}\cap B_{\oplus}$. 
   Therefore $\vec{V}+\R v$ satisfies $\mathscr{P}$ and by induction we deduce that $\vec{V_0}$ satisfies $\mathscr{P}$. In particular, there exists $h_y\in G_{twin}$ such that $h_y.\A=B$ and such that $h_y$ fixes $\A_\oplus\cap B_{\oplus}$ and $y$. We conclude the proof of the theorem by a similar reasoning.
\end{proof}

\begin{rema*}
The theorem above is true if we replace $\sha_{twin}$ and $G_{twin}$ by $\sha_{pol}$ and $G_{pol }$ respectively. The proof is similar since we mainly used that $G_{twin}\subset G_{pol}$ and our preliminary study of $G_{pol}$. 
\end{rema*}

\subsection{Decompositions of $G_{twin}$ and $G_{pol}$}\label{2.17}%%%%%%%%%%%%

\subsubsection{Twin Iwasawa decomposition}\label{2.5}%%%%%%%%%%%%

Recall that $C_{\oplus}=germ_{0_{\oplus}}(C\zv_{f})$ is  the fundamental positive local chamber in $\A_{\oplus}$ and $I=I_{\oplus}$ (\resp $I_{twin}$) is the fixator of $C_{\oplus}$ in $G=\g G(\shk)$ (\resp $G_{twin}=G_{\sho}$).
From Corollary \ref{cor_Iwasawa} and Remark \ref{rk_Iwa_Gpol}, we get:

\begin{prop*} Let $\epsilon\in \{-,+\}$. Then we have:
 \[G_{twin}=U_{twin}^{\epsilon\epsilon}.N_{twin}.I_{twin}\text{  and }G_{pol}=U_{twin}^{\epsilon\epsilon}.N_{twin}.(I_{\oplus}\cap G_{pol}).\]
\end{prop*}

%\sout{\par 3) Clearly one gets also $G_{pol}=U^ {\qe\qe}_{twin}.N_{twin}.(I_{\oplus}\cap G_{pol})$}

\par \begin{NB} In $\A_{\ominus}\subset{\SHI_{\ominus}}$, one considers the fundamental negative local chamber ${C_{\infty}=}germ_{0_{\ominus}}(-C\zv_{f})$ and its fixator or stabilizer the negative Iwahori subgroup {$I_{\ominus}$} of $G$ (acting on $\SHI_{\ominus}$).
One writes $I_{\infty}={I_{\ominus}}\cap G_{twin}$ and the (negative) Iwasawa decomposition may be written:

\par\qquad\qquad\qquad $G_{twin}=U^ {\qe\qe}_{twin}.N_{twin}.I_{\infty}$ and  $G_{pol}=U_{twin}^{\epsilon\epsilon}.N_{twin}.(I_{\ominus}\cap G_{pol})$.
%One may also generalize 1), 2) and 3) above in this negative case.
\end{NB}

\begin{lemm*} Let $\qe=+$ or $\qe=-$ and $A\in {\sha_{\oplus twin}}$ such that $A\supset \g Q_{{\qe\infty}}$. Then there is a $u\in U^\qe_{twin}$ such that $A=u.{\A_{\oplus}}$.
\end{lemm*}

\begin{NB} 1) $u$ is unique and Corollary \ref{cor_Masure_GR}.2 tells, more or less, that $U^ {\qe\qe}_{twin}$ is ``dense'' in $U^ {\qe}_{twin}$.

\par 2) Such results are also true for all pairs ``sector germ $\subset$ twinnable apartment of $\SHI_{\oplus}$ or $\SHI_{\ominus}$'' with $u\in G$ fixing the sector germ, by  \ref{3.1}.3 and  \ref{n1.3.1}.
\end{NB}

\begin{proof} There are $g_{1}\in G_{twin}, g_{2}\in U^\qe$ such that $A=g_{1}.{\A_{\oplus}}=g_{2}.{\A_{\oplus}}$.
So $g_{2}^ {-1}g_{1}\in Stab_{G}({\A_{\oplus}})=N$ and $g_{1}\in G_{twin}\cap (U^\qe.N)=U^\qe_{twin}.N_{twin}$ by Lemma \ref{1.7b}.
One writes $g_{1}=u.n$ with $u\in U^\qe_{twin}$ and $n\in N$ (stabilizing $\A_{\oplus}$), so $A=u.{\A_{\oplus}}$ and the lemma is proved.
\end{proof}

\subsubsection{Decomposition of twin Iwahori subgroups ?}\label{3.6}%%%%%%%%%%%%

\par We saw in \ref{3.1} that the fixator in $G$ of the fundamental positive local chamber $C_{\oplus}$ in $\SHI_{\oplus}$, may be written $I_{\oplus}=U^+_{C_{\oplus}}.U^-_{C_{\oplus}}.\g T(\shk_{\qo_{\oplus}=0})$, with $U^\pm_{C_{\oplus}}=I_{\oplus}\cap U^\pm$.
We would like such a decomposition of $I_{twin}=I_{\oplus}\cap G_{twin}$ or $I_{pol}=I_{\oplus}\cap G_{pol}$.
But this is impossible in general as shown by the following counterexample for $\g G=\g{\mathrm{SL}}_{2}$ (semi-simple).

\par Then $I_{\oplus}$ is the group of the products $\left(\begin{matrix}1&u  \cr 0&1\cr \end{matrix}\right)\left(\begin{matrix}1&0 \cr v &1\cr \end{matrix}\right)\left(\begin{matrix}z&0  \cr 0 & z^ {-1}\cr \end{matrix}\right)$ with $u,v,z\in\shk, \qo_{\oplus}(u)\geq0,\ \qo_{\oplus}(v) >0$ and $\qo_{\oplus}(z)=0$.
But the  fixator in $\mathrm{SL}_{2}(\shk)$ of $0_{\g v}\in\SHI_{\g v}$ is $\mathrm{SL}_{2}(\sho_{\g v})$ \cite{BrT72}, so such a product fixes $0_{\g v}$ if, and only if, $\qo_{\g v}(z)\leq0$, $\qo_{\g v}(zv)\geq0$, $\qo_{\g v}(z^ {-1}u)\geq0$ and $\qo_{\g v}(z(1+uv))\geq0$; hence it 
is in $G_{pol}$ if, and only if, $z^ {-1}\in\sho,\ zv\in\sho,\ z^ {-1}u\in\sho$ and $z(1+uv)\in\sho$. 
Actually then $\left(\begin{matrix}1&u  \cr 0&1\cr \end{matrix}\right)\left(\begin{matrix}1&0 \cr v &1\cr \end{matrix}\right)\left(\begin{matrix}z&0  \cr 0 & z^ {-1}\cr \end{matrix}\right)=\left(\begin{matrix}z(1+uv)& z^ {-1}u  \cr zv & z^ {-1}\cr \end{matrix}\right)\in \mathrm{SL}_{2}(\sho)$, and $\mathrm{SL}_{2}(\sho)=(\mathrm{SL}_{2})_{twin}$ as $\sho$ is a principal ideal domain and $\mathrm{SL}_{2}$ is semisimple.

\par One chooses $P\in\k[\qp]$ an irreducible polynomial, $P\neq\qp$ and writes Bezout $1=-\qp u'+Pv'$, with $u',v'\in \k[\qp]$, we may choose $v'\in\k$.
One chooses $z^ {-1}:=P,\ z^ {-1}u:=u', \ zv:=\qp$, \ie $u=u'/P,\ v=P\qp$, so $z(1+uv)=P^ {-1}(1+u'\qp)=v'$. 
Hence $g:=\left(\begin{matrix}v' &u'  \cr \qp &P\cr \end{matrix}\right)=\left(\begin{matrix}1&u'/P  \cr 0&1\cr \end{matrix}\right)\left(\begin{matrix}1&0 \cr P\qp &1\cr \end{matrix}\right)\left(\begin{matrix}P^ {-1}&0  \cr 0 & P\cr \end{matrix}\right)=\left(\begin{matrix}1&0  \cr \qp/v'&1\cr \end{matrix}\right)\left(\begin{matrix}1&u'v' \cr 0 &1\cr \end{matrix}\right)\left(\begin{matrix} v' &0  \cr 0 & v'^ {-1}\cr \end{matrix}\right)$ is in the Iwahori subgroup of $\mathrm{SL}_{2}$ and in $(\mathrm{SL}_{2})_{twin}$, but its (unique) decomposition in $U^+U^-T$ involves factors not in $(\mathrm{SL}_{2})_{pol}$.
Nevertheless the last decomposition shows that $g$ is in $U_{+}H=\langle H, (U_{\qa+r\qx})_{\qa+r\qx\in \QF_{a+}}\rangle$. This agrees with the fact that, in reductive cases, the answer to the question in \S \ref{1.9} DR5, NB3 is yes.

\subsubsection{Groups associated with spherical vectorial facets}\label{2.16}%%%%%%%%%%%%

\par We choose now to work in $\SHI_{\oplus}$, but the similar results in $\SHI_{\ominus}$ are also true.

\par So we consider a spherical vectorial facet $F\zv\subset \A_{\oplus}$.

\par{\bf (1)} Following \cite[6.2.1, 6.2.2, 6.2.3, 12.5.2]{Re02} we associate to the facet $F\zv$ a parabolic subgroup of $G=\g G(\shk)$ with a Levi decomposition: $P(F\zv)=M(F\zv)\ltimes U(F\zv)$.
Actually $M(F\zv)$ is a $\shk-$split reductive subgroup with maximal $\shk-$split torus $\g T$ and root system $\QF\zm(F\zv)=\set{\qa\in\QF \mid \qa(F\zv)=0}$.
It is generated by $T$ and the $U_{\qa}$ for $\qa\in\QF\zm(F\zv)$.
And $U(F\zv)$ is the smallest normal subgroup of $P(F\zv)$ containing all $U_{\qa}$ for $\qa\in \QF\zu(F\zv)=\set{\qa\in\QF \mid \qa(F\zv)>0}$.

\par{\bf (2)} Parabolics and $G_{twin}$. One defines $U_{twin}(F\zv):=U(F\zv)\cap G_{twin}$, $M_{twin}(F\zv):=\langle T_{twin}; \g U_{\qa}(\sho), \qa\in\QF\zm(F\zv) \rangle$ and $P_{twin}(F\zv):=M_{twin}(F\zv)\ltimes U_{twin}(F\zv)$.

\par One has clearly $U_{twin}(F\zv)\supset \langle \g U_{\qa}(\sho) \mid\qa\in\QF\zu(F\zv)\rangle$, $M_{twin}(F\zv)\subset M(F\zv)\cap G_{twin}$ and $P_{twin}(F\zv)\subset P(F\zv)\cap G_{twin}$.
These three inclusions may certainly be strict in general.

\par From the definition in \S \ref{1.5}, one gets easily that $U^ {++}_{twin}\subset P_{twin}(F\zv)$ when $F\zv\subset \ov{C\zv_{f}}$.

\par One may also define $U_{pol}(F\zv):=U(F\zv)\cap G_{pol}$, $M_{pol}(F\zv):=M(F\zv)\cap G_{pol}$ and $P_{pol}(F\zv):=M_{pol}(F\zv)\ltimes U_{pol}(F\zv)$.

\par{\bf (3)} Twin Iwasawa decomposition. Let $C_{1}$ be a  local facet in $\A_{\oplus}$ or $\A_{\ominus}$.
As in \S \ref{2.5} or \S \ref{1.2}.2 one defines $I_{twin}(C_{1})$ or $I_{pol}(C_{1})$ as the stabilizer (or fixator) in $G_{twin}$ or $G_{pol}$ of $C_{1}$.
So, from Remark \ref{rk_Iwa_Gpol}, one gets the following Iwasawa decompositions:

\par\qquad $G_{twin}=P_{twin}(F\zv).N_{twin}.I_{twin}(C_{1})$

\par  and \ $G_{pol}=P_{twin}(F\zv).N_{twin}.I_{pol}(C_{1})=P_{pol}(F\zv).N_{pol}.I_{pol}(C_{1})$.

%\par One might expect formulas like $I_{\infty}\cap Fix(\QO)= ...$ with $\QO\subset{^\pm\A}$.

\subsubsection{Parabolo-parahoric subgroups}\label{n2.17}
\par We consider now a splayed chimney $\g r_{0}=cl(F,F\zv)$ in $\A_{\oplus}$ (with direction $F\zv$) and its germ $\g R_{0}$.

\par{\bf (1)} Following \cite[6.5]{R11}, we define $P^\qm(\g r_{0})=P^\qm(\g R_{0})=M^\qm(\g r_{0})\ltimes U(F\zv)$, where $M^\qm(\g r_{0})=M^\qm(\g R_{0})$ is the parahoric subgroup of the reductive group $M(F\zv)$, fixator of the local facet $F$ (or of $\g r_{0},\g R_{0}$, as $\g r_{0}$ is in the enclosure of $F$ for the reductive group $M(F\zv)$).
From \cite[6.5, 6.6]{R11}, we get that the group $P^\qm(\g r_{0})$ fixes the chimney germ $\g R_{0}$.
 It depends only on $\g R_{0}$, but it is not clear that it is the whole fixator of $\g R_{0}$ in $G$.

\par{\bf (2)} We consider also the subgroup $P^\qm_{twin}(\g r_{0})=P_{twin}^\qm(\g R_{0})=M^\qm_{twin}(\g r_{0})\ltimes U(F\zv)$ of $P^\qm(\g r_{0})\cap G_{twin}$, where $M^\qm_{twin}(\g r_{0})=\langle \g T(\k); U_{\qa+r\qx},\qa\in\QF\zm(F\zv), (\qa+r\qx)(F)\geq0 \rangle \subset M^\qm(\g r_{0})\cap G_{twin}$.

\par Actually $M^\qm_{twin}(\g r_{0})$ is the parabolic subgroup of the affine Kac-Moody group $M_{twin}(F\zv)$ associated to the local facet $F\subset \A_{\oplus}$.
To see precisely $M_{twin}(F\zv)$ as an affine Kac-Moody group, one has to write it $\g M(F\zv)(\k[\qp,\qp^ {-1}])$ where $\g M(F\zv)$ is the split reductive algebraic group (or group-scheme) with root system $\QF\zm(F\zv)$ and split maximal torus $\g T$.

\begin{theo}\label{2.18}%%%%%%%%%%%%%
With the above notations in \S \ref{2.16} and \S \ref{n2.17}, we have:
$$G_{twin}= P^\qm_{twin}(\g r_{0}).N_{twin}.I_{twin}(C_{1})$$
\end{theo}

\begin{NB} (a) This is the mixed twin Iwasawa decomposition.
It mixes an Iwasawa decomposition in $G_{twin}$ and a Bruhat decomposition (if $C_{1}\subset\A_{\oplus}$) or a Birkhoff decomposition (if $C_{1}\subset\A_{\ominus}$) in the Kac-Moody group $M_{twin}(F\zv)$ (which is actually reductive).

\par (b) One has also $G_{pol}= P^\qm_{twin}(\g r_{0}).N_{twin}.I_{pol}(C_{1})$.
\end{NB}

\begin{proof} Let $g\in G_{twin}$ (\resp $g\in G_{pol}$).
From \S \ref{2.16}, one gets $p\in P_{twin}(F\zv)$, $n\in N_{twin}$, $q\in I_{twin}(C_{1})$ (\resp $q\in I_{pol}(C_{1})$) and $u\in U_{twin}(F\zv)$, $m\in M_{twin}(F\zv)$ with $g=pnq$ and $p=um$.

\par Then one uses the Bruhat (\resp Birkhoff) decomposition in the affine Kac-Moody group $M_{twin}(F\zv)$ associated to the local facets $F\subset \A_{\oplus}$ and $n(C_{1})\subset \A_{\oplus}$ (\resp $n(C_{1})\subset \A_{\ominus}$). So :

\par\qquad $m=p_{1}n_{1}q_{1}$\quad with $p_{1}\in M^\qm_{twin}(\g r_{0})$\quad, \quad$n_{1}\in N_{twin}\cap M_{twin}(F\zv)$
\par \qquad \quad and $q_{1}\in \langle \g T(\k) ; U_{\qa+r\qx}, \qa\in \QF\zm(F\zv), (\qa+r\qx)(n(C_{1}))\geq0 \rangle$.

\par Now $n^ {-1}q_{1}n\in I_{twin}(C_{1})$  and $g=up_{1}n_{1}nn^ {-1}q_{1}nq$ is in $P^\qm_{twin}(\g r_{0}).N_{twin}.I_{twin}(C_{1})$ (\resp $P^\qm_{twin}(\g r_{0}).N_{twin}.I_{pol}(C_{1})$).
\end{proof}

\begin{coro}\label{2.19}%%%%%%%%%%%%%
Let $C$ be a local facet in $\SHI_{\oplus}$ (\resp in $\SHI_{\ominus}$) and $\g R$ a splayed chimney germ in $\SHI_{\oplus}$.
Then $C$ and $\g R$ are always contained in a same twin apartment $A$: $\g R\subset A_{\oplus}$ and $C\subset A_{\oplus}$ (\resp $C\subset A_{\ominus}$).
\end{coro}

%\red{Il me semble que ce resultat est faux, ne faut-il pas supposer que $\mathfrak{R}$ est un germe de cheminee?}
%\blue{c'est exact, j'ai rajoute germ ci-dessus.}

\begin{NB}Mutatis mutandis, one may also clearly suppose $\g R\subset\SHI_{\ominus}$.
\end{NB}

\begin{proof} There are $g,h\in G_{twin}$ with $C_{1}=g^ {-1}C\subset\A_{\oplus}$ (\resp $C_{1}=g^ {-1}C\subset\A_{\ominus}$) and $\g R_{0}=h^ {-1}\g R\subset\A_{\oplus}$.
From Theorem \ref{2.18}, one gets $p\in P^\qm_{twin}(\g R_{0})$, $n\in N_{twin}$ and $q\in I_{twin}(C_{1})$ such that $h^ {-1}g=pnq$.
Now $p$ fixes $\g R_{0}$ (by \S \ref{n2.17}) and $q$ fixes $C_{1}$ (by definition).
So $C=gC_{1}=hpnC_{1}\subset hp(\A_{\oplus})$ (\resp $\subset hp(\A_{\ominus})$) and $\g R=h\g R_{0}=hp\g R_{0}\subset hp(\A_{\oplus})$.
We conclude now with $A=hp(\A)$ as $hp\in G_{twin}$.
\end{proof}

\begin{remas}\label{2.20}%%%%%%%%%%%%% 
When $\g R$ is a sector germ and $C\subset\SHI_{\oplus}$, this corollary is a consequence of Corollary \ref{cor_Masure_GR}.2.
When $\g R$ is still a sector germ and $C\subset\SHI_{\ominus}$, then this corollary may also be deduced from Corollary \ref{cor_Masure_GR}.2: actually we have bijections between the sets of sector germs in $\SHI_{\oplus}$ or in $\SHI_{\ominus}$ (and with the set of chambers in $\zv\!\!\SHI$).

\par When $\g R$ is no longer a sector germ and $C\subset\SHI_{\ominus}$, this corollary or theorem \ref{2.18} gives a kind of non trivial link between $\SHI_{\oplus}$ and $\SHI_{\ominus}$.
It may be considered as a weak twinning of $\SHI_{\oplus}$ and $\SHI_{\ominus}$.
The twinning that may be hoped is a Birkhoff decomposition looking like \ref{2.18}, with $C_{1}\subset\A_{\ominus}$ and $\g r_{0}$ replaced by a local facet in $\A_{\oplus}$ (well chosen with respect to $C_{1}$).  See \ref{4.0} below.
\end{remas}

%%%%%%%%%%%%%%%%%%%%%%%%%%%%%%%%%%%%%%%%%%%%
\subsection{Expected Birkhoff decompositions  and retraction centered at $C_\infty$}%\label{s4}
\label{4.0}

 {Let $H$ be $G_{twin}$ (\resp $G_{pol}$), let $E_{+}\subset \A_{\oplus}$, $E_{-}\subset\A_{\ominus}$ be either points or local facets and let $H_{E_{\pm}}$ be  their fixators in $H$. 
Then a {\it Birkhoff decomposition} in $H$ is a decomposition $H=H_{E_{+}}.\mathrm{Stab}_{H}(\A).H_{E_{-}}$; one may also consider a decomposition $H'=H_{E_{+}}.(\mathrm{Stab}_{H}(\A)\cap H').H_{E_{-}}$ for a subsemigroup $H'$ of $H$.
As in \ref{ss_Retractions}, the existence of such a decomposition means that any $h_{+}.E_{+}$ and $h_{-}.E_{-}$ (for $h_{+},h_{-}\in H$, with some conditions in the case of $H'$) are in a same twin apartment 
$A\in\sha_{twin}$, if $H\subset G_{twin}$ (or in a same $G_{pol}-$twin apartment $A\in\sha_{pol}$,  if $H\subset G_{pol}$).
In the case where $\g G$ is a reductive group, then $\SHI=(\SHI_{\oplus},\SHI_{\ominus})$ is a twin building with a strongly transitive action of the affine Kac-Moody group $G_{twin}=G_{pol}$ (see Remark \ref{1.2}.2). 
Then the Birkhoff decomposition, for $G_{twin}$, is well known (see \eg \cite{Re02}).

\subsubsection{Conjectures}\label{4.4} { One would perhaps have liked that any pair of chambers $C_{x}\subset{\SHI_{\oplus}}$, $C_{y}\subset{\SHI_{\ominus}}$ is twin-friendly, \ie there exists a twin apartment $({ A_{\oplus}},{ A_{\ominus}})$ with $C_{x}\subset {A_{\oplus}}$, $C_{y}\subset { A_{\ominus}}$.
This would correspond to a Birkhoff decomposition  $H=H_{E_{+}}.N_{H}(\A).H_{E_{-}}$ for $H=G_{twin}$ and $E_{+},E_{-}$ as in \S \ref{4.0}.

\par But the experience of masures leads to think that this is not true in general. 
A counterexample is actually given below in Section \ref{s6}.
From this it is reasonable to think that a condition like $x\leq y$ or $y\leq x$ has to be added.

\par For Muthiah's purposes, we may restrict to the case $C_{y}\subset{\A_{\ominus}}\subset{\SHI_{\ominus}}$, $C_{y}=C_{\infty}=germ_{0}(-C\zv_{f})$ is the fundamental chamber in $\SHI_{\ominus}$.
Then we write $0_{\oplus}$ the element $0\in{\A_{\oplus}}$.%\marge{lien avec convexit\'e}

\par We give below two conjectures, the first one closely related to Muthiah's framework.

\medskip
\parni{\bf Conjecture.} For $x\in{\SHI_{\oplus}}$ such that $x\leq{0_{\oplus}}$ or $x\geq{0_{\oplus}}$, then $(C_{x},C_{\infty})$ is twin friendly.

\medskip
\par Actually  Muthiah needs a weaker result: For $x\in\SHI_{\oplus}$, with $x\leq 0_{\oplus}$ and $(x,0_{\ominus})$ twin friendly, then, for any $z\in[0_{\oplus},x]$, the pair $(z,0_{\ominus})$ is twin friendly.

\par But, using Proposition \ref{4.1} and the following Proposition \ref{4.5}, we get from such a result the general conjecture above (at least for $x \stackrel{\circ}{\leq} 0_{\oplus}$).

\medskip
 \parni{\bf Enhanced conjecture} For $x\in{\SHI_{\oplus}}$ and $y\in{\SHI_{\ominus}}$, we write $x\leq y$ (\resp $x\geq y$) if there is a twin apartment $A=({A_{\oplus}},{A_{\ominus}})$ with $x\in {A_{\oplus}}, y\in{A_{\ominus}}$ and $op_{A}(y)\geq x$ (\resp $op_{A}(y)\leq x$), where $op_{A}(y)$ is the point in $A_{\oplus}$ opposite $y$.

\par Then, for $x'\in{\SHI_{\oplus}}$ and $y'\in{\SHI_{\ominus}}$ with $x'\leq x$ and $y\leq y'$ (\resp $x'\geq x$ and $y\geq y'$) one has $x'\leq y'$ (\resp $x'\geq y'$).

\medskip
\par This second conjecture seems to be a reasonable generalization of the result known in masures.

\par Note that these two conjectures are certainly more reasonable, if we replace everywhere $\leq$ by $\stackrel{\circ}{\leq}$ and $\geq$ by $\stackrel{\circ}{\geq}$.

In a recent preprint \cite{patnaik2024local}, Manish Patnaik looks at the above conjecture in the untwisted affine case, \ie for loop groups.
Unfortunately the Birkhoff decomposition he gets is, up to now, proved only in a completion of the Kac-Moody group.
%\red{il me semble que ce n'est pas vraiment $G^{pma}$ mais plutot une autre completion (qui est peut-etre egale je ne sais pas)}

\begin{prop}\label{4.5} For $x \stackrel{\circ}{\leq} y$ in $\SHI_{\oplus}$, there is a $z\in \SHI_{\oplus}$ such that $x\in[z,y]$ and $(z,C_{\infty})$ is twin friendly.
\end{prop}

\begin{proof} One may suppose $x\neq y$.
There is an apartment $A_{\oplus}$ in $\SHI_{\oplus}$ containing $x$ and $y$.
One may consider in $A_{\oplus}$ the spherical vectorial facet $F\zv$ of $\vect{ A_{\oplus}}$ containing $\vect{yx}$, the ray $\qd=y+\R_{+}\vect{yx}$ and the splayed chimney $\g r=\g r(F(y,F\zv),F\zv)$.
By Corollary \ref{2.19}, there is a twin apartment $(A'_{\ominus},A'_{\oplus})$ such that $C_{\infty}\subset A'_{\ominus}$ and $A'_{\oplus}$ contains the germ $\g R$ of $\g r$, \ie $A'_{\oplus}$ contains a shortening $\g r(F(y+k\vect{yx},F\zv),F\zv)$ of $\g r$ (for some $k\in\R$ supposed $\geq1$).
Then $A'_{\oplus}$ contains $z=y+k\vect{yx}$ (and the ray $z+\R_{+}\vect{yx}$).
So $(z,C_{\infty})$ is twin friendly and $x\in[z,y]$ (as $k\geq1$).
\end{proof}
}

\subsubsection{Retraction centered at $C_\infty$}\label{n3.15.1}

Our main motivation to study twin masures is the study of the Kazhdan-Lusztig polynomials introduced by Muthiah in \cite{Mu19b} in the Kac-Moody frameworks. His definition involves the cardinalities of sets of the form 
\begin{equation}\label{eq_coset}
K_{twin} \qp^\lambda K_{twin}\cap I_\infty \qp^\mu K_{twin}/K_{twin},
\end{equation} 
where $K_{twin}$ is the fixator of $0_{\oplus}$ in $G_{twin}$ and $\lambda,\mu \in Y^+=Y\cap\sht$ (and $\qp^\lambda$ is defined in \S \ref{ss_N}). The strategy he proposes to compute these cardinalities follows the   steps below.
 \begin{enumerate}
\item Define a retraction $\rho_{C_\infty}:\shi_{\oplus,\leq 0_{\oplus}}=\set{x\in\SHI_{\oplus}\mid x\leq 0_{\oplus}}\rightarrow \A_{\oplus,\leq 0_{\oplus}}=\A_{\oplus}\cap\shi_{\oplus,\leq 0_{\oplus}} $ centered at $C_\infty$. Then the coset \eqref{eq_coset} is in bijection with 
\begin{equation}\label{eq_retrac} 
\{x\in \shi_{\oplus,\leq 0_{\oplus}}\mid d\zv(0_{\oplus},x)=-\lambda\text{ and }\rho_{C_{\infty}}(x)=-\mu\}, 
\end{equation}(see \ref{5.2} for the definition of $d\zv$).

\parni Recall that for us, following Tits, $\qp^\ql$ acts on $\A_{\oplus}$ by the translation of vector $-\ql$: see \S \ref{1.8}.

\item Study the images by $\rho_{C_\infty}$ of line-segments of $\shi_{\oplus,\leq 0_{\oplus}}$.  He proves in \cite{Mu19b} that such an image is a piecewise linear path of $\sha_{\oplus}$ satisfying certain conditions. He calls such paths $I_\infty$-Hecke paths.

\item Prove that an $I_{\infty}-$Hecke path from $0_{\oplus}$ to $-\qm$ in $\A_{\oplus}$, of shape $-\ql$, has only a finite (computable) number of liftings as line segments of $\shi_{\oplus,\leq 0_{\oplus}}$ from $0_{\oplus}$ to $x\in\SHI_{\oplus}$ with $d\zv(0_{\oplus},x)=-\ql$.

\item Prove that, for $\ql$ and $\qm$ given, there is only a finite number of $I_{\infty}-$Hecke paths from $0_{\oplus}$ to $-\qm$ in $\A_{\oplus}$, of shape $-\ql$.
Together with 3. this gives the cardinality of the set \ref{eq_retrac}.

\end{enumerate}

In \cite{Mu19b}, Muthiah achieves steps 2 and 3 in general and step 4 in certain cases (when $G$ is untwisted affine of type $A$, $D$ or $E$, see \cite[Theorem 5.54]{Mu19b}). 
Step 4 is achieved in full generality in \cite[Corollary 3.11]{hebert2024kazhdan}.
However, step 1 is only conjectural. 

\par We now explain {step 1, \ie}how to define  $\rho_{C_\infty}$ under the assumption that $(G_{twin})^+_\oplus$ (or $(G_{twin})^-_\oplus$) admits a Birkhoff decomposition (which is still  conjectural).
 Steps 2 and 3 will be explained with great details in Section \ref{s5}, see particularly Subsections \ref{5.3}, \ref{5.13} and Theorem \ref{5.7}.
In step 3, it seems that our formula for the number of liftings of a $C_{\infty}-$Hecke path is more precise than Muthiah's formula.
We shall tell nothing about step 4.

\medskip

Let $\she=I_\infty.\A_{\oplus}$. Then  $\she$ is the set of elements $x\in \shi_{\oplus}$ such that $x\cup C_\infty$ is $G_{twin}$-friendly. Indeed, take $x\in \she$ and write $x=i_\infty.y$, with $i_\infty\in I_\infty$ and $y\in \A_{\oplus}$.  Then $A:=i_\infty.\A$ contains $x\cup C_\infty$. Conversely, let $x\in \shi_\oplus$ be such that $x\cup C_\infty$ is $G_{twin}$-friendly. Then there exists $g\in G_{twin}$ such that $A:=g.\A$ contains $x\cup C_\infty$. Then by Theorem~\ref{thmIsomorphisms_fixing_twin_apartments}, there exists $h\in G_{twin}$ such that $h.\A=A$ and $h$ fixes $A\cap \A$. Then $h\in I_\infty$ and there exists $y\in \A_{\oplus}$ such that $h.y=x$, so $x\in \she$.

\begin{lemm}\label{l_def_ret}
Let $z\in \A_{\oplus}$ and $i_\infty\in I_\infty$ be such that $i_\infty.z\in \A_{\oplus}$. Then $i_{\infty}.z=z$. 
\end{lemm}

\begin{proof}
Let $A=i_\infty.\A=(i_\infty.\A_{\oplus},i_{\infty}.\A_{\ominus})$. By Theorem~\ref{thmIsomorphisms_fixing_twin_apartments}, there exists $h\in G_{twin}$ such that $h.A=\A$ and $h$ fixes $A\cap \A$. Then $hi_\infty$ stabilizes $\A$ and thus it belongs to $N_{twin}$. As $hi_{\infty}$ fixes $C_\infty$, it fixes an open subset of $\A_{\ominus}$. Therefore $hi_{\infty}$ fixes $\A_{\ominus}$. By \S \ref{1.7} d), $hi_\infty$ lies in   $\g T(\k)$ and thus it also fixes $\A_{\oplus}$. Therefore $hi_{\infty}.z=z=i_\infty.z$. 
\end{proof}

  We  define $\rho_{C_\infty}:\she\rightarrow \A_{\oplus}$\index{r@$\rho_{C_\infty}$} by $\rho_{C_\infty}(i_\infty.x)=x$ for $x\in \A_{\oplus}$ and $i_\infty\in I_\infty$. This is well-defined by the lemma above. Moreover it is $I_\infty$-invariant and $\rho_{C_\infty}(x)=x$ for all $x\in \A_{\oplus}$, so it satisfies the conditions of \cite[Proposition 2.4]{Mu19b}, with $Q=I_\infty$.

It is however difficult to describe explicitly $\she$. It is related to the existence of Birkhoff decompositions on $G$ by the lemma below. For our purpose, we would like that $\she$ contains $\shi_{\oplus,\geq 0_{\oplus}}$ (or $\shi_{\oplus,\leq 0_{\oplus}}$, since our sign conventions differ from the ones of Muthiah). 
In the following of this subsection \ref{4.0} we work with $\shi_{\oplus,\geq 0_{\oplus}}$, but the same results are true for $\shi_{\oplus,\leq 0_{\oplus}}$.

We set $(G_{twin})^+_\oplus=\{g\in G_{twin}\mid g.0_{\oplus}\geq 0_{\oplus}\}$.

\begin{lemm}
\begin{enumerate}
\item Let $J=\bigcap_{x\in \A_{\oplus}} I_\infty N_{twin} (G_x\cap G_{twin})\text{ and }J^+=\bigcap_{x\in \A_{\oplus,\geq0_{\oplus}}} I_\infty N_{twin} (G_x\cap G_{twin}).$ Then $\she\supset J.\A_{\oplus}\cup J^+.\A_{\oplus,\geq 0_{\oplus}}$. 

\item If $\she=\shi_{\oplus}$, then $G_{twin}=J$. 

\item If $\g G$ is reductive, then $\she=\shi_{\oplus}$.

\item We have $(G_{twin})^+_{{\oplus}}.\A_{\oplus,\geq 0_{\oplus}}=\shi_{\oplus,\geq 0_{\oplus}}$. 

\item We have $J^+\supset (G_{twin})^+_\oplus$ if and only if $\she\supset \shi_{\oplus,\geq 0_{\oplus}}$.

\end{enumerate}

\end{lemm}

\begin{proof}
2) Suppose $\she=\shi_{\oplus}$. Let $g\in G_{twin}$   and $x\in \A_\oplus$. Then $g.x\in \she$ and thus there exists $i_\infty\in I_\infty$, $y\in \A_{\oplus}$ such that $g.x=i_\infty.y$ and $(i_\infty)^{-1}g.x=y$. Let $h\in G_{twin}$ be such that $h(i_\infty)^{-1}g.\A=\A$ and such that $h$ fixes $\A\cap (i_\infty)^{-1}g.\A$ (Theorem \ref{thmIsomorphisms_fixing_twin_apartments}). 
Set $n=h(i_\infty)^{-1}g$. Then $n\in N_{twin}$ and  $y=n.x$. Then $g.x=i_\infty n.x$ and hence $n^{-1}(i_\infty)^{-1} g\in G_x$. Consequently, $g\in I_\infty N_{twin} (G_x\cap G_{twin})$ and $G_{twin}=J$.

(1) Let $x\in J.\A_{\oplus}$ and $j\in J$, $y\in \A_{\oplus}$ be such that $x=j.y$. Write $j=i_\infty n k$, where $(i_\infty,n,k)\in I_\infty\times N_{twin}\times (G_y\cap G_{twin})$. Then $x=i_\infty.(n.y)\in \she$, so $\she\supset J.\A_{\oplus}$. Similarly we have $J^+.\A_{\oplus,\geq 0_{\oplus}}\subset \she$. 

(3) Suppose $\g G$ is reductive. Then we have $G_{twin}=I_{\infty} N_{twin} I_{twin}$, by the Birkhoff decomposition in the affine Kac-Moody group over $\k$, $\g G(\k[\qp,\qp^ {-1}])=G_{twin}$. 
Therefore we have $G_{twin}=I_\infty  N_{twin} mI_{twin} m^{-1}$   for every $m\in N$. Take $x\in \A_{\oplus}$. Then there exists $m\in N_{twin}$ such that $m^{-1}.x\in \overline{C_{\oplus}}$. Then $G_x\cap G_{twin}\supset mI_{twin} m^{-1}$, which proves (3) using (1).

(4) Let $g\in (G_{twin})^+_{\oplus}$ and $x\in \A_{\oplus,\geq 0_{\oplus}}$. Then $x\geq 0_{\oplus}$ and by $G$-invariance of $\leq$ we have $g.x\geq g.0_{\oplus}$. By definition of $(G_{twin})^+_{\oplus}$, we have $g.0_{\oplus}\geq 0_{\oplus}$. 
By transitivity of $\leq$, $g.x\geq 0_{\oplus}$, thus $(G_{twin})^+_{\oplus}.\A_{\oplus,\geq 0_{\oplus}}\subset \shi_{\oplus,\geq 0_{\oplus}}$. 
Let $x\in \shi_{\oplus,\geq 0_{\oplus}}$.  
Then there exists $g\in G_{twin}$ such that $g.x,g.0_{\oplus}\in \A_{\oplus}$ and $g.x\geq g.0_{\oplus}$.
 We can moreover assume that $g.0_\oplus=0_{\oplus}$ (see Corollary \ref{cor_Bruhat}.1 and Proposition \ref{2.11}). 
 Then $x=g^{-1}.(g.x)$ and $g^{-1}\in (G_{twin})^+_{\oplus}$, hence $x\in (G_{twin})^+_{\oplus}.\A_{\oplus,\geq 0_{\oplus}}$. Therefore $\shi_{\oplus,\geq 0_{\oplus}}\subset (G_{twin})^+_{\oplus}.\A_{\oplus,\geq 0_{\oplus}}$ which proves 4). 
 
 5) By 1) and 4), we already have the implication ``$\Rightarrow$''. Assume $\she \supset \shi_{\oplus,\geq 0_{\oplus}}$ and take $g\in (G_{twin})^+_{\oplus}$ and $x\in \A_{\oplus,\geq 0_{\oplus}}$. Then by $G$-invariance of $\leq$, we have $g.x\geq g.0_{\oplus}\geq 0_{\oplus}$, so $g.x\in \shi_{\oplus,\geq 0_{\oplus}}\subset \she$. Therefore there exists $y\in \A_{\oplus}$ and $i_\infty\in I_\infty$ such that $g.x=i_\infty.y$. As in the proof of 2), we have $y\in N_{twin}.x$,  thus $g\in I_\infty N_{twin} (G_x\cap G_{twin})$ and the lemma follows.

\end{proof}

As we shall see in \ref{ss_count_ex}, $J\neq G_{twin}$ in general. We conjecture that $J^+\supset (G_{twin})^+_{\oplus}$ which is equivalent to $\she\supset \shi_{\oplus,\geq 0_{\oplus}}$, by the Lemma above. We also expect similar results for $\shi_{\oplus,\leq 0_{\oplus}}$, $(G_{twin})^-_{\oplus}$ and $J^-$ (where $J^-$ is defined similarly to $J^+$).

\begin{rema}
It seems also natural to define $\she'=I_\ominus.\A_{\oplus}$ and then define $\rho_{C_\infty}':\she'\rightarrow \A_{\oplus}$ by $\rho_{C_\infty}(i.x)=x$ for $i\in I_{\ominus}$, $x\in \A_{\oplus}$. However this is not defined in general because the fixator of $\A_{\ominus}$ in $G$ does not fix $\A_{\oplus}$. Indeed, let $z\in \k(\qp)$ be such that $\omega_\oplus(z)\neq 0$ and $\omega_{\ominus}(z)= 0$ and $\lambda\in Y\setminus\{0\}$. Set $z^\lambda=\lambda(z)\in T$ (recall that $Y=\mathrm{Hom}(\g {Mult},\g T)$). Then $z^\lambda$ acts by translation of vector $-\omega_{\oplus}(z)\lambda$ on $\A_{\oplus}$ and by translation of vector $\omega_{\ominus}(z)\lambda=0$ on $\A_{\ominus}$. Actually, $\nu_{\oplus}(\g T(\sho_{\ominus}^*))=Y$, so we can define $\rho'_{C_\infty}:\she'\rightarrow \A_{\oplus}/Y$. Then we can define the image by $\rho_{C_\infty}'$ of a line-segment of $\she'$ (up to an element of $Y$) by demanding its image to be continuous. So it might be  helpful to look for a  Birkhoff decomposition of $G$ instead of a Birkhoff decomposition of $G_{twin}$, in order to study Kazhdan-Lusztig polynomials. 
\end{rema}

%%%%%%%%%%%%%%%%%%%%%%%%%%%%%%%%%%%%%%%%%%%%
\section{$C_{\infty}-$Hecke paths}\label{s5}

\par As explained above in \S \ref{4.4}, we do not get what is expected to define the retraction $\qr_{I_{\infty}}=\qr_{C_{\infty}}$ (on a great part of $\SHI_{{{\oplus}}}$). 
One would like that : $\forall x\in {\SHI_{\oplus}}, x\geq0_{\oplus}$ (or $x\leq0_{\oplus}$), then $(x,{C_{\infty}})$ is twin friendly.
Actually we get interesting results if, at least, $(z,{C_{\infty}})$ is twin friendly for any $z\in [0_{\oplus},x]$.
Then $\qr_{I_{\infty}}=\qr_{C_{\infty}}$ is defined on $[0_{\oplus},x]$ (by  Theorem~\ref{thmIsomorphisms_fixing_twin_apartments} or by \S\ \ref{n3.15.1}).
In this section we shall prove, using Proposition \ref{4.1},  that $\qr_{C_{\infty}}([0_{\oplus},x])$ is  an $I_{\infty}-$Hecke path (as defined in \cite{Mu19b}).
Actually  $C_{\infty}$ is the canonical (negative) local chamber in ${\SHI_{\ominus}}$ and $\qr_{C_{\infty}}=\qr_{I_{\infty}}$ is the retraction of (a part of) $\SHI_{{\oplus}}$ onto (a part of) $\A_{{\oplus}}$ with center $C_{\infty}$; it is also defined on a part of $\SHI_{\ominus}$ (using a Bruhat decomposition in $\SHI_{\ominus}$).

 More precisely, under the above hypothesis on $[0_{\oplus},x]$, we prove that $\qr_{C_{\infty}}([0_{\oplus},x])$ is a $\ql-$path (with $\ql=d\zv(0_{\oplus},x)$) and may be endowed with a superdecoration (\S \ref{5.2}, \ref{5.3}).
Conversely we prove that any superdecorated $\ql-$path is the image by $\qr_{C_{\infty}}$ of a line segment $[0_{\oplus},x]$ with $\ql=d\zv(0_{\oplus},x)$ and we count the number of these possible $x$ (Theorem \ref{5.7}).
Then, starting from \S \ref{5.10}, we get that the underlying path of a superdecorated $\ql-$path is a $C_{\infty}-$Hecke path of shape $\ql$, for the definition of D. Muthiah (\S \ref{5.15}).

\subsection{Projections and retractions}\label{5.1}%%%%%%%%%%%%

\par \quad\;1) One considers a twin friendly pair $(C_{y},x)$ with $C_{y}$ a local chamber in ${\SHI_{\ominus}}$ and $x\in {\SHI_{\oplus}}$. So one may suppose $C_{y}\subset {\A_{\ominus}}$ and $x\in {\A_{\oplus}}$ (up to an element of $G_{twin}$).

\par By  paragraph \ref{n3.15.1}   the retraction $^+\qr_{C_{y}}$ of $\sht^\pm_{x}(\SHI_{\oplus})$ onto $\sht^\pm_{x}(\A_{\oplus})$ with center $C_{y}$ is well defined.
This means that $^+\qr_{C_{y}}([x,z))$  or $^+\qr_{C_{y}}(C_{x})$ is well defined for $z\in{\SHI_{\oplus}}$ and $x\leq z$ (\resp $z\leq x$) or when $C_{x}$ is a local chamber at $x$ in $\SHI_{\oplus}$ with positive (\resp negative) direction (recall that $[x,z)$ is the germ of $[x,z]$ at $x$).

\par 2) Projections: One defines:

\par $\widetilde{pr}_{x}(C_{y})$ (\resp $pr_{x}(C_{y})$, also written $C^\infty_{x}$ when $C_{y}=C_{\infty}$)\index{p@$pr_{x}$} is the germ in $x$ of the intersection of the half-apartments $D_{\oplus}(\qa+k)$ with $\qa\in\QF$, $k\in\Z$ (\resp of the open-half-spaces $D_{\oplus}^\circ(\qa+k)$ with $\qa\in\QF$, $k\in\R$) such that $D_{twin}(\qa+k\qx) \supset \set x \cup C_{y}$.
 By Theorem \ref{thmIsomorphisms_fixing_twin_apartments},  $\widetilde{pr}_{x}(C_{y})$ (\resp $pr_{x}(C_{y})$) is independent of the choice of $(\A_{\ominus},{\A_{\oplus}})$ containing $(C_{y},x)$.

\par One may remark that $\QF_{a}(C_{y}):=\set{\qa+k\qx\in\QF_{a} \mid {D_{twin}(\qa+k\qx)\supset C_{y}}}$ looks like a system of positive roots in $\QF_{a}$ (in a clear sense).
%But, a priori,  $\QF_{aff}(C_{y})$ is conjugated by $W$ to $\QF_{aff-}$ only if $y\in W.0=Y$.
But it is not clear that $C^\infty_{x}$ is a local chamber (its direction might be outside the Tits cone). 

\par 3) We are mostly interested in the case $C_{y}=C_{\infty}$, hence $^+\qr_{C_{y}}=\qr_{C_{\infty}}=\qr_{I_{\infty}}$.
Then $\QF_{a}(C_{y})=\QF_{a-}$ \ie $C_{y}\subset {D_{twin}(\qa+k\qx)} \iff C_{\oplus} \subset {D_{twin}(-\qa-k\qx)}$.
So (if $x\in\pm\sht^\circ$, more precisely $x\stackrel{\circ}{>}0_{\oplus}$ or $x\leq0_{\oplus}$), $C^\infty_{x}$ is the local chamber opposite at $x$ to $pr_{x}(C_{\oplus})$ (defined similarly to $pr_{x}(C_{y})$ above, see \cite[2.1]{BPR19} for details); its sign is $+$ if $x\stackrel{\circ}{>}0_{\oplus}$ and $-$ if $x\leq0_{\oplus}$.
Moreover $\widetilde{pr}_{x}(C_{y})$ is the closed chamber in the restricted sense (see \cite[\S 4.5]{GR08}) containing  $pr_{x}(C_{y})=C^\infty_{x}$.
If $x$ is a special vertex, $\widetilde{pr}_{x}(C_{y}) = \overline{pr_{x}(C_{y})}$.

\begin{NB} a) Note that we chose above to suppose (up to $G_{twin}$) that $C_{y}\subset {\A_{\ominus}}$ and $x\in {\A_{\oplus}}$.
So, in general, when we speak of $C_{0}=C_0(A_\oplus)$ (\resp $0$) in this section \ref{s5}, it means the positive local chamber (\resp the vertex $0(A_{\oplus})$) opposite $C_{\infty}$ (\resp $0_{\ominus}$) in a twin apartment $A_{twin}\supset {A_{\ominus}}\cup{A_{\oplus}}$ (in the sense of \ref{2.6} and \ref{1.3}) such that $C_{\infty}\subset{A_{\ominus}}$ and $x\in {A_{\oplus}}$.
By Theorem \ref{thmIsomorphisms_fixing_twin_apartments} the condition $x\stackrel{\circ}{>}0(A_{\oplus})$ or $x\leq0(A_{\oplus})$ does not depend of the choice of $A_{twin}$.

\par b) In this case $C_{y}=C_{\infty}$ and $x\stackrel{\circ}{>}0$ or $x\leq0$, we proved that $C^\infty_{x}$ is a local chamber.
\end{NB}

\begin{enonce*}[plain]{\quad\;4) Lemma} Let $C_{x}$ be a local chamber at $x$ in $\SHI_{\oplus}$.
 Then there are affine roots $\qa_{1}+k_{1}\qx, \ldots , \qa_{n}+k_{n}\qx \in \QF_{a}(C_{y})$ with $(\qa_{i}+k_{i}\qx)(x)=0$ and elements $u_{i} \in U_{\qa_{i}+k_{i}\qx}\subset G_{twin}\cap G_{x}\cap G_{C_{y}}$ (possibly $u_{i}=1$) such that $^+\qr_{C_{y}}(C_{x})=u_{n}.\ldots .u_{1}.C_{x}$.

\par In particular $^+\qr_{C_{y}}$ restricted to $\sht^\pm_{x}(\SHI_{\oplus})$ is induced by elements of the group $G^ {min}_{twin}(x)=\langle U_{\qb+r\qx} \mid \qb+r\qx \in \QF_{a}(C_{y}); (\qb+r\qx)(x)=0 \rangle \subset G_{twin}\cap G_{x}$, which fixes $\widetilde{pr}_{x}(C_{y})$.
Hence (in the case of 3) above) this restriction (of $^+\qr_{C_{y}}=\qr_{C_{\infty}}$) is the retraction $\qr'$ of $\sht^\pm_{x}(\SHI_{\oplus})$ onto $\sht^\pm_{x}(\A_{\oplus})$ with center $\widetilde{pr}_{x}(C_{y})$ (or $pr_{x}(C_{y})=C^\infty_{x}$).

\end{enonce*}

\begin{NB}
 $G_{twin}\cap G_{x}\cap G_{C_{y}}$ has the same restriction to $\sht^\pm_{x}(^+\SHI)$ as $G^ {min}_{twin}(x)$.
\end{NB}

\begin{proof} Let $C^0\subset\A_{\oplus}, C^1,\ldots,C^{n}=C_{x}$ be a minimal gallery of local chambers at $x$ in $\SHI_{\oplus}$, with origin in $\A_{\oplus}$ and end $C_{x}$.
One argues by induction on $n$; it is clear for $n=0$.  
If $n\geq1$, one considers the hyperplane $M_{\oplus}(\qa_{1}+k_{1}\qx)$ (with $\qa_{1}\in\QF,k_{1}\in\R$) of $\A_{\oplus}$ containing the local panel common to $C^0$ and $C^1$.
One may suppose $(\qa_{1}+k_{1}\qx)(C_{y})\geq0$. 
If $k_{1}\not\in\Z$, this hyperplane is not a wall and $C^1\subset{\A_{\oplus}}$.
By induction $^+\qr_{C_{y}}(C_{x})=u_{n}.\ldots .u_{2}.C_{x}$ (with clear notations) and we are done (we replace $k_{1}$ by any $k_{1}\in\Z$ and take $u_{1}=1$).
If $k_{1}\in\Z$, then $\qa_{1}+k_{1}\qx\in \QF_{a}(C_{y})$, and, as in Proposition \ref{4.1}, one sees that there exists $u_{1} \in U_{\qa_{1}+k_{1}\qx}$ such that $u_{1}C^1\subset{\A_{\oplus}}$.
One considers the gallery $u_{1}C^1,\ldots,u_{1}C^{n}=u_{1}C_{x}$.
By induction there are $\qa_{2}+k_{2}\qx, \ldots , \qa_{n}+k_{n}\qx \in \QF_{a}(C_{y})$ with $(\qa_{i}+k_{i}\qx)(x)=0$ and elements $u_{i} \in U_{\qa_{i}+k_{i}\qx}\subset G_{twin}\cap G_{x}\cap G_{C_{y}}$ such that $^+\qr_{C_{y}}(u_{1}C_{x})=u_{n}.\ldots .u_{2}.u_{1}.C_{x}$.
So $^+\qr_{C_{y}}(C_{x})={^+\qr}_{C_{y}}(u_{1}C_{x})=u_{n}.\ldots .u_{2}.u_{1}.C_{x}$ as expected.
As each $u_{i}$ fixes $\widetilde{pr}_{x}(C_{y})$ and $pr_{x}(C_{y})$, it is also equal to $\qr'(C_{x})$.

\end{proof}

\subsection{$C_{\infty}-$friendly line segments in $\SHI_{\oplus}$}\label{5.2}%%%%%%%%%%%%

\par\quad\; 1) Let $x,y\in \SHI_{\oplus}$ be such that $x \stackrel{\circ}{<}y$ (\resp $x \stackrel{\circ}{>}y$).
There is a $G-$apartment $g.\A_{\oplus}$ containing $\set{x,y}$, so $g^ {-1}y - g^ {-1}x$ is in $\sht^\circ$ (\resp $-\sht^\circ$).
We define the vectorial distance $\ql=d\zv(x,y)$\index{d@$d\zv(-,-)$} as the unique element in $\ov{C}\zv_{f} \cap \sht^\circ$ (\resp $-\ov{C}\zv_{f} \cap \sht^\circ$) conjugated by $W\zv$ to $g^ {-1}y - g^ {-1}x$.
It does not depend on the choices made (see \eg \cite[\S 1.6]{BPR19}).

\par The line segment $[x,y]$ in $\SHI_{\oplus}$ is said $C_{\infty}-$friendly\index{C@$C_{\infty}-$friendly} if, moreover, $\forall z\in[x,y]$, $(C_{\infty},z)$ is twin friendly.
By Proposition \ref{4.1} we may ask that $A_{\oplus}$ contains $[z,x)$ or $[z,y)$.
 We actually parametrize $[x,y]$ by $[0,1]$ : $\qf:[0,1]\to[x,y]$ is an affine bijection.
 We define  $\qe(\qf)=+1$ if $x \stackrel{\circ}{<}y$ and  $\qe(\qf)=-1$ if $x \stackrel{\circ}{>}y$. 
\par\qquad In the following we suppose $[x,y]$ $C_{\infty}-$friendly.

\par 2) By the usual argument using the compactness of $[x,y]$ and Proposition~\ref{4.1}, we get points $z_{0}=x,z_{1},\ldots,z_{n}=y$ in this order in $[x,y]$ and twin apartments $(A^\ominus_{i},{A^\oplus_{i}})$, $1\leq i\leq n$, with $C_{\infty}\subset{A^\ominus_{i}}$ and $[z_{i-1},z_{i}]\subset{A^\oplus_{i}}$.
We set $z_{i}=\qf(t_{i}), t_{0}=0<t_{1}<\cdots<t_{n}=1$.
By Theorem \ref{thmIsomorphisms_fixing_twin_apartments} or \S \ref{n3.15.1},  we know that $\qr_{C_{\infty}}$ is defined on $[x,y]$, and also on all local chambers $C_{z}$ with vertex $z\in[x,y]$ by Proposition \ref{4.1}.
The above result tells that $\qr_{C_{\infty}}([x,y])$ (or better $\pi=\qr_{C_{\infty}}\circ\qf$) is a piecewise linear continuous path in $\A_{\oplus}$.
It is actually a {\bf $\ql-$path}, as defined in \cite[1.7]{BPR19}, \cite[1.8]{GR14} or  \cite[5.1]{GR08},
 \ie it is a piecewise linear continuous path  $\pi:[0,1]\to\A$ such that, $\forall t\in[0,1]$, $\pi'_{\pm}(t)\in W\zv.\ql$ (which is in $\pm\sht^\circ$).\index{lambda-path@$\ql-$path}
We shall investigate its properties more closely and then call it an $I_{\infty}-$Hecke path (to follow \cite{Mu19b}) or a $C_{\infty}-$Hecke\index{C@$C_{\infty}-$Hecke} path or (more precisely) a Hecke path of shape $\ql$ in $\A_{\oplus}$ with respect to $C_{\infty}$ (in $\A_{\ominus}$).

\par 3) We suppose now moreover that $C^\infty_{x}$ is a local chamber, more precisely that, in the apartment $A^\oplus_{1}$, one has $x \stackrel{\circ}{>}0_{1}$ (\resp $x\ {\leq}\ 0_{1}$), where $0_{1}$ means the opposite in  $A^\oplus_{1}$ of $0_{\ominus}\in{A^\ominus_{1}}$.
By Theorem \ref{thmIsomorphisms_fixing_twin_apartments} this condition does not depend on $A_{1}$ or $[x,y)$ but only on $(C_{\infty},x)$.
In particular the sign of $C^\infty_{x}$ is positive (\resp negative).
We may decorate $[x,y]$ by the use of $C^\infty_{x}$:

\par For $z\in [x,y[$ we set $C^+_{z,\qf}=pr_{[z,y)}(C^\infty_{x})$\index{p@$pr_{[z,y)}$}\index{C@$C^\pm_{z,\qf},C^\pm_{p,\pi}$} and for $z\in ]x,y]$ we set $C^-_{z,\qf}=pr_{[z,x)}(C^\infty_{x})$, \ie $C^+_{z,\qf}$ (\resp $C^-_{z,\qf}$) is the local chamber containing $[z,y)$ (\resp $[z,x)$) in its closure that is the closest to $C^\infty_{x}$, for details
 see \cite[\S 2.1 and Def 2.4]{BPR19} where $C^-_{z,\qf}$ is written $C''_{z}$.
 One has to be careful that, contrary to \lc, we may have  $x \stackrel{\circ}{>}y$ (\ie $\qe(\qf)=-1$) and then $C^+_{z,\qf}$ (\resp $C^-_{z,\qf}$) has a negative (\resp positive) direction.
When $z=\qf(t)$ we write also $C^\pm_{z,\qf}=C^\pm_{t,\qf}$.
We write $\un \qf$ or $\un {[x,y]}$ this decorated line segment.

\par We recall the notations for some segment germs: $\qf_{+}(t)=\qf_{+}(z)=\qf([t,1))=[z,y)$, $\pi_{+}(t)=\pi_{+}(p)=\pi([t,t+\eta))$ (\resp $\qf_{-}(t)=\qf_{-}(z)=\qf([t,0))=[z,x)$, $\pi_{-}(t)=\pi_{-}(p)=\pi([t,t-\eta))$ if $t<1$ (\resp $0<t$) and $z=\qf(t),p=\pi(t)$, $\eta>0$ small; also the right (\resp left) derivatives $\pi'_{+}(t)$ (\resp $\pi'_{-}(t)$).

\par We may also define $C^\pm_{p,\pi}=C^\pm_{t,\pi}:=\qr_{C_{\infty}}(C^\pm_{z,\qf})$\index{C@$C^\pm_{z,\qf},C^\pm_{p,\pi}$} when $p=\pi(t)=\qr_{C_{\infty}}(z)=\qr_{C_{\infty}}(\qf(t))$.
We get thus a decoration of $\pi$:

\par {\bf Definition} \cite[Def. 2.6]{BPR19} A \textbf{decorated} \index{decorated@decorated} $\ql-$path is a triple ${\un\pi}=(\pi,(C^+_{t,\pi})_{t<1},(C^-_{t,\pi})_{t>0})$ such that:
 $\pi$ is a $\ql-$path, $C^+_{t,\pi}$ (\resp $C^-_{t,\pi}$) is a local chamber with the same (\resp opposite) sign as $\ql$, with vertex $\pi(t)$, containing $\pi_{+}(t)$ (\resp $\pi_{-}(t)$) in its closure.
Moreover, for some subdivision $t'_{0}=0<t'_{1}<\cdots<t'_{n}=1$ of $[0,1]$ such that $\pi\vert_{[t'_{i-1},t'_{i}]}$ is a line segment and for any $t'_{i-1}\leq t,t'\leq t'_{i}$, we ask that $C^+_{t,\pi}=pr_{\pi_{+}(t)}(C^\pm_{t',\pi})$ (\resp $C^-_{t,\pi}=pr_{\pi_{-}(t)}(C^\pm_{t',\pi})$) (here we exclude $C^-_{t'_{i-1},\pi}$ and $C^+_{t'_{i},\pi}$ of these equalities).
 
We get easily these properties in our context, as the apartment $A^\oplus_{i}$ above contains $C^+_{z_{i-1},\qf}$ and $C^-_{z_{i},\qf}$ (hence all $C^\pm_{z,\qf}$ for $z\in ]z_{i-1},z_{i}[$).
So, for $p_{i}=\pi(t_{i})$, the restriction $\pi\vert_{[t_{i-1},t_{i}]}$ is a line segment from $p_{i-1}$ to $p_{i}$  and $\qr_{C_{\infty}}([x,y])=[p_{0},p_{1}]\cup[p_{1},p_{2}]\cup \cdots \cup[p_{n-1},p_{n}]$.

\subsection{Retractions of $C_{\infty}-$friendly line segments}\label{5.3}%%%%%%%%%%%%

\par\quad\; 1) We suppose $[x,y]\subset{\SHI_{\oplus}}$, $C_{\infty}-$friendly and parametrized by $\qf$ as in \ref{5.2}.1.
We suppose moreover $[x,y) \subset {\A_{\oplus}}$.
We may then decorate $[x,y]$ (\ie $\qf$) by the use of $C^\infty_{x}$, if $x\leq0$ or $x\stackrel{\circ}{>}0$ (actually we assume often $x=0$), \cf \S \ref{5.2}.3.
We get also a decoration on the $\ql-$path $\qr_{C_{\infty}}([x,y])$ (\ie on $\pi=\qr_{C_{\infty}}\circ\qf$); we keep the notations of \S \ref{5.2}.

\par 2) We suppose $0\stackrel{\circ}{\leq}x \stackrel{\circ}{<}y$ hence $\qe=\qe(\qf)=+1$ (\resp $0\stackrel{\circ}{\geq}x \stackrel{\circ}{>}y$ hence $\qe=\qe(\qf)=-1$); for this we may eventually exchange $x$ and $y$ if \eg $0\stackrel{\circ}{\leq}y \stackrel{\circ}{<}x$.
From this we deduce (by induction on $i$) that, for any $z\in ]z_{i-1},z_{i}[$, one has $z\stackrel{\circ}{>} 0({A^\oplus_{i}})$ (\resp $z\stackrel{\circ}{<} 0({A^\oplus_{i}})$); in particular $C^\infty_{z}$ is a well defined local chamber of sign $\qe$.

\par We now consider $t\in]0,1[$, $z=\qf(t)$, $p=\pi(t)=\qr_{C_{\infty}}(z)$.
We write $({A_{\ominus}},{A_{\oplus}})$ a twin apartment containing $C_{\infty}$ and $C^-_{z,\qf}$.
By \ref{5.1}.4 the restriction of $\qr_{C_{\infty}}$ to $\sht^\pm_{z}({\SHI_{\oplus}})$ (whose image is $\sht^\pm_{p}({\A_{\oplus}})$) is the retraction $\qr_{C^\infty_{z}}$ (of $\sht^\pm_{z}({\SHI_{\oplus}})$ onto $\sht^\pm_{z}({A_{\oplus}})$ with center $C^\infty_{z}$) followed by the isomorphism $\psi$ of $\sht^\pm_{z}({A_{\oplus}})$ onto $\sht^\pm_{p}({\A_{\oplus}})$ induced by $\qr_{C_{\infty}}$ (hence by an element of $I_{\infty}$).
Note that $\psi(C^\infty_{z})=C^\infty_{p}$.

\par We saw that $C^\infty_{z}$ and $C^+_{z,\qf}$ have the same sign $\qe$.
So we may consider a minimal gallery $C^0=C^\infty_{z}, C^1,\ldots,C^m=C^+_{z,\qf}$, of length $m=m_{z}=m_{t}$; we write $\mathbf i_{z}=\mathbf i_{t}$\index{i@$\mathbf i_{z}$} its type. 
We suppose that $C^0,C^1,\ldots, C^ {m'_{t}}$ is a minimal gallery from $C^\infty_{z}$ to $\qf_{+}(z)$.
Now $(C^i_{p}=\qr_{C_{\infty}}(C^i))_{0\leq i\leq m}$ is a minimal gallery in $\sht^\pm_{p}({\A_{\oplus}})$ of type $\mathbf i_{p}:=\mathbf i_{t}$ from $C^\infty_{p}=\qr_{C_{\infty}}(C^\infty_{z})$ to $C^+_{p,\pi}=\qr_{C_{\infty}}(C^+_{z,\qf})$. It is minimal as we retract with respect to $C_p^\infty$, which is the first chamber of the gallery, see \ref{5.1} 4) Lemma.
\par 3) So the $\ql-$path $\pi$ is decorated by the datum $((C^+_{t,\pi})_{t<1},(C^-_{t,\pi})_{t>0})$, with $C^+_{0,\pi}=pr_{\pi_{+}(0)}(C^\infty_{\pi(0)})$.
For any $t\in]0,1[$, one has chosen the type $\mathbf i_{t}$ of a minimal gallery of local chambers in $\sht^\qe_{p}({\A_{\oplus}})$ from $C^\infty_{p}$ to $C^+_{p,\pi}$; its length is $m=m_{p}=m_{t}$.
We supposed also that this minimal gallery begins by a minimal gallery (of length $m'_{t}$) from $C^\infty_{p}$ to $\pi_{+}(t)$ and continues by a gallery of local chambers dominating $\pi_{+}(t)$.

\par For any $t\in]0,1[$ we may consider a gallery} $\mathbf c_{p}=\mathbf c_{t}$\index{c@$\mathbf c_{p}$} of local chambers in $\sht^\qe_{p}({\A_{\oplus}})$ from $C^\infty_{p}=pr_{p}(C_{\infty})$ to the projection $C^ {(+)}_{p,\pi}$ of $C^-_{p,\pi}$ on the segment germ $\pi_{(+)}(t)=\pi(t)+\pi'_{-}(t).[0,1)$ (opposite $\pi_{-}(t)$), that is of type $\mathbf i_{t}$ and centrifugally folded with respect to $C^-_{p,\pi}$, see \cite[\S 2.2]{BPR19}.

\par Such galleries may not exist in general. But we saw above that the decorated line segment $\un\qf$ or $\un{[x,y]}$ gives rise to such galleries.%this superdecoration $\un{\un\pi}$ of $\un\pi$.

\subsection{Superdecorated $C_{\infty}-\ql$ paths}\label{5.4}%%%%%%%%%%%%

\par Let $\pi$ be a $\ql-$path in $\A_{\oplus}$, with $\ql\in \qe(\ov C\zv_{f}\cap\sht^\circ)$ and $\pi(0) \stackrel{\circ}{\geq}0$ if $\qe=1$, $\pi(0) \stackrel{\circ}{\leq}0$ if $\qe=-1$.
Clearly we have $\pi(]0,1]) \subset \qe\sht^\circ$.

\par 1) We consider the numbers $0=t'_{0}<t'_{1}<\cdots<t'_{n}=1$ of \S \ref{5.2}.2 and the points $p'_{i}=\pi(t'_{i})$ where $\pi$ may be folded.
For $t'_{i}\leq t < t'_{i+1}$(\resp $t'_{i} < t \leq t'_{i+1}$) the derivative $\pi'_{+}(t)$ (\resp $\pi'_{-}(t)$) is a constant.
The derivative $\pi'_{\pm}(t)\in W\zv.\ql$ is in $\qe\sht^\circ$.

\begin{enonce*}[plain]{\quad\; 2) Lemma} There is only a finite number of pairs $(M,t)$ with a wall $M$ containing a point $p=\pi(t)$ for $0<t<1$, such that $\pi_{+}(t)$ is not in $M$ and $C^\infty_{p}$ is not in the same side of $M$ as $\pi_{+}(t)$.
\end{enonce*}

\begin{proof} We may restrict to the $t\in [t'_{i},t'_{i+1}[$, more precisely to the $t$ in a small open set $\QO$ in $[t'_{i},t'_{i+1}]$. 
We write $M=M_{\oplus}(\qa+k\qx)$ with $\qa+k\qx \in \QF_{a+}$ (so $k\geq0$). 
The conditions are thus $(\qa+k\qx)(\pi(t))=0$ (hence $\qa(\pi(t))\leq0$), $\qa(\pi'_{+}(t))\neq0$ and more precisely $\qa(\pi'_{+}(t))>0$ (as $C^\infty_{t}\subset D(-\qa-k\qx)$).
Suppose $\qe=+1$, then $\pi(\QO)$ (\resp $\pi'_{+}(t)$, which is independent of $t\in [t'_{i},t'_{i+1}[$) is in the open Tits cone $\sht^\circ$ (as $t>0$), so $\qa(\pi(t))\leq0$ (for some $t\in\QO$) (\resp $\qa(\pi'_{+}(t))>0$) is possible only for a finite number of positive (\resp negative) roots $\qa$.
Hence there is a finite number of possible $\qa$ (by \cite[Proposition 3.12 c)]{K90}, and, then, the condition $(\qa+k\qx)(\pi(t))=0$ is possible for only a finite number of $k\in\Z$.
Moreover $t\in\QO$ is uniquely determined by $\qa+k\qx$  as $\qa(\pi'_{+}(t))\neq0$. 
 We get now the expected finiteness by using the compactness of $[t'_{i},t'_{i+1}]$. 

\par In the case $\qe=-1$, one argues similarly, just exchanging positive and negative roots.
\end{proof}

\par 3) Suppose now that $\pi$  is the underlying path of a decorated $\ql-$path  $\un\pi=(\pi,(C^+_{t,\pi})_{t<1},(C^-_{t,\pi})_{t>0})$
 with $\ql\in \qe(\ov C\zv_{f}\cap\sht^\circ)$ and $\pi(0) \stackrel{\circ}{\geq}0$ if $\qe=1$ (\resp $\pi(0) \stackrel{\circ}{\leq}0$ if $\qe=-1$).
 Moreover, for any $t\in]0,1[$, one supposes the existence of a gallery $\mathbf c_{t}$ satisfying the
 conditions of \S \ref{5.3}.3.
 
\par The fact that $\un\pi=(\pi,(C^+_{t,\pi})_{t<1},(C^-_{t,\pi})_{t>0})$ is a decorated $\ql-$path tells that there are   numbers $0=t'_{0}<t'_{1}<\cdots<t'_{r}=1$ such that, for any $1\leq i\leq r$, $\set{\pi(t) \mid t'_{i-1}\leq t \leq t'_{i}}$ is a segment $[\pi(t'_{i-1}),\pi(t'_{i})]$ and $\un{[\pi(t'_{i-1}),\pi(t'_{i})]}=([\pi(t'_{i-1}),\pi(t'_{i})],(C^+_{t,\pi})_{t'_{i-1}\leq t < t'_{i}},(C^-_{t,\pi})_{t'_{i-1} <  t \leq t'_{i}})$ is a decorated segment ( defined in \cite[Def. 2.6]{BPR19}).

\par In particular the direction $C^ {+\zzv}_{t,\pi}$ of $C^+_{t,\pi}$ for $t'_{i-1}\leq t < t'_{i}$ (\resp $C^ {-\zzv}_{t,\pi}$ of $C^-_{t,\pi}$ for $t'_{i-1} < t \leq t'_{i}$) is constant of sign $\qe$ (\resp $-\qe$), the same (\resp opposite) as the sign of the direction $C^{\infty\zzv}_{\pi(t)}$ of  $C^{\infty}_{\pi(t)}$ (if $t\neq0$).
We write $w_{i-1}^+=d\zw(C^{\infty\zzv}_{\pi(t_{i-1})},C^ {+\zzv}_{t_{i-1},\pi})$ if $i\geq2$ (\resp $w_{i}^-=d^{*\zzw}(C^{\infty\zzv}_{\pi(t_{i})},C^ {-\zzv}_{t_{i},\pi})=d\zw(C^{\infty\zzv}_{\pi(t_{i})},-C^ {-\zzv}_{t_{i},\pi})$\index{d@$d\zw(-,-),d^ {*\zzw}(-,-)$} the corresponding Weyl distance (\resp codistance), \cite[5.133]{AB08}.
We then clearly have $\pi'_{+}(t_{i})=w_{i}^+.\ql$ (\resp $\pi'_{-}(t)=w_{i}^-.\ql$) if one considers $C^{\infty\zzv}_{\pi(t)}$ as a new fundamental vectorial chamber (for $t\neq0$).

\begin{enonce*}[plain]{\quad\; 4) Lemma} One writes $p_{0}=\pi(t_{0}),p_{1}=\pi(t_{1}),\ldots,p_{\ell_{\pi}}=\pi(t_{\ell_{\pi}})$ with $0=t_{0}<t_{1}<\cdots<t_{\ell_{\pi}-1}<t_{\ell_{\pi}}=1$ the points $p=\pi(t)$ satisfying (for some wall $M$) the conditions of Lemma \ref{5.4}.2 above (or $t=0$, $t=1$).
Then any point $t$ where the path $\pi$ is folded at $\pi(t)$ appears in the set $\set{t_{k} \mid 1\leq k \leq\ell_{\pi}-1}$.
\end{enonce*}

\begin{proof} If $\pi$ is folded at $p=\pi(t)$ (for $t\in]0,1[$), one has $\pi'_{+}(t)\neq\pi'_{-}(t)$, \ie $\pi_{(+)}(t)\neq\pi_{+}(t)$.
And, as $\pi_{(+)}(t)$ (\resp $\pi_{+}(t)$) is the segment germ in $\ov{C^ {(+)}_{p,\pi}}$ (\resp $\ov{C^ {+}_{p,\pi}}$) with the same type as $\ql$, one has $C^ {(+)}_{p,\pi}\neq C^ {+}_{p,\pi}$. 
So the gallery $\mathbf c_{p}$ from $C^\infty_{p}$ to $C^ {(+)}_{p,\pi}$ is folded.
This is possible only if there is at least one  wall $M$ separating $C^\infty_{p}$ from $C^ {+}_{p,\pi}$; as $\pi_{(+)}(t)\neq\pi_{+}(t)$ we may also assume $\pi_{+}(t)\not\subset M$.
So $t\in \set{t_{k} \mid 1\leq k \leq\ell_{\pi}-1}$.
\end{proof}

\par {\bf 5) Definition.} 

A superdecorated\index{superdecoration@superdecoration} $C_{\infty}-\ql$ \index{C@$C_{\infty}-\ql$} path is a quadruple $\un{\un\pi}=(\pi,(C^+_{t,\pi})_{t<1},(C^-_{t,\pi})_{t>0}, (\mathbf c_{t})_{0<t<1})$ where $\un{\pi}=(\pi,(C^+_{t,\pi})_{t<1},(C^-_{t,\pi})_{t>0})$ is a decorated $\ql-$path and each $\mathbf c_{t}$ is a gallery of type $\mathbf i_{t}$ satisfying the conditions of \S\ref{5.3}.3.
We ask moreover that the local chamber $C^+_{0,\pi}$ is the projection $pr_{\pi_{+}(0)}(C^\infty_{\pi(0)})=pr_{\pi_{+}(0)}(C_\infty)$.

\par

\medskip
\par 6) It is interesting to describe the properties of the underlying $\ql-$path of a superdecorated $C_{\infty}-\ql$ path}.
 We shall do this in \S \ref{5.13} to \S \ref{5.15}, after some auxiliary results about twin buildings in \S \ref{5.10} to \S \ref{5.12}.
This underlying $\ql-$path is a $C_{\infty}-$Hecke path, as in \cite[5.3.1]{Mu19b} (and similar to \cite[def. 5.2]{GR08}).

\par A $\ql-$path $\pi: [0,1] \to\A$ (with $\ql\in \qe(\ov C\zv_{f}\cap\sht^\circ)$) has only a finite number (possibly $0$, if it is not a $C_{\infty}-$Hecke path) of compatible superdecorations $\un{\un\pi}=(\pi,(C^+_{t,\pi})_{t<1},(C^-_{t,\pi})_{t>0}, (\mathbf c_{t})_{0<t<1})$.
Actually, by \S \ref{5.6} and Theorem \ref{5.7} below, such a superdecoration is the image by $\qr_{C_{\infty}}$ of a $C_{\infty}-$friendly line segment (as explained in \S \ref{5.3}.3) and these line segments depend only of the data $(C^+_{p_{k},\pi})_{0\leq k\leq \ell_{\pi}-1}$, $(C^-_{p_{k},\pi})_{1\leq k\leq \ell_{\pi}-1}$ and $(\mathbf c_{t_{k}})_{1\leq k\leq \ell_{\pi}-1}$.
Now, as $\ql$ is spherical, the number of possible local chambers $C^\pm_{p_{k},\pi}\subset\A$ containing $\pi_{\pm}(t_{k})$ in their closure is finite.
The type $\mathbf i_{t_{k}}$ is the type of a specific minimal gallery in $\sht^\qe_{p}(\A_{\oplus})$ between the chambers $C_{p}^\infty$ and $C^+_{p,\pi}$ (which are well defined by the decoration and $C_{\infty}$); so there is only a finite number of possible such types (moreover we shall fix one of them).
Therefore the number of galleries $\mathbf c_{t_{k}}$ in $\A$ of type $\mathbf i_{t_{k}}$ from $C^\infty_{p_{k}}$ to $C^ {(+)}_{p_{k},\pi}$ is also finite.

\subsection{Liftings of superdecorated $C_{\infty}-\ql$ paths}\label{5.6}%%%%%%%%%%%%

\par\quad\; 1) One considers a superdecorated $C_{\infty}-\ql$  path  $\un{\un\pi}=(\pi,(C^+_{t,\pi})_{t<1},(C^-_{t,\pi})_{t>0}, (\mathbf c_{t})_{0<t<1})$ of shape $\ql\in\qe(\ov C\zv_{f})$, as above in \S \ref{5.4}.5.
One considers also a point $x$ that is $C_{\infty}-$friendly (\ie there is a twin apartment $({A_{\ominus}},{A_{\oplus}})$ with $x\in {A_{\oplus}}$ and $C_{\infty}\subset{A_{\ominus}}$) and such that $\qr_{C_{\infty}}(x)=p_{0}=\pi(0)$.
By Theorem \ref{thmIsomorphisms_fixing_twin_apartments}, we have moreover $x\stackrel{\circ}{\geq}0({A_{\oplus}})$  if $\qe=+1$ and $x\stackrel{\circ}{\leq}0({A_{\oplus}})$ if $\qe=-1$.

\par We aim to prove that there is a $C_{\infty}-$friendly line segment $[x,y]$ with $d\zv(x,y)=\ql\in\qe(\ov C\zv_{f})$ such that $\un{\un\pi}$ is the ``image'' of $[x,y]$ by $\qr_{C_{\infty}}$ (as constructed in \S \ref{5.3}).
We want also a formula for the number of these $[x,y]$.

\par The idea is to build $[x,y]$ progressively, starting from $x$. So we look locally.

\medskip
\par 2) We look first for the segment germs $[x,x_{+})$ of sign $\qe$ such that $\qr_{C_{\infty}}([x,x_{+}))=\pi_{+}(0)=p_{0}+\pi'_{+}(0).[0,1)$, more precisely to local chambers $C^+_{x}$ of sign $\qe$ such that $\qr_{C_{\infty}}(C^+_{x})=C^+_{p_{0},\pi}$ (then $[x,x_{+})$ is the segment in $\ov{C^+_{x}}$ with the same type as $\ql$; so $\qr_{C_{\infty}}([x,x_{+}))=\pi_{+}(0)$ and $C^+_{x}=pr_{[x,x_{+})}(C^\infty_{x})$).

\begin{prop*} There is a local chamber $C^+_{x}$ of sign $\qe$ such that $\qr_{C_{\infty}}(C^+_{x})=C^+_{p_{0},\pi}$.
In case $\qe=+1$, we suppose now moreover $p_{0} \stackrel{\circ}{>}0$ (\ie $p_{0}\neq 0$), then the number of these $C^+_{x}$ (or of the corresponding segment germ $[x,x_{+})$) is finite (if $q=\vert\k\vert$ is finite) and equal to $q^ {m_{0}}$ if $p_{0}$ or $x$ is a special vertex, where $m_{0}$ is the length of $w_{0}^+$ (\cf \S \ref{5.4}.3) \ie the length of a minimal gallery $\mathbf d$ in $\sht^\qe_{p_{0}}({\A_{\oplus}})$ from $C^\infty_{p_{0}}$ to $C^+_{p_{0},\pi}$.
If $p_{0}$ is not special, one has to replace $m_{0}$ by the number $m''_{0}$ of walls separating $C^\infty_{p_{0}}$ from $C^+_{p_{0},\pi}$.
\end{prop*}

\begin{rema*} When $p_{0}=0$, then $C^\infty_{0}$ is negative and  $C^+_{0,\pi}$ of sign $\qe$, so there is a problem if $\qe=+1$.
(Fortunately, for   Muthiah's purpose one has $p_{0}=0$ but $\qe=-1$, as $\qp^\ql$ acts by the translation of vector $-\ql$.)
In this problematic case the condition for $C^+_{x}$  involves codistances: it is $d^{*\zzw}(C^\infty_{x},C^+_{x})=w_{0}^+:=d^{*\zzw}(C_{0}^\infty,C^+_{0,\pi})$.
By the retraction $\qr_{C_{\infty}}$, it is clearly equivalent to $d^{*\zzw}(C^\infty_{0},\qr_{C_{\infty}}(C^+_{x}))=w_{0}^+$, \ie to $\qr_{C_{\infty}}(C^+_{x})=C^+_{p_{0},\pi}$.
There are infinitely many solutions for this condition.
\end{rema*}

\begin{proof} We avoid the problematic case $\qe=+1$, $p_{0}=0$.
Then the equality $\qr_{C_{\infty}}(C^+_{x})=C^+_{p_{0},\pi}$ is equivalent to $d^{\zzw}(C^\infty_{x},C^+_{x})=w_{0}^+:=d^{\zzw}(C_{0}^\infty,C^+_{0,\pi})$.
This is clear as we saw (in \ref{5.3}.2) that $\qr_{C_{\infty}}$ restricted to $\sht^\pm_{x}({\SHI_{\oplus}})$ is equal to a retraction $\qr_{C_{x}^{\infty}}$ (of $\sht^\pm_{x}({\SHI_{\oplus}})$ onto $\sht^\pm_{x}({A_{\oplus}})$ with center $C^\infty_{x}$) followed by an isomorphism $\psi$ of $\sht^\pm_{x}({A_{\oplus}})$ onto $\sht^\pm_{p_{0}}({\A_{\oplus}})$ (which sends $C^\infty_{x}$ to $C^\infty_{p_{0}}$).

Now $d^{\zzw}(C^\infty_{x},C^+_{x})=w_{0}^+$ is equivalent to the existence of a minimal gallery of type $\mathbf i$ (the type of a fixed minimal decomposition of $w_{0}^+$), hence of length $m_{0}=\ell(w_{0}^+)$, in $\SHI_{\oplus}$ from $C^\infty_{x}$ to $C^+_{x}$.
There are $q^ {m_{0}}$ (or more generally $q^ {m''_{0}}$) such galleries.
\end{proof}

\medskip
\par 3) For $0<t<1$, we suppose now given a $z=\qf(t)$, a local chamber $C^-_{z,\qf}$ hence a segment germ $\qf_{-}(t)\subset\ov{C^-_{z,\qf}}$ (of the same type as $-\ql$) such that the pair $(C_{\infty},z)$ (hence also $(C_{\infty},C^-_{z,\qf})$ or $(C_{\infty},\qf_{-}(t))$) is twin friendly and $\qr_{C_{\infty}}(z)=\pi(t)=p,\ \qr_{C_{\infty}}(C^-_{z,\qf})=C^-_{p,\pi} ,\  \qr_{C_{\infty}}(\qf_{-}(t))=\pi_{-}(t)$.
We write $({A_{\ominus}},{A_{\oplus}})$ a twin apartment with $C_{\infty}\subset{A_{\ominus}},\ C^-_{z,\qf},\ \qf_{-}(t) \subset {A_{\oplus}}$.
We now look for a segment germ $[z,z_{+})$ of sign $\qe$ opposite $\qf_{-}(t)$, such that $\qr_{C_{\infty}}([z,z_{+}))=\pi_{+}(t)=p+\pi'_{+}(t).[0,1)$; more precisely we look for a local chamber $C^+_{z}$ of sign $\qe$ opposite $\qf_{-}(t)$, such that $\qr_{C_{\infty}}(C^+_{z})=C^+_{p,\pi}$ and $C^+_{z}=pr_{[z,z_{+})}(C^-_{z,\qf})$.

\begin{prop*} a) There is a local chamber $C^+_{z}$ of sign $\qe$ in $\sht^\qe_{z}({\SHI_{\oplus}})$  such that $\qr_{C_{\infty}}(C^+_{z})=C^+_{p,\pi}$ and that the segment germ $[z,z_{+})$ in $\ov{C^+_{z}}$ of the same type as $\ql$ is opposite $\qf_{-}(t)$.

\par Actually we add the condition that the minimal gallery of type $\mathbf i_{t}$ from $C^\infty_{z}$ to $C^+_{z}$ retracts onto $\mathbf c_{p}$ by the retraction $\qr_{C_{z,\qf}^{-}}$ (of $\sht^\pm_{z}({\SHI_{\oplus}})$ onto $\sht^\pm_{z}({A_{\oplus}})$ with center $C_{z,\qf}^{-}$) followed by the isomorphism $\psi$ of $\sht^\pm_{z}({A_{\oplus}})$ onto $\sht^\pm_{p}({\A_{\oplus}})$ induced by $\qr_{C_{\infty}}$.
This implies $C^+_{z}=pr_{[z,z_{+})}(C^-_{z,\qf})$. 

\par b) Suppose $q=\vert\k\vert$ finite. Then the number of these local chambers is finite (non zero) and equal to the cardinality of the set $\SHC^m_{C^-_{p,\pi}}(C^\infty_{p},\mathbf c_{p})$ of all minimal galleries in $\sht^\qe_{p}({\SHI_{\oplus}})$ starting from $C^\infty_{p}$ and retracting onto $\mathbf c_{p}$ by the retraction of $\sht^\qe_{p}({\SHI_{\oplus}})$ onto $\sht^\qe_{p}({\A_{\oplus}})$ with center $C^-_{p,\pi}$.
(Compare with \cite[\S 3.3 (b)]{BPR19}).

\par c) If $\pi$ is not folded at $p=\pi(t)$, then $\pi_{(+)}(t)=\pi_{+}(t)$.
The number of expected local chambers  $C^+_{z}$ (or of expected segment germs $[z,z_{+})$) is then $q^ {m''_{t}}$, where $m''_{t}$ is the number of  walls that separate $C^\infty_{p}$ from $C^+_{p,\pi}$ and do not contain $\pi_{+}(t)$ (or equivalently $\pi_{-}(t)$).
If $q=\vert\k\vert$ may be infinite, we have at least that $[z,z_{+})$ and $C^+_{z}$ are unique when $m''_{t}=0$. 

\par There is a twin apartment $({A_{\ominus}'},{A_{\oplus}'})$ with ${A_{\ominus}'}\supset C_{\infty}$ and ${A_{\oplus}'}\supset \ov{C^+_{z}}\cup C^-_{z,\qf} \supset [z,z_{+})$.

\par d) In particular, if $t$ is not one of the $t_{i}$ in Lemma \ref{5.4}.4, then $m''_{t}=0$ and $C^+_{z}$ is unique; more precisely this unique $C^+_{z}$ is in ${A_{\oplus}}$, which already contains $C^-_{z,\qf}$ (and $C_{\infty}\subset A_{\ominus}$). 
In particular $C^+_{p,\pi}=C^{(+)}_{p,\pi}$. All this is true for any cardinality of $\k$.

\end{prop*}

\begin{NB} From d) above, one deduces that a superdecorated $C_{\infty}-\ql$ path $\pi$ satisfies the condition of definition of decorated $\ql-$paths in \S\ref{5.2}.3 above with the subdivision $t_{0}=0<t_{1}<\cdots<t_{\ell_{\pi}}=1$ of Lemma \ref{5.4}.4.
Moreover, for $t$ different from each $t_{i}$, the gallery $\mathbf{c}_{\pi(t)}$ is minimal, uniquely determined by its type $\mathbf{i}_{t}$.
\end{NB} 

\begin{proof} a) + b) We write $g\in I_{\infty}$ an element (of $G_{twin}$ fixing $C_{\infty}$) sending $A_{\oplus}$ to $\A_{\oplus}$ and $z$ to $p$; it exists by paragraph \ref{n3.15.1}.
By \S \ref{5.2}.3 the restriction of $\qr_{C_{\infty}}$ to $\sht^\pm_{z}({\SHI_{\oplus}})$ is $g$ restricted to $\sht^\pm_{z}({\SHI_{\oplus}})$ (sending isomorphically $\sht^\pm_{z}({\SHI_{\oplus}})$ onto $\sht^\pm_{p}({\SHI_{\oplus}})$) followed by the retraction $\qr_{C^{\infty}_{p}}$ (of $\sht^\pm_{p}({\SHI_{\oplus}})$ onto $\sht^\pm_{p}({\A_{\oplus}})$ with center $C^{\infty}_{p}$).
The expected $C^+_{z}$  and $[z,z_{+})$ correspond thus bijectively (by $g$) to pairs $(C^+_{p},[p,p_{+}))$ where $C^+_{p}$ is a local chamber in $\sht^\qe_{p}({\SHI_{\oplus}})$ such that $\qr_{C^{\infty}_{p}}(C^+_{p})=C^+_{p,\pi}$ and that  $[p,p_{+})$ is the segment germ in $\ov{C^+_{p}}$ of the same type as $\ql$ and is opposite $\pi_{-}(t)$.

\par But $\mathbf c_{p}=\mathbf c_{t}$ is a gallery in $\sht^\qe_{p}({\A_{\oplus}})$ starting from $C^\infty_{p}$, of type $\mathbf i_{t}$, the type of a minimal gallery from $C^\infty_{p}$ to $C^+_{p,\pi}$.
Hence any minimal gallery in  $\sht^\qe_{p}({\SHI_{\oplus}})$ starting from $C^\infty_{p}$ of type $\mathbf i_{t}$ %retracts by $\qr_{C^{\infty}_{p}}$ onto a gallery which 
ends with a chamber $C^+_{p}$ such that $\qr_{C^{\infty}_{p}}(C^+_{p})=C^+_{p,\pi}$.
Moreover $\mathbf c_{p}$ is centrifugally folded with respect to $C^-_{p,\pi}$ and ends with the chamber $C^{(+)}_{p,\pi}$ projection of $C^-_{p,\pi}$ onto the segment germ $\pi_{(+)}(t)=\pi(t)+\pi'_{-}(t).[0,1)$ (of type $\ql$)  opposite $\pi_{-}(t)$ (of type $-\ql$ in $\ov{C^-_{p,\pi}}$).
The set $\SHC^m_{C^-_{p,\pi}}(C^\infty_{p},\mathbf c_{p})$ is thus exactly the set of all galleries retracting by  $\qr_{C^{\infty}_{p}}$ onto the minimal gallery of type $\mathbf i_{t}$ from $C^\infty_{p}$ to $C^+_{p,\pi}$ and retracting by $\qr_{C^-_{p,\pi}}$ onto $\mathbf c_{p}$.
In particular the last chamber $C^+_{p}$ of such a gallery satisfies $\qr_{C^{\infty}_{p}}(C^+_{p})=C^+_{p,\pi}$ and the segment germ $[p,p_{+})$ in $\ov{C^+_{p}}$ of the same type as $\ql$ retracts by $\qr_{C^-_{p,\pi}}$ onto the segment germ $\pi_{(+)}(t)$.
So a) and b) are proved, as a consequence of \cite[\S 2.3]{BPR19} (mutatis mutandis), which tells that $\SHC^m_{C^-_{p,\pi}}(C^\infty_{p},\mathbf c_{p})$ is non empty and finite (if $q=\vert\k\vert<\infty$) and gives a formula for its cardinality.

\par c)  If $\pi$ is not folded at $p=\pi(t)$, then $\pi_{(+)}(t)=\pi_{+}(t)$ and $\mathbf c _{t}$ is a gallery of type $\mathbf i _{t}$ and length $m_{p}$.
By the convention for $\mathbf i _{t}$ (\cf \S \ref{5.3}.3) the gallery $\mathbf c _{t}$ shortened by removing the chambers of numbering $>m'_{t}$ is minimal from $C^\infty_{p}$ to $\pi_{+}(t)$ and the chambers of numbering $\geq m'_{t}$ contain $\pi_{+}(t)$ in their closure.
So the number of possible choices for $[z,z_{+})$ is the number of possible liftings of the gallery $\mathbf c_{t}$ shortened (and then $C^+_{z}=pr_{[z,z_{+})}(C^-_{z,\qf})$ is well determined).
One considers the hyperplanes $M$ cutting this shortened gallery $\mathbf c _{t}$ along a panel and their contribution to a factor of this number of liftings,
see \cite[\S 2.3]{BPR19} (mutatis mutandis).
If $M$ is not a wall, its contribution is $1$. 
The  walls cutting this shortened gallery $\mathbf c _{t}$ \ie between the chambers $C^0$ and $C^ {m'_{t}}$ are exactly the walls that separate $C^\infty_{p}$ from $C^+_{p,\pi}$ and do not contain $\pi_{+}(t)$; the contribution of each of them is $q$.
If $m''_{t}=0$, each contribution is $1$ and $[z,z_{+})$ is unique.

\par To get the twin apartment $A'$, we just have to modify $\A$ by elements of $U_{\qa+k\qx}$ where $M=M_{\oplus}(\qa+k\qx)$ cuts $\mathbf c _{t}$ between the chambers $C^0$ and $C^ {m'_{t}}$ and $D_{\ominus}(\qa+k\qx)\supset C_{\infty}$, $D_{\oplus}(\qa+k\qx)\supset C^-_{p,\pi}$ and then apply $g^ {-1}$.
The modified apartment $A'$ contains $C_{\infty}, C^-_{z,\qf}$ and $[z,z_{+})$, hence also $\ov{C^+_{z}}$.

\par d) In this case $t\not\in\set{t_{1},\ldots,t_{\ell_{\pi}}}$, one has $m''_{t}=0$ and $q^ {m''_{t}}=1$.
By the above procedure we get $A'$ just by applying $g^ {-1}$ to $\A$. 
So $A'=A=g^ {-1}\A$.
 As $g\in I_{\infty}$ fixes $C_{\infty}$, we have $C^+_{p,\pi}=\qr_{C_{\infty}}(C^+_{z})=\qr_{C_{\infty}}(pr_{[z,z+)}(C^-_{z}))=pr_{\pi_{+}(t)}(C^-_{p})=C^ {(+)}_{p,\pi}$.
 \end{proof}

\begin{theo}\label{5.7}%%%%%%%%%%%%%%
Let  $\un{\un\pi}=(\pi,(C^+_{t,\pi})_{t<1},(C^-_{t,\pi})_{t>0}, (\mathbf c_{t})_{0<t<1})$ be a superdecorated $C_{\infty}-\lambda-$Hecke path in $\A_{\oplus}$ of shape $\ql\in\qe(\ov C\zv_{f}\cap\sht^\circ)$ with $\pi(0)\stackrel{\circ}{\geq}0$ if $\qe=+1$ (\resp $\pi(0)\stackrel{\circ}{\leq}0$ if $\qe=-1$).
We consider also a point $x\in{\SHI_{\oplus}}$ that is $C_{\infty}-$friendly (\ie there is a twin apartment $({A_{\ominus}},{A_{\oplus}})$ with $C_{\infty}\subset {A_{\ominus}}$ and $x\in {A_{\oplus}}$) and such that $\qr_{C_{\infty}}(x)=\pi(0)$.

\par (1) There is a $C_{\infty}-$friendly line segment $[x,y]$ with $d\zv(x,y)=\ql$, such that $\un{\un\pi}$ is the ``image'' of $[x,y]$ by $\qr_{C_{\infty}}$ (as constructed in \S \ref{5.3}).

\par (2) Except in the case $\qe=+1$ and $\pi(0)=0$, the number of these line segments is finite (provided that $q=\vert\k\vert<\infty$) and given by the following formula (for the notations see \S \ref{5.6}.2, Lemma \ref{5.4}.4 and Proposition \ref{5.6}.3)

$$\#\set{[x,y]}=q^ {m''_{0}}\times\prod_{k=1}^ {\ell_{\pi}-1}\,\#\SHC^m_{C^-_{p_{k},\pi}}(C^\infty_{p_{k}},\mathbf c_{p_{k}})$$

\par This number is equal to $q^n.(q-1)^ {n'}$ for some $n,n'\in\Z_{\geq0}$ depending only on $\un{\un\pi}$ (in $\A$), not of $\k$, see \cite[\S 2.3]{BPR19}.

\par (3) If $\pi$ is the parametrization of such a $[x,y]$, we write $z_{0}=x=\qf(t_{0}=0),z_{1}=\qf(t_{1}),\ldots,z_{k}=\qf(t_{k}),\ldots,z_{\ell_{\pi}}=\qf(t_{\ell_{\pi}}=1)=y$.
Then there exist twin apartments $({A^\ominus_{k}},{A^\oplus_{k}})$ (for $1\leq k\leq\ell_{\pi}$) such that ${A^\ominus_{k}}\supset C_{\infty}$ and ${A^\oplus_{k}} \supset \un{[z_{k-1},z_{k}]}$ (\ie ${A^\oplus_{k}}$ contains $[\qf(t_{k-1}),\qf(t_{k}]$ and all $C^+_{t,\qf}$ (for $t_{k-1}\leq t < t_{k}$), $C^-_{t,\qf}$ (for $t_{k-1}< t \leq t_{k}$)).
\end{theo}

\begin{proof} Suppose the line segment $[x=\qf(0),z_{k}=\qf(t_{k})]$ constructed with the expected properties. This is clearly satisfied for $k=0$.
We now construct $[x,z_{k+1}]$.

\par If $k=0$, we investigated the possibilities for $\qf_{+}(0)=[z_{0},z_{1})$ in \S \ref{5.6}.2.
Their number is $\geq1$ and equal to $q^ {m''_{0}}$ under the conditions of (2).
Now the proposition \ref{5.6}.3.d tells that each possibility for $[z_{0},z_{1})$ corresponds to one and only one possibility for $[z_{0},z_{1}]$ and there is a twin apartment $({A^\ominus_{1}},{A^\oplus_{1}})$ such that $C_{\infty}\subset {A^\ominus_{1}}$ and $[x,z_{1}] \subset {A^\oplus_{1}}$; hence $C^\infty_{x}\subset {A^\oplus_{1}}$, $\un{[x,z_{1}]}\subset {A^\oplus_{1}}$.

\par If $k\geq1$, we investigated the possibilities for $\qf_{+}(t_{k})=[z_{k},z_{k+1})$ in \ref{5.6}.3 a) b).
Their number is $\geq1$ and equal to $\#\SHC^m_{C^-_{p_{k},\pi}}(C^\infty_{p_{k}},\mathbf c_{p_{k}})$.
Now the proposition \ref{5.6}.3.d tells that each possibility for $[z_{k},z_{k+1})$ corresponds to one and only one possibility for $[z_{k},z_{k+1}]$.
If we choose a twin apartment $({A^\ominus_{k+1}},{A^\oplus_{k+1}})$ such that $C_{\infty}\subset {A^\ominus_{k+1}}$ and $[z_{k},z_{k+1}) \subset \ov{C^+_{z_{k},\qf}} \subset {A^\oplus_{k+1}}$, then ${A^\oplus_{k+1}}$ contains ${[z_{k},z_{k+1}]}$and $\un{[z_{k},z_{k+1}]}$.
\end{proof}

Note that Theorem~\ref{5.7} is obtained with a slightly different method in \cite[4.5.2]{hebert2024kazhdan}. 

\subsection{Folding measure of superdecorated $C_{\infty}-\ql$ paths}\label{5.8}%%%%%%%%%%%%

\par Let   $\un{\un\pi}=(\pi,(C^+_{t,\pi})_{t<1},(C^-_{t,\pi})_{t>0}, (\mathbf c_{t})_{0<t<1})$ be  a superdecorated $C_{\infty}-\ql$ path  in $\A_{\oplus}$ of shape $\ql\in\qe(\ov C\zv_{f})$, as above in \ref{5.4}.3.
We consider the numbers $0=t_{0}<t_{1}< \cdots<t_{\ell_{\pi}}=1$ and the points $p_{i}=\pi(t_{i})$ as in \ref{5.4}.4.
We recall (\ref{5.3}.2) that, for $p=\pi(t)$ with $t>0$, $C^ {(+)}_{p,\pi}$ is the projection of $C^-_{p,\pi}$ on the segment germ $\pi_{(+)}(t)=\pi(t)+\pi'_{-}(t).[0,1)$; when $t_{i-1}<t<t_{i}$,  $C^ {(+)}_{p,\pi}=C^ {+}_{p,\pi}$ (see \cite[Lemma 2.5]{BPR19} and Proposition \ref{5.6} (3.d) above).
In the following of this subsection we drop $\pi$ in the notations $C^ {\pm}_{p,\pi}=C^ {\pm}_{t,\pi}$ and $C^ {(+)}_{p,\pi}=C^ {(+)}_{t,\pi}$.

\begin{figure}[h] 
% le [h] sert ? positionner la figure ? l'endroit d'insertion, sinon [t] pour haut de page, [b] pour bas de page et [p] pour page s?par?e
   \begin{center}
  \includegraphics[width=18cm]{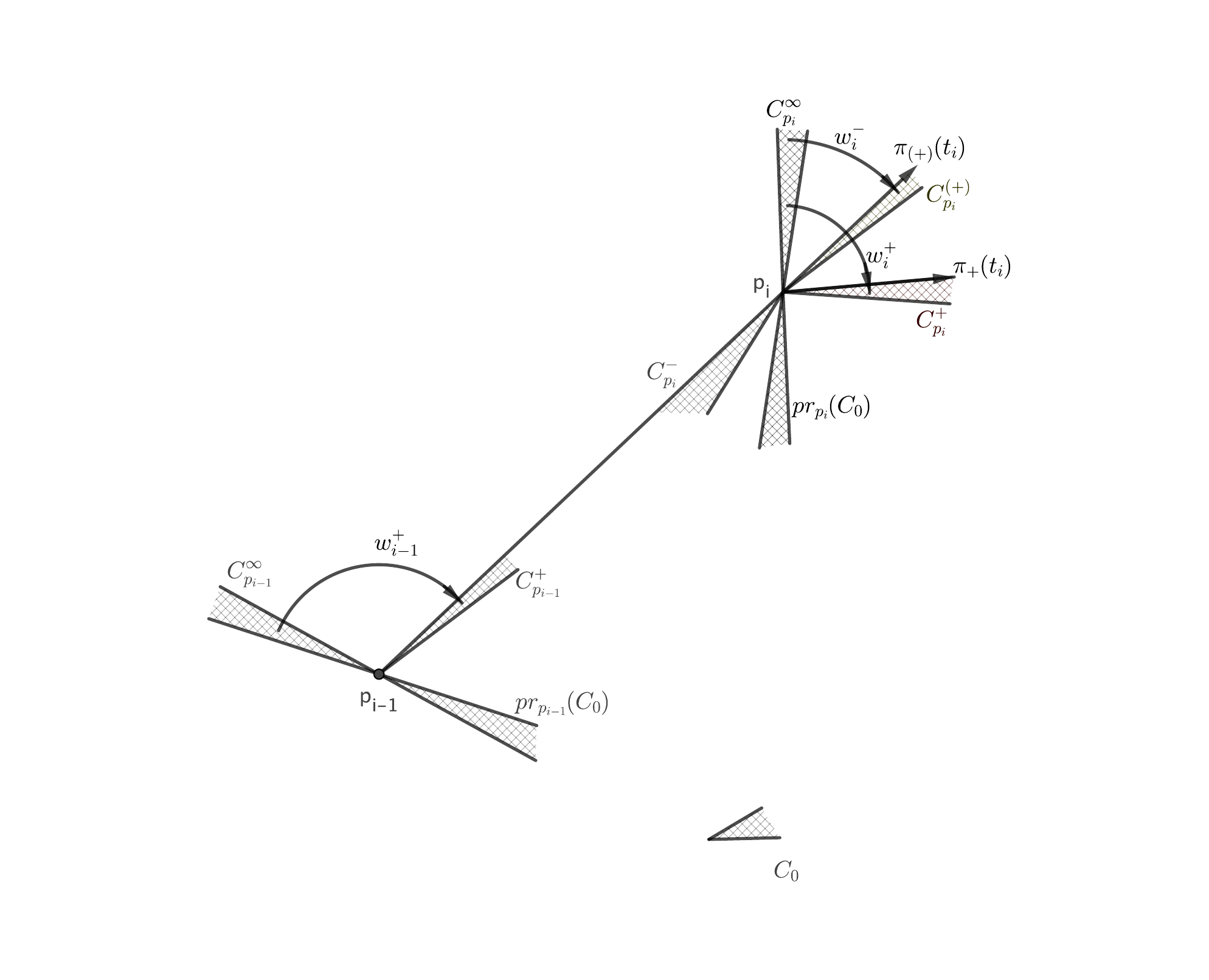}
  %\label{fig1.2.2}}
   \end{center}
\end{figure}

\par The direction $C^ {+\zzv}_{t}$ of $C^+_{t}$ for $t_{i-1}\leq t < t_{i}$ (\resp $C^ {-\zzv}_{t}$ of $C^-_{t}$ for $t_{i-1} < t \leq t_{i}$) is constant of sign $\qe$ (\resp $-\qe$), the same (\resp opposite) as the sign of the direction $C^{\infty\zzv}_{\pi(t)}$ of  $C^{\infty}_{\pi(t)}$ (if $t\neq0$); here we may replace the $t_{i}$ by the $t'_{j}$ of \ref{5.4}.3.
From \cite[2.9.2]{BPR19} it is also clear that, for $t_{i-1}< t \leq t_{i}$, the direction $C^ {(+)\zzv}_{t}$ of $C^ {(+)}_{t}$ is constant of sign $\qe$ and equal to $C^ {+\zzv}_{p_{i-1}}$.
For $i\geq1$, we write $w_{i}^+=d\zw(C^{\infty\zzv}_{p_{i}},C^ {+\zzv}_{p_{i}})$  (if $i<\ell_{\pi}$) (\resp $w_{i}^-=d\zw(C^{\infty\zzv}_{p_{i}},C^ {(+)\zzv}_{p_{i}})=d\zw(C^{\infty\zzv}_{p_{i}},C^ {+\zzv}_{p_{i-1}})$).
Then we clearly have $\pi'_{+}(t_{i})=w_{i}^+.\ql$ (for $i<\ell_{\pi}$) (\resp $\pi'_{-}(t_{i})=w_{i}^-.\ql$ (for $i>0$)) if one considers $C^{\infty\zzv}_{p_{i}}$ as a new fundamental vectorial chamber.

\begin{prop*}%%%%%%%%%%%%%%%%%
For the Bruhat order in $W\zv$, one has $w_{i-1}^+ \geq w_{i}^-$ for $i\geq2$ and $w_{i}^-\leq w_{i}^+$ for $1\leq i <\ell_{\pi}$.
\end{prop*}

\begin{remas*} 1) Unfortunately this gives no inequality between the $w_{i}^+$ (or the $w_{i}^-$).
Perhaps one can get some inequalities with other definitions of $w_{i}^\pm$.

\par 2) In the case of Hecke paths in a masure with respect to a sector germ %(\resp a local chamber) 
one gets $w_{i-1}^+ = w_{i}^-$ %(\resp $w_{i-1}^+ \leq w_{i}^-$ ?) 
and $w_{i}^-\leq w_{i}^+$.
So one gets inequalities between the $w_{i}^+$ (or the $w_{i}^-$).
This case of sector germs is in \cite{GR14}.
It should be possible to prove similarly the case of a Hecke path with respect to a local chamber, but it is written nowhere.
\end{remas*}

\begin{proof} The second inequality is clear: $\mathbf c_{t_{i}}$ is a gallery fom $C^\infty_{p_{i}}$ to $C^ {(+)}_{p_{i}}$, with the same type as a minimal gallery from $C^\infty_{p_{i}}$ to $C^ {+}_{p_{i}}$ (type associated to a minimal decomposition of $w_{i}^+=d\zw(C^{\infty\zzv}_{p_{i}},C^ {+\zzv}_{p_{i}})$).
For the first inequality recall that $pr_{p_{i}}(C_{0})\zv$ is the vectorial chamber containing the $\vect{p_{i}x}$ for $x\in C_{0}$ sufficiently near from $0$.
So $C^ {\infty\zzv}_{p_{i}}=opp(pr_{p_{i}}(C_{0})\zv)$ is the vectorial chamber containing the $\vect{xp_{i}}$ for these $x$.
But we have $\vect{x{p_{i}}}=\vect{x{p_{i-1}}}+\vect{p_{i-1}{p_{i}}}$ and $\vect{x{p_{i-1}}}\in C^ {\infty\zzv}_{p_{i-1}}$, $\vect{p_{i-1}{p_{i}}}\in\ov{C^ {+\zzv}_{p_{i-1}}}$.
Hence $C^ {\infty\zzv}_{p_{i}}$ meets the closed convex hull of $C^ {\infty\zzv}_{p_{i-1}}$ and $C^ {+\zzv}_{p_{i-1}}=C^ {(+)\zzv}_{p_{i}}$.
So $C^ {\infty\zzv}_{p_{i}}$ is in their enclosure, \ie $C^ {\infty\zzv}_{p_{i}}$ is a vectorial chamber of a minimal gallery from $C^ {\infty\zzv}_{p_{i-1}}$ to $C^ {+\zzv}_{p_{i-1}}=C^ {(+)\zzv}_{p_{i}}$.
This proves that $w_{i}^-=d\zw(C^{\infty\zzv}_{p_{i}},C^ {(+)\zzv}_{p_{i}}) \leq d\zw(C^{\infty\zzv}_{p_{i-1}},C^ {+\zzv}_{p_{i-1}}) = w^+_{i-1}$.
\end{proof}

\subsection{Opposite segment germs and retractions in masures or twin buildings}\label{5.10}%%%%%%%%%

\par From \S \ref{5.10} to \S \ref{5.12}, we consider $\SHI$ a  twin building and $\A$ its canonical [twin] apartment.
We use the notation [twin] to indicate the reference to a classical notation in twin buildings, not to \S \ref{2.6}.
We think of $\A$ as a vector space $V=\vect\A$, even if it is more precisely the union of two opposite Tits cones in $V$.
These Tits cones are associated to a root system $\QF$, a Weyl group $W\zv$ and a fundamental chamber $C\zv_{f}$; but the thick walls of $\SHI$ are associated to some particular roots called thick roots.

\par Actually we think very strongly to the case where $\SHI$ is the tangent space (with its unrestricted building structure) at a point $p$ to a thick masure, $\A\ni p$ is an apartment of this masure, $\QF$ is in the dual of $V=\vect\A$ and the thick walls in $\SHI$ are associated to the walls of this masure containing $p$ (\ie the direction $\ker\qb$ of this wall satisfies $\qb(p)\in\Z$ : $\qb$ is a thick root).

\par In the following lines up to the proposition (included), we indicate between parentheses some words we may add when we think to a masure.

\par We consider:
\medskip

\parni $C^-_{p}$ a negative (local) chamber (with vertex $p$) in $\A$

\parni $\qx,\eta$ positive segment germs of origin $0$ (or $p$) in $\A$

\parni $-\qx,-\eta$ their negative opposites in $\A$

\parni $C_{-\qx}$ a negative (local) chamber in $\A$ (with vertex $p$) containing $-\qx$ in its closure

\parni$\mathbf i$ the type of a minimal gallery from $C^-_{p}$ to $C_{-\qx}$

\parni $\g Q$ a positive (local) chamber in $\A$ (with vertex $p$) containing $\eta$  {in its closure

\noindent \begin{minipage}[l]{7cm}  
\medskip

In the picture, everything not in dotted lines is in $\A$.
\medskip
\par One writes $\qr=\qr_{\A,C^-_{p}}$ (\resp $\qr_{\g Q}=\qr_{\A,\g Q}$) the retraction with center $C^-_{p}$ (\resp $\g Q$) and image $\A$ ($=\sht_{p}(\A)$) defined on $\SHI$.
\medskip
\par One asks that $\qx,\eta$ are generated by vectors in $W\zv.\ql$ for $\ql$ a dominating vector in $\A$ (\ie $\ql\in \ov{C\zv_{f}}$).

\medskip
\par $W\zv_{p}$ is the subgroup of $W\zv$ generated by the $r_{\qb}$ for $\qb$ a thick root.
\bigskip
\end{minipage}
\begin{minipage}[r]{12cm}  

% le [h] sert ? positionner la figure ? l'endroit d'insertion, sinon [t] pour haut de page, [b] pour bas de page et [p] pour page s?par?e

\includegraphics[width=12cm]{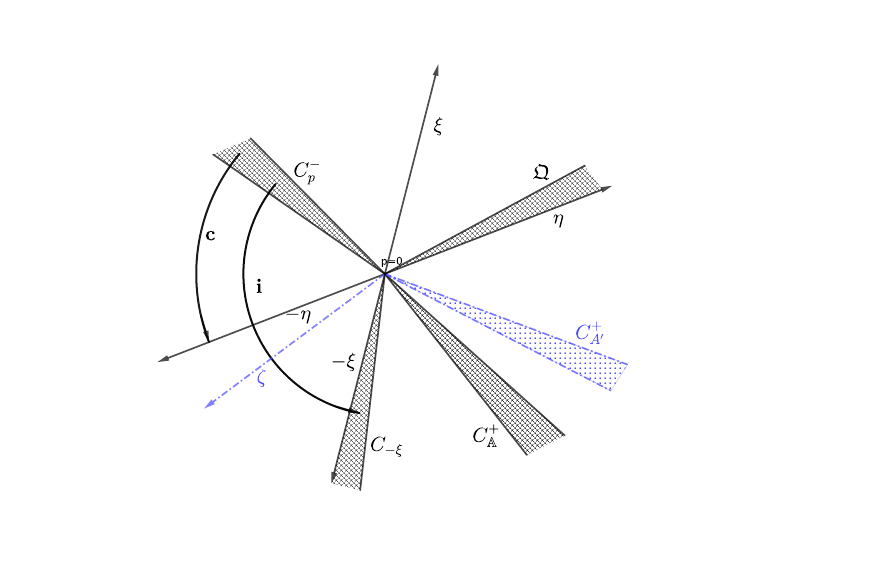}
  
  %\label{fig1.2.2}}

\end{minipage}

\begin{prop*} \cf \cite[4.6]{GR14}

\parni (1) The following conditions are equivalent:

\par (a) There exists an opposite $\qz$ to $\eta$ in $\SHI$ (with vertex $p$) such that $\qr(\qz)=-\qx$.

\par (b) There exists a gallery $\mathbf c$ of (local) chambers in $\A$ (with vertex $p$), of type $\mathbf i$ for some choice of $C_{-\qx}$, that is centrifugally folded with respect to $\g Q$ (in particular folded along thick walls) with first chamber $C^-_{p}$ and last chamber containing $-\eta$ in its closure.

\par (c) $\eta\leq_{W\zv_{p}} \qx$, \ie there exist $\qx_{0},\qx_{s}\in V\setminus\set0$ such that $\eta=[0,1)\qx_{s}$, $\qx=[0,1)\qx_{0}$ and a $W\zv_{p}-$chain from $\qx_{0}$ to $\qx_{s}$, \ie finite sequences $(\qx_{0},\qx_{1},\ldots,\qx_{s})$ of vectors in $V=\vect\A$ and $(\qb_{1},\ldots,\qb_{s})$ of (real) roots satisfying the following (for $1\leq i\leq s$):

\par\quad (i) $r_{\qb_{i}}(\qx_{i-1})=\qx_{i}$,

\par\quad (ii) $\qb_{i}(\qx_{i-1})<0$,

\par\quad (iii) $\ker\qb_{i}$ is a thick wall, \ie $\qb_{i}$ is a thick root (\ie $\qb_{i}(p)\in\Z$ for a masure),

\par\quad (iv) $\qb_{i}\in\QF^+=\QF^+(-C^-_{p})$, \ie $\qb_{i}(C^-_{p})<0$.

\medskip
\parni (2) If moreover $\mathbf i$ is minimal  (\ie $C_{-\qx}$ is the (local) chamber ``containing'' $-\qx$ nearest to $C_{p}^-$, \ie $C_{-\qx}=pr_{-\qx}(C_{p}^-)$, then the possible $\qz$ are in one to one correspondence with the disjoint union of the $\shc\zm_{\g Q}(\mathbf{c})=\set{\textrm{minimal galleries } \mathbf{m} \textrm{ with origin } C_{p}^- \textrm{ and type } \mathbf{i} \textrm{ with image } \mathbf{c} \textrm{ by } \qr_{\A,\g Q}}$, when $\mathbf c$ runs in the set $\QG^+_{\g Q}(\mathbf{i},-\eta)$ of galleries satisfying (b) above with this type $\mathbf i$ (fixed).
\end{prop*}

\begin{remas*} With these choices of signs, $\QF^+$ is of positive type, \ie the associated vectorial chamber $C^+_{\A}=-C^-_{p}$ is in the positive Tits cone $\sht$, but perhaps not equal to $C\zv_{f}$.

\par Contrary to \cite{GR14}, we do not suppose in (1) above that $\mathbf i$ is minimal. This gives more flexibility for applications.

\par We repeat below the main lines of the proofs in \cite{GR14} and \cite[6.1, 6.3]{GR08}. We give details of a proof of \cite[6.1]{GR08} independent of the existence of a strongly transitive group.
\end{remas*}

\noindent \textit{Proof.} {\bf (a) $\implies$ (b)} Let $\mathbf{m}=(C^-_{p}=M_{0},M_{1},\ldots,M_{r}\ni\qz)$ be a minimal gallery in $\SHI$ from $C^-_{p}$ to $\qz$.
Its retraction by $\qr$ is a minimal gallery  from $C^-_{p}$ to $-\qx$.
Hence, under the additional hypothesis of (2), one may suppose $\mathbf m$ of type $\mathbf i$ and then $\qz$ determines $\mathbf m$.
If one retracts now $\mathbf m$ into $\A$ by $\qr_{\g Q}$ (with center $\g Q$), one gets a gallery $\mathbf{c}=\qr_{\A,\g Q}(\mathbf{m})$ satisfying (b) (and of type $\mathbf i$, under the hypothesis of (2)).
This is a result of \cite[4.4]{GR14} which is independent of the existence of a strongly transitive group.

\medskip
\parni {\bf (b) $\implies$ (a)} If $\mathbf{c}=(C^-_{p}=C_{0},C_{1},\ldots,C_{r})$ satisfies (b), there exists a minimal gallery $\mathbf{m}=(C^-_{p}=C'_{0},C'_{1},\ldots,C'_{r})$ retracting by $\qr_{\g Q}=\qr_{\A,\g Q}$ onto $\mathbf{c}$, with the same type $\mathbf i$ (\cf \cite[4.4]{GR14}).
Let $\qz\subset\ov C'_{r}$ retracting by $\qr_{\g Q}$ on $-\eta\subset\ov C_{r}$; as $\eta\subset \ov{\g Q}$, this implies in particular that $\qz$ is opposite $\eta$.
As $\mathbf c$ and $\mathbf m$ are of type $\mathbf i$, one has $\qr(C'_{r})=C_{-\qx}$.
Hence $\qr(\qz)$ is in $\ov C_{-\qx}$ as $-\qx$.
Thus $\qr(\qz)=-\qx$, as they are both opposite $\eta$, up to a conjugation by $W\zv$.

\medskip
\parni  {\bf (2)} Under the hypothesis of (2), the $\qz$ are in one to one correspondence with the $\mathbf m$, which are exactly the galleries in $\coprod_{\mathbf{c}}\shc\zm_{\g Q}(\mathbf{c})$ as announced.

\medskip
\parni  {\bf (a) $\implies$ (c)} This generalizes \cite[Prop. 6.1]{GR08}, just taking $\pi_{+}=\eta,\pi_{-}=\qz,\qr\pi_{-}=-\qx$.

\par One considers [twin] apartments $A^0$ containing $\eta\cup \qz$, $A^+$ containing $C^-_{p}\cup \eta$ and $A^-$ containing $C^-_{p}\cup \qz$.
One defines $\qr_{-}=\qr_{A^-,C^-_{p}}$ (recall that $\qr=\qr_{\A,C^-_{p}}$).
But we shall first modify $A^-$ by the following Lemma.

\begin{lemm}\label{5.11} Let $\SHI=(\SHI^+,\SHI^-)$ be a twin building, $C^-$ a chamber in $\SHI^-$ and $A=(A^+,A^-)$ a [twin] apartment.
Then there exists a chamber $C^+$ in $A^+$ that is opposite $C^-$.
We write then $B=(B^+,B^-)$ the unique [twin] apartment containing $C^-$ and $C^+$.

\par If moreover $D$ is a chamber in $A^+$ (\resp $A^-$), one may choose $C^+$ in such a way that $D\subset B^+$ (\resp $D\subset B^-$).
\end{lemm}

\begin{NB} This Lemma seems well known when $\SHI$ is spherical, but we did not find a reference,
see \cite[2.2.11]{R22}. It is likely that this twin case is also already known.
\end{NB}

\begin{proof} One assumes first $D\subset A^+$.
We choose a [twin] apartment $A_{1}=(A_{1}^+,A_{1}^-)$ containing $C^-$ (in $A^-_{1}$) and $D$ (in $A^+_{1}$), and we write $C''=opp_{A_{1}}(C^-)\subset  A^+_{1}$.
As $D \subset  A^+_{1}$, with $A_{1}$ generated by $C^-$ and $C''$, one has $d^ {*\zzw}(D,C^-)=d\zw(D,C'')$ (see the Chasles relation (4) in \cite[5.173]{AB08}, as $d^ {*\zzw}(C^-,C'')=1$).

\par Let $C^+\subset A^+$ be the chamber such that $d\zw(D,C^+)=d^ {*\zzw}(D,C^-)$; this means that there exists in $A^+$ a minimal gallery $(C_{0}=D,\ldots,C_{s}=C^+)$ of type $\mathbf{i}=(i_{1},\ldots,i_{s})$, where $r_{i_{1}}.\ldots.r_{i_{s}}$ is a minimal decomposition of $d^ {*\zzw}(D,C^-)=d\zw(D,C'')$.
Let us prove that $C^+$ and $C^-$ are opposite.
One calculates $d^ {*\zzw}(C_{j},C^-)$ by induction on $j$: $d^ {*\zzw}(C_{0},C^-)=d\zw(D,C^+)=d^ {*\zzw}(D,C^-)=r_{i_{1}}.\ldots.r_{i_{s}}$.
One bets that $d^ {*\zzw}(C_{j},C^-)=r_{i_{j+1}}.\ldots.r_{i_{s}}$ (this will give $d^ {*\zzw}(C^+,C^-)=1$, qed).
But $d\zw(C_{j+1},C_{j})=r_{i_{j+1}}$ and, by induction hypothesis, $\ell(r_{i_{j+1}}d^ {*\zzw}(C_{j},C^-))=\ell(r_{i_{j+2}}.\ldots.r_{i_{s}})=\ell(d^ {*\zzw}(C_{j},C^-))-1$.
So, by the axiom (Tw2) in \cite[5.133 p. 266]{AB08}, one gets $d^ {*\zzw}(C_{j+1},C^-)=r_{i_{j+2}}.\ldots.r_{i_{s}}$ which concludes the induction.
One has now to prove that $D\subset B^+$.
But $d^ {*\zzw}(D,C^-)=d\zw(D,C^+)$; so this is a consequence of \cite[5.175 p. 278]{AB08}.

\par Let us now look at the case $D\subset A^-$.
We choose a [twin] apartment $A_{1}=(A_{1}^+,A_{1}^-)$ containing $C^-\cup D$ (in $A^-_{1}$) and write $C''=opp_{A_{1}}(C^-)\subset  A^+_{1}$.
The Chasles relation gives $d\zw(D,C^-)=d^ {*\zzw}(D,C'')$.
Let $C^+\subset A^+$ be  such that $d^ {*\zzw}(D,C^+)=d^ {\zzw}(D,C^-)$.
There is in $A^-_{1}$ a minimal gallery $(C_{0}=D,\ldots,C_{s}=C^-)$ of type $\mathbf{i}=(i_{1},\ldots,i_{s})$, where $r_{i_{1}}.\ldots.r_{i_{s}}$ is a minimal decomposition of $d^ {*\zzw}(D,C^+)=d^ {*\zzw}(D,C'')=d\zw(D,C^-)$.
Let us prove that $C^+$ and $C^-$ are opposite.
For this one  calculates $d^ {*\zzw}(C_{j},C^+)$ by induction on $j$: $d^ {*\zzw}(C_{0},C^+)=d^ {*\zzw}(D,C^+)=d\zw(D,C^-)=d^ {*\zzw}(D,C'')=r_{i_{1}}.\ldots.r_{i_{s}}$.
One bets that $d^ {*\zzw}(C_{j},C^+)=r_{i_{j+1}}.\ldots.r_{i_{s}}$ (this will give $d^ {*\zzw}(C^-,C^+)=1$, qed).
But $d\zw(C_{j+1},C_{j})=r_{i_{j+1}}$ and, by induction hypothesis, $\ell(r_{i_{j+1}}d^ {*\zzw}(C_{j},C^+))=\ell(r_{i_{j+2}}.\ldots.r_{i_{s}})=\ell(d^ {*\zzw}(C_{j},C^+))-1$.
So, by the axiom (Tw2), one gets $d^ {*\zzw}(C_{j+1},C^+)=r_{i_{j+2}}.\ldots.r_{i_{s}}$ and the induction is OK.
One has $d^ {*\zzw}(D,C^+)=d\zw(D,C^-)$, hence $D\subset B^-$ by \cite[5.175]{AB08}
\end{proof}

\subsection{End of proof of Proposition \ref{5.10}}\label{5.12}%%%%%%%%%%%%

\par \quad We no longer differentiate the two parts of a [twin] apartment by an exponent $\pm$.

\medskip
\parni  {\bf (a) $\implies$ (c)} We write $A^-_{1}$ the [twin] apartment $B$ of Lemma \ref{5.11} (obtained by setting $C^-:=C^-_{p}$, $A:=A^0$ and $\ov D\supset \qz$).
We shall replace $A^-$ by $A^-_{1}$ but not change $A^0$.
One has $A^-_{1}\supset C^-_{p}\cup\qz\cup C^+$ and $C^+_{A^-_{1}}:=C^+\subset A^0$ is opposite $C^-_{p}$ in $A^-_{1}$.
Recall that $\qr_{-}=\qr_{A^-,C^-_{p}}$ and $\qr=\qr_{\A,C^-_{p}}$.

\begin{rema*} In \cite[prop. 6.1]{GR08}, $C^+$ is written $C_{0}$ and $C^-_{p}=germ(\g s)$.
Both $A^-_{1}$ and $A^-$ contain $C^-_{p}$ and $\qz$, so they are isomorphic by an isomorphism $\qth^-:A^-_{1}\to A^-$ fixing $C^-_{p}$ and $\qz$.
  If one supposes $\qth^-$ induced by an automorphism $\qth^-$ of the twin building (\eg if there is a strongly transitive automorphism group, as in \cite{GR08}), one may define $A^1=\qth^-(A^0)$.
  This apartment contains $\qz$ and a segment germ $\eta^1=\qth^-(\eta)$ (opposite $\qz$) such that $\qr(\eta^1)=\qr(\eta)$ (as $C^-_{p}$ is fixed by $\qth^-$).
  So we are exactly in the situation of \cite{GR08}, second paragraph of the proof of 6.1 ($\eta^1$ is written $\pi^1_{+}$ there).
  
  \par In this proof of 6.1, one takes a minimal gallery $m=(c_{0},c_{1},\ldots,c_{n})$ in $A^1$ from $c_{0}=\qth^-(C^+)=C^+_{A^-}=opp_{A^-}(C^-_{p})$ to the opposite $\eta^1=\qth^-(\eta)$ of $\qz$.
  And then one takes its retraction $\qd=\qr_{-}(m)$.
  We shall replace $m$ by $m'=(\qth^-)^ {-1}(m)$, which is a minimal gallery in $A^0$ from $C^+_{A^-_{1}}=C^+=opp_{A^-_{1}}(C^-_{p})=(\qth^-)^ {-1}(C^+_{A^-})$ to $(\qth^-)^ {-1}(\eta^1)=\eta$.
  So  $\qd=\qr_{-}(m)= \qr_{-}(m')$ and this will avoid to suppose $\qth^-$ induced by an automorphism of $\SHI$.
\end{rema*}

\par  Back to the proof of  {\bf (a) $\implies$ (c)} (without assuming the existence of a strongly transitive group).

\par We assume $A^-=A^-_{1}$.
As in \cite{GR08} we consider $\qr_{-}$ instead of $\qr$; they are almost the same as $\qr_{-}=\qth\circ\qr$, if $\qth:\A\to A^-$ is the isomorphism fixing $C^-_{p}$.

\par By hypothesis, there are $w_{\pm}\in W\zv$, such that $\qx=[0,1)w_{-}\ql$ and $\eta=[0,1)w_{+}\ql$. We choose $w_{\pm}$ minimal for this property. 
Here we consider $C^+_{\A}=opp_{\A}(C^-_{p})$ as the fundamental vectorial chamber of $\A$, to precise the action of $W\zv$ on $\A$ and the relation $\ql\in\A$ (\ie $\ql\in \ov C^+_{\A}$).

\par In $A^0$ one considers a minimal gallery $m'=(c_{0},c_{1},\ldots,c_{n})$ of type $\mathbf{i}=(i_{1},\ldots,i_{n})$ between $c_{0}=C^+=C^+_{A^-_{1}}=C^+_{A^-}$ and $c_{n}\supset\eta$. The retracted gallery $\qd=\qr_{-}(m')=(c_{0},c'_{1}=\qr_{-}(c_{1}),\ldots,c'_{n}=\qr_{-}(c_{n}))$ in $A^-$ is centrifugally folded with respect to $C^-_{p}$.
It satisfies $c_{0}=\qr_{-}(C^+_{A^-})=C^+_{A^-}$ and $c'_{n}\supset\qr_{-}(\eta)$.

\par One has $\eta\in w_{-}C^+_{A^-}=w_{-}C^+$ in $A^0$, with $w_{-}$ minimal (to be precise, $W\zv$ is considered here as a group of automorphisms of $A^0$ by considering $C^+$ as its fundamental vectorial chamber):
Actually $\eta$ is opposite $\qz$ in $A^0\supset C^+$; so, using the isomorphism $\qth^0:A^0\to A^-$ fixing $\qz$ and $C^+$, it is sufficient to prove that the opposite $opp_{A^-}(\qz)$ of $\qz$ in $A^-$ is in $w_{-}C^+=w_{-}C^+_{A^-}$.
But $\qr$ induces the  isomorphism $\qth^ {-1}:A^-\to\A$ (fixing $C^-_{p}$) which sends $\qz$ onto $-\qx$ (by the hypothesis (a)) hence $opp_{A^-}(\qz)$ onto $\qx$ and $C^+_{A^-}$ onto $C^+_{\A}$.
As $\qx\in w_{-}C^+_{\A}$ with $w_{-}$ minimal, by the above definition of $w_{-}$, we are done.

\par From the definition of $m'$, one gets that $w_{-}=r_{i_{1}}.\ldots.r_{i_{n}}$ is a reduced decomposition.
Using once more the isomorphism $\qth^ {-1}:A^-\to\A$ (which sends $\qr_{-}(\eta)$ to $\qr(\eta)=\eta$ and $C^+_{A^-}$ to $C^+_{\A}$), one gets $\qr_{-}(\eta)\in w_{+}C^+_{A^-}$ with $w_{+}$ minimal.

\par In $A^0$, the chambers $c_{j}$ and $c_{j+1}$ are separated by a (thick or thin) wall $H^1_{j}$ and one writes $H_{j}$ the (thick or thin) wall in $A^-$ containing $\qr_{-}(H^1_{j}\cap \ov {c_{j}})=\qr_{-}(H^1_{j}\cap\ov{ c_{j+1}})$.
We denote by $j_{1},\ldots,j_{s}$ the indices such that $c'_{j}=\qr_{-}(c_{j})=\qr_{-}(c_{j+1})=c'_{j+1}$. Then, for all $k$, $H^1_{j_{k}}$ and $H_{j_{k}}$ are thick walls (it is a part of the definition of a centrifugally folded gallery).
One writes $\qb_{k}\in\QF^+$ the positive root such that $H_{j_{k}}$ has direction $\ker\qb_{k}$ (here $\QF^+$ is defined as in condition (iv) of (c) but in $A^-$: $\qb\in\QF^+\iff \qb(C^-_{p})<0$).

\par Actually we get $\qd=\qr_{-}(m')$ from a minimal gallery $\qd^0=(c^0_{0}=c_{0}=C^+_{A^-},c^0_{1},\ldots,c^0_{n})=\qth^0(m')$, of type $\mathbf i$ in $A^-$ from $c_{0}$ to $c^0_{n}=w_{-}c_{0}\supset \qth^0(\eta)$, by applying successive foldings along the walls $H_{j_{1}},H_{j_{2}},\ldots,H_{j_{s}}$.
At each step one gets a gallery $\qd^k=(c^k_{0}=c_{0}=C^+_{A^-},c^k_{1},\ldots,c^k_{n})$, of type $\mathbf i$ in $A^-$, centrifugally folded with respect to $C^-_{p}$, which ends more and more closely to $c_{0}$.

\par One writes $\qx_{0}=w_{-}\ql\in c^0_{n}=w_{-}c_{0}\subset A^-$ %(generating $\qth^0(\eta)$ in $A^-$) 
and $\qx_{k}=r_{\qb_{k}}.\ldots.r_{\qb_{1}}.\qx_{0}\in c^k_{n}\subset A^-$.
In particular $\qx_{s}\in c^s_{n}=c'_{n}$ and $c'_{n}\supset\qr_{-}(\eta)$.
As $\qx_{s}\in W\zv\ql$ and $\eta$ is generated by a vector in $W\zv\ql$, one sees that $\qx_{s}$ generates this segment germ $\qr_{-}(\eta)\subset A^-$.
Similarly, we see that $\qx_{0}=w_{-}\ql$ generates $\qth^0(\eta)\subset A^-$.

\par Actually the isomorphism $\qth^ {-1}\circ\qth^0:A^0\to A^-\to \A$ sends $C^+=C^+_{A^-}$ onto $C^+_{\A}$ and $\eta$ onto $\qx$ (as we saw above that it sends $\qz$ onto $-\qx$): $\qr(\qth^0(\eta))=\qx$ in $\A$.
The isomorphism $\qth^ {-1}:A^-\to\A$ sends $\qth^0(\eta)$ onto $\qx$ and $\qr_{-}(\eta)$ onto $\eta$.
So the condition (c) we aim to prove is equivalent to the conditions (i,ii,iii,iv) for $(\qx_{0},\ldots,\qx_{s})$ and $(\qb_{1},\ldots,\qb_{s})$ in $A^-$.
Actually (i), (iii) (as $H_{j_{k}}$ is a thick wall) and (iv) are clearly satisfied.
Let us prove now (ii): $\qb_{k}(\qx_{k-1})<0$:

\par $\qd^0$ is a minimal gallery from $c_{0}=C^+_{A^-}$ to $w_{-}c_{0}\supset[0,1)w_{-}\ql$.
So, for any $j$, $c^0_{j+1},\ldots,c^0_{n}$ and $[0,1)w_{-}\ql$ are on the same side of the wall separating $c^0_{j}$ and $c^0_{j+1}$; in particular $(c^k_{j_{k}+1},\ldots,c^k_{n})$ is a minimal gallery, entirely on the same side of $H_{j_{k}}$ and $[0,1)\qx_{k}\not\subset H_{j_{k}}$.
But $c^k_{j_{k}}=\qr_{- }(c_{j_{k}})=\qr_{- }(c_{j_{k}+1})=c^k_{j_{k}+1}$ and, as we have centrifugal foldings (with respect to $C^-_{p}$, opposite $C^+_{A^-}=c_{0}$ in $A^-$), this chamber is on the positive side of the wall $H_{j_{k}}$ (with direction $\ker\qb_{k}$).
So $c^k_{j_{k}+1},\ldots,c^k_{n}$ are in this positive side; this means that $\qb_{k}(\qx_{k})>0$, \ie $\qb_{k}(\qx_{k-1})<0$. This proves that (a) $\implies$ (c).

\medskip
\parni  {\bf (c) $\implies$ (a)} This generalizes a part of \cite[th. 6.3]{GR08}.

\par We have $\qx$ (\resp $\eta$) generated by $\qx_{0}=w_{-}\ql$ (\resp $\qx_{s}=w_{+}\ql$) and $w\pm\in W\zv$ is chosen minimal for this property.
We write $w_{-}=r_{i_{1}}.\ldots.r_{i_{n}}$ a minimal decomposition (of type $\mathbf i$, the type of a minimal gallery from $C^-_{p}$ to $-\qx$, as in the hypothesis of (2)).

\par The segment germs $\qz$ such that $\qr(\qz)=-\qx$ are in bijection with the galleries of type $\mathbf i$ $m^-=(c^-_{0}=C^-_{p},c^-_{1},\ldots,c^-_{n})$ that are minimal (\ie non stammering) starting from $C^-_{p}$.
This bijection is given by the relation $\qz\subset \ov {c^-_{n}}$.
We have now to prove that we may choose $m^-$ in such a way that $\qz$ is opposite $\eta$.
This is for this that the $W\zv_{p}-$chain will be useful.

\par We write $\qd^0=(c^0_{0}=C^+_{\A},c^0_{1},\ldots,c^0_{n})$ the minimal gallery of type $\mathbf i$ in $\A$ starting from $C^+_{\A}$.
It is thus stretched from $C^+_{\A}$ to $\qx$ (generated by $\qx_{0}=w_{-}\ql$).
We shall first fold this gallery step after step, using the $W\zv_{p}-$chain:

\par As $\qb_{1}(\qx_{0})<0$ (by (ii)) and $\qb_{1}(C^+_{\A})>0$ (by (iv)), the wall  $\ker\qb_{1}$ (thick by (iii)) separates $c^0_{0}=C^+_{\A}$ from $c^0_{n}$: so it is the wall between two adjacent chambers $c^0_{j_{1}-1}$ and $c^0_{j_{1}}$ (actually here $j_{1}$ is well determined).
One writes $\qd^1=(c^1_{0}=c^0_{0}=C^+_{\A},c^1_{1}=c^0_{1},\ldots,c^1_{j_{1}-1}=c^0_{j_{1}-1},c^1_{j_{1}}=c^1_{j_{1}-1}=r_{\qb_{1}}c^0_{j_{1}},\ldots,c^1_{n}=r_{\qb_{1}}c^0_{n})$.
It is a gallery of type $\mathbf i$ and $c^1_{n}\supset r_{\qb_{1}}(\qx_{0})=\qx_{1}$ (by (i)).
But $\qb_{2}(\qx_{1})<0$ and $\qb_{2}(C^+_{\A})>0$, so the wall $\ker\qb_{2}$ separates $c^1_{0}=C^+_{\A}$ from $c^1_{n}$; it is the wall between two strictly adjacent chambers $c^1_{j_{2}-1}$ and $c^1_{j_{2}}$.
One writes $\qd^2=(c^2_{0}=c^1_{0}=c^0_{0}=C^+_{\A},c^2_{1}=c^1_{1},\ldots,c^2_{j_{2}-1}=c^1_{j_{2}-1},c^2_{j_{2}}=c^2_{j_{2}-1}=r_{\qb_{2}}c^1_{j_{2}},\ldots,c^2_{n}=r_{\qb_{2}}c^1_{n})$.
It is a gallery of type $\mathbf i$ and $c^2_{n}\supset r_{\qb_{2}}(\qx_{1})=\qx_{2}$.
But $\qb_{3}(\qx_{2})<0$ , \etc
At the end of the day, one gets a gallery $\qd^s=(c^s_{0}=c^0_{0}=C^+_{\A},c^s_{1},\ldots,c^s_{n})$ of type $\mathbf i$ in $\A$ starting from $C^+_{\A}$ and finishing in $c^s_{n}\supset \qx_{s}=w_{+}\ql$ (generating $\eta$).
This gallery is folded along thick walls (this is condition (iii)), but perhaps not centrifugally folded (with respect to $C^-_{p}$), contrary to what is written (too quickly) in \cite{GR08}, line 3 on page 2650.

\par To prove now that $\qz$ and $\eta$ are opposite segment germs, it is equivalent to prove that $c^s_{n}$ ($\supset\eta$) and $c^-_{n}$ ($\supset\qz$) are opposite chambers (as $\qz$ and $\eta$ are generated by vectors in $\pm W\zv\ql$).
For this we shall choose carefully the successive chambers $c^-_{i}$ and prove more than necessary: by induction on $j$, $c^-_{j}$ and $c^s_{j}$ are opposite for $0\leq j\leq n$; this is true for $j=0$.
Let us suppose  $c^-_{j-1}$ and $c^s_{j-1}$ opposite.
Then  $c^s_{j}$ is adjacent to $c^s_{j-1}$ (\resp $c^-_{j}$ has to be strictly adjacent to $c^-_{j-1}$) along a panel (in the unrestricted sense) of type $i_{j}$.
If the wall containing this panel is thin, then  $c^s_{j}$ and $c^s_{j-1}$ (\resp $c^-_{j}$ and $c^-_{j-1}$) are in the same apartments and $c^s_{j} \neq c^s_{j-1}$ (\resp $c^-_{j} \neq c^-_{j-1}$) so $c^-_{j}$ and $c^s_{j}$ are automatically opposite.
If, on the contrary this wall is thick, then (from the theory of twin buildings, see \eg \cite[5.139 and 5.134]{AB08}) one knows that all chambers adjacent (or equal) to $c^-_{j-1}$ (along a panel of type $i_{j}$) except exactly one, are opposite $c^s_{j}$.
As the wall is thick, we can always choose $c^-_{j}$ opposite $c^s_{j}$ and strictly adjacent to $c^-_{j-1}$.   \qed

\subsection{$C_{\infty}-$Hecke paths}\label{5.13}%%%%%%%%%%%%

\par We consider, as before \S \ref{5.10}, a thick masure $\SHI$ and a (canonical) apartment $\A$ considered as a vector space with origin $0=0_{\A}$.
It is endowed with a Weyl group $W\zv$, a  root system $\QF$ (in $\A^*$), a fundamental vectorial chamber $C\zv_{f}$ and a Tits cone $\sht=W\zv.C\zv_{f}$. %\marge{\cf d\'ebut 4.4}
    We   consider a spherical dominant or antidominant vector $\ql\in \qe(\ov C\zv_{f}\cap\sht^\circ)$. 

\par Recall the definition and properties of $\ql-$paths from \S \ref{5.2}.3, \S \ref{5.4}.1 and Lemma \ref{5.4}.2.

\par We consider now the case where $\SHI={\SHI_{\oplus}}$ is the positive part of a twin masure and $\A={\A_{\oplus}}$ is the canonical twin apartment.
So $\A_{\ominus}$ contains the fundamental negative local chamber $C_{\infty}$.
For any $p\in\A_{\oplus}$ satisfying $p \stackrel{\circ}{>}0$ or $p\leq0$, we defined in \S\ref{5.1}.3 the local chamber $C^\infty_{p}=pr_{p}(C_{\infty})$; its sign is $+$ if $p \stackrel{\circ}{>}0$ (\ie $p\in\sht^\circ$) and $-$ if $p\leq0$ (\ie $p\in -\sht$).

\par We suppose the origin $\pi(0)$ of $\pi$ in $\qe\sht$.
By the choice of $\ql$, we have $\pi(t)\in\qe\sht^\circ$, for any $t\in]0,1]$.
So $C^\infty_{\pi(t)}$ is well defined and of sign $\qe$ for $t>0$.

\begin{defi*} Such a $\ql-$path $\pi$ is called a $C_{\infty}-$Hecke\index{C@$C_{\infty}-$Hecke} path of type $\ql$ (with sign $\qe$) if, for any $0<t<1$, the left and right derivatives $\pi'_{\pm}(t)\in\qe\sht^\circ$ at $p=\pi(t)$ satisfy $\pi'_{-}(t)\leq_{W\zv_{p},C^\infty_{p}} \pi'_{+}(t)$, which means that there is a $(W\zv_{p},C^\infty_{p})-$chain from $\pi'_{+}(t)$ to $\pi'_{-}(t)$, \ie finite sequences $(\qx_{0}=\pi'_{+}(t),\qx_{1},\ldots,\qx_{s}=\pi'_{-}(t))$ of vectors in $\A$ and $(\qb_{1},\ldots,\qb_{s})$ of real roots (in $\A^*$) such that, for all $i=1,\ldots,s$,

\par (i) $r_{\qb_{i}}(\qx_{i-1})=\qx_{i}$,

\par (ii) $\qb_{i}(\qx_{i-1})<0$,

\par (iii) $\qb_{i}(p)\in\Z$ (\ie $\exists$ a  wall of direction $\ker\qb_{i}$ containing $p$),

\par (iv) $\qb_{i}\in\QF^+(C^\infty_{p})$, \ie $\qb_{i}(C^\infty_{p})>\qb_{i}(p)$.

\end{defi*}

\begin{remas*} 1) When $p$ is not a folding point of $\pi$ (\ie $\pi'_{-}(t)=\pi'_{+}(t)$), the above conditions (i) to (iv) are always fulfilled with $s=0$.

\par 2) $W\zv_{p}$ is the subgroup of $W\zv$ generated by the $r_{\qb}$, for $\qb\in\QF$ and $\qb(p)\in\Z$.

\par 3) The condition (iv), more precisely the definition of $\QF^+(C^\infty_{p})$, is opposite the definitions in \cite[1.8 (iv) and 1.8 (2)]{GR14}, \cite[3.3]{BPGR16} or \cite[2.5]{BPR19}.
Actually in these references the analogue of $C^\infty_{p}$ (which determines locally the investigated retraction) is naturally of negative sign.
In our case $C^\infty_{p}$ is of positive sign for $\qe=1$ and negative sign for $\qe=-1$; this is one of the reasons for our choice of definition of $\QF^+(C^\infty_{p})$.
Notice that this $\QF^+(C^\infty_{p})$ will be also opposite the $\QF^+$ of Proposition \ref{5.10} 1.c (iv), when we shall use this proposition.

\par 4) We write $C^ {\infty\zzv}_{p}\subset \A$ the vectorial chamber (of sign $\qe$) which is the direction of $C^\infty_{p}$.
We consider the linear action of $W\zv$ on $\A$ obtained by identifying $(\A,C\zv_{f})$ and $(\A,C^ {\infty\zzv}_{p})$.
As $\pi'_{\pm}(t)$ is also of sign $\qe$, there is $w_{\pm}(t)\in W\zv$ such that $\pi'_{\pm}(t)\in w_{\pm}(t).C^ {\infty\zzv}_{p}\subset V$; we actually choose $w_{\pm}(t)$ minimal for this property.
Then the condition $\pi'_{-}(t)\leq_{W\zv_{p},C^\infty_{p}} \pi'_{+}(t)$ implies $w_{-}(t)\leq w_{+}(t)$:

\par Actually one may define $\qs_{i}\in W\zv$ minimal such that $\qx_{i}\in\qs_{i}.C^ {\infty\zzv}_{p}$ (hence $w_{-}(t)=\qs_{s}$ and $w_{+}(t)=\qs_{0}$) and we prove now that $\qs_{i}\leq\qs_{i-1}$.
Clearly $\qx_{i}\in r_{\qb_{i}}\qs_{i-1}.C^ {\infty\zzv}_{p}$, so $\qs_{i}\leq r_{\qb_{i}}\qs_{i-1}$.
But $\qx_{i-1}\in \qs_{i-1}.C^ {\infty\zzv}_{p}$, $\qb_{i}(\qx_{i-1})<0$ and $\qb_{i}(C^ {\infty\zzv}_{p})>0$.
Therefore $C^ {\infty\zzv}_{p}$ and $\qs_{i-1}.C^ {\infty\zzv}_{p}$ are on opposite sides of the wall $\ker\qb_{i}$.
This proves that $\ell(r_{\qb_{i}}\qs_{i-1})<\ell(\qs_{i-1})$ and $r_{\qb_{i}}\qs_{i-1}\leq \qs_{i-1}$.

\par 5) The relation $\pi'_{-}(t)\leq_{W\zv_{p},C^\infty_{p}} \pi'_{+}(t)$ is also opposite the relation appearing in the above references \cite{GR14}, \cite{BPGR16} or \cite{BPR19}.
 This is really a new phenomenon.
We saw in remark (4) above that it implies $w_{-}(t)\leq w_{+}(t)$.
This reminds us the relation $w^-_{i}\leq w^+_{i}$ of Proposition \ref{5.8}, but it is opposite the relation in \cite[5.4]{GR08}.

\par One may note that, in this reference the definition of $w_{\pm}(t)$ compares classically $\pi'_{\pm}(t)$ with the fundamental vectorial chamber $C\zv_{f}$ (which is opposite the analogue of $C^\infty_{p}$), while the definition above compares it with $C^\infty_{p}$ (which is seldom of direction $C\zv_{f}$).
\end{remas*}

\subsection{$C_{\infty}-$Hecke paths as retractions of $C_{\infty}-$ friendly line segments}\label{5.14}%%%%%%%%%%%%

\par A line segment $[x,y]$ in $\SHI_{\oplus}$ is said $\qe-C_{\infty}-$friendly\index{C@$C_{\infty}-$friendly} if it is $C_{\infty}-$friendly in the sense of \ref{5.2}.1 with $x\stackrel{\circ}{<}y$ if $\qe=+1$ (\resp $x\stackrel{\circ}{>}y$ if $\qe=-1$) and moreover, in a twin apartment $({A^\ominus_{0}},{A^\oplus_{0}})$ containing $C_{\infty}\subset{A^\ominus_{0}}$ and $x\in{A^\oplus_{0}}$ (or even $[x,y)\subset{A^\oplus_{0}}$) one has $x\geq {0_{A^\oplus_{0}}}$ (\resp $x\leq  {0_{A^\oplus_{0}}}$), where $ {0_{A^\oplus_{0}}}$ is the element in ${A^\oplus_{0}}$ opposite the vertex of $C_{\infty}$.

\par We write $\qf:[0,1]\to[x,y]$ an affine parametrization of $[x,y]$, with $x=\qf(0),y=\qf(1)$ and $\ql=d\zv(x,y)$; actually $\qe\ql$ is in the interior of the Tits cone and in $\ov C\zv_{f}$.

\par By definition, the retraction $\qr_{C_{\infty}}$ (of a part of $\SHI_{\oplus}$ into $\A_{\oplus}$, with center $C_{\infty}$) is defined on $[x,y]$ and we saw in \S \ref{5.2}.2 that the image $\pi=\qr_{C_{\infty}}\circ\qf$ of $[x,y]$ by $\qr_{C_{\infty}}$ is a $\ql-$path.

\begin{prop*} 1) Let $[x,y]\subset{\SHI_{\oplus}}$ be an $\qe-C_{\infty}-$friendly line segment and $\ql=d\zv(x,y)$, then its image $\pi$ by $\qr_{C_{\infty}}$ is a $C_{\infty}-$Hecke path of type $\ql$ (with sign $\qe$).

\par 2) Conversely let $\pi$ be a $C_{\infty}-$Hecke path of type $\ql$ (with sign $\qe$) in $\A_{\oplus}$ with origin $p_{0}\geq{0_{\A_{\oplus}}}$ (\resp $p_{0}\leq{0_{\A_{\oplus}}}$) if $\qe=+1$ (\resp $\qe=-1$) and $x\in{\SHI_{\oplus}}$ be such that $(C_{\infty},x)$ is twin friendly and $\qr_{C_{\infty}}(x)=p_{0}$, then there is  an $\qe-C_{\infty}-$friendly line segment $[x,y]$ such that $\pi=\qr_{C_{\infty}}([x,y])$; moreover $\ql=d\zv(x,y)$.
\end{prop*}

\begin{proof} We consider first the case $\qe=-1$.

\par (1) Clearly $p_{0}=\pi(0)$ satisfies $p_{0}\leq{0_{\A_{\oplus}}}$, \ie $p_{0}\in-\sht$.
For any $t\in]0,1[$ we write $p=\pi(t)$; we have now to find a $(W\zv_{p},C^\infty_{p})-$chain from $\pi'_{+}(t)$ to $\pi'_{-}(t)$.
For this we use Proposition \ref{5.10} in the tangent space $\sht_{p}(\SHI_{\oplus})$ and we change first $\A$ in order that it contains $C_{\infty}$ and $[\qf(t),x)$: this does not change $\pi$, up to an isomorphism which is a restriction of $\qr_{C_{\infty}}$.
We then have $p=\pi(t)=\qf(t)$.
For the chamber $C^-_{p}$ we take the negative chamber $C^\infty_{p}$ of \S \ref{5.1}.3 (we identify local chambers at $p$ and chambers in $\sht_{p}(\SHI_{\oplus})$).
For $\qz$ we take the negative segment germ $\qf_{+}(t)=[p,y)$.
For $\eta$ we take  the positive segment germ $\pi_{-}(t)=\qf_{-}(t)=p-[0,1)\pi'_{-}(t)\subset\A_{\oplus}$ (we identify segment germs of origin $p$ in $\SHI_{\oplus}$ and segment germs of origin $0$ in $\sht_{p}(\SHI_{\oplus})$).
And for $-\qx$ we take  the negative segment germ $\pi_{+}(t)=p+[0,1)\pi'_{+}(t)\subset\A_{\oplus}$ (so $\qx=p-[0,1)\pi'_{+}(t)\subset\A_{\oplus}$ is a positive segment germ).

\par We have $\qr_{C_{\infty}}(\qz)=-\qx$, \ie $\qr(\qz)=-\qx$, as the restriction of $\qr_{C_{\infty}}$ to $\sht_{p}(\SHI_{\oplus})$ is $\qr=\qr_{\A,C^\infty_{p}}$ (see Lemma \ref{5.1}.4).
We are exactly in the situation of Proposition \ref{5.10} 1.a, except that the $\ql$ in \lc is our $-\ql\in-\qe\sht^\circ=\sht^\circ$: $\eta$ (\resp $\qx$) is generated by $-\pi'_{-}(t)\in-W\zv\ql$ (\resp $-\pi'_{+}(t)\in-W\zv\ql$).
From (a) $\implies$ (c) in this proposition, we get $\eta\leq\qx$, or more precisely sequences $(\qx'_{0}=-\pi'_{+}(t),\qx'_{1},\ldots,\qx'_{s}=-\pi'_{-}(t))$  and $(\qb'_{1},\ldots,\qb'_{s})$ satisfying the conditions (i) to (iv) of Proposition \ref{5.10} 1.c.
Considering the sequences $(\qx_{0}=-\qx'_{0}=\pi'_{+}(t),\qx_{1}=-\qx'_{1},\ldots,\qx_{s}=-\qx'_{s}=\pi'_{-}(t))$  and $(\qb_{1}=-\qb'_{1},\ldots,\qb_{s}=-\qb'_{s})$, we get the expected $(W\zv_{p},C^\infty_{p})-$chain, as the $\QF^+(C^\infty_{p})$ of \ref{5.13} (iv) is opposite the $\QF^+$ of Proposition \ref{5.10} 1.c.iv.

\par (2) Now  $\pi$ is a $C_{\infty}-$Hecke path of shape $\ql$ (with sign $\qe=-1$) in $\A_{\oplus}$ with origin $p_{0}\in\A_{\oplus}$ and $x\in\SHI_{\oplus}$ satisfies $\qr_{C_{\infty}}(x)=p_{0}$.
By definition there is a subdivision $0=t_{0}<t_{1}<\cdots<t_{\ell_{\pi}}=1$ of $[0,1]$ such that $\pi([0,1])=[p_{0},p_{1}]\cup[p_{1},p_{2}]\cup\cdots\cup[p_{\ell_{\pi}-1},p_{\ell_{\pi}}]$, if we write $p_{i}=\pi(t_{i})$.
We take a twin apartment $A_{0}$ containing $C_{\infty}$ and $x$, then $\qr_{C_{\infty}}\vert_{A_{0}}$ is an isomorphism of $A_{0}$ onto $\A$ fixing $C_{\infty}$ and sending $x$ to $p_{0}$; so $x\leq{0_{\A_{0}^{\oplus}}}$ as expected.
We shall prove by induction that, for $i\geq1$, there is a $(-1)-C_{\infty}-$friendly line segment $[x,z_{i}]$ such that $\qr_{C_{\infty}}([x,z_{i}])=\pi([0,t_{i}])$.
We define $[x,z_{1}]=(\qr_{C_{\infty}}\vert_{A^\oplus_{0}})^ {-1}([p_{0},p_{1}])$, it is a solution for $i=1$.
We assume now the result for $i$ and prove it for $i+1$.
Up to an isomorphism, we may assume $\A\supset C_{\infty}\cup [z_{i},x)$.
Let $p:=p_{i}=z_{i}$, we get the situation of Proposition \ref{5.10}, by setting $C^-_{p}:=C^\infty_{p}$, $\eta:=[z_{i},x)=[p_{i},p_{i-1})$, $-\qx=[p_{i},p_{i+1})$.
The condition (c) of \lc is fulfilled (see above in (1) the translation between chains).
So the implication (c) $\implies$ (a) provides us a segment germ $\qz$ opposite $\eta$ with origin $z_{i}$ satisfying $\qr_{C_{\infty}}(\qz)=-\qx=[p_{i},p_{i+1})$.
We write $A_{i}$ a twin apartment containing $C_{\infty}$ and $\qz$.
Then  $\qr_{C_{\infty}}\vert_{A_{i}}$ is an isomorphism from $A_{i}$ onto $\A$ fixing $C_{\infty}$ and we define $[z_{i},z_{i+1}]=(\qr_{C_{\infty}}\vert_{A^\oplus_{i}})^ {-1}([p_{i},p_{i+1}])$.
We have $\qr_{C_{\infty}}([x,z_{i}]\cup[z_{i},z_{i+1}])=\pi([0,t_{i+1}])$.
But $[z_{i},x)=\eta$ and $[z_{i},z_{i+1})=\qz$ are opposite.
So $[x,z_{i}]\cup[z_{i},z_{i+1}]$ is a line segment by \cite[4.9]{GR14} and we are done.

\medskip
\par We deal now with the case $\qe=+1$.

\par (1) As above $p_{0}=\pi(0)$ satisfies $p_{0}\geq{0_{\A_{\oplus}}}$, \ie $p_{0}\in\sht$.
For any $t\in]0,1[$ we write $p=\pi(t)$; we have now to find a $(W\zv_{p},C^\infty_{p})-$chain from $\pi'_{+}(t)$ to $\pi'_{-}(t)$.
We want to use Proposition \ref{5.10} in $\sht_{p}(\SHI_{\oplus})$, but now $C^\infty_{p}$ is a positive local chamber in $\A_{\oplus}$.
Luckily the signs in Proposition \ref{5.10} are not important, as \eg $\QF^+$ is defined in 1.c.iv by reference to $C^-_{p}$, not to the signs in $\A_{\oplus}$. The important fact is that $\qx$ and $\eta$ (\resp $C^-_{p},\qz$ and $-\qx$) are of the same sign.
We change first $\A$ in order that it contains $C_{\infty}$ and $[\qf(t),x)$, so $p=\pi(t)=\qf(t)$.
We take now $C^-_{p}:=C^\infty_{p}$, $\qz:=\qf_{+}(t)$, $\eta:=\pi_{-}(t)=\qf_{-}(t)=p-[0,1)\pi'_{-}(t)$ and $-\qx=\pi_{+}(t)=p+[0,1)\pi'_{+}(t)$, so $\qx=p-[0,1)\pi'_{+}(t)$.
We have $\qr_{C_{\infty}}(\qz)=-\qx$ and we are exactly in the situation of Proposition \ref{5.10} 1.a, except for the signs; in particular $\eta$ (\resp $\qx$) is generated by $-\pi'_{-}(t)\in-W\zv\ql$ (\resp  $-\pi'_{+}(t)\in-W\zv\ql$), so the $\ql$ in \lc is our $-\ql$.
From (a) $\implies$ (c) in this proposition, we get $\eta\leq\qx$ which seems to mean $-\pi'_{-}(t)\leq -\pi'_{+}(t)$.
By the same trick as above for the case $\qe=-1$, we get the expected $(W\zv_{p},C^\infty_{p})-$chain from $\pi'_{+}(t)$ to $\pi'_{-}(t)$.

\par The proof of the converse result (2) is the same, mutatis mutandis, as the one given above in the case $\qe=-1$.
\end{proof}

\subsection{Consequences}\label{5.15}%%%%%%%%%%%%

\par (1) We considered in \S \ref{5.3}.2, $C_{\infty}-$friendly line segments $[x,y]$ which were actually $\qe-C_{\infty}-$friendly.
We endowed them with a decoration.
Then $\pi=\qr_{C_{\infty}}([x,y])\subset \A_{\oplus}$ is endowed with a superdecoration (\S \ref{5.3}.3) which makes it a  superdecorated $C_{\infty}-\ql$ path  (see \S \ref{5.4}.5 and \S \ref{5.4}.6).
Conversely we proved in Theorem \ref{5.7} that a superdecorated $C_{\infty}-\ql$ path is the image by $\qr_{C_{\infty}}$ of a $C_{\infty}-$friendly line segment.

\par Comparing with the above Proposition \ref{5.14}, we get that:

\par (a) The underlying path of a super-decorated $C_{\infty}-\ql$ path  is a $C_{\infty}-$Hecke path.

\par (b) Any $C_{\infty}-$Hecke path $\pi\subset\A$ may be endowed with a super-decoration (provided that $\pi(0) \stackrel{\circ}{\geq}  0_{\oplus}$ or $\pi(0) \stackrel{\circ}{\leq}  0_{\oplus}$).

\par (c) The number of these possible super-decorations is finite (see \S \ref{5.4}.6).

\par Actually the consideration of (super-)decorations is useful to count the number of line segments with a given $C_{\infty}-$Hecke path as image under $\qr_{C_{\infty}}$ (see \ref{5.7}).
But the definition we gave of a super-decoration is perhaps too precise.
Other choices of the decorations $C^\pm_{t,\pi}$ may be interesting, \eg to compare with Muthiah's results in \cite{Mu19b}.

\begin{NB} The reader should note that a decorated $C_{\infty}-\ql$ path cannot always be endowed with a super-decoration.
One should, at least, assume the condition $C^+_{t,\pi}=pr_{\pi_{+}(t)}(C^-_{t,\pi})$, when $t$ is not among the $t_{i}$ of Lemma \ref{5.4}.4.
See analogously Proposition 2.7 N.B. and Remark (3) in \S 3.3 of \cite{BPR19}.

\medskip
\par (2) We indicated in \S \ref{n3.15.1} that our main motivation, according to Muthiah's goals, was to calculate the cardinality of sets of the form $(K_{twin}\qp^{-\ql} K_{twin}\cap I_{\infty}\qp^{-\qm} K_{twin})/K_{twin}$ for $\ql,\qm\in \qe(\ov{C\zv_{f}}\cap\sht^\circ\cap Y)\subset\A_{\oplus}$.
Such a set is in one to one correspondence with the set of points $x\in\SHI_{\oplus}$ such that $d\zv(0_{\oplus},x)=\ql$ and $\qr_{C_{\infty}}(x)$ is defined and equal to $\qm$.
Due to the lack of a Birkhoff decomposition, we are only able to calculate the cardinality of a subset: the set of the $x$ as above such that, moreover, $\qr_{C_{\infty}}(z)$ is defined for any $z\in[0_{\oplus},x]$. 
The formula we get for this cardinality is as follows:
it is the sum of the numbers $\#\set{[x,y]}$ in Theorem \ref{5.7}.2, where the sum runs on the set of all superdecorated $C_{\infty}-\ql$ paths in $\A_{\oplus}$ of shape $\ql$ from $0_{\oplus}$ to $\qm$ (with the type $\mathbf{i}_{t}$ fixed for any $t\in]0,1[$).
One can notice that this set of paths depends only on $\A_{\oplus}$, $\ql$ and $\qm$ (not of $\k$) and that it is finite, at least if the root system of $\g G$ is untwisted affine of type $A$, $D$ or $E$ (see (1.c)  above and the result \ref{n3.15.1}.4 of D. Muthiah).
So (in the case of $A,D,E$ with $\ql\in -(\ov{C\zv_{f}}\cap\sht^\circ\cap Y)$) this cardinality is a well defined polynomial in the cardinality $q$ of $\k$, depending only on $\A_{\oplus}$, $\ql$ and $\qm$.
\end{NB}

%%%%%%%%%%%%%%%%%%%%%%%%%%%%%%%%%%%%%%%%%%%%%%
\section{The case of affine $\mathrm{SL}_2$: counter-example to the Birkhoff decomposition and examples of Hecke paths}\label{s6}
In this section, we begin by proving  that when $G$ is affine $\mathrm{SL}_2$ over $\k(\qp)$, the Birkhoff decomposition does not hold, that is $G_{twin}\not\subset I_\infty N K$. Actually, many Kac-Moody groups over $\k(\qp)$ can be considered as affine $\mathrm{SL}_2$ over $\k(\qp)$: we will work with $G^\loo$, $G$ and $\widetilde{G}$, which are respectively $\mathrm{SL}_2(\k(\qp){[u,u^ {-1}]})$, $G^\loo\rtimes \k(\qp)^*$ and a central extension of $G$. Their maximal tori have dimensions $1$, $2$ and $3$. In $G^\loo$, neither the simple coroots nor the simple roots are free, in $G$ the simple coroots are not free but the simple roots are free and in $\widetilde{G}$ both simple roots and simple coroots are free so that $G$ and $\widetilde{G}$ fulfil the assumptions  of \ref{1.2}.  To prove that the Birkhoff decomposition does not hold, we work in $G^\loo$, in which the computations are easier, and then deduce the results for $G$ and $\widetilde{G}$.

We exhibit an element of $G_{twin}\setminus I_\infty N K$. Our element lies in $G\setminus (G^+_{\oplus}\cup G^-_{\oplus})$, where the index $\oplus$ means that $G^+$ and $G^-$ are defined with respect to $\shi_{\oplus}$. 
 This suggests that we need to work in $G^+_{\oplus}$ or $G^-_{\oplus}$ to obtain a Birkhoff decomposition {(see \ref{4.4})}. This was expected, since this is already the case for the Cartan decomposition

\par We end this section with some examples of Hecke paths associated with $G$.

\subsection{Notation and projection of $\widetilde{G}$ on $G$}

We begin by defining $\widetilde{G}$, which is a central extension of  $\mathrm{SL}_2\left(\k(\qp)[u,u^{-1}]\right)\rtimes \k(\qp)^*$, by defining a root generating system, in the sense of Bardy-Panse \cite{Ba96}.  Let $\widetilde{Y}=\Z\aleph^\vee\oplus \Z c\oplus \Z d$, where $\aleph^\vee,c,d$ are some symbols, corresponding to the positive root of $\mathrm{SL}_2(\k(\qp))$, to  the central extension and to the semi-direct extension by $\k(\qp)^*$ respectively. Let $\widetilde{X}=\Z \aleph\oplus \Z \delta\oplus \Z\Lambda_0$, where $\aleph$\index{a@$\aleph$}, $\delta$\index{d@$\delta$}, $\Lambda_0$\index{L@$\Lambda_0$} $:\widetilde{Y}\rightarrow \Z$ are the $\Z$-module morphisms defined by $\aleph(\aleph^\vee)=2$, $\aleph(c)=\aleph(d)=0$,  $\delta(\aleph^\vee)=0=\delta(c)$,  $\delta(d)=1$, $\Lambda_0(c)=1$ and  $\Lambda_0(\aleph^\vee)=\Lambda_0(d)=0$.  Let $\widetilde{\alpha_0}=\delta-\aleph$, $\widetilde{\alpha_1}=\aleph$, $\widetilde{\alpha_0^\vee}=c-\aleph^\vee$ and $\widetilde{\alpha_1^\vee}=\aleph^\vee$. Then $\widetilde{\mathcal{S}}=(\left(\begin{smallmatrix} 2 & -2\\ -2 & 2\end{smallmatrix}\right),\widetilde{X},\widetilde{Y},\{\widetilde{\alpha_0},\widetilde{\alpha_1}\},\{\widetilde{\alpha_0^\vee},\widetilde{\alpha_1^\vee}\})$ is a root generating system. Let $\widetilde{G}$ be the  Kac-Moody group associated with $\widetilde{\mathcal{S}}$ over $\k(\qp)$. Then by \cite[13]{Kum02} and \cite[7.6]{Mar18}, $\widetilde{G}$ is a central extension of $G:=\mathrm{SL}_2\left(\k(\qp)[u,u^{-1}]\right)\rtimes \k(\qp)^*$, where $u$ is an indeterminate and if $(M,z),(M_1,z_1)\in G$, with $M=\begin{psmallmatrix}
a(\qp,u) & b(\qp,u)\\
c(\qp,u) & d(\qp,u)
\end{psmallmatrix}$, $M_1=\begin{psmallmatrix}
a_1(\qp,u) & b_1(\qp,u)\\ c_1(\qp,u) & d_1(\qp,u)
\end{psmallmatrix}$,  we have \begin{equation}\label{eqDef_semi_direct_product}
(M,z).(M_1,z_1)=(M\begin{psmallmatrix}
a_1(\qp,z u) & b_1(\qp,zu)\\ c_1(\qp,z u) & d_1(\qp,z u)
\end{psmallmatrix},zz_1) .
\end{equation}

 Let $X=\Z\aleph\oplus \Z \delta$ and $Y=\Z\aleph^\vee\oplus \Z d$. We regard $X$ as a set of maps from $Y$ to $\Z$ by restricting them to $Y$. Let $\alpha_0=\delta-\aleph$, $\alpha_1=\aleph$, $\alpha_0^\vee=-\aleph^\vee$ and $\alpha_1^\vee=\aleph^\vee$. Then  $\mathcal{S}=(\left(\begin{smallmatrix} 2 & -2\\ -2 & 2\end{smallmatrix}\right),X,Y,\{\alpha_0,\alpha_1\},\{\alpha_0^\vee,\alpha_1^\vee\})$ is a root generating system and $G$ is the associated Kac-Moody group  (over $\k(\qp)$).

Note that the family $(\alpha_0^\vee,\alpha_1^\vee)$ is not free. We have $\Phi=\{\alpha+k\delta\mid\alpha\in \{\pm \aleph\}, k\in \Z\}$ and $(\alpha_0,\alpha_1)$ is a basis of this root system. We denote by $\Phi^+$ (resp. $\Phi^-$) the set $\Phi\cap (\N\alpha_0+\N\alpha_1)$ (resp $-\Phi_+$). For $k\in \Z$ and $y\in \k(\qp)$, we set $x_{\aleph+k\delta}(y)=(\begin{psmallmatrix}
1 & u^k y\\ 0 & 1
\end{psmallmatrix},1)\in G$ and $x_{-\aleph+k\delta}(y)=(\begin{psmallmatrix}
1 & 0 \\ u^k y& 1 
\end{psmallmatrix},1)\in G$.

\par The tori of $\widetilde G$, $G$ and $G^\loo$ are different (with respective dimensions $3$, $2$ and $1$).
On the contrary the maximal unipotent subgroups $\widetilde U^\pm$, $U^\pm$ and $U^\pm_{\loo}$ are naturally isomorphic \cite[1.9.2]{R16}.

We set $T=\{(\begin{psmallmatrix}
y & 0\\ 0 & y^{-1}
\end{psmallmatrix},z)\mid y,z\in \k(\qp)^*\}$. Then $T$ is a maximal split torus of $G$. Let $N$ be the normalizer of $T$ in $G$. We have  \[N=G\cap \Bigg( (\begin{psmallmatrix}
\k(\qp)^* u^\Z & 0 \\
0 & \k(\qp)^* u^\Z
\end{psmallmatrix},\k(\qp)^*)\sqcup (\begin{psmallmatrix} 0& \k(\qp)^* u^\Z\\
\k(\qp)^* u^\Z & 0
\end{psmallmatrix},\k(\qp)^*)\Bigg).\] 

Recall that $\sho=\k[\qp,\qp^{-1}]$. We have \[\mathfrak{N}(\sho)=G\cap\Bigg( ( \begin{psmallmatrix}
\k^* \qp^\Z u^\Z & 0 \\
0 & \k^*\qp^\Z u^\Z
\end{psmallmatrix},\k^* \qp^\Z)\sqcup ( \begin{psmallmatrix} 0& \k^*\qp^\Z u^\Z\\
\k^* \qp^\Z u^\Z & 0
\end{psmallmatrix},\k^*\qp^\Z)\Bigg)\] and 
$\mathfrak{U}(\sho):=\langle x_{\alpha+k\delta}(\sho)\mid\alpha\in \{\pm \aleph\},k\in \Z\rangle$, so that 
\[G_{twin}=\langle \mathfrak{N}(\sho),\mathfrak{U}(\sho)\rangle \subset \mathrm{SL}_2(\sho[u,u^{-1}])\rtimes \k^*\qp^\Z.\]

The group $G$ (resp $\widetilde{G}$) acts on the masures $\shi_{\oplus},\shi_{\ominus}$ (resp $\widetilde{\shi_{\oplus}},\widetilde{\shi_{\ominus}}$). 
We denote with a tilde the objects related to the masures $\widetilde{\shi}_\oplus$ and $\widetilde{\shi}_{\ominus}$ (for example the vertex $\widetilde{0_\oplus}$ and the local chamber $\widetilde{C_\infty}$).
 Let $K$ (resp. $\widetilde{K}$) be the fixator of $0_{\oplus}$  (resp. of $\widetilde{0_\oplus}$) in $G$ (resp. in $\widetilde{G}$) and $I_\infty$ (\resp $\widetilde I_\infty$) be the fixator of $C_\infty$ in $G_{twin}$ (\resp of $\widetilde C_\infty$ in $\widetilde G_{twin}$). Let $\g v\in \{\ominus,\oplus\}$. The standard apartment $\widetilde{\A_{\g v}}$ can be written as  $\A_{\g v}\oplus \R c$, where $c\in Y$ corresponds to the center, so that $\A_{\g v}$ can be considered as the quotient of $\widetilde{\A_{\g v}}$ by $\R c$. Let $\pi:\widetilde{G}\twoheadrightarrow G$\index{p@$\pi$} denote the natural projection and denote also by $\pi:\widetilde{\A}_{\g v}=\A_{\g v}\oplus \R c\twoheadrightarrow \A_{\g v}$ the natural projection. Then we have the following easy lemma.  

\begin{lemm}\label{lemQuotient_masure}
 The map $\pi:\widetilde{\A}_{\g v}\rightarrow \A_{\g v}$  uniquely  extends to a map $\pi:\widetilde{\shi}_{\g v}\rightarrow \shi_{\g v}$  such that $\pi(g.a)=\pi(g).\pi(a)$ for $g\in \widetilde{G}$, $a\in \widetilde{\A}_{\g v}$. In particular,  we can regard $\shi_{\g v}$ as a quotient of $\widetilde{\shi}_{\g v}$ by $\R c$. 
\end{lemm}

Let $\g v\in \{\ominus,\oplus\}$ and $f(\qp),g(\qp)\in \k(\qp)^*$ be  such that $\qo_{\g v}(f(\qp))=\qo_{\g v}(g(\qp))=0$. Let $\ell,n\in \Z$.  Then $\left(\begin{psmallmatrix} f(\qp) \qp^\ell  & 0 \\ 0 & f(\qp)^{-1} \qp^{-\ell} \end{psmallmatrix},g(\qp) \qp^n\right)$ acts on $\A_{\g v}$ by the translation of vector $-\mathrm{sgn}(\g v)(\ell \aleph^\vee+nd)$.

\par The kernel $C$ of $\pi:\widetilde G\to G$ is a one-dimensional split central torus (actually the reduced connected component of the center of $\widetilde G$, which is contained in $\widetilde T$), with cocharacter group $\Z.c\subset \widetilde Y$ (cocharacter group of $\widetilde T$).
So  there exists an isomorphism $T_C:\k(\qp)^*\rightarrow C$ such that $T_C(a)$ acts by the translations of vectors $-\omega_\oplus(a)c$ on $\widetilde \A_{\oplus}$ and $-\omega_\ominus(a)c$ on $\widetilde\A_{\ominus}$ (see \ref{ss_N} (2)).
We set $t_c=T_C(\qp^ {-1})\in \widetilde{\g T}(\k[\qp,\qp^ {-1}]) \subset \widetilde G_{twin}$.

\begin{lemm}
Let $i\in \widetilde{G}$ be such that $\pi(i)\in I_\infty$. Then  $i\in \widetilde I_{\infty}C\subset \widetilde{I}_\infty \widetilde{T}$. 
\end{lemm}

\begin{proof}
We have $\pi(i.\widetilde{0_{\ominus}})=0_\ominus$ and hence $i.\widetilde{0_{\ominus}}=0_{\ominus}+kc$, for some $k\in \Z$. Then $it_c^{k}.\widetilde{0_{\ominus}}=\widetilde{0_{\ominus}}$. Then $it_c^{k}.\widetilde{C_\infty}$ is a local chamber based at $\widetilde{0_\ominus}$ and we have $\pi(it_c^{k}.\widetilde{C_\infty})=C_\infty$. Therefore $it_c^{k}\in\widetilde I_{\infty}$ and $i\in  \widetilde I_{\infty}C\subset \widetilde{I}_\infty \widetilde{T}$.
\end{proof}

\begin{lemm}
Recall that $\widetilde{\she}=\widetilde{I_\infty}.\widetilde{\A}$ and that $\she=I_\infty.\A$. Then $\widetilde{\she}=\pi^{-1}(\she)$ and $\widetilde{\rho}_{\widetilde{C_\infty}}=\pi\circ \rho_{C_\infty}$. More precisely, let $g\in G$ and  $x\in \A$ be such that $g.x\in \she$. Let $\tilde{g}\in \pi^{-1}(g)$ and  $\tilde{x}\in \pi^{-1}(x)$. 
Then $\tilde{g}.\tilde{x} \in \widetilde{\she}$ and $\pi\circ\widetilde{\rho}_{\widetilde{C_\infty}}(\tilde{g}.\tilde{x})=\rho_{C_\infty}(g.x)$. 
\end{lemm}

\begin{proof}
Let $\tilde{x}\in\widetilde{\she}$. Write $\tilde{x}=\tilde{i}.\tilde{y}$, where $\tilde{i}\in \widetilde{I_\infty}$ and $\tilde{y}=\widetilde{\rho}_{\widetilde{C_\infty}(\tilde x)}\in \A$. Then $\pi(\tilde{x})=\pi(\tilde{i}).\pi(\tilde{y})\in I_\infty.\pi(\tilde{y})\subset \she$. Moreover, $\pi(\widetilde{\rho}_{\widetilde{C_\infty}}(\tilde{x}))=\pi(\tilde{y})=\rho_{C_\infty}(\pi(x))$. Conversely, take $x\in \she$. Write $x=i.y$, with $i\in I_\infty$ and $y\in \A$. Let $\tilde{x}\in \pi^{-1}(\{x\})$,  $\tilde{y}\in \pi^{-1}(\{y\})$ and $\tilde{i}\in \pi^{-1}(\{i\})$. Then $\pi(\tilde{i}.\tilde{y})=\pi(\tilde{x})$ and hence there exists $k\in \Z$ such that $\tilde{x}=(t_c)^k\tilde{i}.\tilde{y}=\tilde{i}(t_c)^k.\tilde{y}$. Therefore $\tilde{x}\in \widetilde{\she}$, which proves that $\widetilde{\she}=\pi^{-1}(\she)$. 

Take $g\in G$ and $x\in \A$ such that $g.x\in \she$. Let $\tilde{g}\in\pi^{-1}(\{g\})$ and $\tilde{x}\in \pi^{-1}(\{x\})$. Write $g.x=i.y$, with $i\in I_\infty$ and $y\in \A$. Take $\tilde{i}\in \pi^{-1}(\{i\})$ and $\tilde{y}\in \pi^{-1}(\{y\})$. 
Then $\pi(\tilde{g}.\tilde{x})=\pi(\tilde{i}.\tilde{y})$, so there exists $k\in \Z$ such that $\tilde{g}.\tilde{x}=\tilde i(t_c)^{k}.\tilde{y}\in \widetilde{\she}$. 
Therefore $\widetilde{\rho}_{\widetilde{C}_\infty}(\tilde{g}.\tilde{x})=(t_c)^k.\tilde{y}$, and the lemma follows. 
\end{proof}

In \ref{ss_Root_generating}, we defined actions of $W^v$ on $\widetilde{\A}$ and $\A$. We denote by $\tilde{.}$ the action of $W^v$ on $\tilde{\A}$. We have $w\tilde{.}x\in w.x+\R c$, for all $x\in \A\subset \widetilde{\A}$. 

\begin{lemm}\label{lem_Act_tldW}
Let $\tilde{\lambda}\in \overline{\widetilde{C^v_f}}$ and $\lambda=\pi(\tilde{\lambda})$. Let $v,w\in W^v$ be such that $v.\lambda=w.\lambda$. Then $v\tilde{.}\tilde{\lambda}=w\tilde{.}\tilde{\lambda}$. 
\end{lemm}

\begin{proof}
Let $i\in I=\{0,1\}$. We have $r_i.\lambda=\lambda-\alpha_i(\lambda)\alpha_i^\vee$ and $r_i\tilde{.}\lambda=\lambda-\widetilde{\alpha_i}(\lambda)\widetilde{\alpha_i}^\vee=\lambda-\alpha_i(\lambda)\widetilde{\alpha_i}^\vee$, with $\widetilde{\alpha_i}^\vee\in \alpha_i^\vee+\R c$. Moreover, $(W^v)\tilde{.}c=\{c\}$, so by induction on $\ell(w')$, we have $(w')\tilde{.}\lambda\in w'.\lambda+\R c$, for all $w'\in W^v$.

Write $\tilde{\lambda}=\lambda+tc$, with $t\in \R$. We have $v^{-1}w.\lambda=\lambda$ and therefore: \[v^{-1}w\tilde{.}\tilde{\lambda}=v^{-1}w\tilde{.}(\lambda+t c)=v^{-1}w\tilde{.}\lambda+t v^{-1}w\tilde{.}c=v^{-1}w\tilde{.}\lambda+tc\in v^{-1}w.\lambda+\R c=\lambda+\R c=\tilde{\lambda}+\R c.\]

Consequently $v^{-1}w\tilde{.}\tilde{\lambda}\in  \overline{\widetilde{C^v_f}}\cap W^v\tilde{.}\tilde{\lambda}=\{\tilde{\lambda}\}$  
\end{proof}

\begin{lemm}
 Let $\tilde{\lambda}\in  \ov{\widetilde{C^v_f}}\cap \widetilde\sht^\circ$, $\lambda=\pi(\tilde{\lambda})$, $\tau:[0,1]\rightarrow \A$ be a $\lambda$-path (for the action  $.$ of $W^v$ on $\A$) and $\tilde{a_0}\in \widetilde{\A}$  be  such that $\pi(\tilde a_{0})=\qt(0)$.
 Then there exists a unique $\tilde{\lambda}$-path $\tilde{\tau}:[0,1]\rightarrow\tilde{\A}$ (for the action $\tilde{.}$ of $W^v$ on $\tilde{\A}$) such that $\pi\circ\tilde{\tau}=\tau$ and $\tilde\qt(0)=\tilde a_{0}$. 
\end{lemm}

\begin{proof}
Let $n\in \N$ and $0\leq t_0<t_1<\ldots < t_n=1$  be such that $\tau$ is differentiable (with constant derivative) on $]t_i,t_{i+1}[$ for all $i\in \{0,\ldots,n-1\}$. For $i\in  \{0,\ldots,n-1\}$ and $t\in ]t_i,t_{i+1}[$, choose $w_i\in W^v$ such that $\tau'(t)=w_i.\lambda$. Let $\tilde{\tau}:[0,1]\rightarrow \tilde{\A}$ be a $\tilde{\lambda}$-path  with $\pi\circ\tilde{\tau}=\tau$.  
Maybe increasing the number of $t_i$, we may assume that $\tilde{\tau}$ is differentiable on $]t_i,t_{i+1}[$ for all $i\in \{0,\ldots,n-1\}$. Let $i\in \{0,\ldots,n-1\}$ and $t\in ]t_i,t_{i+1}[$. Then $\pi(\tilde{\tau}'(t))=w_i.\lambda$. By Lemma~\ref{lem_Act_tldW} we deduce that $\tilde{\tau}'(t)=w_i.\tilde{\lambda}$. So $\tilde{\tau}(t)-\tilde{\tau}(0)$ is well-determined by $\tau$ for every $t\in [0,1]$, which proves the desired uniqueness.

For the existence, it suffices to set $\tilde{\tau}(t)=\tilde{a_0}+\int_0^t \tilde{\tau}'$, for $t\in [0,1]$. 
\end{proof}

Let $g\in G$ and $\varphi:[0,1]\rightarrow \A$ be a parametrization of a preordered segment of $\A$. We assume moreover that $g.\varphi(t)\in \she$ for all $t\in [0,1]$. Let $\tilde{g}\in \pi^{-1}(\{g\})$. Then from what we proved,  for every $t\in [0,1]$, $\tilde{g}.\varphi(t)\in \widetilde{\she}$ and $\pi\left(\widetilde{\rho}_{\widetilde{C_\infty}}(\tilde{g}.\varphi)\right)=\rho_{C_\infty}(g.\varphi)$, and we can recover $\pi\left(\widetilde{\rho}_{\widetilde{C_\infty}}(\tilde{g}.\varphi)\right)$ from $\widetilde{\rho}_{\widetilde{C_\infty}}(\tilde{g}.\varphi)$.
}

\medskip

 Let $\widetilde{I_\infty}$ be the fixator of $\widetilde{C_\infty}$ in $\widetilde{G}_{twin}$, $\widetilde{N}=\pi^{-1}(N)$. Let $\widetilde{g}\in \widetilde{I_\infty} \widetilde{N} \widetilde{K}$ and $g=\pi(\widetilde{g})$. Then by Lemma~\ref{lemQuotient_masure}, $g\in I_\infty N K$. Therefore, in order to prove that $\widetilde{I_\infty}\widetilde{N} \widetilde{K}\nsupseteq \widetilde{G}_{twin}$,  it suffices to prove that $I_\infty N K\nsupseteq G_{twin}$ and we now work with $G$ instead of $\widetilde{G}$.

\subsection{Reduction to a problem in $G^\loo$}
\label{6.3}

We {have} $G^{\loo}=\mathrm{SL}_2\left(\k(\qp)[u,u^{-1}]\right)\rtimes \{1\}\subset G$.  We set $I_\infty^\loo=I_\infty\cap G^\loo$ and $K^\loo=K\cap G^\loo$. We denote by   $\pr^{sd}:G\rightarrow \k(\qp)^*$ the projection on the second coordinate. 
We begin by proving that we can get rid of the semi-direct product and work in $G^\loo$.   We regard $\delta$ as a linear form $\A_{twin}\rightarrow \R$. For $\g v\in \{\ominus,\oplus\}$, we denote by $\delta_{\g v}:\A_{\g v}\rightarrow \R$ the restriction of $\delta$ to $\A_{\g v}$. As $\delta_\g v(\aleph^\vee)=0$, $\delta_\g v$ is $W\zv$-invariant.  Let $\rho_{+\infty,\g v}:\shi_\g v\rightarrow\A_{\g v}$ be the retraction with respect to the sector germ  $C^v_f$. We extend $\delta_\g v$ to $\shi_{\g v}$ by setting $\delta_\g v(x)=\delta_\g v(\rho_{+\infty,\g v}(x))$, for $x\in \shi_\g v$.   Actually by \cite[Proposition 8.3.2 2)]{Heb18e}, we have $\delta_\g v=\delta_\g v\circ \rho$, for any retraction $\rho:\shi_\g v\rightarrow \A_{\g v}$ centred at a sector germ.

Recall from \ref{3.1} that  \[U_{C_\infty}=\langle x_\alpha(y)\mid\alpha\in \Phi, y\in \k(\qp), x_\alpha(y)\in G_{C_\infty}\rangle\] and \[T_{0,\ominus}=\mathfrak{T}(\{y\in \k(\qp)\mid\omega_{\ominus}(y)=0\})=\{(\begin{psmallmatrix}
y & 0\\ 0 & y^{-1}
\end{psmallmatrix},z)\mid y,z\in \k(\qp)^*, \omega_{\ominus}(y)=\omega_{\ominus}(z)=0\}.\]

\begin{lemm}\label{lemEffect_semi_direct_delta}
Let $g\in G$, $\g v\in \{\ominus,\oplus\}$ and $x\in \SHI_{\g v}$. Write $g=(g^\loo,g^ {sd})$ with $g^ {sd}\in\shk^*$.
Then $\delta_{\g v}(g.x)=\delta_{\g v}(x)+\qo_{\g v}(g^ {sd})$.
\end{lemm}

\begin{proof}
Suppose that $\g v=\oplus$. By the Iwasawa decomposition  (\cite[Proposition 4.7]{R16}) we can write $g=v_1t_1k$, with $v_1\in U^+$, $t_1\in T$ and $k\in K$. By \cite[Proposition 4.14]{R16} applied with the point $0_\oplus$ we can write $k=v_+v_-n$, where $v_+\in U^+$, $v_-\in U^-$ and $n\in N\cap K$.  Write $x=v_2.y$, with $v_2\in U^+$. Then $\delta_\oplus(g.x)=\delta(v_1t_1v_+v_-nv_2.y)$. As $T$ normalizes $U^+$ and $U^-$ , we have $\delta_{\oplus}(g.x)=\delta_{\oplus}(v_1v_+'v_-'t_1nv_2.y)$, for some $v_+'\in U^+$ and $v_-'\in U^-$. By \cite[Proposition 8.3.2 2)]{Heb18e}, we deduce that $\delta_\oplus(g.x)=\delta_{\oplus}(t_1nv_2.y)$. As $t_1nv_2(t_1n)^{-1}$ fixes the sector germ $t_1n.(+\infty)$, \cite[Proposition 8.3.2 2)]{Heb18e} implies that \[\delta_{\oplus}(g.x)=\delta_{\oplus}(t_1n.y).\]

We have $g=v_1t_1v_+v_-n$ and thus $\pr^{sd}(g)=g^ {sd}=\pr^{sd}(t_1)\pr^{sd}(n)$. As $n\in N\cap K$, we have $\omega_\oplus(\pr^{sd}(n))=0$. Therefore $\ell:=\omega_{\oplus}(\pr^{sd}(g))=\omega_\oplus(\pr^{sd}(t_1n))$. Therefore $\delta_{\oplus}(t_1n.y)=\delta_{\oplus}(y)+\ell=\delta_{\oplus}(x)+\ell$, which proves the lemma when $\g v=\oplus$. The case where $\g v=\ominus$ is similar.
\end{proof}

\begin{remas*}
(1) From the Lemma above we deduce that if $\g v\in \{\ominus,\oplus\}$, then the masure $\shi^\loo_{\g v}$ of $G^\loo$ is actually $\{x\in \shi_\g v\mid\delta_{\g v}(x)=0\}$.

\par (2) Suppose $\g v$ is any place of $\shk$ and write $g=(g^\loo,g^ {sd})\in \mathrm{SL}_{2}(\shk[u,u^ {-1}])\rtimes \shk^*=G$.
Let $\qd_{\g v}$ be the map $\SHI_{\g v}\to\R$ whose restriction on the canonical apartment $\A_{\g v}$ is $\qd:Y\otimes_{\Z}\R\to\R$ as in \ref{6.3}.
Then the above Lemma may be generalized easily to get $\delta_{\g v}(g.x)=\delta_{\g v}(x)+\qo_{\g v}(g^ {sd})$.

\end{remas*}

\begin{lemm}\label{6.4b} The Laurent polynomial versions of $G^\loo$ and $G$ are $G^\loo_{pol}=\mathrm{SL}_{2}(\sho[u,u^ {-1}])$ and $G_{pol}=\mathrm{SL}_{2}(\sho[u,u^ {-1}])\rtimes \sho^*$, where $\sho=\k[\qp,\qp^ {-1}]$, hence $\sho^*=\sqcup_{j\in\Z}\k\qp^j$.
\end{lemm}

\begin{proof} For any place $\g v$, we know from \cite[4.12.3.b]{R16}, that $\set{g\in \mathrm{SL}_{2}(\shk[u,u^ {-1}]) \mid g.0_{\g v}=0_{\g v}}$ is equal to $\mathrm{SL}_{2}(\sho_{\g v}[u,u^ {-1}])$.
Taking now the intersection in $\mathrm{SL}_{2}(\shk[u,u^ {-1}])$ of all these groups for $\g v\neq 0,\infty$, we get $G^\loo_{pol}=\mathrm{SL}_{2}(\sho[u,u^ {-1}])$ (see \ref{1.2}.3.a).
Now from Remark \ref{lemEffect_semi_direct_delta}.2, we see that the component in $\shk^*$ of an element in $G_{pol}$ has to be in $\sho^*$.
So we get clearly $G_{pol}=\mathrm{SL}_{2}(\sho[u,u^ {-1}])\rtimes \sho^*$.
\end{proof}

\begin{rema}\label{6.4c} Comparison of $G_{twin}$ and $G_{pol}$:

\par (1) Inside $G^\loo_{pol}$ (\resp $G_{pol}$) the twin group $G^\loo_{twin}$ (\resp $G_{twin}$) is generated by the diagonal and upper or lower triangular matrices in $\mathrm{SL}_{2}(\sho[u,u^ {-1}])$ (\resp and by $\sho^*$).
So the problem of the equality of $G^\loo_{pol}$ with $G^\loo_{twin}$ (\resp $G_{pol}$ with $G_{twin}$) is exactly equivalent to the problem of the generation of $\mathrm{SL}_{2}(\k[\qp,\qp^ {-1},u,u^ {-1}])$ by its elementary matrices.
Unfortunately, in \cite[\S 2 p. 228]{Cos88}, the author tells that he knows no answer for this problem (while many closely related cases are known).

\par (2) We may also look more generally to affine $\mathrm{SL}_{n}$ over $\shk=\k(\qp)$, \ie replace above $\mathrm{SL}_{2}$ by $\mathrm{SL}_{n}$ for $n\geq3$.
One gets easily that, as above for $\mathrm{SL}_{2}$, $G_{twin}=G^ {loop}_{twin}\rtimes \sho^*$ and $G_{pol}=G^ {loop}_{pol}\rtimes \sho^*$
Moreover $G^ {loop}_{twin}$ is the subgroup of $\mathrm{SL}_{n}(\k(\qp)[u,u^ {-1}])$ generated by its unipotent elementary matrices with coefficients in $\sho[u,u^ {-1}]$; it is a subgroup of $\mathrm{SL}_{n}(\k[\qp,\qp^ {-1},u,u^ {-1}])$.

\par Now, for any place $\g v$, $\sho_{\g v}$ is a discrete valuation ring (in particular a local ring); so, following \cite[page 14]{Coh66}, $\sho_{\g v}$ is a $GE-$ring: $\mathrm{SL}_{n}(\sho_{\g v})$ is generated by its unipotent elementary matrices.
Following \cite[page 223]{Sus77}, $SK_{1}(\sho_{\g v})=\set0$.
And from [\lc Cor. 7.10], $\mathrm{SL}_{n}(\sho_{\g v}[u,u^ {-1}])$ is generated by its unipotent elementary matrices, for $n\geq3$ (as $\sho_{\g v}$ is of dimension $1$).
We have got what is needed to generalize \cite[4.12.3.b]{R16} from $\mathrm{SL}_{2}$ to $\mathrm{SL}_{n}$.
So $\mathrm{SL}_{n}(\sho_{\g v}[u,u^ {-1}])$ is the group of elements $g\in \mathrm{SL}_{n}(\shk[u,u^ {-1}])$ fixing the origin $0_{\g v}$ of the masure $\SHI_{\g v}$ of $\mathrm{SL}_{n}(\shk[u,u^ {-1}])$ associated to the valuation $\qo_{\g v}$.

\par Taking now the intersection in $\mathrm{SL}_{n}(\shk[u,u^ {-1}])$ of all these groups for $\g v\neq 0,\infty$, we get $G^\loo_{pol}=\mathrm{SL}_{n}(\sho[u,u^ {-1}])$ (as above in Lemma \ref{6.4b}).
But Corollary 7.11 of \cite{Sus77} tells that $\mathrm{SL}_{n}(\k[\qp,\qp^ {-1},u,u^ {-1}])$ is generated by its elementary unipotent matrices for $n\geq3$.
So $G^ {loop}_{pol}=G^ {loop}_{twin}$ and $G_{pol}=G_{twin}$.

\end{rema}

\begin{lemm}\label{lemCompatibility_decomposition_Gloo}
Let $g\in I_\infty N K\cap G^\loo$. Then $g\in I_\infty^\loo N^\loo K^\loo$, where $N^\loo=N\cap G^\loo$. 
\end{lemm}

\begin{proof}
Let $G_{C_\infty}$ be the fixator of $C_{\infty}$ in $G$. We have $I_\infty=G_{C_\infty}\cap G_{twin}$ and  by Proposition~\ref{3.2}, $G_{C_\infty}=U_{C_\infty}.T_{0,\ominus}$,

Write $g=v t_0 n k$, where $v\in U_{C_\infty}$,$t_0\in T_{0,\ominus}$, $n\in N$ and $k\in K$. Write $k=(k_1,k_2)$, with $k_2\in \k(\qp)^*$. Then by Lemma~\ref{lemEffect_semi_direct_delta}, we have $\omega_{\oplus}(k_2)=0$ and hence $(1,k_2)\in K$. By \eqref{eqDef_semi_direct_product} we deduce that $(1,k_2^{-1}).k\in K^\loo$. We have \[g=v. t_0 n (1,k_2) .(1,k_2^{-1})k\in I_\infty^\loo N K^\loo\cap G^\loo.\] As $\pr^{sd}$   is a group morphism, we deduce $t_0 n(1,k_2)\in N^\loo$, which proves the lemma. 
\end{proof}

\subsection{Towards a counter-example in $G^\loo$}

We now prove that $ I_\infty^\loo N^\loo K^\loo\neq G^\loo\cap G_{twin}$. We now identify $G^\loo$ with $\mathrm{SL}_2\left(\k(\qp)[u,u^{-1}]\right)$.

We begin by describing $I_\infty^\loo$ (or more precisely a group containing it). After that, we  regard $G^\loo$ as a subgroup of $\overline{G^\loo}=\mathrm{SL}_2\left(\k(\qp)(\!(u^{-1})\!)\right)$, and define ``completions'' $\overline{K^\loo}$ and $\overline{I_\infty}$ of $K^\loo$ and $I_\infty$ in  $\overline{G^\loo}$. We then define an element $g\in G^\loo\cap G_{twin}$, that admits a decomposition $g=\overline{i}\overline{k}$, with $(\overline{i},\overline{k})\in \overline{I_\infty}\setminus I_\infty\times \overline{K^\loo}\setminus K^\loo$, and  by a uniqueness property for these decompositions, we deduce that $g\notin I_\infty N^\loo K^\loo$.

Recall that $\sho_{\oplus}=\{y\in \k(\qp)\mid\omega_{\oplus}(y)\geq 0\}$. 

\begin{lemm}\label{lemDescriptionK}
We have $K= \mathrm{SL}_{2}(\sho_{\oplus}[u,u^{-1}])\rtimes \sho_{\oplus}^*$. 

\end{lemm}

\begin{proof}
By \cite[Proposition 4.14]{R16}, we have $K=U_{0_\oplus}^{nm-}U_{0_\oplus}^+ \widehat{N}_{0_\oplus},$ where $U_{0_{\oplus}}^+=U^+\cap \langle \bigcup_{\alpha\in \Phi} \{u\in U_\alpha\mid u.0_{\oplus}=0_{\oplus}\} \rangle$, $\widehat{N}_{0_\oplus}$ is the fixator of $0_{\oplus}$ in $N$ and $U_{0_\oplus}^{nm-}$ is defined in   \ref{3.1}. 
By \cite{R16}[Example 4.12 3) b)], $U_{0_\oplus}^{nm-} \subset \mathrm{SL}_{2}(\sho_{\oplus}[u,u^{-1}]) \rtimes \{1\}$. As $\widehat N_{0_\oplus}$ and $U_{0_{\oplus}}^+$ are contained in $\mathrm{SL}_{2}(\sho_{\oplus}[u,u^{-1}])\rtimes \sho_{\oplus}^*$, we deduce that $K\subset \mathrm{SL}_{2}(\sho_{\oplus}[u,u^{-1}])\rtimes \sho_{\oplus}^*$. By \cite{R16}[Example 4.12 3) b)], we have $K^\loo=\mathrm{\mathrm{SL}}_2(\sho_{\oplus}[u,u^{-1}])$ and as $\{1\}\rtimes \sho_{\oplus}^*$ fixes $0_{\oplus}$ (it fixes $\A_{\oplus}$), we deduce that $K=\mathrm{SL}_{2}(\sho_{\oplus}[u,u^{-1}])\rtimes \sho_{\oplus}^*$.
\end{proof}

\begin{lemm}\label{lemDescription_I_infty}
 We have \[I_\infty^\loo\subset \begin{psmallmatrix}
\qp^{-1}\k[\qp^{-1}][u,u^{-1}]+\k[u^{-1}]&\  & \qp^{-1}\k[\qp^{-1}][u,u^{-1}]+u^{-1}\k[u^{-1}] \\
\qp^{-1}\k[\qp^{-1}][u,u^{-1}]+\k[u^{-1}]&\  & \qp^{-1}\k[\qp^{-1}][u,u^{-1}]+\k[u^{-1}] \end{psmallmatrix}.\]

\end{lemm}

\begin{proof}
Recall that $G_{C_\infty}$ is the fixator of $C_\infty$ in $G$.  Let $y\in \k(\qp)^*$ and $k\in \Z$. If $k\geq 0$, then $x_{\aleph+k\delta}(y)\in G_{C_\infty}$ if and only if $\omega_{\ominus}(y)> 0$ and if $k<0$, then $x_{\aleph+k\delta}(y)\in G_{C_\infty}$ if and only if $\omega_{\ominus}(y)\geq 0$. Indeed, the fixed point set of $x_{\aleph+k\delta}(y)$ is $D:=\{a\in \A_{\ominus}\mid\aleph(a)+k\delta(a)+\omega_{\ominus}(y)\geq 0\}$. \begin{itemize}
\item[$\bullet$] If $\omega_{\ominus}(y)>0$, then $D$ contains a neighborhood of $0_{\ominus}$ in $\A_{\ominus}$ and thus $D$ contains $C_\infty$. 

\item[$\bullet$] If $C_\infty\subset D$, then $0_{\ominus}\in D$ and thus $\omega_{\ominus}(y)\geq 0$.

\item[$\bullet$]  Assume that  $k\geq 0$ and that $C_\infty\subset D$. Let $\Omega$ be a neighborhood of $0_{\ominus}$ in $\A_{\ominus}$  such that $\Omega\cap -C^v_{f,\ominus}$ is contained in $D$. Then for all $a\in \Omega\cap -C^v_{f,\ominus}$, we have $\omega_{\ominus}(y)\geq (-\aleph(a)-k\delta(a))>0$ and thus $\omega_{\ominus}(y)>0$.

\item[$\bullet$] Assume that $k<0$. As $\{\aleph,\delta-\aleph\}$ is a basis of $\Phi^+$, we have that $(\aleph-\delta)(C_\infty)>0$, and thus $(\aleph+k\delta)(C_\infty)>0$. Therefore if $\omega_{\ominus}(y)=0$, then $x_{\aleph+k\delta}(y)\in G_{C_\infty}$. 
\end{itemize}     

Similarly, if $k>0$, then $x_{-\aleph+k\delta}(y)\in G_{C_\infty}$ if and only if $\omega_{\ominus}(y)> 0$ and if $k\leq 0$, then $x_{-\aleph+k\delta}(y)\in G_{C_\infty}$ if and only if $\omega_{\ominus}(y)\geq 0$.

By Proposition~\ref{3.2}, we have $G_{C_\infty}=U_{C_\infty}.T_{0,\ominus}$.

Take  $v\in U_{C_\infty}$ and  write it  $v=(\begin{psmallmatrix}
a_{1,1} & a_{1,2}\\ a_{2,1} & a_{2,2}
\end{psmallmatrix},1)$, with $a_{1,1},a_{1,2},a_{2,1},a_{2,2}\in \k(\qp)[u,u^{-1}]$. Take $t\in T_{0,\ominus}$ and write it $(\begin{psmallmatrix}
y & 0 \\ 0 & y^{-1}
\end{psmallmatrix},z)$, with $y,z\in \k(\qp)^*$ such that $\omega_{\ominus}(y)=\omega_{\ominus}(z)=0$. Then $vt=(\begin{psmallmatrix}
a_{1,1}y & a_{1,2}y^{-1}\\ a_{2,1}y & a_{2,2}y^{-1}
\end{psmallmatrix},z)$. Let $i,j\in \{1,2\}$. By the first part of the proof, we can write \[a_{i,j}=\sum_{k\leq -1,\ell\in \Z} \qp^{-k} f_{k,\ell}(\qp) u^\ell+\sum_{\ell\in \N} f_{0,\ell}(\qp)u^{-\ell},\] where $f_{k,l}(\qp)\in \k(\qp)$ satisfies $\omega_{\ominus}(f_{k,\ell}(\qp))=0$ for all $k,\ell$, with $f_{0,0}(\qp)=0$ if $(i,j)=(1,2)$. Lemma follows by intersecting $G_{C_\infty}$, $G_{twin}$ and $G^\loo$. 

\end{proof}

\subsection{Calculations in a completion}\label{6.9b}%%%%%%%%%%%%

Let  $\overline{G^\loo}=\mathrm{SL}_2\left(\k(\qp)(\!(u^{-1})\!)\right)\supset G^\loo$. 
By \cite[4.12.3.b]{R16} this group is the negative Mathieu completion $(G^\loo)^ {nma}$ of $G^\loo
$ (\cf \ref{2.1}, \ref{3.1}.2).

 Let  \[\overline{K^\loo}=\begin{psmallmatrix}
\sho_{\oplus}(\!(u^{-1})\!) & \ & \sho_{\oplus}(\!(u^{-1})\!)\\
  \sho_{\oplus}(\!(u^{-1})\!) & \ &  \sho_{\oplus}(\!(u^{-1})\!)
\end{psmallmatrix}\cap \overline{G^\loo}\] 
and \[\overline{I_\infty^\loo}=\begin{psmallmatrix}
\qp^{-1} \k [\qp^{-1}](\!(u^{-1})\!)+\k[\![u^{-1}]\!]&\ &  \qp^{-1} \k [\qp^{-1}](\!(u^{-1})\!)+u^{-1}\k[\![u^{-1}]\!]\\
\qp^{-1} \k [\qp^{-1}](\!(u^{-1})\!)+\k[\![u^{-1}]\!]&\ & \qp^{-1} \k [\qp^{-1}](\!(u^{-1})\!)+\k[\![u^{-1}]\!]
\end{psmallmatrix}\cap \overline{G^\loo} .\]

By Lemmas~\ref{lemDescriptionK} and \ref{lemDescription_I_infty},  $K^\loo\subset \overline{K^\loo}$, $I_{\infty}^\loo\subset \overline{I_{\infty}^\loo}$ and  $\overline{I_\infty^\loo}\cap \overline{K^\loo}=\begin{psmallmatrix}
\k[\![u^{-1}]\!] &  u^{-1}\k[\![u^{-1}]\!]\\
\k[\![u^{-1}]\!]&	\k[\![u^{-1}]\!]
\end{psmallmatrix}\cap \mathrm{SL}_{2}(\k[\![u^ {-1}]])$.

\begin{lemm*} The subgroup $U_{C_{\infty}}^ {ma-}$ of $\ov{G^\loo}$ introduced in \ref{3.1}.2 is the intersection $H$ of $\mathrm{SL}_{2}(\sho_{\ominus}[\![u^ {-1}]\!])$ with 
$ \begin{psmallmatrix}
1+u^ {-1}\sho_{\ominus}[\![u^{-1}[\!] & \ & u^ {-1}\sho_{\ominus}[\![u^{-1}[\!] \\
 \sho_{\ominus}[\![u^{-1}[\!]  & \ &  1+u^ {-1}\sho_{\ominus}[\![u^{-1}[\!] 
\end{psmallmatrix}$ .
Its intersection with $G^\loo$ (\resp $G^\loo_{twin} =G^\loo\cap G_{twin}$) is $U_{C_{\infty}}^ {pm-}$ (\resp is in $I_{\infty}^\loo$).
Its intersection with $G^\loo_{pol}$ is the intersection of $\mathrm{SL}_{2}(\k[\qp^ {-1},u^ {-1}])$ with 
$\begin{psmallmatrix}
1+u^ {-1}\k[\qp^ {-1},u^ {-1}] & \ & u^ {-1}\k[\qp^ {-1},u^ {-1}] \\
 \k[\qp^ {-1},u^ {-1}]  & \ &  1+u^ {-1}\k[\qp^ {-1},u^ {-1}] 
\end{psmallmatrix}  $.
\end{lemm*}

\begin{NB} $U_{C_{\infty}}^ {ma-}$ is not in $\overline{I_{\infty}^\loo}$.
One should replace $\qp^ {-1}\k[\qp^ {-1}]$ by $\set{x\in \k(\qp) \mid \qo_{\ominus}(x)>0}$ in the definition of this last group to get such an inclusion.

\end{NB}
\begin{proof} An easy calculation in $\mathrm{SL}_{2}$ proves that a matrix is in $H$ if, and only if, it may be written 
$\begin{psmallmatrix}1 &  0 \\ c & 1 \end{psmallmatrix}
\begin{psmallmatrix}1+d &  0 \\ 0 & (1+d)^ {-1} \end{psmallmatrix}
\begin{psmallmatrix}1 &  b \\ 0 & 1 \end{psmallmatrix}$,
with $c\in \sho_{\ominus}[\![u^{-1}[\!]$ and $b,d \in u^ {-1}\sho_{\ominus}[\![u^{-1}[\!]$.
On the other side we saw in \ref{3.1}.2 that (taking $\g g_{\Z}=\g{sl}_{2}(\Z[u,u^ {-1}])$) the elements in $U_{C_{\infty}}^ {ma-}$ are written $\prod_{\qa\in\QD^-} X_{\qa}(\g g_{\qa,\Z}\otimes_{\Z}\sho_{\ominus})$ (as $f_{C_{\infty}}(\qa)=0$ for $\qa\in\QD^-$).
And we may choose any order on the set $\QD^-$ of negative roots.
We consider first (on the left) the roots $-\aleph-n\qd$ for $n\geq0$, then (in the middle) the imaginary roots $-n\qd$ for $n>0$ and last (on the right) the roots $\aleph-n\qd$ for $n>0$.
For $\qa=-\aleph-n\qd$, $\g g_{\qa,\Z}=\begin{psmallmatrix} 0 &  0 \\ \Z u^ {-n} & 0 \end{psmallmatrix}$, so $X_{\qa}(\g g_{\qa,\Z}\otimes_{\Z}\sho_{\ominus})=x_{\qa}(\sho_{\ominus})=\begin{psmallmatrix} 1 &  0 \\ \sho_{\ominus} u^ {-n} & 1 \end{psmallmatrix}$; hence the (commutative) product of these terms for $n\geq0$ is written $\begin{psmallmatrix}1 &  0 \\ c & 1 \end{psmallmatrix}$ with $c\in \sho_{\ominus}[\![u^{-1}[\!]$.
Similarly the (commutative) product of the $X_{\qa}(\g g_{\qa,\Z}\otimes_{\Z}\sho_{\ominus})$  for $\qa=\aleph-n\qd$ with $n>0$ is written $\begin{psmallmatrix}1 &  b \\ 0 & 1 \end{psmallmatrix}$ with $b\in u^ {-1}\sho_{\ominus}[\![u^{-1}[\!]$.

\par To get the first assertion of the Lemma, the last thing to do now is to identify the commutative products of the $X_{\qa}(\g g_{\qa,\Z}\otimes_{\Z}\sho_{\ominus})$ for $\qa=-n\qd,n>0$ with matrices $\begin{psmallmatrix}1+d &  0 \\ 0 & (1+d)^ {-1} \end{psmallmatrix}$ as above.
But a basis of $\g g_{-n\qd,\Z}$ is $h_{n}=\begin{psmallmatrix} u^ {-n} &  0 \\  0 & -u^ {-n} \end{psmallmatrix}$.
The expression $X_{-n\qd}(h_{n}\otimes\ql)$ of \ref{2.1}, \ref{3.1} is actually written $[exp](\ql h_{n})$ in \cite[2.12]{R16} and is equal to $\begin{psmallmatrix}v_{1} &  0 \\ 0 & v_{2} \end{psmallmatrix}$ with $v_{1}=1+\ql u^ {-n}+\ql^2u^ {-2n}+\ldots$ and $v_{2}=v_{1}^ {-1}=1-\ql u^ {-n}$.
Moreover such an element is in $U_{C_{\infty}}^ {ma-}$ if, and only if, $\ql\in\sho_{\ominus}$ (as $f_{C_{\infty}}(-n\qd)=0$).
Now an easy induction proves that any  element in $1+u^ {-1}\sho_{\ominus}[\![u^{-1}[\!]$ may be written as an infinite product of terms of the shape $1-\ql u^ {-n}$ with $\ql\in\sho_{\ominus}$ and $n>0$.
So we get the equality $U_{C_{\infty}}^ {ma-}=H$.

\par Now the last assertions of the Lemma are easy consequences of the definitions and \ref{6.4b}.
\end{proof}

\subsection{An element in $G^\loo_{twin}\setminus I_\infty N K$}\label{ss_count_ex}

 Let $g=x_{-\aleph}(\qp u^{-1} )x_{\aleph}(\qp^{-1} u^{-1})\in G^\loo$.   We have \begin{equation}\label{eqDecomposition_g}
 g=x_\aleph(\frac{\qp^{-1}u^{-1}}{1+u^{-2}})\begin{psmallmatrix}
 \frac{1}{1+u^{-2}} & 0\\ 0 & 1+u^{-2}
 \end{psmallmatrix} x_{-\aleph}(\frac{\qp u^{-1}}{1+u^{-2}})=\overline{i}\overline{k},
 \end{equation} where $\overline{i}=\begin{psmallmatrix}
 \frac{1}{1+u^{-2}} & \qp^{-1} u^{-1} \\ 0 & 1+u^{-2}
\end{psmallmatrix}\in \overline{I^\loo_\infty}$ and $\overline{k}=x_{-\aleph}(\frac{\qp u^{-1}}{1+u^{-2}})\in \overline{K^\loo}$.  Therefore $g\in \overline{I_\infty^\loo} \overline{K^\loo}$.

{\par Actually $g.0_{\oplus}\neq0_{\oplus}$ (as the first factor in $g$ fixes $0_{\oplus}$ and the second one does not fix it).
But $\qd_{\oplus}(g.0_{\oplus})=\qd_{\oplus}(0_{\oplus})$ by \ref{lemEffect_semi_direct_delta}.
So neither $g.0_{\oplus}\geq0_{\oplus}$ nor $g.0_{\oplus}\leq0_{\oplus}$, \ie $g\not\in G_{\oplus}^+\cup G_{\oplus}^-$.
}

\begin{lemm}\label{lemNon_belonging}
The element $g$ does not belong to $I_\infty^\loo K^\loo$. 
\end{lemm}

\begin{proof}

 Suppose $g=ik$, with $i\in I_\infty^\loo\subset \overline{I_\infty^\loo}$ and $k\in K^\loo\subset \overline{K^\loo}$. Set $h=i^{-1}\overline{i}=k\overline{k}^{-1}\in \overline{K^\loo}\cap \overline{I_\infty^\loo}$. Therefore $\overline{i}h^{-1}=i$. Write $i=\begin{psmallmatrix}
A & B \\ C & D
\end{psmallmatrix}$ and $h^{-1}=\begin{psmallmatrix}
a & b\\ c & d
\end{psmallmatrix}$, with $a,b,c,d\in \k[\![u^{-1}]\!]$ and $A,B,C,D\in \qp^{-1}\k[\qp^{-1}][u,u^{-1}]+\k[u^{-1}]$. We have \[\frac{a}{1+u^{-2}}+\qp^{-1}u^{-1}c=A, \frac{b}{1+u^{-2}}+\qp^{-1}u^{-1} d=B, (1+u^{-2}) c =C, (1+u^{-2})d= D.\]

Therefore $\tilde{a}:=\frac{a}{1+u^{-2}}\in \k[u^{-1}]$ and $\tilde{b}:=\frac{b}{1+u^{-2}}\in \k[u^{-1}]$.  We have $\qp^{-1}u^{-1}c\in \k[\qp^{-1}][u,u^{-1}]$ and thus $c\in \k[u,u^{-1}]$. Moreover $(1+u^{-2})c\in \k[u^{-1}]$ and thus $c\in \k[u^{-1}]$. Similarly, $d\in \k[u^{-1}]$. As $\det(i)=1$, we have $ad-bc=(1+u^{-2})(\tilde{a}d-\tilde{b}c)=1$ and thus $1+u^{-2}$ is invertible in $\k[u^{-1}]$: we reach a contradiction. Consequently $g\notin I_\infty^\loo K^\loo$.
\end{proof}

It is easy to check that $N^\loo K^\loo=T^\loo K^\loo$ and $N^\loo \overline{K^\loo}=T^\loo \overline{K^\loo}$, where $T^\loo=\{\begin{psmallmatrix} y & 0\\ 0 & y^{-1}\end{psmallmatrix}\mid y\in \k(\qp)^*\}=G^\loo\cap T$.
\begin{lemm}
Let $t,t'\in T^\loo$ be such that $\overline{I_\infty^\loo}t \overline{K^\loo}\cap \overline{I_\infty^\loo}t'\overline{K^\loo}\neq \emptyset$. Then $t\overline{K^\loo}=t'\overline{K^\loo}$.
\end{lemm}

\begin{proof}
There exists $(i,k)\in \overline{I_\infty}\times \overline{K^\loo}$ such that $itk=t'$, or equivalently, $t'^{-1}i t =k^{-1}$. Write $t=\begin{psmallmatrix}\gamma & 0\\ 0 & \gamma^{-1}\end{psmallmatrix}$ and  $t'=\begin{psmallmatrix}\gamma' & 0\\ 0 & \gamma'^{-1}\end{psmallmatrix}$, with $\gamma,\gamma'\in \k(\qp)$. Write $i =(a_{m,n})_{m,n\in \{1,2\}}$, with $a_{m,n}\in \k(\!(u^{-1})\!)[\qp^{-1}]$   and $k=(b_{m,n})_{m,n\in \{1,2\}}$, with $b_{m,n}\in \sho_{\oplus}(\!(u^{-1})\!)$, for $m,n\in \{1,2\}$. 

Suppose $a_{1,1}a_{2,2}=0$. Then $a_{1,2}a_{2,1}=-1$. Let $\tilde{a}_{1,2}\in u^{-1}\k[\![u^{-1}]\!]$ and $\tilde{a}_{2,1}\in \k[\![u^{-1}]\!]$ be the evaluations of $a_{1,2}$ and $a_{2,1}$ at $\qp^{-1}=0$. Then $\tilde{a}_{1,2}\tilde{a}_{2,1}=-1$: we reach a contradiction. Therefore $a_{1,1}a_{2,2}\neq 0$. 

 We have $a_{1,1}\gamma'^{-1}\gamma=b_{1,1}$ and $a_{2,2}(\gamma'^{-1}\gamma)^{-1}=b_{2,2}$. For $m,n\in \{1,2\}$, write $a_{m,n}=\sum_{p\leq 0} a_{m,n,p}(u) \qp^p$, where $a_{m,n,p}(u)\in \k(\!(u^{-1})\!)$, for all $m,n,p$.  Let $\ell=\omega_{\oplus}(\gamma'^{-1}\gamma)$ and set $f(\qp)=\qp^{-\ell}\gamma'^{-1}\gamma$.  Then \[a_{1,1} \gamma'^{-1}\gamma=\sum_{p\leq 0} a_{1,1,p}(u)  f(\qp)\qp^{\ell+p}\in \sho_{\oplus}(\!(u^{-1})\!)\] and thus $\ell\geq 0$. As $a_{2,2}\gamma' \gamma^{-1}\in \sho_{\oplus}(\!(u^{-1})\!)$, we also have $\ell\leq 0$. Therefore $\ell=0$.  This proves that $t'^{-1}t\in \overline{K^\loo}$.
\end{proof}

We deduce that $g\notin I_\infty N K$.  Indeed, otherwise, by Lemma~\ref{lemCompatibility_decomposition_Gloo} we could write $g=itk$, with $i\in I_\infty^\loo$, $t\in T^\loo$ and $k\in K^\loo$. Then $t\in T\cap \overline{K^\loo}\subset K^\loo$ and thus $g\in I_\infty^\loo K^\loo$, which would contradict Lemma~\ref{lemNon_belonging}. In particular, $G_{twin}\supsetneq I_\infty N K$.

\subsection{Examples of Hecke paths}

The $C_{\infty}-$Hecke paths, which are the image by the retraction $\rho_{C_\infty}$ of $C_{\infty}-$ friendly line segments have very different behaviors than the Hecke paths considered in the  references \cite{GR14}, \cite{BPGR16} or \cite{BPR19}.
We study here some examples of such $C_{\infty}-$Hecke paths in the case of affine $\mathrm{SL}_2$ .

In the context of  Lemma~\ref{lemQuotient_masure}, we consider the action of the subgroup  $G^\loo$ of $G$ on $\SHI_{\oplus}$.
We choose the parametrization of the line segment $[0\,\,-d]$ (with $\qd (-d)=-1$ and $\aleph(-d)=0$)  in $\A$ given by   $\varphi:[0,1]\to \A_{\oplus}\subset \widetilde\A_{\oplus}$ such that   $\varphi(t)=-td$ and will study  $C_{\infty}-$Hecke paths $\rho_{ C_\infty}(g.\varphi([0 ,1]))$ for some $g\in G^\loo$. 
They are the images, by the map $\pi$ of Lemma \ref{lemQuotient_masure}, of the $C_{\infty}-$Hecke paths $\widetilde\rho_{ C_\infty}(\widetilde g.\varphi([0 ,1]))$, for some $\widetilde g\in \widetilde G$ with image $g$ in $G$.
We have to prove, along the way, that these retractions $\widetilde\rho_{ C_\infty}(\widetilde g.\varphi([0 ,1]))$ and $\rho_{ C_\infty}(g.\varphi([0 ,1]))$ are well defined; for this we shall prove some Birkhoff type decompositions of some elements in $G$.

 These elements $g$ are products of terms $\begin{psmallmatrix}1\quad  &  \qp^{k-1} u^{k} \\ 0\quad  & 1 \end{psmallmatrix}=x_{\aleph+k\qd+ (k-1)\xi} (1)$ for $k\in \Z_{>0}$.
So they are in $U^+=\widetilde U^+\subset\widetilde G$ and act on $\widetilde\SHI_{\oplus}$.
 One recall that $x_{\aleph+k\qd+ (k-1)\xi} (1)$ fixes $D^+_{1-k}:=\{a\in \A_{\oplus}\mid \aleph (a)+k\qd (a)+ (k-1)\geq 0\}$ and its analog $\widetilde D^+_{1-k}$ in $\widetilde\A_{\oplus}$. 
 This half-apartment contains $C_{\oplus}$ and is limited by $M_{1-k}$ (line of equation $x=-ky+1-k$) in $\A_{\oplus}$ with cartesian system such that $x$ corresponds to $\aleph$ and $y$ to $\qd$. The matrix  $\begin{psmallmatrix}1 &  0 \\ \qp^{-k+1} u^{-k} & 1 \end{psmallmatrix}$ fixes $D^-_{1-k}:=\{a\in \A_{\oplus}\mid \aleph (a)+k\qd (a)+ (k-1)\leq 0\}$.
 
 Moreover ($\S$\ref{1.9}) $x_{\aleph+k\qd+ (k-1)\xi} (1)x_{-\aleph-k\qd- (k-1)\xi} (-1)x_{\aleph+k\qd+ (k-1)\xi} (1)=\begin{psmallmatrix}0\quad  &  \qp^{k-1} u^{k} \\ -\qp^{1-k} u^{-k}\quad  & 0 \end{psmallmatrix}$ stabilizes $\A_{\oplus}$ and its class in $W$ is the reflexion $R_{k-1}$ fixing $M_{1-k}$.  
 We denote $t_{k}:={{k-1}\over k}  \in [0, 1]$, so that $\varphi(t_{k})=(0, -t_{k})\in M_{1-k}$. 

 In order to write decompositions of the elements $g$ (written as a product) with a left term in $I_\infty^\loo$,  
 we use the two following  formulas  in  $\mathrm{SL}_{2}$, 
 
 $\begin{pmatrix}1 &  a \\ 0 & 1 \end{pmatrix}= \begin{pmatrix}1 &  0 \\ a^{-1} & 1 \end{pmatrix} \begin{pmatrix}0 &  a \\ -a^{-1} & 0 \end{pmatrix}  \begin{pmatrix}1 &  0 \\ a^{-1} & 1 \end{pmatrix}    $ and 

 $\begin{pmatrix}1 &  a \\ 0 & 1 \end{pmatrix}
\begin{pmatrix}1 &  0 \\ b & 1 \end{pmatrix}=
\begin{pmatrix}1 &  0 \\ b(1+ab)^{-1} & 1 \end{pmatrix}\begin{pmatrix}1+ab &  0 \\ 0 &(1+ab)^{-1} \end{pmatrix}\begin{pmatrix}1 &  a(1+ab)^{-1} \\ 0 & 1 \end{pmatrix}.$

\medskip 
\noindent  EXAMPLE 1: 
 
For $N\geq 1$,  we consider $\widetilde g_{N}=g_{N}=\displaystyle{\prod_{k=1}^N}\begin{psmallmatrix}1\quad  &  \qp^{-1} (\qp u)^{3k} \\ 0 \quad & 1 \end{psmallmatrix}$ and want to study the $C_{\infty}-$Hecke paths $\rho_{ C_\infty}(g_{N}.\varphi([0 ,1]))$.

In Figure~\ref{f}, we represent $\rho_{ C_\infty}(g_{2}.\varphi([0,1]))$ in blue  and $\rho_{ C_\infty}(g_{3}.\varphi([0,1]))$ (blue and red).

\begin{figure}[h]
% le [h] sert ? positionner la figure ? l'endroit d'insertion, sinon [t] pour haut de page, [b] pour bas de page et [p] pour page s?par?e
   \begin{center}
\caption{$C_\infty$-Hecke path}\label{f}
  \includegraphics[width=16cm]{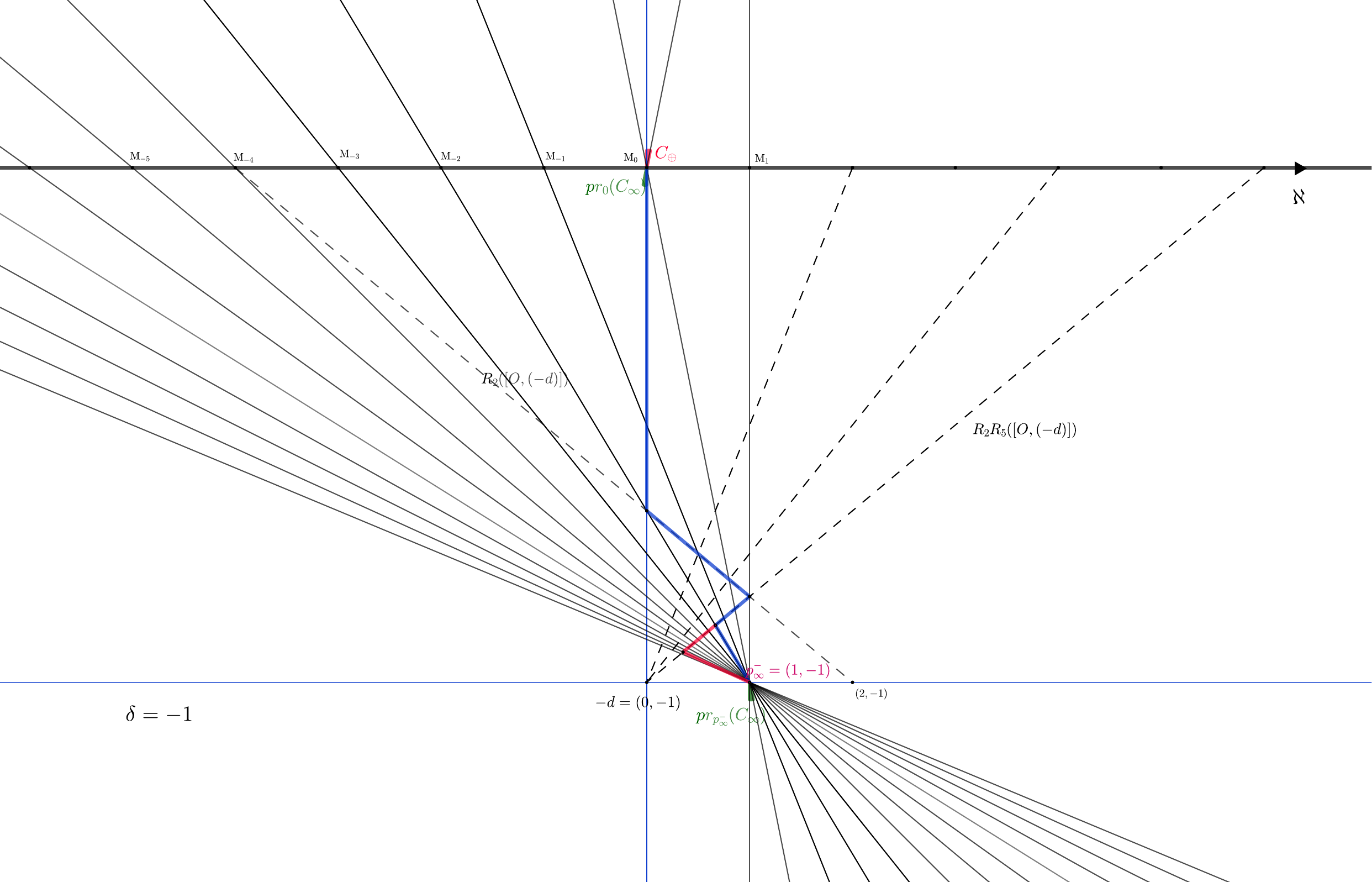}
  
  %\label{fig1.2.2}}
   \end{center}
\end{figure}

For $N=2$, we give details of the study. 

The element $g_{2}=\begin{psmallmatrix}1 & \qp^{2} u^{3}\\ 0 & 1 \end{psmallmatrix}\begin{psmallmatrix}1 &  \qp^{5} u^{6} \\ 0 & 1 \end{psmallmatrix}$ fixes $\varphi(t)$ for $ t\in [0,t_{3}]$, so, for such a $t$,   $\rho_{ C_\infty}(g_{2}.\varphi(t))$ is well defined and equal to $\varphi(t)$. 
For $ t\in [t_{3},t_{6}]$, we use \[\begin{psmallmatrix}1 & \qp^{2} u^{3}\\ 0 & 1 \end{psmallmatrix}=\begin{psmallmatrix}1 & 0 \\  \qp^{-2} u^{-3} & 1 \end{psmallmatrix}\begin{psmallmatrix}0&  \qp^{2} u^{3} \\ - \qp^{-2} u^{-3} & 0 \end{psmallmatrix}\begin{psmallmatrix}1 & 0 \\  \qp^{-2} u^{-3} & 1 \end{psmallmatrix},\] then as $\begin{psmallmatrix}1 & 0 \\  \qp^{-2} u^{-3} & 1 \end{psmallmatrix}$ and $\begin{psmallmatrix}1 &  \qp^{5} u^{6} \\ 0 & 1 \end{psmallmatrix}$  fix $\varphi(t)$, \[\rho_{ C_\infty}(g_{2}.\varphi(t))=\rho_{ C_\infty}(\begin{psmallmatrix}1 & 0 \\  \qp^{-2} u^{-3} & 1 \end{psmallmatrix}\begin{psmallmatrix}0 &  \qp^{2} u^{3} \\ - \qp^{-2} u^{-3} &0 \end{psmallmatrix}\varphi(t))=\rho_{ C_\infty}(\begin{psmallmatrix}0 &  \qp^{2} u^{3} \\ - \qp^{-2} u^{-3} & 0 \end{psmallmatrix}\varphi(t))\] (if it exists), because $\begin{psmallmatrix}1 & 0 \\  \qp^{-2} u^{-3} & 1 \end{psmallmatrix}\in I_{\infty}^\loo$ ($\S$\ref{6.9b}). So $\rho_{ C_\infty}(g_{2}.\varphi(t))$ is well defined and equal to $R_{2}\varphi(t)$  for $t\in [t_{3}, t_{6}]$ .

For $t\geq t_{6}$, we can write,  successively using the two formulas above, 

 $g_{2}=\begin{psmallmatrix}1 & \qp^{2} u^{3}\\ 0 & 1 \end{psmallmatrix}\begin{psmallmatrix}1 &  \qp^{5} u^{6} \\ 0 & 1 \end{psmallmatrix}=\begin{psmallmatrix}1 & \qp^{2} u^{3}\\ 0 & 1 \end{psmallmatrix}\begin{psmallmatrix}1 & 0 \\  \qp^{-5} u^{-6} & 1 \end{psmallmatrix}\begin{psmallmatrix}0 &  \qp^{5} u^{6} \\ - \qp^{-5} u^{-6} & 0 \end{psmallmatrix}\begin{psmallmatrix}1 & 0 \\  \qp^{-5} u^{-6} & 1 \end{psmallmatrix}$

$=
\begin{psmallmatrix}1 &0\\ {\qp^{-5} u^{-6}\over {1+\qp^{-3} u^{-3}}} & 1 \end{psmallmatrix}\begin{psmallmatrix}{1+\qp^{-3} u^{-3}} & 0 \\ 0 & {1\over {1+\qp^{-3} u^{-3}}} \end{psmallmatrix}\begin{psmallmatrix}1 &{\qp^{2} u^{3}\over {1+\qp^{-3} u^{-3}}} \\ 0& 1 \end{psmallmatrix}
\begin{psmallmatrix}0 &  \qp^{5} u^{6} \\ - \qp^{-5} u^{-6} & 0 \end{psmallmatrix}\begin{psmallmatrix}1 & 0 \\  \qp^{-5} u^{-6} & 1 \end{psmallmatrix}$. 

Using ${1\over {1+\qp^{-3} u^{-3}}}  =1+\sum_{k\geq1}(-1)^k\qp^{-3k} u^{-3k}$, we find the existence of a matrix $A\in \mathrm{SL}_{2}(\sho_{\ominus}[\![u^ {-1}]\!])\cap  \begin{psmallmatrix}
1+u^ {-1}\sho_{\ominus}[\![u^{-1}[\!] & \ & u^ {-1}\sho_{\ominus}[\![u^{-1}[\!] \\
 \sho_{\ominus}[\![u^{-1}[\!]  & \ &  1+u^ {-1}\sho_{\ominus}[\![u^{-1}[\!] 
\end{psmallmatrix}$  such that \[g_{2}=A\begin{psmallmatrix}1 &-\qp^{-1} \\ 0& 1 \end{psmallmatrix} \begin{psmallmatrix}1 &\qp^{2} u^{3} \\ 0& 1 \end{psmallmatrix}
\begin{psmallmatrix}0 &  \qp^{5} u^{6} \\ - \qp^{-5} u^{-6} &0 \end{psmallmatrix}\begin{psmallmatrix}1 & 0 \\  \qp^{-5} u^{-6} & 1 \end{psmallmatrix}.\]

By the lemma of  \S\ref{6.9b},   $A\in U_{C_{\infty}}^ {ma-}\subset\ov{G^\loo}$ and more precisely, as $g_{2}$ and the other matrices are in $G^\loo\cap G_{twin}$ so is A, and  we have $A\in I_{\infty}^\loo$. Moreover $\begin{psmallmatrix}1 &\qp^{-1} \\ 0& 1 \end{psmallmatrix}=x_{\aleph-\xi} (1) $ fixes $C_{\infty}\subset\{a\in \A_{{\ominus}}\mid\aleph (a)+1\geq 0\}$, so $A\begin{psmallmatrix}1 &-\qp^{-1} \\ 0& 1 \end{psmallmatrix}\in I_{\infty}^\loo$.

For $t\geq t_{6}$, we obtain $\rho_{ C_\infty}(g_{2}.\varphi(t))=\rho_{ C_\infty}(\begin{psmallmatrix}1 &\qp^{2} u^{3} \\ 0& 1 \end{psmallmatrix}
\begin{psmallmatrix}0 &  \qp^{5} u^{6} \\ - \qp^{-5} u^{-6} & 0 \end{psmallmatrix}\begin{psmallmatrix}1 & 0 \\  \qp^{-5} u^{-6} & 1 \end{psmallmatrix}\varphi(t))$ (if it exists). 
But, we know that $ \begin{psmallmatrix}1 & 0 \\  \qp^{-5} u^{-6} & 1 \end{psmallmatrix}$   fixes $\varphi(t) $ and $\begin{psmallmatrix}0 &  \qp^{5} u^{6} \\ - \qp^{-5} u^{-6} & 0 \end{psmallmatrix}$ acts by $R_{5}$  on it.

As  $\begin{psmallmatrix}1 &\qp^{2} u^{3} \\ 0& 1 \end{psmallmatrix}$ fixes $D^+_{-2}$,  for $t\geq t_{6}$, this matrix acts on $R_{5}(\varphi(t))$ if and only if $t<t_{9}$ (as $R_{5}(\varphi(t))\in D^+_{-2}\iff \varphi(t)\in D^-_{-8}$)

So, for $t\in[t_{6}, t_{9}]$ by the same argument as in $[t_{3}, t_{6}]$, $\rho_{ C_\infty}(g_{2}.\varphi(t))$  is well defined and equal to $R_{2}R_{5}(\varphi(t))$.
Moreover for $t\in [t_{9}, 1]$, $\rho_{ C_\infty}(g_{2}.\varphi(t))= R_{5}(\varphi(t))$.
We see that the Hecke path has exactly 3 folding points $p_{3}=\rho_{ C_\infty}(g_{2}.\varphi(t_{3}))=\qf(t_{3})$, $p_{6}=\rho_{ C_\infty}(g_{2}.\varphi(t_{6}))=R_{2}\qf(t_{6})$,  $p_{9}=\rho_{ C_\infty}(g_{2}.\varphi(t_{9}))=R_{2}R_{5}\qf(t_{9})$, with the line segment  $[p_{6}\,\,p_{9}]\subset R_{2}R_{5}(\varphi([0,1])=[(6,0) (-d)]$  and his last direction is that of  $R_{5}\varphi([t_{9},1]$.

For all $N\geq 2$,   $\rho_{C_{\infty}}(g_{N}.\varphi([0 ,1]))$ is well defined and has 3 folding points $p_{3}=\rho_{ C_\infty}(g_{2}.\varphi(t_{3}))$, $p_{6}=\rho_{ C_\infty}(g_{2}.\varphi(t_{6}))$,  $p_{3(N+1)}=\rho_{ C_\infty}(g_{N}.\varphi(t_{3(N+1)}))$, moreover $[p_{6}\,\,p_{3(N+1)}] =R_{2}R_{5}(\varphi ([t_{6}, t_{{3(N+1)})}])) $ and  is included in the line segment $[(6,0)\,(-d)]$ and the last direction  of the Hecke path is that of the segment germ  $R_{3N-1}\varphi((0,1])$.
\medskip 

This result is easily obtained by induction. As $g_{N+1}=g_{{N}}\begin{psmallmatrix}1 \quad &  \qp^{-1}(\qp u)^{3(N+1)} \\ 0\quad & 1 \end{psmallmatrix}$, for $t\leq t_{3(N+1)}$ we have $ \rho_{ C_\infty}(g_{N+1}.\varphi(t))=\rho_{ C_\infty}(g_{N}.\varphi(t))$ and so the Hecke path has the two folding points  $p_{3}=\rho_{ C_\infty}(g_{2}.\varphi(t_{3}))$, $p_{6}=\rho_{ C_\infty}(g_{2}.\varphi(t_{6}))$. We will see that we have no folding at $p_{3(N+1)}$.
For the calculus, we remark that if $u_{k}= \qp^{-1}q^k$ and $S_{N}=\sum_{k=1}^{k=N}u_{k}$, then ${S_{N}\over{1+u_{N+1}^{-1}S_{N}}}- \qp^{-1}q^N\in - \qp^{-1}+q^{-1} \k[[q^{-1}]]$ and we will use the same method as before. 

We write  $g_{N+1}=g_{N}\begin{psmallmatrix}1 & 0 \\  \qp\!(\qp u)^{-3(N+1)} & 1 \end{psmallmatrix}\begin{psmallmatrix}0 &  \qp^{-1}\!(\qp u)^{3(N+1)} \\ - \qp\!(\qp u)^{-3(N+1)} & 0 \end{psmallmatrix}\begin{psmallmatrix}1 & 0 \\  \qp\!(\qp u)^{-3(N+1)} & 1 \end{psmallmatrix}.$ For $t\geq t_{3(N+1)}$,  we know that $\begin{psmallmatrix}1 & 0 \\  \qp\!(\qp u)^{-3(N+1)} & 1 \end{psmallmatrix}$ fixes $\varphi(t)$ and $\begin{psmallmatrix}0 &  \qp^{-1}\!(\qp u)^{3(N+1)} \\ - \qp\!(\qp u)^{-3(N+1)} &0 \end{psmallmatrix}$ acts as $R_{{3N+2}}$.

We consider $q=(\qp u)^{3}$ and  $u_{k}= \qp^{-1}\!(\qp u)^{3k}=\qp^{-1}q^k $ and see that : 

$\displaystyle{g_{N}\begin{psmallmatrix}1 & 0 \\  \qp\!(\qp u)^{-3(N+1)} & 1 \end{psmallmatrix}=\begin{psmallmatrix}1 & 0 \\ {{(u_{N+1})^{-1}}\over{1+u^{-1}_{N+1}S_{N}}} & 1 \end{psmallmatrix}\begin{psmallmatrix}{1+u^{-1}_{N+1}S_{N}}& 0 \\0 &( {1+u^{-1}_{N+1}S_{N}})^{-1} \end{psmallmatrix}\begin{psmallmatrix}1\quad  & {S_{N}\over{1+u_{N+1}^{-1}S_{N}}}\\ 0\quad  & 1 \end{psmallmatrix}}$ so it can be written $A' \begin{psmallmatrix}1 &-\qp^{-1} \\ 0& 1 \end{psmallmatrix}\begin{psmallmatrix}1 &\qp^{-1}q^N\\ 0 & 1 \end{psmallmatrix}$ with $A'\in \mathrm{SL}_{2}(\sho_{\ominus}[\![u^ {-1}]\!])\cap  \begin{psmallmatrix}
1+u^ {-1}\sho_{\ominus}[\![u^{-1}[\!] & \ & u^ {-1}\sho_{\ominus}[\![u^{-1}[\!] \\
 \sho_{\ominus}[\![u^{-1}[\!]  & \ &  1+u^ {-1}\sho_{\ominus}[\![u^{-1}[\!] 
\end{psmallmatrix}$.
 As before, we can see that  $A' \begin{psmallmatrix}1 &-\qp^{-1} \\ 0& 1 \end{psmallmatrix}\in  I_{\infty}^\loo$.

By \S\ref{6.9b}, for $t\geq t_{3(N+1)}$, $ \rho_{ C_\infty}(g_{N+1}.\varphi(t))= \rho_{ C_\infty}(\begin{psmallmatrix}1 &\qp^{-1}\!(\qp u)^{3N}\\ 0 & 1 \end{psmallmatrix}R_{{3N+2}} (\varphi(t)))$ (if they exist). 

For $t$ large enough, the last direction of the Hecke path is  $R_{3N+2}\varphi((0,1])$.

More precisely $\begin{psmallmatrix}1 &\qp^{-1}\!(\qp u)^{3N}\\ 0 & 1 \end{psmallmatrix}$ acts on $R_{{3N+2}}( \varphi(t))$ iff $t\leq t_{3(N+2)}$ (because   $R_{2+3N}(D^-_{-(3N-1)})=D^+_{-(3(N+2)-1)}$). 
But, as in the first calculus for $t\in[t_{6}, t_{9}]$, we can see that, modulo $I^\loo_{\infty}$, this matrix acts by $R_{3N-1}$ and we have $R_{3N-1}R_{{3N+2}}=R_{2} R_{5}$. 
So $ \rho_{ C_\infty}(g_{N+1}.\varphi([t_{3(N+1)}, t_{3(N+2)}]))$ is well defined, is equal to $R_{2} R_{5}([\varphi(t_{3(N+1)}), \varphi(t_{3(N+2)})]) $ and is included in $[(6,0)\, (-d)]$ so  there is no more folding at $p_{3(N+1)}$ and we have the expected result.  The third folding point is  $p_{3(N+2)}=\rho_{ C_\infty}(g_{N}.\varphi(t_{3(N+2)}))$.

\medskip 
\noindent  EXAMPLE 2: 

In the second example, we want to consider a new family $(g'_{N})$, with a growing number of folding points. In the analog of previous calculus, we want that the action of the ``new term'' doesn't affect the previous folding points. 

We consider for $N\geq 0$,   $\widetilde g'_{N}=g'_{N}=\displaystyle{\prod_{k=0}^N}\begin{psmallmatrix}1 \quad &  \qp^{-1} (\qp u)^{3.2^k} \\ 0 \quad & 1 \end{psmallmatrix} \in G^\loo_{twin}$.

Let us prove that for $N\geq 1$,  $ \rho_{ C_\infty}(g'_{N}.\varphi([0 ,1]))$ is well defined, has at least $N$ folding points and there exists $t_{3.2^{N}}\leq T_{N}<t_{3.2^{N+1}}$ such that $\rho_{C_\infty}(g'_{N}.\varphi([T_{N}\,, 1]))=R_{{3.2^{N}-1}}(\varphi([T_{N}\,, 1]))$.

As  $g'_{1}=g_{2}$, we know the corresponding Hecke path and the result is true in this case (with $t_{9}= T_{1}<t_{12}$). 

We consider for $N\geq 1$,  $g'_{N+1}=g'_{N}\begin{psmallmatrix}1 &  \qp^{-1} (\qp u)^{3.2^{N+1}} \\ 0 & 1 \end{psmallmatrix}$. 

As before, if it is well defined, we have  $\rho_{C_\infty}(g_{N+1}.\varphi([0\, , t_{3.2^{N+1}}]))= \rho_{C_\infty}(g_{N}.\varphi([0\, , t_{3.2^{N+1}}]))$. We know by induction that $\rho_{C_\infty}(g_{N+1}.\varphi([T_{N} , t_{3.2^{N+1}}]))= R_{{3.2^{N}-1}}(\varphi([T_{N},   t_{3.2^{N+1}}]))$ is well defined and that this Hecke path has at least N folding points there. 
As previously,  we write  
$g'_{N+1}=g'_{N}\begin{psmallmatrix}1 &\, 0\\  \qp  (\qp u)^{-3.2^{N+1}}  & \,1 \end{psmallmatrix}\begin{psmallmatrix}0 &  \qp^{-1}  (\qp u)^{3.2^{N+1}}\\  -\qp  (\qp u)^{-3.2^{N+1}} & 0 \end{psmallmatrix}\begin{psmallmatrix}1 &\, 0\\  \qp  (\qp u)^{-3.2^{N+1}}  & \,1 \end{psmallmatrix},$ in order to study the case $t>t_{3.2^{N+1}}$. With  $D=1+\sum_{k=0}^N(\qp u)^{3.(2^k-2^{N+1})}$ and $a=\qp^{-1}\sum_{k=0}^N(\qp u)^{3.2^k}$, we
obtain 

 $g'_{N+1}=\begin{psmallmatrix}1 &\, 0\\  {\qp  (\qp u)^{-3.2^{N+1}}\over D}  & \,1 \end{psmallmatrix}\begin{psmallmatrix}D& 0\\ 0 & 1/D \end{psmallmatrix}\begin{psmallmatrix}1 & {a\over D}\\ 0 & 1 \end{psmallmatrix}\begin{psmallmatrix}0 &  \qp^{-1}  (\qp u)^{3.2^{N+1}}\\  -\qp  (\qp u)^{-3.2^{N+1}} & 0 \end{psmallmatrix}\begin{psmallmatrix}1 & 0\\  \qp  (\qp u)^{-3.2^{N+1}}  & 1 \end{psmallmatrix}$ 
 In fact ${a\over D}-a+\qp^{-1}\in \sho_{\ominus}[\![u^ {-1}]\!] $, and so (as before) modulo  $I^\loo_{\infty}$ , $g'_{N+1}$ acts on $\varphi([ t_{3.2^{N+1}}\,,  1])$  as $\begin{psmallmatrix}1 & a\\ 0 & 1 \end{psmallmatrix}\begin{psmallmatrix}0 &  \qp^{-1}  (\qp u)^{3.2^{N+1}}\\  -\qp  (\qp u)^{-3.2^{N+1}} & 0 \end{psmallmatrix}$ so as $ g'_{N}\circ R_{{3.2^{N+1}-1}}$.
 
But for $k\leq N$, $R_{{3.2^{N+1}-1}}(\varphi( t))\in M_{1-3.2^k}$ if, and only if, $\varphi( t)\in M_{1-3.(2^{N+2}-2^k)}$ (\ie $t=t_{3.(2^{N+2}-2^k)}$). So $g'_{N}$ really acts on $R_{{3.2^{N+1}-1}}(\varphi(t))$ only for some $t$ in $ [t_{3.(2^{N+1})}\, , t_{3.(2^{N+2}-2^N)}]$ and there exists $T_{N+1}$ with $t_{3.2^{N+1}}\leq T_{N+1}<t_{3.2^{N+2}} $ such  that $\rho_{C_\infty}(g'_{N+1}.\varphi([T_{N+1}\,, 1]))=R_{{3.2^{N+1}-1}}(\varphi([T_{N+1}\,, 1]))$ and, as the direction of this line segment is different from that of $ R_{{3.2^{N}-1}}(\varphi([T_{N}\,,  t_{3.2^{N+1}}]))$, there is a new folding point for this Hecke path, so at least $N+1$ folding points.

\begin{rema*} It is interesting to look at what happens in these two examples when $N$ goes to infinity.
Actually $\cup_{N=1}^\infty\,\widetilde g_{N}\qf([0,t_{3N}])$ (\resp $\cup_{N=1}^\infty\,\widetilde g'_{N}\qf([0,t_{3.2^N}])$) is an increasing union of $C_{\infty}-$friendly line segments in $\widetilde \SHI$; and the same is true for their images in $\SHI$.
So we get a half-open $C_{\infty}-$friendly line segment written (abstractly) $\widetilde g_{\infty}\qf([0,1[)$ (\resp $\widetilde g'_{\infty}\qf([0,1[)$) in $\widetilde\SHI$ and $ g_{\infty}\qf([0,1[)$ (\resp $ g'_{\infty}\qf([0,1[)$) in $\SHI$.
A question is whether they can reasonably be completed in a ``closed'' $C_{\infty}-$friendly line segment.
The answer is clearly no for example 2: this would lead to a $C_{\infty}-$Hecke path $\qr_{C_{\infty}}(\ov{g'_{\infty}\qf([0,1[)})$ with an infinite number of folding points, contrary to Definition \ref{5.13} and Proposition in \S \ref{5.14}.

\par On the contrary we can make further calculations for example 1, as $\widetilde g_{N}=g_{N}$ is associated to a geometric sequence in $\k[\qp,\qp^ {-1},u,u^ {-1}]$.
We consider the matrix $g^1_{N}=\begin{psmallmatrix}\sum_{k=0}^N(\qp u)^ {3k} \,&\, \qp^ {-1}(\qp u)^ {3N+3}\\  -\qp    & 1 -(\qp u)^3 \end{psmallmatrix} \in G^\loo_{pol}$.
So $g^1_{\infty}:=g_{N}g^1_{N}=\begin{psmallmatrix}1\, &\, \qp^2u^3 \\  -\qp \,   & \,1-(\qp u)^3 \end{psmallmatrix}=\begin{psmallmatrix}1 & 0\\  -\qp   & 1 \end{psmallmatrix}\begin{psmallmatrix}1 & \qp^2u^3\\  0  & 1 \end{psmallmatrix}$ is a fixed element in $G^\loo_{twin}$ (so $g^1_{N}\in G^\loo_{twin}$).
By the following Lemma $g^1_{N}$ fixes $\qf([0,t_{3N+3}])$.
So $ g_{\infty}\qf([0,1[)$ is actually equal to $ g^1_{\infty}\qf([0,1[)$.
We shall prove now that $ g^1_{\infty}\qf([0,1])$ is a $C_{\infty}-$friendly line segment.
The associated $C_{\infty}-$Hecke path is then clearly $[0 \,\,p_{3}]\cup[p_{3}\,\,p_{6}]\cup[p_{6}\,\,-d]$.

\par We have to find a good Birkhoff decomposition for $g^1_{\infty}$.
The details of the calculations are similar to those above and left to the reader.

\par $g^1_{\infty}=\begin{psmallmatrix}1 & 0\\  \qp^ {-2}u^ {-3}   & 1 \end{psmallmatrix}\begin{psmallmatrix}1 & 0\\  -\qp   & 1 \end{psmallmatrix} \begin{psmallmatrix}0 & \qp^2u^3\\  -\qp^ {-2}u^ {-3}   & 0 \end{psmallmatrix}\begin{psmallmatrix}1 & 0\\  \qp^ {-2}u^ {-3}   & 1 \end{psmallmatrix}$

\par\quad\;
$=\begin{psmallmatrix}1 & 0\\  \qp^ {-2}u^ {-3}   & 1 \end{psmallmatrix} \begin{psmallmatrix}1 & -\qp^ {-1}\\  0   & 1 \end{psmallmatrix} \begin{psmallmatrix}0 & \qp^2u^3\\  -\qp^ {-2}u^ {-3}   & 0 \end{psmallmatrix} \begin{psmallmatrix}0 & \qp^5u^6\\  -\qp^ {-5}u^ {-6}   & 0 \end{psmallmatrix} \begin{psmallmatrix}1 & 0\\  \qp^ {-5}u^ {-6}   & 1 \end{psmallmatrix}  \begin{psmallmatrix}1 & 0\\  \qp^ {-2}u^ {-3}   & 1 \end{psmallmatrix}$

\par Now $\begin{psmallmatrix}1 & 0\\  \qp^ {-2}u^ {-3}   & 1 \end{psmallmatrix}$ (\resp $\begin{psmallmatrix}1 & 0\\  \qp^ {-5}u^ {-6}   & 1 \end{psmallmatrix}$) fixes $\qf([t_{3},1])$ (\resp $\qf([t_{6},1])$), $\begin{psmallmatrix}0 & \qp^5u^6\\  -\qp^ {-5}u^ {-6}   & 0 \end{psmallmatrix}$ (\resp $\begin{psmallmatrix}0 & \qp^2u^3\\  -\qp^ {-2}u^ {-3}   & 0 \end{psmallmatrix} $) stabilizes $\A_{\oplus}$ and induces on it $R_{5}$ (\resp $R_{2}$); moreover $\begin{psmallmatrix}1 & 0\\  \qp^ {-2}u^ {-3}   & 1 \end{psmallmatrix}$ and $ \begin{psmallmatrix}1 & -\qp^ {-1}\\  0   & 1 \end{psmallmatrix}$ are in $I^\loo_{\infty}$.
So the last expression for $g^1_{\infty}$ is a Birkhoff decomposition, telling that the pair $\set{C_{\infty},g^1_{\infty}\qf([t_{6},1])}$ is friendly.
One can also deduce from these expressions the shape of $\qr_{C_{\infty}}(g^1_{\infty}\qf([0,1]))$.
\end{rema*}

\begin{lemm*} $g^1_{N}\in (U^ {ma+}_{\qf([0,t_{3N+3}])}U_{-\aleph,\qf([0,1])})\cap G^\loo_{twin}$ fixes $\qf([0,t_{3N+3}])$.
\end{lemm*}

\begin{proof} In $\mathrm{SL}_{2}(\k[\qp][\![u]\!])\subset \g G^ {pma}(\shk)$, one may write $g^1_{N}=\begin{psmallmatrix}1 & a \\  0   & 1 \end{psmallmatrix} \begin{psmallmatrix}c^ {-1} & 0\\  0   & c \end{psmallmatrix} \begin{psmallmatrix}1 & 0 \\  b   & 1 \end{psmallmatrix} \begin{psmallmatrix}1 & 0 \\  -\qp   & 1 \end{psmallmatrix}$, with $c=1-(\qp u)^3$, $a=\frac{\qp^ {-1}(\qp u)^ {3N+3}}{1-(\qp u)^3}=\sum_{k=0}^\infty \qp^ {3N+2+3k}u^ {3N+3+3k}$ and $b=\frac{-\qp}{1-(\qp u)^3}+\qp=\sum_{k=1}^\infty -\qp(\qp u)^ {3k}$.
Now $\begin{psmallmatrix}1 & 0 \\  -\qp   & 1 \end{psmallmatrix}=x_{-\aleph}(-\qp)\in U_{-\aleph,\qf([0,1])}$ fixes $\qf([0,1])$.
And $\begin{psmallmatrix}1 & 0 \\  b   & 1 \end{psmallmatrix}=\prod_{k=1}^\infty x_{-\aleph+3k\qd}(-\qp^ {3k+1})\in U^ {ma+}_{\qf([0,1])}$ fixes also $\qf([0,1])$, as $f_{\qf([0,1])}(-\aleph+3k\qd)=3k$ (see \ref{3.1}).
Moreover $\begin{psmallmatrix}1 & a \\  0   & 1 \end{psmallmatrix}=\prod_{k=0}^\infty x_{\aleph+(3N+3+3k)\qd}(\qp^ {3N+2+3k})\in U^ {ma+}_{\qf([0,t_{3N+3}])}$ fixes $\qf([0,t_{3N+3}])$, as $f_{\qf([0,t])}(\aleph+(3N+3+3k)\qd)=(3N+3+3k)t$.

\par The last thing is now to prove that $\begin{psmallmatrix}c^ {-1} & 0\\  0  \, & c \end{psmallmatrix}$ is in $U^ {ma+}_{\qf([0,1])}$.
We argue as in \S \ref{6.9b} or \cite[2.12]{R16}.
The matrix $h_{n}^+=\begin{psmallmatrix}u^ {n} & 0\\  0   & -u^n \end{psmallmatrix}$ is a basis of $\g g_{n\qd,\Z}$, hence $X_{n\qd}(h_{n}^+\otimes\ql)=[exp](\ql h_{n}^+)= \begin{psmallmatrix} v_{1} & 0\\  0   & v_{2} \end{psmallmatrix}$, with $v_{1}=1+\ql u^n+\ql^2u^ {2n}+\cdots$ and $v_{2}=1-\ql u^n$.
We take $n=3,\ql=\qp^3$, so $\begin{psmallmatrix}c^ {-1} & 0\\  0   & c \end{psmallmatrix}=[exp](\qp^3 h_{3}^+)$.
But $f_{\qf([0,1])}(3\qd)=\inf\set{r\in\Z \mid (3\qd)(\qf([0,1]))+r\geq0}=3$, so $\begin{psmallmatrix}c^ {-1} & 0\\  0   & c \end{psmallmatrix} \in U^ {ma+}_{\qf([0,1])}$.
\end{proof}

%%%%%%%%%%%%%%%%%%%%%%%%%%%%%%%%%%%%%%%%%%%%%%%%%%

\printindex

%\newpage
%%%%%%%%%%%%%%%%%%%%%%%%%%%%%%%%%%%%%%%%%%%%%%%
%%%%%%%%%%%%%%%%%%%%%%%%%%%%%%%%%%%%%%%%%%%%%%%%%%   \ar[u]@{^{(}->}

\medskip
\par Universit\'e de Lorraine, CNRS, IECL, F-54000 Nancy, France

\noindent Nicole.Panse@univ-lorraine.fr ; Auguste.Hebert@univ-lorraine.fr ; Guy.Rousseau@univ-lorraine.fr

\end{document}